\documentclass[11pt]{amsart}

\newif\ifdebug
\debugfalse

\newif \iffig
\figfalse

\newif \iftable

\usepackage{sfilip_package_settings}
\usepackage{sfilip_abbreviations}
\usepackage{sfilip_thm_style_long}

\newcommand{\PP}{\mathbb{P}}

\newcommand{\Flag}{\operatorname{Flag}}
\newcommand{\R}{\mathbb{R}}
\newcommand{\N}{\mathbb{N}}

\newcommand{\Id}{\operatorname{Id}}

\newcommand{\cross}{\times}

\newcommand{\Leb}{\operatorname{Leb}}
\newcommand{\tnorm}[1]{{\left\vert\kern-0.25ex\left\vert\kern-0.25ex\left\vert #1 
    \right\vert\kern-0.25ex\right\vert\kern-0.25ex\right\vert}}
\newcommand{\st}{\;\: : \;\:}         

\usepackage{xcolor}

\title[AC of invariant measures for random diffeomorphisms]{Absolute continuity of invariant measures for random diffeomorphisms}

\hypersetup{
	pdftitle={Absolute continuity of invariant measures for random diffeomorphisms},	
	pdfauthor={Aaron Brown, Alex Eskin, Simion Filip, Federico Rodriguez Hertz},	
	pdfsubject={math},		
	pdfkeywords={MSC},	
	pdfnewwindow=true,		
	linkcolor=darkblue,		
	citecolor=darkred,		
	filecolor=darkblue,		
	urlcolor=darkblue,		
	pdfborder={0 0 0},
	breaklinks=true
}

\thanks{Revised \textsc{\today} }

\date{October 2026}

\author[A.\ Brown]{
	Aaron Brown
}
\address{
	\parbox{0.5\textwidth}{
		Department of Mathematics\\
		Northwestern University\\}
}
\email{{awb@math.northwestern.edu}}

\author[A.\ Eskin]{
	Alex Eskin
}
\address{
	\parbox{0.5\textwidth}{
		Department of Mathematics\\
		University of Chicago\\
		5734 S University Ave\\
		Chicago, IL 60637\\}
}
\email{{eskin@math.uchicago.edu}}

\author[S.\ Filip]{
	Simion Filip
}
\address{
	\parbox{0.5\textwidth}{
		Department of Mathematics\\
		University of Chicago\\
		5734 S University Ave\\
		Chicago, IL 60637\\}
}
\email{{sfilip@math.uchicago.edu}}

\author[F.\ Rodriguez Hertz]{
	Federico Rodriguez Hertz
}
\address{
	\parbox{0.5\textwidth}{
		Department of Mathematics\\
		Pennsylvania State University\\
		328 McAllister Building\\
		State College, PA 16802\\}
}
\email{{fjr11@psu.edu}}
\begin{document}

\begin{abstract}
	We develop criteria showing that invariant measures for random walks generated by diffeomorphisms of manifolds are in the Lebesgue class.
	Our assumptions are first, that there is a uniform expansion gap in the Lyapunov spectrum of the random walk, and second, that the conditionals of the invariant measure along the appropriate (partial) stable manifolds are in the Lebesgue class.
\end{abstract}

%
\maketitle
%
\noindent\hrulefill
\tableofcontents
\nointerlineskip
\noindent \hrulefill
\ifdebug
   \listoffixmes
\fi


\section{Introduction}
	\label{sec:introduction}
In this work, we consider a not necessarily compact manifold $M$ with an action of a collection of diffeomorphisms in $\Diff(M)$.
Let $\mu$ be a probability measure on $\Diff(M)$ with compact support (see \autoref{sssec:random_dynamics_and_words}) and assume $\nu$ is a measure on the manifold $M$ which is invariant under every diffeomorphism in the support of $\mu$.
Our main results are a number of criteria which guarantee that $\nu$ is absolutely continuous with respect to a volume form on $M$.

There are two inputs to our main result, which is \autoref{theorem:Absolute-continuity-and-ergodicity} below.
The first is an assumption which ensures a uniform expansion gap in position $k$ in the Lyapunov spectrum of the random walk on $M$ generated by $\mu$.
This is used to show that partial stable manifolds of dimension $k$ of the random walk are smooth and have uniform geometry, at every point.
The second assumption is that the conditionals of $\nu$ along the corresponding $k$-dimensional partial stable foliation are of Lebesgue class.
This last assumption can, in some situations, be obtained using the main results of \cite{BEFRH25}.

We note that our results apply when the Lyapunov spectrum has zero exponents, and also when the manifold $M$ is not compact, but we assume the existence of a Margulis function to control the escape to infinity.

\paragraph{Uniform expansion gaps.}
Before proceeding to a more detailed description of our results, we give a brief overview of the possible assumptions on expansion gaps which lead to uniform geometry of the partial stable manifolds.

In brief, we define first for $m\leq \dim M$ the function $\alpha_m(f,L)$ where $f\in \Diff(M)$ and $L\subset TM(q)$ is an $m$-dimensional tangent plane, which measures the change in volume along $L$ under $Df$ (see \autoref{eq:def:sigma:f:L}).
Uniform expansion in dimension $m$ is the assumption that there exists $N\in \bN$ and $C>0$ such that for all $q\in M$ and all $L\subset TM(q)$ of dimension $m$, the expected value of $\alpha_m(\bullet,L)$ is at least $C$ after $N$ steps.

Similarly, one can define
\begin{align*}
	\alpha_{k}'(f,L,L')&:=\alpha_{k}(f,L)-\alpha_{k+1}(f,L')\\
	\alpha_{k}''(f,L'',L,L')&:=2\alpha_{k}(f,L)-\alpha_{k-1}(f,L'')-\alpha_{k+1}(f,L')
\end{align*}
for nested subspaces $L''\subset L\subset L'\subset TM(q)$ of dimensions $k-1,k$, and $k+1$ respectively.

We will assume, see \autoref{ssec:dynamical:assumptions}, that the expected value of $\alpha_{k}'$ is uniformly negative after $N$ steps.
Additionally, there will be two possible cases, (FS) and (PS), depending on whether we will consider the full stable manifold or a partial stable manifold of dimension $d-k$.
In case (PS), we will assume that the expected value of $\alpha_{k}''$ is uniformly negative after $N$ steps.
In case (FS), we will assume that the expected value of $\alpha_{k-1}'$ can be made arbitrarily small after sufficiently many steps.

\paragraph{(Partial) stable manifolds}
Under the assumptions outlined above, we show in \autoref{theorem:Good-stables-everywhere} that stable manifolds of dimension $d-k$ exist at \emph{every point} of $M$, and that they are smooth and have uniform geometry (for a large set of random trajectories) in \autoref{theorem:good:words}.

We note that the uniform geometry of stable manifolds allows us to show ergodicity of volume measure when $\mu$ is supported on volume-preserving diffeomorphisms, see \autoref{theorem:Absolute-continuity-and-ergodicity}, using an argument originally appearing in Zhang \cite{Zhang2019_On-stable-transitivity-of-finitely-generated-groups-of-volume-preserving}.

\paragraph{Relation to other works}
Closely related results, under related assumptions, were obtained by B\'enard and Zhang \cite{BenardZhang2026_Rigidity-and-equidistribution-of-random-walks-by-diffeomorphisms-near} and by DeWitt, Dolgopyat, and Zhang \cite{DeWittDolgopyatZhang2026_Ergodicity-of-coexpanding-on-average-random-dynamical-systems}.
Like our work, a crucial point in these results is to allow zero Lyapunov exponents.
A substantial difference and source of technicalities is that our work is concerned with noncompact manifolds and involves the use of a Margulis function, whereas \cite{BenardZhang2026_Rigidity-and-equidistribution-of-random-walks-by-diffeomorphisms-near,DeWittDolgopyatZhang2026_Ergodicity-of-coexpanding-on-average-random-dynamical-systems} treat only compact manifolds.

Specifically, \cite[Thm.~1]{DeWittDolgopyatZhang2026_Ergodicity-of-coexpanding-on-average-random-dynamical-systems} establish ergodicity of volume measure under a random walk which has a codimension $1$ hyperbolicity gap, which is a stronger assumption than our uniform expansion gap assumption.
This result follows by applying our \autoref{theorem:Absolute-continuity-and-ergodicity} to an ergodic component of volume.
We note that the ergodicity part of our \autoref{theorem:Absolute-continuity-and-ergodicity} allows one to extend \cite[Cor.~2]{DolgopyatKrikorian2007_On-simultaneous-linearization-of-diffeomorphisms-of-the-sphere} to odd dimensions, as in \cite[Thm.~1.6]{DeWittDolgopyatZhang2026_Ergodicity-of-coexpanding-on-average-random-dynamical-systems}.

The work \cite{BenardZhang2026_Rigidity-and-equidistribution-of-random-walks-by-diffeomorphisms-near} has stronger assumptions and stronger conclusions compared to ours.
On the one hand, the assumptions on expansion and gaps are stronger, but at the same time their results are effective and also yield equidistribution of spheres in the random walk.

\paragraph{Applications.}
We will use the criteria in this paper, combined with \cite{BEFRH25}, to obtain a measure and orbit closure rigidity result for random walks on character varieties of surface groups, with target a compact Lie group, and with action given by algebro-geometric subgroups of the mapping class group.

\paragraph{Acknowledgements.}
This material is based upon work supported by the National Science Foundation under Grant No.
DMS-2305394 and DMS-2604568 (SF),
DMS-2400191 (AB), 
DMS-2453688 (FRH), 
and by the Simons Foundation (AE).

\paragraph{AI use acknowledgement}
The ideas and arguments in this paper are due to the authors and go back to winter of 2026.
AI was only used to assist with the drafting of some routine parts of the text.



					\section{Assumptions and results}
					\label{sec:Assumptions-and-results}

In this section, we give the precise assumptions for our main results.
Readers interested in compact manifolds only can skip \autoref{ssec:assumptions_on_the_total_space}.
In \autoref{ssec:Uniform-Expansion-with-gaps} we define our functions $\alpha_{k}^{\bullet}(f,\bullet)$ and view them as cocycles on appropriate flag bundles over $M$.
We then list the necessary assumptions in \autoref{ssec:dynamical:assumptions} and state our main result in \autoref{ssec:Main-Theorem}.

\subsection{Assumptions on the total space}
\label{ssec:assumptions_on_the_total_space}

In this subsection we record the standing assumptions on the ambient manifold
and on the random dynamics.

Let $M$ be a smooth, possibly non-compact, connected
manifold equipped with a background
Riemannian metric.  All tangent-space norms and distances below are taken with
respect to this metric unless otherwise specified.  We write $TM(q)$ for the
tangent space at $q$, and $B(q;r)$ for the radius-$r$ ball in $TM(q)$.  The
Riemannian distance on $M$ is denoted $d^M$.

Instead of assuming that $M$ is compact, which is insuffient for our
intended applications, we postulate a more flexible setup with a
systole function and an admissible Margulis function. In this
subsection, we make the
precise definitions.

\subsubsection{Norms on jet bundles}
\label{sssec:norms_on_jet_bundles}
Denote by \index{$J_kM$}$J_kM$ the bundle of $k$-jets of functions on $M$ and
its dual \index{$J_k^\dual M$}$J_k^\dual M$.  For example, $J_0M$ is the
trivial $1$-dimensional bundle, and
$T^\dual M = \ker\left(J_1 M \onto J_0 M\right)$.

The Riemannian metric on $M$ yields:
\begin{enumerate}
\item Norms on the $k$-jets $J_kM$, and induced norms on the duals
$J_k^\dual M$.  The norms are compatible, i.e. the embedding
$J_k^\dual M\into J_{k+1}^\dual M$ is an isometry.
\item The norm on $TM$ gives a distance function on $M$, by minimization along
paths.
\end{enumerate}
We will denote these ambient norms by
\index{$\norm$@$\tnorm{\bullet}$}$\tnorm{\bullet}$, and if necessary indicate
with a subscript the order of the jets.

\subsubsection{Admissible charts with bounds on jets}
\label{sssec:charts_with_bounds_on_jets}
We assume that there exists an ambient radius function
\index{$r_a$}$r_a\colon M\to (0,1]$ such that for every $q\in M$ there is a
smooth chart, which we will call an admissible exponential chart,
\begin{displaymath}
\index{$e$@$\exp_q$}\exp_q\colon B(q;r_a(q))\to M,
\end{displaymath}
where $B(q;r_a(q))$ is the $\tnorm{\bullet}$-ball in $TM(q)$.  We require
$\exp_q$ to be a diffeomorphism onto its image and to be normalized by
\begin{equation}
\label{eq:chart:normalization}
\exp_q(0)=q, \qquad D_0\exp_q=\Id_{TM(q)}.
\end{equation}
The term ``exponential'' is only terminology; these charts need not be the
geodesic exponential maps of the Riemannian metric. 

We assume the following compatibility between these charts and the background
splittings of higher order jets.  For every $k\geq 0$ there exists a function
\index{$C_k(q)$}$C_k\colon M\to [1,\infty)$ such that, for any
$v\in B(q;r_a(q))$ and $q':=\exp_q(v)$, the induced map between higher order
jets
\begin{displaymath}
J_k(TM(q))(v) \xleftarrow{J_k\exp_q} J_kM(q')
\end{displaymath}
has norm bounded by $C_k(q)$.  The jet bundles over a vector space equipped
with a norm are naturally split and have natural metrics induced from the
metric of the vector space.

The functions $r_a(q)$ and $C_k(q)$ will be controlled by the systole function
in \autoref{sssec:systole_function} below.

\begin{remark}[On distances and norms]
\label{rmk:on_distances_and_norms}
If the chosen charts are the Riemannian exponential charts, then the
Levi-Civita connection provides the above splittings.

For every $n$, there exists a constant $c_n$ such that for every Finsler norm
$\norm{-}_F$ on an $n$-dimensional vector space, there exists a
positive-definite inner product $\norm{-}_J$ provided by the John ellipsoid of
the unit ball of the Finsler metric such that
$\norm{-}_F\leq \norm{-}_J\leq c_n\norm{-}_F$.  Thus, on tangent spaces, we can
freely pass from a Finsler metric to a positive-definite quadratic form, at the
cost of dimensional constants.

\end{remark}

\subsubsection{Random dynamics and words}
\label{sssec:random_dynamics_and_words}

In this paper, all measures are understood to be completed
when necessary, and we retain the same notation for their completions.
For joint integrals, we use the completion of the corresponding product
measure.

Let $\mu$ be a probability measure on $\Diff_\infty(M)$ with compact
support in the $C^\infty$ compact-open topology, 
and let \index{$S$@$\cS_\mu$}$\cS_\mu:=\operatorname{supp}\mu$.  We also write
\begin{displaymath}
\index{$S$@$\cS_\mu^\pm$}\cS_\mu^\pm:=\cS_\mu\cup\{s^{-1}\st s\in\cS_\mu\}.
\end{displaymath}
If $w=(s_0,\ldots,s_{n-1})$ is a word in $\cS_\mu^\pm$, we write
\index{$|w|$}$|w|=n$ and
\begin{displaymath}
\index{$f_w$}f_w:=s_{n-1}\circ\cdots\circ s_0.
\end{displaymath}
The empty word has length $0$ and represents the identity map.  Words in
$\cS_\mu$ without inverses are precisely the maps in the support of the
convolution powers of $\mu$, with multiplicity ignored.

The assumptions below are stated for one letter in $\cS_\mu^\pm$.  Since
$\cS_\mu$ is compact, the same type of estimates hold for words of any fixed
length, after changing the constants.  In particular, replacing $\mu$ by a
convolution power $\mu^{(N)}$ preserves the standing assumptions, with constants
allowed to depend on $N$.

\subsubsection{Systole function for random dynamics}
\label{sssec:systole_function}
\label{sssec:random:systole_function}
We assume there exists a function
\begin{displaymath}
\index{$\phi_a(q)$}\phi_a\colon M\to [2,\infty),
\end{displaymath}
called the ambient \emph{systole function}, such that $\log\phi_a$ is uniformly
continuous for the ambient Riemannian distance.  In examples $\phi_a$ is often
an inverse systole, such as the height on the upper half-plane or a reciprocal
of the distance to a singular set; throughout the paper we keep the shorter
terminology ``systole function''.  In particular, $\phi_a$ is continuous and
therefore bounded on compact sets.

We require $\phi_a$ to control both the local geometry and the compact-support
random dynamics in the following sense.
\begin{description}
\item[Bounded growth of the systole.] There exists a constant $A_\phi\geq 1$
such that, for every $s\in\cS_\mu^\pm$ and every $q\in M$,
\begin{displaymath}
\phi_a(s(q))\leq A_\phi\,\phi_a(q).
\end{displaymath}
Consequently, for every word $w$ in $\cS_\mu^\pm$,
\begin{equation}
\label{eq:systole:growth}
\phi_a(f_w(q))\leq A_\phi^{|w|}\phi_a(q).
\end{equation}

\item[Bounds on chart radii and chart transitions.] There is a constant
$e_0>0$, and for every $k\geq 0$ there is a constant $e_k>0$, such that, for
every $q\in M$,
\begin{displaymath}
r_a(q)\geq \phi_a(q)^{-e_0}, \qquad C_k(q)\leq \phi_a(q)^{e_k}.
\end{displaymath}

\item[Bounds on generator dynamics.] There exists a constant $e>0$ such that,
for the dynamical radius
\begin{displaymath}
\index{$r_d(q)$}r_d(q):=\phi_a(q)^{-e}r_a(q),
\end{displaymath}
and for every $s\in\cS_\mu^\pm$, the inclusion
\begin{displaymath}
s\bigl(\exp_q(B(q;r_d(q)))\bigr)
\subset \exp_{s(q)}\bigl(B(s(q);r_a(s(q)))\bigr)
\end{displaymath}
holds.  We denote the resulting coordinate expression by
\begin{equation}
\label{eqn:big_G_t_definition}
\index{$G_{s,q}$}G_{s,q}:=\exp_{s(q)}^{-1}\circ s\circ\exp_q .
\end{equation}
For every $k\geq 0$ there is a constant $a_k>0$ such that
\begin{equation}
\label{eq:norm:Gsq:bound}
\norm{G_{s,q}}_{C^k\left(B(q;r_d(q))\right)}\leq \phi_a(q)^{a_k}
\end{equation}
for all $s\in\cS_\mu^\pm$ and all $q\in M$.
\end{description}

\begin{remark}[Fixed-length words]
\label{rmk:systole_fixed_length_words}
The preceding one-letter assumptions imply analogous estimates for words of any
prescribed length $L$.  More precisely, after replacing the constants by new
constants depending on $L$, the same statements hold with $s$ replaced by $f_w$
for all words $w$ with $|w|\leq L$.  We will use this fixed-length stability
without changing notation.
\end{remark}

\subsubsection{Some special cases}
\label{sssec:some_special_cases}
If $M$ is compact and $\cS_\mu$ is finite, then we can take $\phi_a\equiv 10$,
or any other fixed constant larger than $1$.  The chart radii, chart-transition
constants, and all $C^k$-bounds for the finitely many maps in $\cS_\mu^\pm$ are
then uniform.  The non-compact hypotheses above are a way of retaining
polynomial control of these constants in terms of $\phi_a(q)$.

\subsubsection{An elementary bound}
\label{sssec:an_elementary_bound}
We record, for future use, the following elementary consequence of compact
support and smoothness.  For every compact set $K\subset M$ and every
$L\in\N$, there exist $\ve_0>0$ and $C=C(K,L,\ve_0)$ such that, whenever
$q\in K$, $d^M(q,q')\leq \ve_0$, and $w$ is a word in $\cS_\mu^\pm$ of length
at most $L$, one has
\begin{equation}
\label{eqn:elementary_compactness_2_point_estimate}
d^M(f_w(q),f_w(q'))\leq C\cdot d^M(q,q').
\end{equation}
Indeed, choose a relatively compact neighborhood $K'\supset K$, and then choose
$\ve_0>0$ so that $q\in K$ and $d^M(q,q')\leq \ve_0$ imply $q'\in K'$.
The words of length at most $L$ form a compact family, and therefore
have uniformly bounded first derivative on
$K'$.


\subsubsection{Margulis functions}
\label{sec:subsec:Margulis:functions}

The systole function controls local geometry, but it does not by itself imply
that typical random trajectories spend most of their time in regions of
controlled geometry.  The following  Margulis function
condition is the recurrence
hypothesis that replaces compactness in the non-compact setting.

For a non-negative measurable function $u$ on $M$, write
\begin{displaymath}
\index{$P_\mu$}P_\mu u(q):=\int_{\Diff_\infty(M)}u(f(q))\,d\mu(f)
\end{displaymath}
for the Markov operator associated to the random walk.

\begin{definition}[Admissible Margulis function]
\label{def:random:margulis:function}
A continuous proper function $h\colon M\to [1,\infty)$ is called an admissible
Margulis function for $(M,\mu,\phi_a)$ if there exist constants $0<c<1$,
$b<\infty$, $\kappa>0$, and $\delta>0$ such that the following hold.  Here
proper means that the sublevel sets $\{q\in M\st h(q)\leq R\}$ are compact for
all $R<\infty$.
\begin{itemize}
\item[\rm (a)] For all $q\in M$,
\begin{equation}
\label{eq:def:random:margulis}
P_\mu h(q)=\int_{\Diff_\infty(M)} h(f(q))\,d\mu(f)\leq c\,h(q)+b.
\end{equation}
\item[\rm (b)] For all $q\in M$,
\begin{equation}
\label{eq:def:random:margulis:systole-domination}
h(q)\geq \kappa\,\phi_a(q)^\delta,
\end{equation}
where $\phi_a$ is the systole function from \autoref{sssec:systole_function}.
\end{itemize}
\end{definition}

\subsection{Uniform Expansion with gaps}
\label{ssec:Uniform-Expansion-with-gaps}

For $x \in M$ let \index{$TM(x)$}$TM(x)$ denote the tangent space to $M$.
Let \index{$G$@$\Gr_m(x)$}$\Gr_m(x)$ denote the Grassmannian of
$m$-dimensional subspaces of $TM(x)$. The background Riemannian metric induces,
for each $x \in M$, a norm
\index{$\| \cdot \|_x$}$\| \cdot \|_x$ on the exterior power $\bigwedge^m(TM(x))$. We then have a cocycle \index{$\alpha_m$}$\alpha_m \colon \Diff_\infty(M) \times \Gr_m(TM) \to \mathbb{R}$, given by
\begin{equation}
\label{eq:def:sigma:f:L}
	\alpha_m(f,L) = \log \frac{\| Df(l_1 \wedge \dots l_m) \|_{f(x)}}{\|l_1 \wedge \dots l_m \|_x}
\end{equation}
where $\{l_1, \dots, l_m\}$ is any basis for $L$. Then $\alpha_m(f,L)$ measures the change in the volume of the subspace $L$ under the diffeomorphism $f$. 


\begin{definition}[Uniform expansion]
\label{def:uniformly_expanding}
The measure $\mu$ is \emph{uniformly expanding in dimension $m$} if 
there exist $C_m > 0$ and $N_m \in
\bN$ such that for all $x \in M$ and all $L \in \Gr_m(x)$,
\begin{equation}
\label{eq:def:uniform:expansion}
	\int_{\Diff_\infty(M)} \alpha_m(f,L) \, d\mu^{(N_m)}(f) > C_m > 0
\end{equation}
where \index{$\mu^{(N_m)}$}$\mu^{(N_m)}$ denotes the $N_m$-fold convolution of $\mu$ with itself.
\end{definition}

\begin{definition}[Uniform expansion gaps]
	\label{def:uniform_expansion_gaps}
	The measure $\mu$ has a \emph{uniform expansion gap in dimension $m$} if 
	there exist $C_m > 0, \delta_m\in\{\pm 1\}$, and  $N_m \in
	\bN$ such that for all $x \in M$ and for all $L_1\subset L_0$ with $L_i \in \Gr_{m-i}(x)$, we have
\begin{displaymath}
	\delta_m\cdot {\int_{\Diff_\infty(M)} \big[\alpha_{m}(f,L_0)-\alpha_{m-1}(f,L_1)\big] \, d\mu^{(N_m)}(f) } > C_m > 0
\end{displaymath}
	where $\mu^{(N_m)}$ denotes the $N_m$-fold convolution of $\mu$ with itself.
\end{definition}

\begin{remark}[On uniform expansion and gaps]
	\label{rmk:on_uniform_expansion_and_gaps}
	Suppose that $\nu$ is an ergodic $\mu$-stationary measure with Lyapunov spectrum $\lambda_1\geq \dots\geq \lambda_{\dim M}$ on $TM$.
	If $\mu$ is uniformly expanding in dimension $m$, then it follows from the same arguments as in the proof of \cite[Thm.~C]{Chung2020_Stationary-measures-and-orbit-closures-of-uniformly-expanding-random}, that the Lyapunov exponents satisfy $\lambda_1+\dots +\lambda_m>C_m/N_m$.
	Analogously, if $\mu$ has uniform expansion gap in dimension $m$ then $\delta_m \lambda_m>C_m/N_m$ where $\delta_m\in \{\pm 1\}$.
	
	Note also that in case $M$ is compact, both \autoref{def:uniformly_expanding} and \autoref{def:uniform_expansion_gaps} are properties that are open upon a $C^1$-perturbation.
\end{remark}

\subsubsection{Flag spaces and cocycles.}
\label{sssec:flag:spaces:cocycles}

To encode first- and second-difference information in the
exterior-power cocycles, we pass from Grassmannians to partial flag
spaces and define the derived cocycles $\alpha_k'$ and  $\alpha_k''$. 

For a vector space $V$, let us define the space of partial flags
\begin{displaymath}
\Flag_{k}'(V) = \{ (L_k,L_{k+1}) \mid  L_k
\subset L_{k+1} \subset V, \dim L_k = k, \dim L_{k+1}=k+1 \}
\end{displaymath}
Similarly, let
\begin{multline*}
\Flag_{k}''(V) = \{ (L_{k-1},L_k,L_{k+1}) \mid L_{k-1} \subset L_k
\subset L_{k+1} \subset V, \\
\dim L_{k-1} = k-1, \dim L_k = k, \dim  L_{k+1} = k+1\}
\end{multline*}
We also define  $\alpha_k': \Diff_\infty(M) \times \Flag_k'(TM) \to \R$ by
\begin{displaymath}
\alpha_k'(f,L_k,L_{k+1}) = \alpha_k(f,L_k) - \alpha_{k+1}(f,L_{k+1})
\end{displaymath}
and $\alpha_k'': \Diff_\infty(M) \times \Flag_k''(TM) \to \R$ by
\begin{displaymath}
\alpha_{k}''(f,L_{k-1},L_k,L_{k+1}) = 2 \alpha_k(f,L_k)-
\alpha_{k-1}(f,L_{k-1}) - \alpha_{k+1}(f,L_{k+1}).
\end{displaymath}

\subsection{Dynamical Assumptions}
\label{ssec:dynamical:assumptions}
Let $d$ be the dimension of $M$. 
Fix $k\in\{1,\dots,d-1\}$. In the rest of the paper, we make the following assumptions:

\begin{itemize}
\item
The random walk is given by the Bernoulli measure, and we consider the
natural filtration generated by the letters.
\item
$M$ satisfies  all the assumptions of 
\S\ref{ssec:assumptions_on_the_total_space}, and  there exists an
admissible Margulis function $h: M \to [1,\infty)$. 
\item 
We assume that there exists a uniform gap in dimension $k+1$
with $\delta_{k+1} < 0$. In other words we assume that there exist $C
> 0$ and $N \in \bN$ such that for all $x \in M$ and all
$F \in \Flag'_k(T_x  M)$,
\begin{equation}
\label{eq:uniform:gap:negative}
  \int_{\Diff_\infty(M)} \alpha_k'(f,F)  \,  d\mu^{(N)}(f)  > C > 0
\end{equation}
In particular, in view of Remark~\ref{rmk:on_uniform_expansion_and_gaps},
if $\nu$ is an ergodic $\mu$-stationary measure with Lyapunov spectrum
$\lambda_1\geq \dots\geq \lambda_d$ on $TM$, then
$\lambda_{k+1} < 0$. 
\item We also need to assume one of the following:
  \begin{itemize}
    \item[{\rm (FS)}] (Full Stable Case). For every $\epsilon > 0$ 
      there exists $N \in \bN$ such that
      for all $x \in M$ and all $F \in \Flag'_{k-1}(T_x M)$,
\begin{equation}
\label{eq:uniform:gap:positive}
\int_{\Diff_\infty(M)} \alpha_{k-1}'(f,F)  \,  d\mu^{(N)}(f)  \le \epsilon.
\end{equation}

     \item[{\rm (PS)}] (Partial Stable Case).
	   There exist $C > 0$ and
       $N \in \bN$ such that for all $x \in M$ and all $F \in
       \Flag''_k(T_x M)$,
\begin{equation}
\label{eq:uniform:gap:second:derivative}
\int_{\Diff_\infty(M)} \alpha_k''(f,F)  \,  d\mu^{(N)}(f) >C > 0.
\end{equation}
\end{itemize}
\end{itemize}
We will see that in the (FS) case, the dimension of the stable manifold is $d-k$, while in the (PS) case the dimension of the stable manifold is strictly larger than $d-k$ and in this case we will only work with the partial stable of dimension $d-k$.

\begin{remark}
\leavevmode
\begin{enumerate}
\item In some cases, zero exponents may be forced for algebraic
  reasons. For example if the algebraic hull is an orthogonal group,
  we may have $\alpha_{k-1}'$ identically zero, which trivializes
  (\ref{eq:uniform:gap:positive}). Also, in the symplectic case,
  (\ref{eq:uniform:gap:positive}) may follow from
  (\ref{eq:uniform:gap:negative}) because of a symmetry.
	
\item Note that if (\ref{eq:uniform:gap:negative}) holds for $N$, it
  holds for $mN$, where $m \in \bN$, with $C$ replaced by
  $mC$. Similar statements hold for (\ref{eq:uniform:gap:positive})
  and (\ref{eq:uniform:gap:second:derivative}). Therefore we may
  assume that all three statements hold with the same $N$.

\item Suppose we are in the case (FS).
Let $C$, $N_-$ be the constants from
(\ref{eq:uniform:gap:negative}). Choose $\epsilon < C/(2N_-)$ in
(\ref{eq:uniform:gap:positive}), and let $N_+$ be the constant $N$
from (\ref{eq:uniform:gap:positive}). Let  $N$ be a common multiple
of $N_+$ and $N_-$. 

Note that $\alpha_k'' = \alpha_k' - \alpha_{k-1}'$.  Then, in view of
(\ref{eq:uniform:gap:negative}) and (\ref{eq:uniform:gap:positive}),
\begin{multline*}
  \int \alpha_k''(f,F)\, d\mu^{(N)}(f) = \\
  =
  \int \alpha_k'(f,F)\, d\mu^{(N)}(f)
  -
  \int \alpha_{k-1}'(f,F)\, d\mu^{(N)}(f) 
   >\frac{N}{N_-}C-\frac{N}{N_+}\epsilon
\geq\frac{NC}{2N_-}>0.
\end{multline*}
Therefore, we may assume that (\ref{eq:uniform:gap:second:derivative})
also holds in case (FS).
		
We also pass to a common multiple to make sure that in both the (FS)
and (PS) cases, (\ref{eq:uniform:gap:negative}) and
(\ref{eq:uniform:gap:second:derivative}) hold for the same $N$.
\end{enumerate}
\end{remark}

\subsubsection{Replacing $\mu$ with $\mu^{(N)}$.}
We now replace $\mu$ with $\mu^{(N)}$, so that we may assume that (both
in the (FS) and the (PS) cases),  
(\ref{eq:uniform:gap:negative}) and 
(\ref{eq:uniform:gap:second:derivative}) 
hold with $N=1$.

\subsubsection{Extra dynamical assumptions}
			\label{ssec:extra:dynamical:assumptions}
We now state some additional dynamical assumptions which imply a
non-concentration result for the (tangent spaces to) the stable manifolds. 

We now assume the following:
\begin{itemize}
\item We are in case (FS).
\item We have uniform expansion in dimension $d-1$. More precisely, 
there exist $C > 0$ and $N \in
\bN$ such that for all $x \in M$ and all $L \in \Gr_{d-1}(x)$,
\begin{equation}
\label{eq:codim:one:uniform:expansion}
	\int_{\Diff_\infty(M)} \alpha_{d-1}(f,L) \, d\mu^{(N)}(f) > C > 0
\end{equation}
\item For every $\epsilon > 0$ there exists  $N \in
\bN$ such that for all $x \in M$ and all $L \in \Gr_{d}(x)$,
\begin{equation}
\label{eq:codim:zero:uniform:expansion}
	\int_{\Diff_\infty(M)} \alpha_{d}(f,L) \, d\mu^{(N)}(f) \le \epsilon.
\end{equation}
(Note that since $\dim M = d$, $\Gr_d(x)$ is a singleton). 
This condition is automatic if all $f$ in the support of $\mu$
preserve the Riemannian volume form on $M$. 
\end{itemize}

As in the previous subsection, we may assume, for any fixed $\epsilon
> 0$, (after replacing $\mu$ by
an appropriate convolution with itself) that
(\ref{eq:codim:one:uniform:expansion}) and
(\ref{eq:codim:zero:uniform:expansion}) hold with $N=1$.


\subsection{Main Theorem}
\label{ssec:Main-Theorem}

\begin{theorem}[Absolute continuity and ergodicity]
\label{theorem:Absolute-continuity-and-ergodicity}
Suppose that the random walk satisfies the assumptions in
\autoref{ssec:dynamical:assumptions}. We further assume that either 
the ``extra dynamical assumptions'' of
\S\ref{ssec:extra:dynamical:assumptions} are satisfied, or that the
conclusion of \autoref{theorem:non:concentration} holds.

Suppose that $\nu$ is a $\mu$-stationary 
probability measure on $M$ which is invariant
under every $f$ in the support of $\mu$. We also assume that 
the conditional measures along (partial) $(d-k)$-stable manifolds are
in the Lebesgue class. Then
\begin{itemize}
\item[{\rm (i)}] $\nu$ is absolutely continuous with a smooth density with
  respect to the Riemannian volume  measure on $M$.
\item[{\rm (ii)}] $\nu$ is fully supported on $M$ and ergodic.
\item[{\rm (iii)}]  If $\mu$ is supported on volume-preserving
diffeomorphisms, then $\nu$ agrees with the normalized
volume measure on $M$, and in particular the volume measure is ergodic
and finite. 
\end{itemize}
\end{theorem}

\begin{remark}[Further strengthenings]
	\label{rmk:Further-strenthenings}
	\leavevmode
	\begin{enumerate}
		\item 
		We show in \autoref{theorem:Good-stables-everywhere} that under these assumptions, stable manifolds exist at \emph{every} point of $M$.
		\autoref{theorem:good:words} shows that stable manifolds have controlled geometry for good futures.
		\item 
Given the  assumptions on $\nu$ of
\autoref{theorem:Absolute-continuity-and-ergodicity} (namely
probability, ergodicity, invariance under every $f$ in the support of $\mu$  and
the conditional measure hypothesis), the same conclusions as in \autoref{theorem:Absolute-continuity-and-ergodicity} hold under weaker assumptions, namely the validity of \autoref{theorem:Good-stables-everywhere}, \autoref{theorem:good:words} and \autoref{theorem:non:concentration}.
		This is proved in \autoref{theorem:ac:conclusion}.
\end{enumerate}
\end{remark}

\noindent The next technical result allows us to assume first that $\nu$ in the main \autoref{theorem:Absolute-continuity-and-ergodicity} is ergodic, and then observe that the conclusions of \autoref{theorem:Absolute-continuity-and-ergodicity} imply its uniqueness.

\begin{proposition}[Ergodicity and absolute continuity on stables]
	\label{prop:Ergodicity-and-absolute-continuity-on-stables}
	Suppose $\nu$ is a probability measure on $M$, invariant under every $f\in \supp \mu$.
	Suppose furthermore that $\nu$ has absolutely continuous conditional measures along (partial) $(d-k)$-stable manifolds.
	
	Then a.e. ergodic component of $\nu$ is also invariant under every $f\in \supp \mu$ and has absolutely continuous conditional measures along (partial) $(d-k)$-stable manifolds.	
\end{proposition}
\begin{proof}
	On the skew-product space $\Omega\times M$ consider the $\sigma$-algebra of sets that are saturated in the first coordinate, and invariant under the dynamics $(\omega,x)\mapsto (\sigma(\omega),f_{\omega_0}(x))$.
	This $\sigma$-algebra (rather, one which is countably-generated and $\mu^{\otimes \bN}\otimes \nu$-equivalent to it) determines the ergodic decomposition of $\nu$.
	By the Birkhoff ergodic theorem, this $\sigma$-algebra is also saturated by the equivalence relation $(\omega,x)\sim (\omega,y)$ if $y$ is in the (partial) $(d-k)$-stable manifold of $x$ for $\omega$.
	
	Therefore, taking ergodic decompositions commutes with disintegration of measures along (partial) $(d-k)$-stable manifolds, and the result follows.
\end{proof}




\section{Uniform geometry of stable manifolds}
\label{sec:Uniform-geometry-of-stable-manifolds}

\subsection{Singular value decomposition}
\label{sssec:Singular-value-decomposition}

Suppose that $V_1,V_2$ are $d$-dimensional Euclidean vector spaces, and
$A\colon V_1\to V_2$ is a linear isomorphism. Write
\[
\sigma_1(A)\ge \cdots \ge \sigma_d(A)>0
\]
for the singular values of $A$, listed with multiplicity.

Equivalently, the positive self-adjoint operators $A^\dagger A$ on $V_1$
and $AA^\dagger$ on $V_2$ have the same eigenvalues
$\sigma_1(A)^2,\dots,\sigma_d(A)^2$, and there are orthogonal
decompositions
\[
V_1=\bigoplus_\sigma V_1^\sigma,
\qquad
V_2=\bigoplus_\sigma V_2^\sigma,
\]
where $\sigma$ runs over the distinct singular values of $A$, such that
\[
A(V_1^\sigma)=V_2^\sigma
\qquad\text{and}\qquad
\|Av\|=\sigma\,\|v\|
\quad\text{for all }v\in V_1^\sigma.
\]

For $1\le k\le d-1$, let
\[
\Delta_k(A):=\frac{\sigma_k(A)}{\sigma_{k+1}(A)}.
\]

\begin{definition}[$k$-approximate image and kernel]
\label{def:def-k-approximate-kernel-and-image}
Assume that $\sigma_k(A)>\sigma_{k+1}(A)$. Then the \emph{$k$-approximate
image} of $A$ is
\[
\im_k A:=\bigoplus_{\sigma\ge \sigma_k(A)} V_2^\sigma,
\]
and the \emph{$(d-k)$-approximate kernel} of $A$ is
\[
\ker_{d-k} A:=\bigoplus_{\sigma\le \sigma_{k+1}(A)} V_1^\sigma.
\]
Equivalently, $\im_k A$ is the sum of the singular spaces corresponding to
the $k$ largest singular values, and $\ker_{d-k}A$ is the sum of the
singular spaces corresponding to the $d-k$ smallest singular values.
\end{definition}

\begin{remark}
If $\sigma_k(A)=\sigma_{k+1}(A)$, then these subspaces are not canonically
defined in general. In this paper we will only use $\im_k A$ and
$\ker_{d-k}A$ in situations where there is a $k$-gap.
\end{remark}

We use the notation  $\dist_{\Gr}$ to denote distance in the Grassmanian. 


	\subsection{Main statements on the geometry of stable manifolds}
		\label{ssec:Main-statement-on-the-geometry-of-stable-manifolds}

We assume that the dynamical assumptions in
\S\ref{ssec:dynamical:assumptions} are satisfied. (We do not use the
extra dynamical assumptions of \S\ref{ssec:extra:dynamical:assumptions} in
this section).   

\subsubsection{Words and Bowen balls}
	\label{sssec:Words-and-Bowen-balls}

Let $w=(i_0,i_1, \dots,  )$ be a finite or infinite word in the random
walk alphabet. Let $|w|$ denote the length of $w$. Then, for $n \le
|w|$, we write $w^{(n)}$ to denote the truncation $(i_0,i_1, \dots,
i_{n-1})$ and 
$f_w^n$ to denote  $f_{i_{n-1}}  \circ \dots  \circ f_{i_0}$. We use
the convention that $f^0_w$ is the identity map. 
We also use the notation
\begin{displaymath}
BB(w,x,\epsilon) = \{ y \in M \,\mid\,  \text{for all $0 \le i \le |w|$, $d(f_w^{i}(y), f_w^{i}(x)) < \epsilon$}\}.
\end{displaymath}
We refer to the sets $BB(w,x,\epsilon)$ as Bowen balls.  
Concatenation of words is written in composition order: $uv$ denotes the word obtained by first following $v$ and then following $u$.

Our main goals are the following two theorems.

\begin{theorem}[Good stables everywhere]
\label{theorem:Good-stables-everywhere}
There exists $\lambda > 0$ and  $R_{\rm init} < \infty$ depending only
on the constants in the dynamical assumptions of
\S\ref{sssec:random_dynamics_and_words} and
\S\ref{ssec:dynamical:assumptions} such that the following holds.

For any $x \in M$ there exists a subset $\Omega_\infty(x) \subset \Omega$ with
$\mu^{\otimes \mathbb N}(\Omega_\infty(x))=1$, such that for all $\tilde{w} \in
\Omega_\infty(x)$ there exists a $C^\infty$ embedded submanifold
\[
\mathcal W_{d-k}(x,\tilde{w})\subset M
\]
of dimension $d-k$, containing $x$ such that for every
$y\in\mathcal W_{d-k}(x,\tilde{w})$,
\[
\limsup_{n\to\infty}
\frac1n
\log d\bigl(f^n_{\tilde{w}}(x),f^n_{\tilde{w}}(y)\bigr)
\le -\lambda/2.
\]
Moreover,
\begin{displaymath}
\lim_{n \to \infty} \Delta_k(Df_{\tilde{w}}^{n}(x)) =  + \infty,
\end{displaymath}
\begin{displaymath}
\lim_{n  \to \infty} \dist_{\Gr}\!\bigl(
\ker_{d-k}(Df_{\tilde{w}}^{n}(x)),
T_x \mathcal W_{d-k}(x,\tilde w)
\bigr) = 0.
\end{displaymath}
and $h(f_{\tilde{w}}^{n}(x)) <  R_{\rm init}$ for infinitely many $n$.  
\end{theorem}
\noindent We call $\cW_{d-k}(x,\tilde{w})$ the local $k$-stable manifold associated to $x$ and $\tilde{w}$.

\begin{theorem}[Good words]
\label{theorem:good:words}
There exist an integer $N$, numbers $\epsilon > 0$, $\lambda > 0$, $\beta>0$ and
$C_{\rm init} > 0$, $R_{\rm init} > 0$ and $\epsilon_{\rm init}  > 0$,
a function $\rho_{\rm init}: \mathbb{N} \to \bR_+$ 
and a finite collection $J$ of words in the random
walk alphabet each of length between $N/2$ and $N$ and for any finite
word $w$ in the random
walk alphabet and any $x \in M$ a preferred word $w_{next} =
w_{next}(x,w) \in J$,  all depending only on the constants of
\S\ref{sssec:random_dynamics_and_words} and 
\S\ref{ssec:dynamical:assumptions}
such that the following holds.

For every
$y\in M$ with
$h(y) \le R_{\rm init}$
and every finite word $u$ in the random walk alphabet, if
$v:=w_{\rm next}(y,u)$
and
$d(f_v^{|v|}(y),z) < \epsilon$,
then
\begin{equation}
\label{eq:good:words:uniform:return}
h(z)<R_{\rm init}.
\end{equation}

Suppose $x \in M$ and 
$w_0$ is any finite word in the random walk, 
and suppose $n \in \mathbb{N}$. Assume
\begin{equation}
\label{eq:initial:x:return:height}
h\left(f_{w_0}^{|w_0|}(x)\right)  \leq R_{\rm init},
\end{equation}
and
\begin{equation}
\label{eq:initial:x:return:gap}
\Delta_k(Df_{w_0}^{|w_0|}(x)) >  e^{\epsilon_{\rm  init}}.
\end{equation}
Suppose
$w = w_n \dots w_1 w_0$
is such that for each $0 \le i < n$ there exists
\begin{displaymath}
x_i \in BB(w_i \dots w_0,x,\epsilon) \cap B(x,\rho_{\rm init}(|w_0|))
\end{displaymath}
with
$y_i = f^{|w_i \dots w_0|}_{w_i \dots w_0}(x_i)$
satisfying
$h(y_i) \le R_{\rm  init}$ and
$w_{i+1} = w_{next}(y_i, w_i)$.
(So in particular, for $1 \le i \le n$,
$w_i\in J$).

Let $\Omega(w)$ be the set of all infinite words
extending $w$, and let $\mathbb{P}_w=\mu^{\otimes \mathbb N}$ denote
the Bernoulli measure on the continuations after the fixed prefix $w$,
viewed as a probability measure on $\Omega(w)$.

Then there exists $n_\ast=n_\ast(w_0) \ge 1$ depending on 
$w_0$ (but not on $x$ as long as (\ref{eq:initial:x:return:height})
and (\ref{eq:initial:x:return:gap}) hold), such that for $n \ge
n_\ast$ there exists a decomposition
\[
\Omega(w)=G(w,x)\sqcup B(w,x)
\]
such that
\begin{equation}
\label{eq:good:words:bad:probability}
\mathbb{P}_w(B(w,x))\le e^{-\beta \sqrt{n}}.
\end{equation}
Moreover, $G(w,x) \subset \Omega_\infty(x)$ (so for every
$\tilde w\in G(w,x)$, the $k$-stable manifold $\cW_{d-k}(x,\tilde{w})$
exists).  Also, for $\tilde{w} \in G(w,x)$, $\cW_{d-k}(x,\tilde{w})$
satisfies the following additional estimates:

The size of $\mathcal W_{d-k}(x,\tilde{w})$ is bounded from
below by a constant depending only on $w_0$ and the constants in the
dynamical assumptions of \S\ref{sssec:random_dynamics_and_words} and
\S\ref{ssec:dynamical:assumptions}. 

For every $r\ge1$, its $C^r$-geometry is
bounded above by a constant depending only on $w_0$, $r$, the
constants in the dynamical assumptions \S\ref{sssec:random_dynamics_and_words} and
\S\ref{ssec:dynamical:assumptions}. 

For $\rho > 0$, let $B^{\cW}(x,\tilde{w},\rho) = \{ y \in  \cW_{d-k}(x,\tilde{w}) \st
d_\cW(x,y) < \rho \}$. Then, there exist
$\rho_{\rm st} >0$ and $C_{\rm st} \ge 1$, both depending only on $w_0$, the
constants in the dynamical assumptions \S\ref{sssec:random_dynamics_and_words} and
\S\ref{ssec:dynamical:assumptions}, 
such that for every $m \ge 0$ and every
$y \in B^{\cW}(x,\tilde{w},\rho_{\rm st})$,
\begin{equation}
\label{eq:good:words:stable:contraction}
\left\|D_y f^{m}_{\tilde{w}}\big|_{T_y\mathcal W_{d-k}(x,\tilde{w})}\right\|
\leq C_{\rm st}e^{-\lambda m/2}.
\end{equation}
Also, for $y \in B^{\cW}(x,\tilde{w},\rho_{\rm st})$, 
\begin{equation}
\label{eq:good:words:jacobian:ratio}
J_{\tilde w}(x,y) := \lim_{m\to\infty}
\frac{ \Jac\!\left(f^m_{\tilde w}|_{\mathcal W_{d-k}(x,\tilde
w)}\right)(y) }{ \Jac\!\left(f^m_{\tilde w}|_{\mathcal
W_{d-k}(x,\tilde w)}\right)(x) }
\end{equation}
exists and satisfies
\begin{equation}
\label{eq:good:words:jacobian:bounds}
C_{\rm Jac}^{-1}\le J_{\tilde w}(x,y)\le
C_{\rm Jac},
\end{equation}
where $C_{\rm Jac}$ depends only on $w_0$,
the constants in the dynamical assumptions \S\ref{sssec:random_dynamics_and_words} and
\S\ref{ssec:dynamical:assumptions}. 

The function $y \to J_{\tilde{w}}(x,y)$ is $C^\infty$ on $B^{\cW}(x,\tilde{w},\rho_{\rm
  st})$. For every $r \ge 1$ there exists $C_{{\rm Jac},r} < \infty$,
depending only on $r$,
$w_0$ and the constants in the dynamical assumptions \S\ref{sssec:random_dynamics_and_words} and
\S\ref{ssec:dynamical:assumptions} such that
\begin{equation}
\label{eq:good:words:jacobian:derivatives}
\|D^r_{\cW,y} J_{\tilde{w}}(x,y)\| \le C_{{\rm Jac},r} \qquad\text{for
  every $y \in B^{\cW}(x,\tilde{w},\rho_{\rm st})$}.
\end{equation}

The tangent space of this stable manifold is controlled by the
slow space of the original word $w_0$:
\begin{multline}
\label{eq:good:words:tangent:control}
\dist_{\Gr}\!\bigl( \ker_{d-k}(Df_{w_0}^{|w_0|}(x)), T_x
\mathcal W_{d-k}(x,\tilde w) \bigr) \\  \le C_{\rm
init}\min\left(1,\Delta_k(Df_{w_0}^{|w_0|}(x))^{-1}\right).
\end{multline}
\end{theorem}


\subsection{Uniformly tempered splittings and stable manifolds}
\label{ssec:uniformly:tempered:splittings:stable:manifolds}

We next recall the analytic input which will be used to pass from the
linear estimates obtained later to actual stable manifolds.

\begin{definition}[Uniformly tempered splitting]
\label{def:uniformly:tempered}
Let
\begin{displaymath}
A_i:V_i\to V_{i+1},
\qquad i\ge0,
\end{displaymath}
be a sequence of linear isomorphisms between $d$-dimensional Euclidean
spaces. We say that $A_i$ admits a $(C,\lambda,\epsilon)$-uniformly
tempered splitting in dimension $k$ if there exist splittings
\begin{displaymath}
V_i=\mathcal U(i)\oplus \mathcal S(i),
\qquad
\dim\mathcal U(i)=k,
\qquad
\dim\mathcal S(i)=d-k,
\end{displaymath}
and norms
$\|\cdot\|'_i$
on $V_i$ such that the following hold.

First, the splitting is equivariant:
\begin{displaymath}
A_i\mathcal U(i)=\mathcal U(i+1),
\qquad
A_i\mathcal S(i)=\mathcal S(i+1)
\qquad
\text{for every }i\ge0.
\end{displaymath}

Second, the adapted norms come from inner products and
are subtempered with respect to the background
Euclidean norms:
\begin{equation}
\label{eq:uniformly:tempered:norms}
C^{-1}e^{-\epsilon i}\|v\|
\le
\|v\|'_i
\le
Ce^{\epsilon i}\|v\|
\qquad
\text{for every }v\in V_i.
\end{equation}

Third, the splitting is orthogonal in the adapted inner products:
\begin{equation}
\label{eq:uniformly:tempered:angle}
(\|u+s\|'_i)^2 = (\|u\|'_i)^2 + (\|s\|'_i)^2
\end{equation}
for every $u\in\mathcal U(i)$ and $s\in\mathcal S(i)$.

Fourth, the stable directions are uniformly contracted:
\begin{equation}
\label{eq:uniformly:tempered:stable:contraction}
\|A_i s\|'_{i+1}
\le
e^{-\lambda}\|s\|'_i
\qquad
\text{for every }s\in\mathcal S(i).
\end{equation}

Finally, the one-step domination holds:
\begin{equation}
\label{eq:uniformly:tempered:one:step:domination}
\|A_i s\|'_{i+1}
\le
e^{-\lambda}\|A_i u\|'_{i+1}
\end{equation}
for every $i\ge0$, every $s\in\mathcal S(i)$ with $\|s\|'_i=1$, and every
$u\in\mathcal U(i)$ with $\|u\|'_i=1$.
\end{definition}

\begin{lemma}[Kernel convergence for uniformly tempered splittings]
\label{lemma:uniformly:tempered:kernel:convergence}
Let
\begin{displaymath}
A_i:V_i\to V_{i+1},
\qquad i\ge0,
\end{displaymath}
be a sequence of linear isomorphisms between $d$-dimensional Euclidean spaces.
Suppose that $A_i$ admits a $(C,\lambda,\epsilon)$-uniformly tempered
splitting in dimension $k$ in the sense of Definition~\ref{def:uniformly:tempered}.
Assume
$0<\epsilon<\frac{\lambda}{4}$.
For $n$ sufficiently large, write
\begin{displaymath}
A(n):=A_{n-1}\cdots A_0:V_0\to V_n,
\qquad
S(n):=\ker_{d-k}A(n)\subset V_0.
\end{displaymath}
Then there are constants $C'>0$ and $\lambda'>0$, depending only on
$C,\lambda,\epsilon,d,k$, such that
\begin{displaymath}
\dist_{\Gr}\bigl(S(n),\mathcal S(0)\bigr)
\le
C'e^{-\lambda'n}
\qquad
\text{for every sufficiently large }n.
\end{displaymath}
In particular,
\begin{displaymath}
S(n)\to \mathcal S(0)
\qquad
\text{exponentially fast.}
\end{displaymath}
In addition,
$\Delta_k(A(n))\longrightarrow+\infty$.
\end{lemma}

\begin{proof}
We write $\|\cdot\|'_i$ for the adapted norms from
Definition~\ref{def:uniformly:tempered}. By the one-step domination condition,
for every $i\ge0$, every $u\in\mathcal U(i)$ with $\|u\|'_i=1$, and every
$s\in\mathcal S(i)$ with $\|s\|'_i=1$,
\begin{displaymath}
\|A_i s\|'_{i+1}
\le
e^{-\lambda}\|A_i u\|'_{i+1}.
\end{displaymath}
Iterating this estimate gives
\begin{equation}
\label{eq:uniformly:tempered:iterated:domination:adapted}
\frac{
\mathfrak m'\bigl(A(n)|_{\mathcal U(0)}\bigr)
}{
\bigl\|A(n)|_{\mathcal S(0)}\bigr\|'
}
\ge
e^{\lambda n},
\end{equation}
where the conorm and norm are computed using $\|\cdot\|'_0$ on $V_0$ and
$\|\cdot\|'_n$ on $V_n$.

We compare the adapted norms with the background Euclidean norms. By
Definition~\ref{def:uniformly:tempered},
\begin{displaymath}
C^{-1}e^{-\epsilon i}\|v\|
\le
\|v\|'_i
\le
Ce^{\epsilon i}\|v\|
\qquad
(v\in V_i).
\end{displaymath}
Therefore
\begin{displaymath}
\mathfrak m\bigl(A(n)|_{\mathcal U(0)}\bigr)
\ge
C^{-2}e^{-\epsilon n}
\mathfrak m'\bigl(A(n)|_{\mathcal U(0)}\bigr),
\end{displaymath}
and
\begin{displaymath}
\bigl\|A(n)|_{\mathcal S(0)}\bigr\|
\le
C^2e^{\epsilon n}
\bigl\|A(n)|_{\mathcal S(0)}\bigr\|'.
\end{displaymath}
Combining these estimates with
\eqref{eq:uniformly:tempered:iterated:domination:adapted}, we obtain
\begin{equation}
\label{eq:uniformly:tempered:iterated:domination:euclidean}
\frac{
\mathfrak m\bigl(A(n)|_{\mathcal U(0)}\bigr)
}{
\bigl\|A(n)|_{\mathcal S(0)}\bigr\|
}
\ge
C^{-4}e^{(\lambda-2\epsilon)n}.
\end{equation}

The adapted transversality estimate in
Definition~\ref{def:uniformly:tempered}, together with the equivalence of
$\|\cdot\|'_0$ and the Euclidean norm, gives
\begin{displaymath}
\angle(\mathcal U(0),\mathcal S(0))\ge c_0
\end{displaymath}
for some $c_0=c_0(C,d,k)>0$.

We now apply Lemma~\ref{lemma:slow:space:from:dominated:splitting} to
\begin{displaymath}
T=A(n),\qquad U=\mathcal U(0),\qquad S=\mathcal S(0).
\end{displaymath}
Using \eqref{eq:uniformly:tempered:iterated:domination:euclidean}, this gives
\begin{displaymath}
\dist_{\Gr}\bigl(\ker_{d-k}A(n),\mathcal S(0)\bigr)
\le
C'e^{-(\lambda-2\epsilon)n}.
\end{displaymath}
Since $\epsilon<\lambda/4$, the exponent $\lambda-2\epsilon$ is positive.
Thus the convergence is exponential, for instance with
\begin{displaymath}
\lambda':=\lambda-2\epsilon.
\end{displaymath}
The final assertion follows from
(\ref{eq:uniformly:tempered:iterated:domination:euclidean}), 
since by the min-max principle,
\begin{displaymath}
\sigma_k(A(n)) \ge \mathfrak{m}(A(n)|_{\mathcal U(0)}), \qquad
\sigma_{k+1}(A(n)) \le \|A(n)|_{\mathcal S(0)}\|. 
\end{displaymath}
\end{proof}

\begin{theorem}[Stable manifold theorem]
	\label{theorem:stable:manifold:theorem}
	Suppose $M$ is a Riemannian manifold satisfying \autoref{ssec:assumptions_on_the_total_space}, with ambient systole function $\phi_a$ as in \autoref{sssec:systole_function}.  Let $(f_i)_{i\geq0}\subset\Diff^{\infty}(M)$ be a sequence of diffeomorphisms such that all $f_i$ and $f_i^{-1}$ satisfy the bounds in \autoref{sssec:systole_function} with the same constants.
	
	Fix $x\in M$, set $x_0=x$, and define
	$x_{i+1}=f_i(x_i).$
	Assume that there exist $C,\lambda,\epsilon>0$ with the following properties.
	
	We have control of the systole:
	\begin{displaymath}
		\phi_a(x_i)
		\le
		C e^{\sqrt{i}}
		\qquad
		\text{for every }i\ge0.
	\end{displaymath}
	Furthermore, the derivative cocycle
	\begin{displaymath}
		Df_i(x_i):T_{x_i}M\to T_{x_{i+1}}M
	\end{displaymath}
	admits a $(C,\lambda,\epsilon)$-uniformly tempered splitting in dimension
	$k$. Suppose also that $\epsilon$ is sufficiently small depending on
	$\lambda$ and the constants in
	\S\ref{ssec:assumptions_on_the_total_space} and 
	\S\ref{sssec:systole_function}.
	
	Then there exists a $C^\infty$ embedded submanifold
	\begin{displaymath}
		\mathcal W_{d-k}(x)\subset M
	\end{displaymath}
	of dimension $d-k$, containing $x$, tangent to the stable space
	$\mathcal S(0)\subset T_xM$, such that for every
	$y\in\mathcal W_{d-k}(x)$,
	\begin{displaymath}
		\limsup_{n\to\infty}
		\frac1n
		\log d\bigl((f_{n-1}\cdots f_0)(x),(f_{n-1}\cdots f_0)(y)\bigr)
		\le
		-\lambda/2.
	\end{displaymath}
	Moreover, the size of $\mathcal W_{d-k}(x)$ is bounded from below by a
	constant depending only on $C,\lambda,\epsilon,M$.
	For every $r\ge1$, its $C^r$-geometry is bounded above by a
	constant $C_{\rm geom}$ depending only on $r,C,\lambda,\epsilon,M$.
	
	For $y\in \mathcal W_{d-k}(x)$, let
	$d_{\mathcal W}(x,y)$ denote intrinsic distance on $\mathcal
	W_{d-k}(x)$, and let $B^{\cW}(x,r)$ denote $\{y \st d_\cW(x,y) < r\}$. 
	There exist a radius $\rho_{\rm st}>0$ and, for every $r\ge1$,
	a constant $C_{{\rm st},r}\ge1$, depending only on
	$C,\lambda,\epsilon$ (and on $r$ for $C_{{\rm st},r}$) and the
	standing geometric and dynamical data in
	\S\ref{ssec:assumptions_on_the_total_space} and
	\S\ref{sssec:systole_function}, such that, for every $m\geq0$ and
	every $y\in B^{\cW}(x,\rho_{\rm st})$,
	\begin{equation}
		\label{eq:stable:all:derivatives}
		\left\|\nabla^r\!\left((f_{m-1}\circ\cdots\circ f_0)
		\big|_{\mathcal W_{d-k}(x)}\right)(y)\right\|
		\le C_{{\rm st},r}e^{-\lambda m/2}.
	\end{equation}
	Here the empty composition is the identity, and covariant derivatives
	use the induced metric on $\mathcal W_{d-k}(x)$ and the ambient metric
	on $M$; at order one this is the differential restricted to the stable
	tangent space.
	Finally, there exists  a constant $C_{\rm Jac}<\infty$, depending
	only on $C,\lambda,\epsilon,M$, such that for $y \in
	B^{\cW}(x,\rho_{\rm st})$ the limit
	\begin{displaymath}
		J(x,y):=
		\lim_{n\to\infty}
		\frac{
			\Jac\!\left((f_{n-1}\cdots f_0)|_{\mathcal W_{d-k}(x)}\right)(y)
		}{
			\Jac\!\left((f_{n-1}\cdots f_0)|_{\mathcal W_{d-k}(x)}\right)(x)
		}
	\end{displaymath}
	exists and satisfies
	\begin{displaymath}
		C_{\rm Jac}^{-1}\le J(x,y)\le C_{\rm Jac}.
	\end{displaymath}
	The function $y \to J(x,y)$ is $C^\infty$ on $B^{\cW}(x,\rho_{\rm
		st})$. For every $r \ge 1$ there exists $C_{{\rm Jac},r} < \infty$,
	depending only on $r$,
	$(C,\lambda,\epsilon)$ and  the standing geometric
	and dynamical data in \S\ref{ssec:assumptions_on_the_total_space} and
	\S\ref{sssec:systole_function}, such that
	\begin{displaymath}
		\|D^r_{\cW,y} J(x,y)\| \le C_{{\rm Jac},r}
	\end{displaymath}
	for all $y \in B^{\cW}(x,\rho_{\rm st})$. 
\end{theorem}

\begin{proof}
	(Reduction of the geometric setup to \autoref{theorem:euclidean:stable:manifold:theorem}).
	Under the hypotheses of
	\autoref{theorem:stable:manifold:theorem}, with
	$\phi_a(x_i)\le Ce^{\sqrt{i}}$ and
	$0<\epsilon<\lambda/64$, there are rescaled coordinate maps to
	which \autoref{theorem:euclidean:stable:manifold:theorem} applies, with
	\[
	p=d-k,\qquad
	\alpha=e^{-7\lambda/8},\qquad
	\kappa=e^{-\lambda},\qquad
	\theta=e^{-3\lambda/4}.
	\]

	The elementary estimate
	\[
	\sup_{t\ge0}e^{A\sqrt{t}-\eta t}
	\le e^{A^2/(4\eta)}\qquad(A\ge0,\ \eta>0)
	\]
	will absorb every fixed-order systole bound into an exponential
	decay factor. We retain the given adapted norms, with their
	fixed tempering exponent $\epsilon$.
	Write $G_i=\exp_{x_{i+1}}^{-1}f_i\exp_{x_i}$ and $A_i=D_0G_i$.
	Choose splitting-preserving isometries
	\[
	I_i:\mathbb R^{d-k}\oplus\mathbb R^k
	\longrightarrow(T_{x_i}M,\|\cdot\|'_i),\qquad
	\|I_i\|,\|I_i^{-1}\|\le Ce^{\epsilon i}.
	\]
	Fix $\ell=1/10$ in \autoref{theorem:euclidean:stable:manifold:theorem}, and let $\delta$ be its smallness
	constant for the parameters above. Put
	\[
	\rho_i=\rho_0e^{-8\epsilon i},\qquad
	\widehat F_i(z)=
	\rho_{i+1}^{-1}I_{i+1}^{-1}G_i(I_i\rho_i z),
	\qquad \widehat L_i=D_0\widehat F_i.
	\]
	The stable and complementary blocks $B_i,D_i$ of
	$\widehat L_i$ satisfy
	\[
	\|B_i\|\le e^{-\lambda+8\epsilon}<\alpha,\qquad
	\|D_i^{-1}\|\|B_i\|\le\kappa.
	\]
	We next choose a single $\rho_0>0$ so that all the remaining
	hypotheses of \autoref{theorem:euclidean:stable:manifold:theorem} hold. First,
	\[
	r_d(x_i)\ge C^{-(e+e_0)}e^{-(e+e_0)\sqrt{i}},\qquad
	\rho_i\|I_i\|\le C\rho_0e^{-7\epsilon i}.
	\]
	The elementary estimate therefore lets us put
	$I_i\rho_i\overline Q$ strictly inside the dynamical chart
	domains for every $i$. Thus $\widehat F_i$ is defined on a
	neighborhood of $\overline Q$. Its normalized remainder is
	\[
	R_i(z)=\widehat L_i^{-1}\widehat F_i(z)-z
	=\rho_i^{-1}I_i^{-1}A_i^{-1}G_i(I_i\rho_i z)-z.
	\]
	The inverse-generator bound gives
	$\|A_i^{-1}\|\le\phi_a(x_{i+1})^{a_1}$. Consequently, for $r\ge2$,
	\[
	\begin{aligned}
		\|D^rR_i\|
		&\le \rho_i^{r-1}\|I_i^{-1}\|\|A_i^{-1}\|
		\|D^rG_i\|\|I_i\|^r\\
		&\le K_r\rho_0^{r-1}
		\exp\bigl(-(7r-9)\epsilon i+(a_r+a_1)\sqrt{i}\bigr).
	\end{aligned}
	\]
	Here $\sqrt{i+1}\le\sqrt{i}+1$ absorbs the shift in the
	inverse-generator estimate into $K_r$. Since $7r-9>0$ for
	every $r\ge2$, these derivatives are uniformly bounded in $i$
	for each $r$, without changing $\epsilon$ or $\rho_0$.
	Since $R_i(0)=D_0R_i=0$, decreasing $\rho_0$ also ensures
	$\sup_i\|R_i\|_{C^2(Q)}\le\delta$. Only low-order bounds are
	used for this choice; higher orders determine the constants $M_r$. 
	No restriction on the growth of the standing exponents $a_r$ with $r$ is needed.
	We also decrease $\rho_0$ so that the differentials of
	$\exp_{x_i}$ are uniformly comparable to isometries on these
	domains. Indeed, their variation from the identity, after
	parallel transport along radial paths, is bounded by
	$K\phi_a(x_i)^{e_2}\rho_i\|I_i\|$, using the second-jet chart
	bound. The same subexponential estimate makes this uniformly
	small.
	Apply \autoref{theorem:euclidean:stable:manifold:theorem} and denote its invariant graphs by $w_i$.
	Define
	\[
	\Sigma_i(s)=\exp_{x_i}(I_i\rho_i(s,w_i(s))),\qquad
	\mathcal W_{d-k}(x)=\Sigma_0(B_{d-k}(1/2)).
	\]
	These plaques are embedded, tangent at their centers to
	$\mathcal S(i)$, and satisfy
	$f_i\Sigma_i=\Sigma_{i+1}\beta_i$.
	At time zero the chart bounds, $\phi_a(x_0)\le C$, and the graph
	bounds give a controlled positive size and all the asserted
	$C^r$ geometry bounds. Choose a controlled intrinsic radius
	$\rho_{\rm st}$ so that its ball is contained in
	$\Sigma_0(B_{d-k}(1/4))$.
	Write $f^{(0)}=\Id$ and $f^{(m)}=f_{m-1}\circ\cdots\circ f_0$.
	\autoref{theorem:euclidean:stable:manifold:theorem} gives $|S_m(s)|\le\theta^m|s|$ and
	$\|D^jS_m\|\le K_j\theta^m$ for every $j\ge1$.
	The chart bounds and $\rho_m\|I_m\|\le C\rho_0e^{-7\epsilon m}$
	give $\|\nabla^j\Sigma_m\|\le K_j$: every term is bounded by
	$K_j e^{-7\epsilon m+A_j\sqrt m}$ for some fixed $A_j$.
	Thus $d(x_m,\Sigma_m(S_m(s)))\le K\theta^m|s|$.
	Applying the chain rule to
	$f^{(m)}\circ\Sigma_0=\Sigma_m\circ S_m$, and using the controlled
	geometry of $\Sigma_0$, proves \eqref{eq:stable:all:derivatives},
	since $\theta^m\le e^{-\lambda m/2}$.
	It remains to check the volume-density hypotheses in \autoref{theorem:euclidean:stable:manifold:theorem}.
	Let $v_i(s)$ be the Riemannian volume density of $\Sigma_i$.
	The chart comparison and differentiation of the volume density give
	\[
	\|D\log v_i\|\le K\left(
	\|I_i\|\|I_i^{-1}\|\|D^2w_i\|
	+\phi_a(x_i)^{e_2}\rho_i\|I_i\|^2\|I_i^{-1}\|
	\right).
	\]
	The graph term is at most $Ke^{2\epsilon i}$. The chart term has
	the uniform bound
	\[
	K\rho_0\exp(-5\epsilon i+e_2\sqrt{i})\le K.
	\]
	Thus
	\[
	\|D\log v_i\|\le K e^{2\epsilon i},\qquad
	e^{2\epsilon}\theta<1.
	\]
	At time zero all $C^r$ norms of $\log(v_0/v_0(0))$ are controlled
	by the chart and graph bounds. \autoref{theorem:euclidean:stable:manifold:theorem} applies with
	$H=K$ and $\tau=2\epsilon$.
	Its quantity $\mathcal J_m(s)$ is exactly the Riemannian
	Jacobian of $f^{(m)}|_{\mathcal W}$ at $\Sigma_0(s)$, because
	$f^{(m)}\Sigma_0=\Sigma_m S_m$. Hence
	\[
	J(x,\Sigma_0(s))=J_v(s).
	\]
	The bounds for $J_v$, together with the controlled time-zero
	geometry, give the required intrinsic derivative bounds, as
	well as uniform upper and lower bounds. This completes the
	reduction.
\end{proof}


\subsection{The abstract good-event formulation}
\label{ssec:abstract:good:event}

We now isolate the linear-probabilistic statement which will be proved in the
rest of this section. The purpose of the formulation is to separate the
finite-dimensional linear estimates from the final compactness argument which
produces a finite collection of words.

We work with random linear isomorphisms
\begin{displaymath}
A_i:V_i\to V_{i+1},\qquad i\geq 0,
\end{displaymath}
between $d$-dimensional Euclidean spaces, adapted to a filtration
$\{\mathcal F_i\}_{i\geq 0}$, where $\mathcal F_0$ is trivial.  Thus
$A_i$ is $\mathcal F_{i+1}$-measurable, and $\mathcal F_i$ is the information
available before applying $A_i$.

We assume that the process is generated by a measurable transition rule assigning a probability distribution for the next step to every finite history, including histories of probability zero. Throughout, conditional probabilities and expectations are computed using this rule and its iterates, with the tower property holding pointwise.

For integers $t\ge 0$ and $\ell\geq1$, set $A(t,0):=\Id_{V_t}$ and write
\begin{displaymath}
A(t,\ell):=A_{t+\ell-1}\cdots A_t \qquad A(n) := A(0,n).
\end{displaymath}
In the non-compact setting a uniform bound on
$\|A_i\|+\|A_i^{-1}\|$ is replaced by a height process.  We assume that
$H_i\geq 1$ is $\mathcal F_i$-measurable and that there are constants
$0<c<1$ and $b<\infty$ such that
\begin{equation}
\label{eq:abstract:margulis:height:drift}
\mathbb E(H_{i+1}\mid \mathcal F_i)\leq cH_i+b
\end{equation}
for all $i\geq 0$.  We also assume that there is a constant
$B_{\rm lin}<\infty$ such that
\begin{equation}
\label{eq:height:controlled:operator:norm}
\log\left(\|A_i\|+\|A_i^{-1}\|\right)
\leq
B_{\rm lin}\bigl(1+\log H_i\bigr)
\end{equation}
for all $i\geq 1$.

To record the systole information needed for
\autoref{theorem:stable:manifold:theorem}, introduce an
$\mathcal F_i$-measurable process $\Phi_i\geq2$ and fixed constants
$\kappa_\Phi,\delta_\Phi>0$ and $A_\Phi\geq1$ such that, for every
history and every $i\geq1$,
\begin{equation}
\label{eq:abstract:systole:height}
H_i\geq\kappa_\Phi\Phi_i^{\delta_\Phi},
\end{equation}
\begin{equation}
\label{eq:abstract:systole:growth}
A_\Phi^{-1}\Phi_i\leq\Phi_{i+1}\leq A_\Phi\Phi_i.
\end{equation}

After increasing $B_{\rm lin}$ by a constant depending only on $d$ and $k$,
\eqref{eq:height:controlled:operator:norm} implies
\begin{equation}
\label{eq:height:controlled:alpha:bounds}
\left|\alpha_k'(A_i,F')\right|
+
\left|\alpha_k''(A_i,F'')\right|
\leq
B_{\rm lin}\bigl(1+\log H_i\bigr)
\end{equation}
for all $F'\in\Flag_k'(V_i)$ and all $F''\in\Flag_k''(V_i)$.

We assume the following conditional drift estimates.  There is a constant
$c_0>0$ such that, for every $i\geq 1$, for every
$\mathcal F_i$-measurable random flag
$F_i''\in\Flag_k''(V_i)$,
\begin{equation}
\label{eq:height:controlled:alphapp:drift}
\mathbb E\left(\alpha_k''(A_i,F_i'')\mid\mathcal F_i\right)\geq c_0,
\end{equation}
and for every $\mathcal F_i$-measurable random flag
$F_i'\in\Flag_k'(V_i)$,
\begin{equation}
\label{eq:height:controlled:alphap:drift}
\mathbb E\left(\alpha_k'(A_i,F_i')\mid\mathcal F_i\right)\geq c_0.
\end{equation}

We require (\ref{eq:abstract:margulis:height:drift}), (\ref{eq:height:controlled:alphapp:drift}), and (\ref{eq:height:controlled:alphap:drift}) to hold after every fixed finite history and for every relevant flag. All subsequent conditional estimates likewise hold for every history satisfying their stated hypotheses, with conditional probabilities and expectations interpreted as above.

In the compact case one may take $H_i\equiv 1$.

\begin{definition}
\label{def:height:controlled}
A finite sequence
\begin{displaymath}
0=s_0<s_1<s_2<\cdots<s_m
\end{displaymath}
is $(R,K)$-\emph{height-controlled} if $H_{s_1} \le R$, for $2 \le
i \le m$, 
\begin{displaymath}
H_{s_i}<R
\end{displaymath}
and, for every $1\leq i < m$,
\begin{displaymath}
\max_{0\leq r\leq s_{i+1}-s_i}
\log\max\left\{
\|A(s_i,r)\|,
\|A(s_i,r)^{-1}\|
\right\}
< K(s_{i+1}-s_i).
\end{displaymath}
\end{definition}

\begin{definition}
\label{def:adapted:to:O}
Let
\begin{displaymath}
\mathcal O\subset \Gr_k(\bR^d) \times GL(d,\bR)
\end{displaymath}
be invariant under orthogonal changes of coordinates:
for every $U,V\in O(d)$,
\begin{displaymath}
(E,Q)\in\mathcal O
\quad\Longleftrightarrow\quad
\left(UE,VQU^{-1}\right)\in\mathcal O.
\end{displaymath}
Given
\begin{displaymath}
0=s_0<s_1<s_2<\cdots<s_m,
\end{displaymath}
put
\begin{displaymath}
B_i:=A(s_i,s_{i+1}-s_i),
\qquad
0\leq i<m.
\end{displaymath}
We say that the sequence is \emph{adapted to $\mathcal O$} if
\begin{displaymath}
(\im_k(B_{i-1}),B_i)\in\mathcal O   \qquad  \text{(and
  $\im_k(B_{i-1})$ is defined)}
\end{displaymath}
for every $1\leq i<m$.
\end{definition}

\begin{proposition}[Good block events]
\label{prop:good:block:events}
There exist constants
\begin{displaymath}
\lambda>0,
\qquad
R_{\rm init}<\infty, \qquad \epsilon_{\rm init} > 0
\end{displaymath}
depending only on the constants of
\S\ref{ssec:abstract:good:event}, with the following properties.

\begin{enumerate}
\item[{\rm (a)}]
For every $\epsilon>0$, let $\mathcal T_\infty(\epsilon)$ be the set
of paths $\omega$ such that there exists $C=C(\omega,\epsilon)$ so
that the following hold:
\begin{itemize}
\item  The cocycle admits a $(C,\lambda,\epsilon)$-uniformly tempered splitting in
dimension $k$.
\item $\Phi_i \le C e^{\sqrt{i}}$ for every $i \ge 1$.
\item $H_i<R_{\rm init}$ for infinitely many $i$.
\end{itemize}
Then $\bP(\cT_\infty(\epsilon)) = 1$.
\item[{\rm (b)}]
Suppose
\begin{displaymath}
H_1 \le R_{\rm init}, \text{ and } \Delta_k(A_0) > e^{\epsilon_{\rm init}}.
\end{displaymath}
For every $\epsilon>0$, there exist an integer $N\geq1$, constants
\begin{displaymath}
K_{\rm  blk}<\infty,
\qquad
p_0>0,
\qquad
\beta>0,
\qquad
C_{\rm init}<\infty,
\end{displaymath}
and an invariant under orthogonal changes of coordinates open set
\begin{displaymath}
\mathcal O
\subset
\Gr_k(\bR^d)\times GL(d,\bR)
\end{displaymath}
with the property
\begin{equation}
\label{eq:good:block:events:output:gap}
(E,Q)\in\mathcal O\Rightarrow
\Delta_k(Q)>e^{\epsilon_{\rm init}}, 
\end{equation}
all depending only on $\epsilon$, $R_{\rm init}$, and the constants of
\S\ref{ssec:abstract:good:event}, but not on the fixed initial prefix
$A_0$, such that the following holds:

Let $m\geq1$, and let
\begin{displaymath}
0=s_0<s_1=1<s_2<\cdots<s_m
\end{displaymath}
be integer-valued stopping times satisfying
\begin{displaymath}
\frac N2<s_{i+1}-s_i\leq N
\end{displaymath}
for every $1\leq i<m$.  Define
\begin{displaymath}
\mathcal E^{(m)}(s_0,\ldots,s_m)
\end{displaymath}
to be the event that the sequence
$s_0,\ldots,s_m$
is $(R_{\rm init},K_{\rm  blk})$-height-controlled and adapted to $\mathcal
O$. Then, on $\mathcal E^{(m)}(s_0,\ldots,s_m)$,
\begin{equation}
\label{eq:good:block:events:extension}
\mathbb P\left(
\bigcup_{\substack{\ell\in\N\\ N/2<\ell\leq N}}
\mathcal E^{(m+1)}
(s_0,\ldots,s_m,s_m+\ell)
\,\middle|\,
\mathcal F_{s_m}
\right)
\geq p_0.
\end{equation}
Moreover, for every sufficiently large $m$, there is an event
\begin{displaymath}
\mathcal T^{(m)}(s_0,\ldots,s_m)
\end{displaymath}
such that, on
$\mathcal E^{(m)}(s_0,\ldots,s_m)$,
\begin{equation}
\label{eq:good:block:events:tempered:extension}
\mathbb P\left(
\mathcal T^{(m)}(s_0,\ldots,s_m)
\,\middle|\,
\mathcal F_{s_m}
\right)
\geq
1-e^{-\beta\sqrt m}.
\end{equation}
On
$\mathcal E^{(m)}(s_0,\ldots,s_m) \cap \mathcal T^{(m)}(s_0,\ldots,s_m)$,
the full cocycle admits a
$(C_{\rm temp},\lambda,\epsilon)$-uniformly tempered splitting in dimension
$k$, where $C_{\rm temp}<\infty$ may depend only on the fixed initial prefix
$A_0$, $\epsilon$, $R_{\rm init}$ and the constants of
\S\ref{ssec:abstract:good:event}. In  addition, on $\mathcal
E^{(m)}(s_0,\ldots,s_m) \cap \mathcal T^{(m)}(s_0,\ldots,s_m)$, one
has $H_j < R_{\rm init}$ for infinitely many $j$, and
\begin{displaymath}
\Phi_i \le C_{\rm temp} e^{\sqrt{i}}, \qquad i \ge 1.  
\end{displaymath}
In  particular, $\mathcal
E^{(m)}(s_0,\ldots,s_m) \cap \mathcal T^{(m)}(s_0,\ldots,s_m) \subset
\cT_\infty(\epsilon)$. 

The corresponding stable space at time $0$ satisfies
\begin{displaymath}
\dist_{\Gr}\bigl(
\mathcal S(0),\ker_{d-k}(A_0)
\bigr)
\leq
C_{\rm init}
\min\left\{
1,\Delta_k(A_0)^{-1}
\right\}.
\end{displaymath}
\end{enumerate}
\end{proposition}

\noindent The proof of \autoref{prop:good:block:events} will take the entirety of \autoref{sec:good:words}.
We assume it for now and use it to derive \autoref{theorem:Good-stables-everywhere} and \autoref{theorem:good:words} in \autoref{ssec:Proof-of-theorems-on-stable-manifolds}.


\subsection{Proof of the main statements  on stable manifolds}
\label{ssec:Proof-of-theorems-on-stable-manifolds}

\subsubsection{Proof of \autoref{theorem:Good-stables-everywhere}}
	\label{sssec:Proof-of-stables-everywhere}
Let $\lambda>0$ be the constant from Proposition~\ref{prop:good:block:events}.
Choose
\begin{displaymath}
0<\epsilon_{\rm temp}<\frac{\lambda}{4}
\end{displaymath}
small enough so that
\autoref{theorem:stable:manifold:theorem} applies.

Fix $x\in M$, and let $\tilde w\in\Omega$ be an infinite word.
Write $x_i=f_{\tilde w}^i(x)$,  
\begin{displaymath}
A_i
:=
D f_{\tilde w_i}\bigl(x_i):
T_{x_i}M
\to
T_{x_{i+1}}M .
\end{displaymath}
Set $H_i=h(x_i)$ and $\Phi_i=\phi_a(x_i)$.
Use the background Riemannian norms and let
$\mathcal F_i=\sigma(\tilde w_0,\ldots,\tilde w_{i-1})$, with
$\mathcal F_0$ trivial.

Then $H_i,\Phi_i$ are $\mathcal F_i$-measurable
and $A_i$ is $\mathcal F_{i+1}$-measurable. After every fixed history,
the next letter has law $\mu$, giving the required transition rule.

We now verify the conditions of \S\ref{ssec:abstract:good:event}. 
The height drift \eqref{eq:abstract:margulis:height:drift} follows from
\eqref{eq:def:random:margulis}. The operator bound
\eqref{eq:height:controlled:operator:norm} follows from
(\ref{eq:chart:normalization}) and (\ref{eq:norm:Gsq:bound}),
applied to $f_{\tilde w_i}$ at $x_i$ and its inverse at $x_{i+1}$,
together with \eqref{eq:systole:growth} and
\eqref{eq:def:random:margulis:systole-domination}.
The absolute cocycle bound \eqref{eq:height:controlled:alpha:bounds}
then follows as in the abstract setup. Finally,
\eqref{eq:height:controlled:alphapp:drift} and
\eqref{eq:height:controlled:alphap:drift} follow respectively from
\eqref{eq:uniform:gap:second:derivative} and
\eqref{eq:uniform:gap:negative}, which already hold with $N=1$.
These estimates hold for every fixed history and every relevant flag.

The abstract bounds \eqref{eq:abstract:systole:height} and
\eqref{eq:abstract:systole:growth} follow respectively from
\eqref{eq:def:random:margulis:systole-domination} and
\eqref{eq:systole:growth}, the latter applied to a generator and its inverse.
Thus, we can apply Proposition~\ref{prop:good:block:events}.
For the exponent $\epsilon_{\rm temp}$ fixed above, define
\[
\Omega_\infty(x):=\mathcal T_\infty(\epsilon_{\rm temp}),
\]
where the event is evaluated for this geometric process, with time
indexed by the individual letters of $\tilde w$.
Part~{\rm (a)} gives $\mu^{\otimes\mathbb N}(\Omega_\infty(x))=1$.
For every $\tilde w\in\Omega_\infty(x)$, its defining properties give
the uniformly tempered splitting, the systole bound for $i\geq1$,
and infinitely many returns below $R_{\rm init}$. Enlarging the common
prefactor $C$ to include $\phi_a(x)$ gives the systole bound also at
$i=0$.

Let
\begin{displaymath}
T_xM=\mathcal U(x,\tilde w)\oplus \mathcal S(x,\tilde w)
\end{displaymath}
be the splitting at time $0$. We apply
Theorem~\ref{theorem:stable:manifold:theorem} to the non-autonomous sequence
\begin{displaymath}
f_i=f_{\tilde w_i},
\qquad
x_i=f_{\tilde w}^i(x).
\end{displaymath}
The derivative cocycle has a
$(C,\lambda,\epsilon_{\rm temp})$-uniformly tempered splitting, and
$\phi_a(x_i)\leq Ce^{\sqrt{i}}$ for every $i\geq0$.
These are precisely the two orbit-dependent hypotheses of
\autoref{theorem:stable:manifold:theorem}. Hence
Theorem~\ref{theorem:stable:manifold:theorem} gives a $C^\infty$ embedded
local stable manifold
\begin{displaymath}
\mathcal W_{d-k}(x,\tilde w)\subset M
\end{displaymath}
of dimension $d-k$, containing $x$, tangent to
$\mathcal S(x,\tilde w)$ at $x$, such that for every
$y\in\mathcal W_{d-k}(x,\tilde w)$,
\begin{displaymath}
\limsup_{n\to\infty}
\frac1n
\log d\bigl(f_{\tilde w}^n(x),f_{\tilde w}^n(y)\bigr)
\le
-\lambda/2.
\end{displaymath}
The remaining assertions in
Theorem~\ref{theorem:Good-stables-everywhere} follow from
Lemma~\ref{lemma:uniformly:tempered:kernel:convergence}. 
\hfill \qed

\begin{lemma}
\label{lemma:expanding:initial:map}
Let $A_i:V_i\to V_{i+1}$ admit a
$(C,\lambda,\epsilon)$-uniformly tempered splitting in dimension $k$.
Factor $A_0=B_{L-1}\cdots B_0$, where $L\geq1$, through
$d$-dimensional Euclidean intermediate spaces, with
$\|B_\ell\|,\|B_\ell^{-1}\|\leq K$ for $0\leq\ell<L$ and $K\geq1$.
Then the expanded cocycle
\[
B_0,\ldots,B_{L-1},A_1,A_2,\ldots
\]
admits a $(C',\lambda,\epsilon)$-uniformly tempered splitting,
preserving both subspaces at every original time, where one may take
$C'=\max\{2C^2e^{\epsilon+\lambda L}K^L, eCK^L\}$.
Moreover, if nonnegative scalars $\Phi_i\leq Ce^{\sqrt{i}}$
for $i\geq1$ are expanded to a sequence satisfying
\[
\widehat\Phi_{L+i-1}=\Phi_i\quad(i\geq1),\qquad
\widehat\Phi_\ell\leq K\widehat\Phi_{\ell+1}\quad(0\leq\ell<L),
\]
then $\widehat\Phi_\ell\leq C'e^{\sqrt{\ell}}$ for all $\ell\geq0$.
\end{lemma}
\begin{proof}
Retain the adapted norms at $V_i$ for $i\geq1$.
For $\ell<L$, put $P_\ell=B_{L-1}\cdots B_\ell$, pull back
$\mathcal U(1)$ and $\mathcal S(1)$ by $P_\ell$, and for $u,s$
in these respective subspaces set
\[
\bigl(\|u+s\|_\ell^*\bigr)^2
=\bigl(\|P_\ell u\|'_1\bigr)^2
+e^{2\lambda(L-\ell)}\bigl(\|P_\ell s\|'_1\bigr)^2.
\]
Along the inserted prefix, the maps on $\mathcal U$ are isometries,
those on $\mathcal S$ contract by exactly $e^{-\lambda}$, and
transversality has constant $1$. The bounds
$\|P_\ell^{\pm1}\|\leq K^L$, together with
\eqref{eq:uniformly:tempered:norms} and
\eqref{eq:uniformly:tempered:angle} at time $1$, give comparison
with the background norms with constant at most $C'$.
For $\ell\geq L$, the old time is $i=\ell-L+1\leq\ell$, so the
original estimates imply the required ones.
Since $P_0=A_0$, both original subspaces at time $0$ are preserved.
Finally, for $\ell<L$,
$\widehat\Phi_\ell\leq K^{L-\ell}\Phi_1\leq K^LCe^\epsilon\leq C'$;
the scalar bound at later times follows from the same index shift.
\end{proof}

\subsubsection{Proof of \autoref{theorem:good:words}}
\label{sssec:Proof-of-Good-Words}
Let $\lambda$, $R_{\rm init}$, $\epsilon_{\rm init}$
be the constants from \autoref{prop:good:block:events}.
Apply \autoref{prop:good:block:events}{\rm (b)} with the exponent
$\epsilon_{\rm temp}$ fixed in
\S\ref{sssec:Proof-of-stables-everywhere}, to obtain  
$N$, $K_{\rm blk}$, $\beta$, $C_{\rm init}$,
and an orthogonally invariant open set
$\mathcal O\subset \Gr_k(\bR^d)\times GL(d,\bR)$.
Put
\begin{displaymath}
K_{\rm init}:=\{q\in M\st h(q)\leq R_{\rm init}\},
\end{displaymath}
and let
\begin{displaymath}
\mathfrak G_{\rm init}
:=
\{(y,E)\st y\in K_{\rm init},\ E\in\Gr_k(T_yM)\}.
\end{displaymath}
This is compact.

For a word $v$ with
\begin{displaymath}
\frac N2<|v|\leq N,
\end{displaymath}
let $\mathscr U_v\subset\mathfrak G_{\rm init}$ be the set of pairs
$(y,E)$ such that the following hold.  First,
\begin{equation}
\label{eq:good:words:block:return}
h(f_v^{|v|}(y))<R_{\rm init}.
\end{equation}
Second, the block $D_yf_v^{|v|}$ is $K_{\rm blk}$-tame.  Third,
$(E,D_y f_v^{|v|}) \in \cO$. 
The sets $\mathscr U_v$ are open.

The sets $\mathscr U_v$, with $N/2<|v|\leq N$, cover
$\mathfrak G_{\rm init}$.  Indeed, fix
\begin{displaymath}
(y,E)\in\mathfrak G_{\rm init}.
\end{displaymath}
Choose a linear isomorphism $P:\bR^d\to T_yM$ with
\begin{displaymath}
\im_k(P)=E,
\qquad
\Delta_k(P)>e^{\epsilon_{\rm init}}.
\end{displaymath}
Set $H_1=h(y)$ and
$H_0=\max\{1,(h(y)-b)/c\}$. 
Regard $P$ as the initial prefix $A_0$.
Let $y_j=f_\omega^j(y)$. For $i\geq1$, put
\[
H_i=h(y_{i-1}),\qquad \Phi_i=\phi_a(y_{i-1}),
\]
and let $A_i$ be the derivative of the next letter at $y_{i-1}$.
At abstract time $i\geq1$, the filtration reveals the first $i-1$
random letters, and $\mathcal F_0$ is trivial.
The verification in \S\ref{sssec:Proof-of-stables-everywhere} applies
for every $i\geq1$. At time $0$,
\eqref{eq:abstract:margulis:height:drift} follows from
$H_1\leq cH_0+b$, by the chosen value of $H_0$.
The operator and systole bounds start at $i=1$, so they impose no
additional restriction on the artificial prefix $P$. We can thus apply
Proposition~\ref{prop:good:block:events}(b). In particular, the $m=1$  case of 
\eqref{eq:good:block:events:extension} produces a word $v$ of length
$N/2<|v|\leq N$ with $h(f_v^{|v|}(y))<R_{\rm init}$, $D_yf_v^{|v|}$ is
$K_{\rm blk}$-tame and  $(E,D_yf_v^{|v|})\in\mathcal O$. Then, 
$(y,E)\in\mathscr U_v$. 

Choose a finite subcover
\begin{displaymath}
\mathfrak G_{\rm init}\subset\bigcup_{v\in J}\mathscr U_v.
\end{displaymath}
Thus $J$ is a finite collection of words, each satisfying
\begin{displaymath}
\frac N2<|v|\leq N.
\end{displaymath}
Fix a metric on $\mathfrak G_{\rm init}$, and let
\begin{displaymath}
\lambda_{\rm Leb}>0
\end{displaymath}
be a Lebesgue number for this cover.  Thus, for every
$\xi\in\mathfrak G_{\rm init}$, there exists $v\in J$ such that
\begin{displaymath}
B(\xi,\lambda_{\rm Leb})\subset\mathscr U_v.
\end{displaymath}
Let
\begin{displaymath}
\mathscr U_v' = \{ \xi \in \mathscr U_v \st \dist(\xi, \mathfrak G_{\rm
  init} \setminus\mathscr U_v)\geq\lambda_{\rm Leb}\}.
\end{displaymath}
Then the $\mathscr U_v'$ are compact and $J$ is finite. Therefore, there exists
$\epsilon_0 > 0$ such that for any  $v \in  J$ and  any $(y,E) \in
\mathscr U_v'$, $h(f_v^{|v|}(y)) < R_{\rm init} - \epsilon_0$. 

For a point $y\in M$ and a finite word $u$, write
\begin{displaymath}
x_u(y):=f_u^{-|u|}(y).
\end{displaymath}
Whenever
\begin{displaymath}
h(y) \le R_{\rm init}
\end{displaymath}
and
\begin{displaymath}
\Delta_k\left(D_{x_u(y)}f_u^{|u|}\right)>e^{\epsilon_{\rm init}/2},
\end{displaymath}
define
\begin{displaymath}
\Xi(y,u):=
\left(
y,\,
\im_k\left(D_{x_u(y)}f_u^{|u|}\right)
\right)
\in\mathfrak G_{\rm init}.
\end{displaymath}

For notational reasons we introduce $\Xi^+$ as the restriction of
$\Xi$ to the smaller domain,
\[
\Delta_k\!\left(D_{x_u(y)}f_u^{|u|}\right)>e^{\epsilon_{\rm init}}.
\]
Thus $\Xi^+(y,u)=\Xi(y,u)$ wherever $\Xi^+$ is defined.

For each $L\geq1$, choose $\rho_{\rm init}(L)>0$ sufficiently small so that
the following initial-prefix robustness holds.  If $|u|=L$, 
$\Xi^+(f_u^L(x),u)$ is defined, and
\begin{displaymath}
x'\in B(x,\rho_{\rm init}(L)) \qquad \text{with
}h(f_u^L(x')) \le R_{\rm init}, 
\end{displaymath}
then $\Xi(f_u^L(x'),u)$ is defined and
\begin{displaymath}
d\left(
\Xi(f_u^L(x'),u),
\Xi(f_u^L(x),u)
\right)
<
\frac{\lambda_{\rm Leb}}{2}.
\end{displaymath}
This is possible by compact support of $\mu$, the finite-time compactness of
words of length $L$ whose endpoint lies in $K_{\rm init}$, and the uniform
continuity controlled by the 
``the fixed-length derivative bounds of
Remark~\ref{rmk:systole_fixed_length_words}''.
We now choose
\begin{displaymath}
0 < \epsilon <\min_{1 \le L \le N} \rho_{\rm init}(L). 
\end{displaymath}
We also decrease $\epsilon$ so that for  any $z$ with $h(z) < R_{\rm
  init} - \epsilon_0$ and  any $z'$  with $d(z,z') < \epsilon$, $h(z')
< R_{\rm init}$. Then, for every $v\in J$ and
$(y,E)\in\mathscr U_v'$, $d(f_v^{|v|}(y),z)<\epsilon$ implies
$h(z)<R_{\rm init}$.

We now define $w_{\rm next}$.  If $\Xi(y,u)$ is defined, choose
\begin{displaymath}
w_{\rm next}(y,u)\in J
\end{displaymath}
so that
\begin{displaymath}
B\bigl(\Xi(y,u),\lambda_{\rm Leb}\bigr)
\subset
\mathscr U_{w_{\rm next}(y,u)}.
\end{displaymath}
If $\Xi(y,u)$ is not defined but 
$y\in K_{\rm init}$ choose any
$E\in\Gr_k(T_y M)$, select $v\in J$ with
$B((y,E),\lambda_{\rm Leb})\subset\mathscr U_v$ and set
$w_{\rm next}(y,u):=v$.
(This gives
\eqref{eq:good:words:uniform:return} for every incoming word.) 
If $y \not\in  K_{\rm init}$, define $w_{\rm next}(y,u)$ arbitrarily with
values in $J$.

Now let $w_0$, $x$, $n$, and
\begin{displaymath}
w=w_n\cdots w_1w_0
\end{displaymath}
satisfy the hypotheses of the theorem.  Consider the derivative cocycle along
the orbit of $x$, with initial block
\begin{displaymath}
A_0:=D_xf_{w_0}^{|w_0|}.
\end{displaymath}
Set $H_1=h(f_{w_0}^{|w_0|}(x))$ and choose the auxiliary
$H_0=\max\{1,(H_1-b)/c\}$. Set
\begin{displaymath}
s_0:=0,\qquad s_1:=1,
\end{displaymath}
and, for $1\leq j\leq n$, set
\begin{displaymath}
s_{j+1}:=1+\sum_{r=1}^j |w_r|.
\end{displaymath}

Put $L=|w_0|$ and $q_\ell=f_{\tilde w}^{\ell}(x)$. For $i\geq1$, set
\[
H_i=h(q_{L+i-1}),\qquad \Phi_i=\phi_a(q_{L+i-1}).
\]
The initial word $w_0$ is fixed; at abstract time $i\geq1$, reveal
the first $i-1$ letters after $w_0$, and take $\mathcal F_0$ trivial.
All the verifications in \S\ref{sssec:Proof-of-stables-everywhere}
apply after this index shift for $i\geq1$, with the same constants.
At time $0$, the chosen $H_0$ gives $H_1\leq cH_0+b$, verifying
\eqref{eq:abstract:margulis:height:drift}.
The bounds \eqref{eq:height:controlled:operator:norm},
\eqref{eq:abstract:systole:height}, and
\eqref{eq:abstract:systole:growth} are required only for $i\geq1$.
Thus, we can apply Proposition~\ref{prop:good:block:events}(b). 

We claim that
\begin{displaymath}
\mathcal E^{(n+1)}(s_0,\ldots,s_{n+1})
\end{displaymath}
holds for this derivative cocycle.  The length conditions follow from
\begin{displaymath}
\frac N2<|w_j|\leq N
\qquad
(1\leq j\leq n).
\end{displaymath}
The initial height and gap conditions are exactly the assumptions
\begin{displaymath}
h(f_{w_0}^{|w_0|}(x)) \le R_{\rm init},
\qquad
\Delta_k(D_xf_{w_0}^{|w_0|})>e^{\epsilon_{\rm init}}.
\end{displaymath}

For $0\leq i\leq n$, put
\begin{displaymath}
z_i:=f_{w_i\cdots w_0}^{|w_i\cdots w_0|}(x).
\end{displaymath}
We prove by induction that, for every $0\leq i\leq n$, the actual incoming datum
$\Xi^+(z_i,w_i)$ is defined and, for $1\leq i\leq n$,
\begin{displaymath}
\Xi^+(z_{i-1},w_{i-1})\in\mathscr U_{w_i}.
\end{displaymath}
For $i=0$, this follows from the initial height and gap assumptions.

Assume $1\leq i\leq n$ and $\Xi^+(z_{i-1},w_{i-1})$ is defined.  By hypothesis there exists
\begin{displaymath}
x_{i-1}\in BB(w_{i-1}\cdots w_0,x,\epsilon)\cap
B(x,\rho_{\rm init}(|w_0|))
\end{displaymath}
such that, with
\begin{displaymath}
y_{i-1}:=f_{w_{i-1}\cdots w_0}^{|w_{i-1}\cdots w_0|}(x_{i-1}),
\end{displaymath}
one has
\begin{displaymath}
h(y_{i-1}) \le R_{\rm init},
\qquad
w_i=w_{\rm next}(y_{i-1},w_{i-1}).
\end{displaymath}

If $i=1$, the definition of $\rho_{\rm init}(|w_0|)$ gives
\begin{displaymath}
d\bigl(\Xi(y_0,w_0),\Xi^+(z_0,w_0)\bigr)
<
\frac{\lambda_{\rm Leb}}{2}.
\end{displaymath}
If $i\geq2$, then $w_{i-1}\in J$, and the choice of $\epsilon$ gives
\begin{displaymath}
d\bigl(\Xi(y_{i-1},w_{i-1}),\Xi^+(z_{i-1},w_{i-1})\bigr)
<
\frac{\lambda_{\rm Leb}}{2}.
\end{displaymath}
In either case, since
\begin{displaymath}
w_i=w_{\rm next}(y_{i-1},w_{i-1})
\end{displaymath}
and
\begin{displaymath}
B\bigl(\Xi(y_{i-1},w_{i-1}),\lambda_{\rm Leb}\bigr)
\subset
\mathscr U_{w_i},
\end{displaymath}
we get
\begin{displaymath}
\Xi(z_{i-1},w_{i-1}) = \Xi^+(z_{i-1},w_{i-1})\in\mathscr U_{w_i}.
\end{displaymath}
Since $x_{w_i}(z_i)=z_{i-1}$,
\eqref{eq:good:words:block:return} and
\eqref{eq:good:block:events:output:gap} give
\[
h(z_i)<R_{\rm init},
\qquad
\Delta_k\!\left(D_{x_{w_i}(z_i)}f_{w_i}^{|w_i|}\right)
>e^{\epsilon_{\rm init}}.
\]
Thus $\Xi^+(z_i,w_i)$ is defined, proving the induction
step, and hence the claim.

By \autoref{prop:good:block:events}{\rm (b)}, for all sufficiently large
$n$ there is an event
\begin{displaymath}
\mathcal T^{(n+1)}(s_0,\ldots,s_{n+1})
\end{displaymath}
such that
\begin{displaymath}
\mathbb P\left(
\mathcal T^{(n+1)}(s_0,\ldots,s_{n+1})
\,\middle|\,
\mathcal F_{s_{n+1}}
\right)
\geq
1-e^{-\beta\sqrt{n+1}}.
\end{displaymath}
On this event the full derivative cocycle admits a uniformly tempered
splitting in dimension $k$, and its stable space $\mathcal S(0)$ satisfies
\begin{displaymath}
\dist_{\Gr}\bigl(
\mathcal S(0),\ker_{d-k}(A_0)
\bigr)
\leq
C_{\rm init}
\min\left\{
1,\Delta_k(A_0)^{-1}
\right\}.
\end{displaymath}

Restricting to the cylinder $\Omega(w)$, define
\begin{displaymath}
G(w,x):=
\Omega(w)\cap
\mathcal T^{(n+1)}(s_0,\ldots,s_{n+1}),
\qquad
B(w,x):=\Omega(w)\setminus G(w,x).
\end{displaymath}
After decreasing $\beta$ if necessary,
\begin{displaymath}
\mathbb P_w(B(w,x))\leq e^{-\beta\sqrt n}.
\end{displaymath}

For every $\tilde w\in G(w,x)$, the event $\mathcal E^{(n+1)}$ also
holds. Since $H_1\leq R_{\rm init}$,
\eqref{eq:def:random:margulis:systole-domination} and
\eqref{eq:systole:growth} give
\[
\phi_a(q_\ell)\leq A_\phi^{L-\ell}
(R_{\rm init}/\kappa)^{1/\delta}\qquad(0\leq\ell\leq L).
\]
By \eqref{eq:chart:normalization} and the $k=1$ case of
\eqref{eq:norm:Gsq:bound}, the prefix derivatives and their inverses
are uniformly bounded in terms of $L$ and the fixed data.
Apply Lemma~\ref{lemma:expanding:initial:map} with
$C=C_{\rm temp}$, with $B_0,\ldots,B_{L-1}$ the one-letter
derivatives along $w_0$, and with
$\widehat\Phi_\ell=\phi_a(q_\ell)$.
Choose $K\geq A_\phi$ also bounding those derivatives and their
inverses; \eqref{eq:systole:growth} for inverse letters verifies the
prefix scalar bound. We obtain a
$(C_*,\lambda,\epsilon_{\rm temp})$-uniformly tempered splitting and
$\phi_a(q_\ell)\leq C_*e^{\sqrt{\ell}}$ for all $\ell\geq0$
in the original indexing, with the same stable space $\mathcal S(0)$.
Together with the infinitely many returns supplied by
\autoref{prop:good:block:events}{\rm (b)},
these are the defining conditions for $\Omega_\infty(x)$ in
\S\ref{sssec:Proof-of-stables-everywhere}. Hence
$G(w,x)\subseteq\Omega_\infty(x)$ pointwise, and we may apply
\autoref{theorem:stable:manifold:theorem} at $x$, with
parameters $(C_*,\lambda,\epsilon_{\rm temp})$.

Thus the stable manifold
\begin{displaymath}
\mathcal W_{d-k}(x,\tilde w)
\end{displaymath}
exists, with
\begin{displaymath}
T_x\mathcal W_{d-k}(x,\tilde w)=\mathcal S(0).
\end{displaymath}
In the proof of \autoref{prop:admissible:to:uniformly:tempered},
\eqref{eq:admissible:to:ut:norm:tempering} for $i\geq t_1=1$
is uniform, and the extension to time $0$ uses only
$A_0^{\pm1}$ and $e^\lambda$.
The prefix bounds above therefore make $C_{\rm temp}$ uniform in
$x$, $n$ and the good continuation; its enlargement for the
$\Phi_i$ bound is uniform as well.
The abstract expansion lemma now gives $C_*$ depending only on
$L=|w_0|$ and the fixed data.
The size, $C^r$-geometry and stable-Jacobian distortion bounds follow
directly from
\autoref{theorem:stable:manifold:theorem}, whose output constants depend
only on $(C_*,\lambda,\epsilon_{\rm temp})$, with the standing data
fixed (and on $r$ for $C^r$ geometry).

Finally, since
\begin{displaymath}
A_0=D_xf_{w_0}^{|w_0|},
\end{displaymath}
the stable-space estimate above gives
\begin{displaymath}
\dist_{\Gr}\!\bigl(
\ker_{d-k}(Df_{w_0}^{|w_0|}(x)),
T_x\mathcal W_{d-k}(x,\tilde w)
\bigr)
\leq
C_{\rm init}
\min\left\{
1,\Delta_k(Df_{w_0}^{|w_0|}(x))^{-1}
\right\}.
\end{displaymath}
\hfill\qed



\section{Good words and their extensions}
\label{sec:good:words}

In this section, we prove Proposition~\ref{prop:good:block:events}.
Throughout this subsection we use the standing assumptions from
\autoref{ssec:abstract:good:event}.

\subsection{Singular values and averages over the Grassmanian}
\label{ssec:singular:average}

\begin{lemma}
\label{lemma:avg:alpha:j}
Suppose $1 \le j  < d$, and let  
$\rho_j$ denote  the $O(d)$-invariant probability measure on
$\Gr_j(\R^d)$. There exists a constant $C_{d,j} > 0$ such  that
for every $A \in GL(d,\R)$, 
\begin{multline*}
\log( \sigma_1(A) \dots \sigma_j(A) ) - C_{d,j} \le
\int_{\Gr_j(\R^d)} \alpha_j(A,E) \, d\rho_j(E) \\
\le \log( \sigma_1(A)\dots\sigma_j(A) )
\end{multline*}
For $j=0,d$, the estimate holds taking $\alpha_0 = C_{d,j} =  0$
and the empty product equal to $1$. 
\end{lemma}

\begin{proof}
By orthogonal invariance of $\rho_j$, the integrals are unchanged if $A$ is
replaced by $UAV$ with $U,V\in O(d)$. Thus, by the singular value
decomposition, it is enough to consider the case
\begin{displaymath}
A=\diag(\sigma_1(A),\dots,\sigma_d(A)).
\end{displaymath}
Let  $e_1, \dots, e_d$ be  the standard basis  for $\R^d$, and let
$v_j = e_1 \wedge \dots \wedge e_j$. Then, for any $E \in \Gr_j(\R^d)$
\begin{displaymath}
\log( \sigma_1(A)\dots\sigma_j(A) ) - f(E)  \le \alpha_j(A,E)
\le \log( \sigma_1(A)\dots\sigma_j(A) ).
\end{displaymath}
where $f(E) = -\log  |\langle v_E, v_j\rangle|$ where $v_E$ is either unit
vector in $\bigwedge^j(E)$.   
The function $f$ is integrable with respect to $\rho_j$
by the logarithmic integrability of Pluecker coordinates. Let
\[
C_{d,j}:=\int_{\Gr_j(\R^d)}f(E)\,d\rho_j(E)<\infty
\]
Integrating the preceding inequalities
proves the lemma.
\end{proof}

\begin{lemma}
\label{lemma:avg-gamma-sigmak}
Let $\rho$ denote the $O(d)$-invariant probability measure on
\begin{displaymath}
\Flag_k'(\R^d)
=
\{L_k\subset L_{k+1}\subset \R^d:\dim L_k=k,\ \dim L_{k+1}=k+1\}.
\end{displaymath}
or the $O(d)$-invariant probability measure on
\begin{displaymath}
\Flag_k''(\R^d)
=
\{L_{k-1} \subset L_k\subset L_{k+1}\subset \R^d:\dim L_{k-1}=k-1, \dim L_k=k,\ \dim L_{k+1}=k+1\}.
\end{displaymath}
Then there exists a constant $C=C(d,k)$ such that for every
$A\in GL(d,\R)$,
\begin{displaymath}
\left|
\int_{\Flag_k'(\R^d)}
\alpha_k'(A,F)\,d\rho(F)
+ \log \sigma_{k+1}(A)
\right|
\le C,
\end{displaymath}
and
\begin{displaymath}
\left|
\int_{\Flag_k''(\R^d)}
\alpha_k''(A,F)\,d\rho(F)
- \log \Delta_k(A)
\right|
\le C.
\end{displaymath}
\end{lemma}

\begin{proof}
This follows immediately  from Lemma~\ref{lemma:avg:alpha:j}.  
\end{proof}

Recall that $\Delta_k(B,E)$ is defined in \ref{sssec:relative:gap}.  
\begin{lemma}
\label{lemma:quantitative:bad:position}
Let $B:V_1\to V_2$ be an invertible linear map  between
$d$-dimensional Euclidean  spaces, $1 \le k < d$
and let $E\subset V_1$ be a
$k$-dimensional subspace. 

Let $\rho_E^-$ be the unique $O(E)$-invariant probability measure on
$\Gr_{k-1}(E)$, and let $\rho_E^+$ be the unique probability measure on the set
\begin{displaymath}
\{L_{k+1}\in \Gr_{k+1}(V_1): E\subset L_{k+1}\}
\end{displaymath}
which is invariant under the stabilizer of $E$ in $O(V_1)$. Equivalently, under the
identification
\begin{displaymath}
\{L_{k+1}\in \Gr_{k+1}(V_1): E\subset L_{k+1}\}\simeq \PP(V_1/E),
\qquad
L_{k+1}\mapsto L_{k+1}/E,
\end{displaymath}
the measure $\rho_E^+$ is the natural $O(V_1/E)$-invariant probability measure on
$\PP(V_1/E)$.

Define
\begin{displaymath}
\overline{\alpha_k''}(B,E)
:=
\int_{L_{k-1}\subset E}\int_{L_{k+1}\supset E}
\alpha_k''\bigl(B,(L_{k-1},E,L_{k+1})\bigr)\,
d\rho_E^+(L_{k+1})\,d\rho_E^-(L_{k-1}).
\end{displaymath}
Then there exists a constant $C_{d,k}$, depending only on $d$ and $k$, such that
\begin{equation}
\label{eq:alpha:k:twoprime:vs:Delta:k:avg}
\log \Delta_k(B,E) \le \overline{\alpha_k''}(B,E)
\le
\log \Delta_k(B,E) + C_{d,k}.
\end{equation}
\end{lemma}

\begin{proof}
Write
\begin{displaymath}
\alpha_k''(B,(L_{k-1},E,L_{k+1}))
=
2\alpha_k(B,E)-\alpha_{k-1}(B,L_{k-1})-\alpha_{k+1}(B,L_{k+1}).
\end{displaymath}
Averaging over $L_{k-1}\subset E$ and $L_{k+1}\supset E$, we obtain
\begin{equation}
\label{eq:avg-second-split-quant}
\overline{\alpha_k''}(B,E)
=
2\alpha_k(B,E)-I_-(B,E)-I_+(B,E),
\end{equation}
where
\begin{displaymath}
I_-(B,E):=
\int_{L_{k-1}\subset E}\alpha_{k-1}(B,L_{k-1})\,d\rho_E^-(L_{k-1}),
\end{displaymath}
and
\begin{displaymath}
I_+(B,E):=
\int_{L_{k+1}\supset E}\alpha_{k+1}(B,L_{k+1})\,d\rho_E^+(L_{k+1}).
\end{displaymath}
Applying Lemma~\ref{lemma:avg:alpha:j} to $B|_E$ in degree $k-1$ we get
\begin{displaymath}
\alpha_k(B,E)-\log \frak{m}(B|_E) -C_{d,k}^{(1)} \le I_-(B,E)\le \alpha_k(B,E)-\log \frak{m}(B|_E).
\end{displaymath}
Also, if $\bar B_E:V_1/E\to V_2/BE$ is the induced quotient map, then every
$(k+1)$-plane $L_{k+1}\supset E$ corresponds to a line $\ell=L_{k+1}/E\subset V_1/E$,
and
\begin{displaymath}
\alpha_{k+1}(B,L_{k+1})
=
\alpha_k(B,E)+\alpha_1(\bar B_E,\ell).
\end{displaymath}
Applying Lemma~\ref{lemma:avg:alpha:j} to $\bar B_E$ in degree
$1$, we obtain
\begin{displaymath}
\alpha_k(B,E) + \log \|\bar B_E\|-C_{d,k}^{(2)} \le I_+(B,E)\le \alpha_k(B,E) + \log \|\bar B_E\|.
\end{displaymath}
Substituting these two bounds into \eqref{eq:avg-second-split-quant}, we get
(\ref{eq:alpha:k:twoprime:vs:Delta:k:avg}). 
\end{proof}

\subsection{Tail estimates}
\label{ssec:tail:estimates}

\begin{proposition}[Tail gap growth from a prefix]
\label{prop:tail:gap:from:prefix}
Fix $R<\infty$ sufficiently large.  There are constants
$C<\infty$, $c_1>0$, and $C_{d,k}<\infty$, depending on $R$ and the standing
constants, such that the following holds.

Let $t\geq 1$, assume $H_t\leq R$, and let
\begin{displaymath}
Q:U\to V_t
\end{displaymath}
be an $\mathcal F_t$-measurable linear isomorphism from a
$d$-dimensional Euclidean space $U$ to $V_t$. Let $E$ be a $\cF_t$
measurable $k$-plane in $U$. Then, for every $L>0$,
\begin{equation}
\label{eq:tail:gap:from:prefix:probability}
\mathbb P\left(
\exists q\geq 0 \st
\log\Delta_k(A(t,q)Q,E)-\log\Delta_k(Q,E)
<
\frac{c_0}{2}q-L-C_{d,k}
\,\middle|\,
\mathcal F_t
\right)
\leq
Ce^{-c_1\sqrt L}.
\end{equation}
The same estimate holds conditionally after a stopping time $t$, on the event
$H_t\leq R$.
\end{proposition}

\begin{proof}
Let $\rho_E^+$  and  $\rho_E^-$ be as in
Lemma~\ref{lemma:quantitative:bad:position}. 
For
$i\geq t$ define
\begin{multline*}
\psi_i  := \\
  \int_{L_{k-1} \subset E}  \int_{L_{k+1} \supset E}
\alpha_k''\bigl(A_i,A(t,i-t)Q(L_{k-1},E,L_{k+1})\bigr)\,d\rho_E^+(L_{k+1})\,
d\rho_E^-(L_{k-1}).  
\end{multline*}
For each $(L_{k-1},E,L_{k+1}) \in\Flag_k''(U)$, the flag $A(t,i-t)Q(L_{k-1},E,L_{k+1})$ is
$\mathcal F_i$-measurable and belongs to $\Flag_k''(V_i)$.  Hence
\eqref{eq:height:controlled:alphapp:drift} gives
\begin{displaymath}
\mathbb E(\psi_i\mid \mathcal F_i)\geq c_0.
\end{displaymath}
By \eqref{eq:height:controlled:alpha:bounds},
\begin{displaymath}
|\psi_i|\leq B(1+\log H_i)
\end{displaymath}
for a constant $B$ depending only on $B_{\rm lin},d,k$.  Applying
\autoref{prop:azuma:margulis} to the shifted process after time $t$ gives,
outside an event of conditional probability at most $Ce^{-c_1\sqrt L}$,
\begin{displaymath}
\sum_{i=t}^{t+q-1}\psi_i
\geq
\frac{c_0}{2}q-L
\end{displaymath}
for every $q\geq 0$.

By the cocycle identity,
\begin{multline*}
\sum_{i=t}^{t+q-1}\psi_i
=
\int_{L_{k-1} \subset E}  \int_{L_{k+1} \supset E}
\alpha_k''(A(t,q)Q,(L_{k-1},E,L_{k+1}))\, d\rho_E^+(L_{k+1})\,
d\rho_E^-(L_{k-1})  \\
-
\int_{L_{k-1} \subset E}  \int_{L_{k+1} \supset E}
\alpha_k''(Q,(L_{k-1},E,L_{k+1}))\,d\rho_E^+(L_{k+1})\,
d\rho_E^-(L_{k-1}).
\end{multline*}
Lemma~\ref{lemma:quantitative:bad:position}  compares the two integrals with
$\log\Delta_k(A(t,q)Q,E)$ and $\log\Delta_k(Q,E)$, respectively, up to an
additive error depending only on $d,k$.  This proves
\eqref{eq:tail:gap:from:prefix:probability}.
\end{proof}

\begin{proposition}[Tail contraction from a prefix]
\label{prop:tail:contracting:from:prefix}
Fix $R<\infty$ sufficiently large.  There are constants
$C<\infty$, $c_1>0$, and $C_{d,k}<\infty$, depending on $R$ and the standing
constants, such that the following holds.

Let $t\geq 1$, assume $H_t\leq R$, and let
\begin{displaymath}
Q:U\to V_t
\end{displaymath}
be an $\mathcal F_t$-measurable linear isomorphism from a
$d$-dimensional Euclidean space $U$ to $V_t$.  Then, for every $L>0$,
\begin{equation}
\label{eq:tail:contracting:from:prefix:probability}
\mathbb P\left(
\exists q\geq 0 \st
-\log\sigma_{k+1}(A(t,q)Q)+\log\sigma_{k+1}(Q)
<
\frac{c_0}{2}q-L-C_{d,k}
\,\middle|\,
\mathcal F_t
\right)
\leq
Ce^{-c_1\sqrt L}.
\end{equation}
The same estimate holds conditionally after a stopping time $t$, on the event
$H_t\leq R$.
\end{proposition}

\begin{proof}
Let $\rho_U$ be the invariant probability measure on $\Flag_k'(U)$, and for
$i\geq t$ define
\begin{displaymath}
\psi_i
:=
\int_{\Flag_k'(U)}
\alpha_k'\bigl(A_i,A(t,i-t)QF\bigr)\,d\rho_U(F).
\end{displaymath}
For each $F\in\Flag_k'(U)$, the flag $A(t,i-t)QF$ is
$\mathcal F_i$-measurable and belongs to $\Flag_k'(V_i)$.  Hence
\eqref{eq:height:controlled:alphap:drift} gives
\begin{displaymath}
\mathbb E(\psi_i\mid \mathcal F_i)\geq c_0.
\end{displaymath}
By \eqref{eq:height:controlled:alpha:bounds},
\begin{displaymath}
|\psi_i|\leq B(1+\log H_i)
\end{displaymath}
for a constant $B$ depending only on $B_{\rm lin},d,k$.  Applying
\autoref{prop:azuma:margulis} to the shifted process after time $t$ gives,
outside an event of conditional probability at most $Ce^{-c_1\sqrt L}$,
\begin{displaymath}
\sum_{i=t}^{t+q-1}\psi_i
\geq
\frac{c_0}{2}q-L
\end{displaymath}
for every $q\geq 0$.

By the cocycle identity,
\begin{displaymath}
\sum_{i=t}^{t+q-1}\psi_i
=
\int_{\Flag_k'(U)}
\alpha_k'(A(t,q)Q,F)\,d\rho_U(F)
-
\int_{\Flag_k'(U)}
\alpha_k'(Q,F)\,d\rho_U(F).
\end{displaymath}
\autoref{lemma:avg-gamma-sigmak} compares the two integrals with
$-\log\sigma_{k+1}(A(t,q)Q)$ and $-\log\sigma_{k+1}(Q)$, respectively, up to
an additive error depending only on $d,k$.  This proves
\eqref{eq:tail:contracting:from:prefix:probability}.
\end{proof}

\begin{proposition}[Initial transverse gap with compact return]
\label{prop:initial:transverse:gap}
Fix $R<\infty$ sufficiently large.  There exist constants
$a>0$, $A<\infty$, $C<\infty$, $c_1>0$, and $p_0\geq 1$, depending on $R$
and the standing constants, such that the following holds.

Let $t\geq 1$, assume $H_t\leq R$, and let $E\subset V_t$ be an
$\mathcal F_t$-measurable $k$-plane.  Then, for every $p\geq p_0$, with
conditional probability at least
\begin{displaymath}
1-Ce^{-c_1\sqrt p}
\end{displaymath}
given $\mathcal F_t$, there exists an integer $p_\ast$ with
\begin{displaymath}
p\leq p_\ast\leq 2p
\end{displaymath}
such that $H_{t+p_\ast} < R$ and
\begin{equation}
\label{eq:initial:transverse:gap:return:cushion}
\Delta_k(A(t,p_\ast),E)
\geq
e^{ap_\ast},
\end{equation}
\begin{equation}
\label{eq:initial:transverse:gap:return:angle}
\angle\bigl(E,\ker_{d-k}A(t,p_\ast)\bigr)
\geq
e^{-Ap_\ast}.
\end{equation}
The same estimate holds conditionally after a stopping time $t$, on the event
$H_t\leq R$.
\end{proposition}

\begin{proof}
Let $\rho_E^-$ be the invariant probability measure on $\Gr_{k-1}(E)$, and
let $\rho_E^+$ be the invariant probability measure on the set of
$(k+1)$-planes containing $E$.  Let $\rho_E:=\rho_E^-\otimes\rho_E^+$ on the
set of flags
\begin{displaymath}
L_{k-1}\subset E\subset L_{k+1}.
\end{displaymath}
For $i\geq t$, define
\begin{displaymath}
\psi_i
:=
\int
\alpha_k''\bigl(A_i,A(t,i-t)F\bigr)\,d\rho_E(F).
\end{displaymath}
For each flag $F$ in the support of $\rho_E$, the flag $A(t,i-t)F$ is
$\mathcal F_i$-measurable and belongs to $\Flag_k''(V_i)$.  Hence
\eqref{eq:height:controlled:alphapp:drift} gives
\begin{displaymath}
\mathbb E(\psi_i\mid\mathcal F_i)\geq c_0.
\end{displaymath}
By \eqref{eq:height:controlled:alpha:bounds},
\begin{displaymath}
|\psi_i|\leq B(1+\log H_i)
\end{displaymath}
for a constant $B$ depending only on $B_{\rm lin},d,k$.

Apply \autoref{prop:azuma:margulis} to the shifted process after time $t$,
with loss $c_0p/4$.  Outside an event of conditional probability at most
$Ce^{-c_1\sqrt p}$, for every $q\geq 0$,
\begin{displaymath}
\sum_{i=t}^{t+q-1}\psi_i
\geq
\frac{c_0}{2}q-\frac{c_0}{4}p.
\end{displaymath}
In particular, for every $p\leq q\leq 2p$,
\begin{equation}
\label{eq:initial:window:averaged:gap}
\overline{\alpha_k''}(A(t,q),E)
\geq
\frac{c_0}{4}q.
\end{equation}

Apply \autoref{lemma:abstract:height:occupation} with exponent $\beta=1$ to
the shifted process after time $t$, with $N=2p$ and $v=D(2p)$.  After
increasing $D$, outside an event of conditional probability at most
$Ce^{-c_1\sqrt p}$ one has
\begin{displaymath}
\sum_{i=t}^{t+2p-1}(1+\log H_i)\leq D(2p).
\end{displaymath}
Together with \eqref{eq:height:controlled:operator:norm}, this implies that,
outside the same exceptional event and for every $0\leq q\leq 2p$,
\begin{equation}
\label{eq:initial:window:product:norm}
\log\left(\|A(t,q)\|+\|A(t,q)^{-1}\|\right)
\leq
A_{\rm prod}p
\end{equation}
for some $A_{\rm prod}<\infty$.

It remains to choose a return time in the window.  By
\autoref{lemma:abstract:margulis:return:tails}, applied to the shifted process
after time $t+p$,
\begin{displaymath}
\mathbb P\left(
H_{t+q} \ge R \text{ for every } p\leq q\leq 2p
\,\middle|\,
\mathcal F_{t+p}
\right)
\leq
C H_{t+p}e^{-c_1p}.
\end{displaymath}
Taking conditional expectation with respect to $\mathcal F_t$ and using
\eqref{eq:abstract:margulis:height:drift}, together with $H_t\leq R$, gives
\begin{displaymath}
\mathbb P\left(
H_{t+q} \ge R \text{ for every } p\leq q\leq 2p
\,\middle|\,
\mathcal F_t
\right)
\leq
Ce^{-c_1p}.
\end{displaymath}
Thus, outside a further exceptional event, there exists
$p_\ast\in[p,2p]$ such that $H_{t+p_\ast} < R$.  Fix such a $p_\ast$ for
which \eqref{eq:initial:window:averaged:gap} and
\eqref{eq:initial:window:product:norm} also hold.

By  (\ref{eq:initial:window:averaged:gap}) and
Lemma~\ref{lemma:quantitative:bad:position},
\begin{displaymath}
\log  \Delta_k(A(t,p_\ast),E) +C_{d,k}  \ge  \frac{c_0 p_\ast}{4},
\end{displaymath}
hence, (after increasing $p_0$ to absorb $C_{d,k}$), 
\begin{displaymath}
\Delta_k(A(t,p_\ast),E) \ge e^{a p_\ast}
\end{displaymath}
with $a=c_0/8$.

Put $B:=A(t,p_\ast)$. Increasing $p_0$ ensures
$\Delta_k(B,E)\geq2$. By Lemma~\ref{lemma:relative:gap:to:angle}
and \eqref{eq:initial:window:product:norm},
\[
\angle(E,\ker_{d-k}B)
\geq \frac{\sigma_{k+1}(B)}{\|B\|}
\geq \frac{1}{\|B\|\|B^{-1}\|}
\geq e^{-2A_{\rm prod}p_\ast}.
\]
This proves \eqref{eq:initial:transverse:gap:return:angle}
with $A=2A_{\rm prod}$.
\end{proof}

\subsection{Regular sequences}
\label{ssec:regular:sequences}

\begin{definition}[$C$-regular sequence]
\label{def:C:regular}
Let $B_i:V_i\to V_{i+1}$, $a\le i\le b$, be a finite sequence of linear
isomorphisms between $d$-dimensional Euclidean spaces. Assume that each
$B_i$ has a strict $k$-gap. Write
\begin{displaymath}
F_i:=\im_k(B_i)\subset V_{i+1},
\qquad
K_i:=\ker_{d-k}(B_i)\subset V_i.
\end{displaymath}
Set
\begin{equation}
\label{eq:C:regular:endpoint}
\theta_a:=\frac{\pi}{2}, 
\end{equation}
\begin{displaymath}
\theta_{b+1} = \frac{\pi}{2},
\end{displaymath}
and, for $a<i\le b$, set
\begin{displaymath}
\theta_i:=\angle(F_{i-1},K_i).
\end{displaymath}

We say that $B_a,\dots,B_b$ is $C$-regular in dimension $k$ if
\begin{equation}
\label{eq:C:regular:interior}
\Delta_k(B_i)\theta_i\theta_{i+1}\ge C
\qquad
\text{for every }a\le i \le b.
\end{equation}
An infinite sequence $B_i$, $i \ge a$ is $C$-regular in dimension $k$
if (\ref{eq:C:regular:interior}) holds for  all $i \ge a$ (with the
convention that (\ref{eq:C:regular:endpoint}) holds). 
\end{definition}
Note that any infinite $C$-regular sequence truncated at $b>a$ is a finite
$C$-regular sequence.

\begin{definition}[$(C,\delta)$-uniformly regular sequence]
\label{def:C:delta:uniformly:regular}
With the notation of Definition~\ref{def:C:regular} and $0  < \delta \leq 1$
we say that $B_a,\dots,B_b$ is $(C,\delta)$-uniformly regular in dimension $k$ if
\begin{displaymath}
\theta_i\ge\delta
\qquad
\text{for every }a<i\le b,
\end{displaymath}
and
\begin{displaymath}
\Delta_k(B_i)\delta^2\ge C
\qquad
\text{for every }a\le i\le b.
\end{displaymath}
\end{definition}

If $B_a,\dots,B_b$ is $(C,\delta)$-uniformly regular, then it is
$C$-regular. Indeed, for $a\le i \le b$ we have
\begin{displaymath}
\Delta_k(B_i)\theta_i\theta_{i+1}
\ge
\Delta_k(B_i)\delta^2
\ge C.
\end{displaymath}
In the rest of this subsection we assume the following setup. Let
\begin{displaymath}
I=[a,b]\cap\mathbb Z
\qquad\text{or}\qquad
I=[a,\infty)\cap\mathbb Z,
\end{displaymath}
and let
\begin{displaymath}
\{B_i:V_i\to V_{i+1}\}_{i\in I}
\end{displaymath}
be a finite or infinite $C$-regular sequence in dimension $k$. For
$i\in I$, write
\begin{displaymath}
K_i:=\ker_{d-k}(B_i),
\qquad
F_i:=\im_k(B_i).
\end{displaymath}

\begin{proposition}[Forward and backward neighborhoods]
\label{prop:C:regular:neighborhoods}
There exist constants
\begin{displaymath}
a_0=a_0(d,k)>0,
\qquad
C_0=C_0(d,k)\geq1,
\end{displaymath}
with $a_0<1/4$, such that if $C\geq C_0$, then the following holds in the
standing setup of this subsection.

For $i=a$, define
\begin{displaymath}
\mathcal V_a^+
:=
\left\{
E\in\Gr_k(V_a):
\dist_{\Gr}(E,K_a^\perp)\leq a_0
\right\}.
\end{displaymath}
For $i\in I$, $i>a$, define
\begin{displaymath}
\mathcal V_i^+
:=
\left\{
E\in\Gr_k(V_i):
\dist_{\Gr}(E,F_{i-1})\leq a_0\theta_i
\right\}.
\end{displaymath}
For $i\in I$, define
\begin{displaymath}
\mathcal V_i^-
:=
\left\{
S\in\Gr_{d-k}(V_i):
\dist_{\Gr}(S,K_i)\leq a_0\theta_i
\right\}.
\end{displaymath}
Then:

\begin{enumerate}
\item[{\rm (i)}] If $i\in I$, $E\in\mathcal V_i^+$, and
$S\in\mathcal V_i^-$, then
$\angle(E,S)\geq \frac12\theta_i$.
\item[{\rm (ii)}] If $i,i+1\in I$, then
$B_i(\mathcal V_i^+)\subset \mathcal V_{i+1}^+$.
\item[{\rm (iii)}] If $i,i+1\in I$, then
$B_i^{-1}(\mathcal V_{i+1}^-)\subset \mathcal V_i^-$.
\item[{\rm (iv)}] In the finite case $I=[a,b]\cap\mathbb Z$, we have
$B_b^{-1}(F_b^\perp)\in\mathcal V_b^-$.
\item[{\rm (v)}] There is a constant $C_{\rm lip}=C_{\rm lip}(d,k)$ such
that, if $i,i+1\in I$ and $S,S'\in\mathcal V_{i+1}^-$, then
\begin{equation}
\label{eq:C:regular:backward:lipschitz:raw}
\dist_{\Gr}\bigl(B_i^{-1}S,B_i^{-1}S'\bigr)
\leq
\frac{C_{\rm lip}}{\Delta_k(B_i)\theta_{i+1}^2}
\dist_{\Gr}(S,S').
\end{equation}
Consequently,
\begin{equation}
\label{eq:C:regular:backward:lipschitz:scaled}
\frac{
\dist_{\Gr}\bigl(B_i^{-1}S,B_i^{-1}S'\bigr)
}{\theta_i}
\leq
\frac{C_{\rm lip}}{C}
\frac{\dist_{\Gr}(S,S')}{\theta_{i+1}}.
\end{equation}
\end{enumerate}
\end{proposition}

\begin{proof}
Choose $a_0=a_0(d,k)>0$ sufficiently small, with $a_0<1/4$, so that the
Grassmannian distance and the angle function are comparable on all balls of
radius $2a_0$, and so that \autoref{lemma:inverse:local:lipschitz} applies
on such balls.  We then choose $C_0=C_0(d,k)$ sufficiently large in the
estimates below.

We first prove {\rm (i)}.  If $i=a$, then $\mathcal V_a^+$ is a small
neighborhood of $K_a^\perp$ and $\mathcal V_a^-$ is a small neighborhood of
$K_a$.  Since $\theta_a=\pi/2$, the result follows by the choice of
$a_0$. Indeed,
\begin{displaymath}
\angle(E,S)\ge\frac{\pi}{2}-a_0-a_0\frac{\pi}{2}\ge\frac{\pi}{4}=\frac{\theta_a}{2}.
\end{displaymath}

If $i>a$, then $E$ is within $a_0\theta_i$ of $F_{i-1}$ and $S$ is within
$a_0\theta_i$ of $K_i$.  Since
\begin{displaymath}
\angle(F_{i-1},K_i)=\theta_i,
\end{displaymath}
the triangle inequality on the Grassmannian, together with the choice of
$a_0$, gives
\begin{displaymath}
\angle(E,S)\geq\frac12\theta_i.
\end{displaymath}

We now prove {\rm (ii)}.  Let $i,i+1\in I$ and $E\in\mathcal V_i^+$.  By
{\rm (i)}, applied with $S=K_i$, one has
\begin{displaymath}
\angle(E,K_i)\geq\frac12\theta_i.
\end{displaymath}
Hence \autoref{lemma:image:of:transverse:plane} gives
\begin{displaymath}
\dist_{\Gr}(B_iE,F_i)
\leq
\frac{16}{\Delta_k(B_i)\theta_i}.
\end{displaymath}
Using $C$-regularity,
\begin{displaymath}
\Delta_k(B_i)\theta_i\theta_{i+1}\geq C,
\end{displaymath}
we get
\begin{displaymath}
\dist_{\Gr}(B_iE,F_i)
\leq
\frac{16}{C}\theta_{i+1}.
\end{displaymath}
Taking $C_0$ sufficiently large gives
\begin{displaymath}
B_iE\in\mathcal V_{i+1}^+.
\end{displaymath}

The proof of {\rm (iii)} is similar.  Let $S\in\mathcal V_{i+1}^-$.  By
{\rm (i)}, applied at $i+1$ with $E=F_i$, one has
\begin{displaymath}
\angle(S,F_i)\geq\frac12\theta_{i+1}.
\end{displaymath}
Thus \autoref{lemma:inverse:estimate:transverse:to:fast:image} gives
\begin{displaymath}
\dist_{\Gr}(B_i^{-1}S,K_i)
\leq
\frac{16}{\Delta_k(B_i)\theta_{i+1}}.
\end{displaymath}
Again using
\begin{displaymath}
\Delta_k(B_i)\theta_i\theta_{i+1}\geq C,
\end{displaymath}
we obtain
\begin{displaymath}
\dist_{\Gr}(B_i^{-1}S,K_i)
\leq
\frac{16}{C}\theta_i.
\end{displaymath}
Taking $C_0$ sufficiently large proves
\begin{displaymath}
B_i^{-1}S\in\mathcal V_i^-.
\end{displaymath}

For {\rm (iv)}, we have the identity
$B_b^{-1}(F_b^\perp)=K_b\in\mathcal V_b^-$, which follows from the
singular subspace definitions.

It remains to prove {\rm (v)}.  Let $i,i+1\in I$ and
$S,S'\in\mathcal V_{i+1}^-$.  Apply
\autoref{lemma:inverse:local:lipschitz} to $B_i$ with
\begin{displaymath}
H:=K_{i+1},
\qquad
F:=F_i,
\qquad
\theta:=\theta_{i+1}.
\end{displaymath}
The hypotheses hold by the definitions of $\mathcal V_{i+1}^-$ and the choice
of $a_0$.  Therefore
\begin{displaymath}
\dist_{\Gr}\bigl(B_i^{-1}S,B_i^{-1}S'\bigr)
\leq
\frac{C_{\rm lip}}{\Delta_k(B_i)\theta_{i+1}^2}
\dist_{\Gr}(S,S'),
\end{displaymath}
which proves \eqref{eq:C:regular:backward:lipschitz:raw}.  Dividing by
$\theta_i$ and using
\begin{displaymath}
\Delta_k(B_i)\theta_i\theta_{i+1}\geq C
\end{displaymath}
gives \eqref{eq:C:regular:backward:lipschitz:scaled}.
\end{proof}

\begin{lemma}[Pullback sets for regular sequences]
\label{lemma:C:regular:pullback:sets}
There exists $C_0=C_0(d,k)\geq1$ with the following property. Assume
$C \ge C_0$. Let $\mathcal V_i^-$ be the neighborhoods from
\autoref{prop:C:regular:neighborhoods}.  For $i<n$, define
\begin{displaymath}
\mathcal C_{i,n}:=
B_i^{-1}B_{i+1}^{-1}\cdots B_{n-1}^{-1}(\mathcal V_n^-).
\end{displaymath}
Then $\mathcal C_{i,n}$ is a nonempty compact subset of $\mathcal V_i^-$.
Moreover, if $i<n<n'$ then
\begin{displaymath}
\mathcal C_{i,n'}\subset \mathcal C_{i,n}.
\end{displaymath}
Finally, after increasing $C_0$ if necessary, there is a constant
$q_{\rm pb}=q_{\rm pb}(C)<1$ such that
\begin{equation}
\label{eq:C:regular:pullback:diameter}
\operatorname{diam}_{\Gr}(\mathcal C_{i,n})
\leq
2a_0\theta_i q_{\rm pb}^{n-i}.
\end{equation}

In the infinite case, for each fixed $i$, the sets $\mathcal C_{i,n}$ shrink
to a single point as $n\to\infty$.
\end{lemma}

\begin{proof}
The nesting property follows immediately from
\autoref{prop:C:regular:neighborhoods}{\rm (iii)}.
Compactness is immediate.

Put
\begin{displaymath}
d_i(S,S'):=\theta_i^{-1}\dist_{\Gr}(S,S')
\end{displaymath}
on $\mathcal V_i^-$.
Let $C_{\rm lip}$ be the constant in
\autoref{prop:C:regular:neighborhoods}{\rm (v)}, and set
\begin{displaymath}
q_{\rm pb}:=\frac{C_{\rm lip}}{C}.
\end{displaymath}
Increasing $C_0$ if necessary, we may assume $q_{\rm pb}<1$.  The
scaled Lipschitz estimate
\eqref{eq:C:regular:backward:lipschitz:scaled} gives
\begin{displaymath}
\operatorname{diam}_{d_i}(\mathcal C_{i,n})
\leq
q_{\rm pb}^{n-i}\operatorname{diam}_{d_n}(\mathcal V_n^-).
\end{displaymath}
Since $\mathcal V_n^-$ has radius $a_0\theta_n$ in the Grassmannian metric,
\begin{displaymath}
\operatorname{diam}_{d_n}(\mathcal V_n^-)\leq 2a_0.
\end{displaymath}
Thus
\begin{displaymath}
\operatorname{diam}_{d_i}(\mathcal C_{i,n})
\leq
2a_0q_{\rm pb}^{n-i},
\end{displaymath}
which is equivalent to
\eqref{eq:C:regular:pullback:diameter}.

In the infinite case the sets $\mathcal C_{i,n}$ are nested nonempty compact
sets and their diameters tend to $0$.  Hence their intersection consists of a
single point.
\end{proof}

\begin{definition}[Fast and slow spaces for regular sequences]
\label{def:C:regular:fast:slow:spaces}
Assume the standing setup of this subsection.  Define the fast spaces by
\begin{displaymath}
\mathcal U_a:=K_a^\perp,
\qquad
\mathcal U_{i+1}:=B_i\mathcal U_i
\qquad
(i\in I).
\end{displaymath}

If $I=[a,b]\cap\mathbb Z$ is finite, define the slow spaces by
\begin{displaymath}
\mathcal S_{b+1}:=F_b^\perp,
\qquad
\mathcal S_i:=B_i^{-1}\mathcal S_{i+1}
\qquad
(a\leq i\leq b).
\end{displaymath}

If $I=[a,\infty)\cap\mathbb Z$ is infinite, define the slow spaces by
\begin{displaymath}
\{\mathcal S_i\}:=\bigcap_{n>i}\mathcal C_{i,n}
\qquad
(i\geq a),
\end{displaymath}
where the sets $\mathcal C_{i,n}$ are those of
\autoref{lemma:C:regular:pullback:sets}.
\end{definition}

\begin{lemma}[Basic properties of the fast and slow spaces]
\label{lemma:C:regular:fast:slow:basic}
There exists $C_0=C_0(d,k)\geq1$ with the following property.  Assume the
standing setup of this subsection, and assume $C\geq C_0$.  Let
$\mathcal U_i$ and $\mathcal S_i$ be the fast and slow spaces from
\autoref{def:C:regular:fast:slow:spaces}.  Then the following hold.
\begin{itemize}
\item[{\rm (i)}] The fast and slow spaces are equivariant:
\begin{equation}
\label{eq:C:regular:fast:slow:equivariant}
B_i\mathcal U_i=\mathcal U_{i+1}, \qquad B_i\mathcal S_i=\mathcal S_{i+1}
\end{equation}
whenever both sides are defined.
\item[{\rm (ii)}] For all $i \in I$, 
\begin{equation}
\label{eq:C:regular:fast:slow:containment}
\mathcal U_i\in\mathcal V_i^+ \qquad \mathcal S_i\in\mathcal V_i^-.
\end{equation}
\item[{\rm (iii)}] For all $i \in I$, $\mathcal U_i \oplus \mathcal S_i =  V_i$.
Furthermore, 
\begin{equation}
\label{eq:C:regular:fast:slow:angle}
\angle(\mathcal U_i,  \mathcal S_i) >  \frac{1}{2} \theta_i.   
\end{equation}
\item[{\rm  (iv)}] The slow spaces are close to the one-step slow
  singular spaces.  
For every $i$ such that $i+1\in I$,
\begin{equation}
\label{eq:C:regular:slow:subspace:near:kernel}
\dist_{\Gr}(\mathcal S_i,K_i)
\leq
\frac{16}{\Delta_k(B_i)\theta_{i+1}}.
\end{equation}
In the finite case,
\begin{equation}
\label{eq:C:regular:slow:terminal:equals:kernel}
\mathcal S_b=K_b.
\end{equation}
\item[{\rm (v)}] For every $i\in I$,
\begin{equation}
\label{eq:C:regular:fast:one:step}
\mathfrak m(B_i|_{\mathcal U_i})
\geq
\frac{1}{4} \theta_i\,\sigma_k(B_i).
\end{equation}
For every $i$ for which $i+1\in I$,
\begin{equation}
\label{eq:C:regular:slow:one:step}
\|B_i|_{\mathcal S_i}\|
\leq
4\, \theta_{i+1}^{-1}\sigma_{k+1}(B_i).
\end{equation}
In the finite case, at the terminal step one has the sharper identity
\begin{equation}
\label{eq:C:regular:slow:terminal:one:step}
\|B_b|_{\mathcal S_b}\|=\sigma_{k+1}(B_b).
\end{equation}
\end{itemize}
\end{lemma}

\begin{proof}
The equivariance of the fast spaces is part of the definition.  In the finite
case, the equivariance of the slow spaces is also part of the definition.  In
the infinite case, for $n>i+1$ one has
\begin{displaymath}
B_i\mathcal C_{i,n}=\mathcal C_{i+1,n}.
\end{displaymath}
Letting $n\to\infty$ gives
\begin{displaymath}
B_i\mathcal S_i=\mathcal S_{i+1}.
\end{displaymath}
This proves \eqref{eq:C:regular:fast:slow:equivariant}.

We next prove the containments.  For the fast spaces, since
$\mathcal U_a=K_a^\perp$,
we have $\mathcal U_a\in\mathcal V_a^+$.  The general containment follows
inductively from \autoref{prop:C:regular:neighborhoods}{\rm (ii)}.

For the slow spaces in the finite case, the terminal definition gives
\begin{displaymath}
\mathcal S_b=B_b^{-1}F_b^\perp=K_b,
\end{displaymath}
so $\mathcal S_b\in\mathcal V_b^-$.  The containment
$\mathcal S_i\in\mathcal V_i^-$
for $a\leq i<b$ follows by backward induction from
\autoref{prop:C:regular:neighborhoods}{\rm (iii)}.  In the infinite case, this
containment is part of \autoref{lemma:C:regular:pullback:sets}.  This proves
\eqref{eq:C:regular:fast:slow:containment}. The equation (\ref{eq:C:regular:fast:slow:angle}) follows from
Proposition~\ref{prop:C:regular:neighborhoods}{\rm (i)}. 

We now prove \eqref{eq:C:regular:slow:terminal:equals:kernel} and
\eqref{eq:C:regular:slow:subspace:near:kernel}.  In the finite case,
\begin{displaymath}
\mathcal S_b=B_b^{-1}F_b^\perp=K_b,
\end{displaymath}
which gives \eqref{eq:C:regular:slow:terminal:equals:kernel}.
Now let $i$ be such that $i+1\in I$.  By the containment just proved,
$\mathcal S_{i+1}\in\mathcal V_{i+1}^-$.
Hence \autoref{prop:C:regular:neighborhoods}{\rm (i)} gives
$\angle(\mathcal S_{i+1},F_i)\geq \tfrac{1}{2}\theta_{i+1}$.
Since $\mathcal S_i=B_i^{-1}\mathcal S_{i+1}$,
\autoref{lemma:inverse:estimate:transverse:to:fast:image} gives
\begin{displaymath}
\dist_{\Gr}(\mathcal S_i,K_i)
\leq
\frac{16}{\Delta_k(B_i)\theta_{i+1}}.
\end{displaymath}
which is (\ref{eq:C:regular:slow:subspace:near:kernel}).

For the fast one-step estimate, the containment
\eqref{eq:C:regular:fast:slow:containment} and
\autoref{prop:C:regular:neighborhoods}{\rm (i)} give 
\begin{displaymath}
\angle(\mathcal U_i,K_i)\geq \frac{1}{2}\,\theta_i
\end{displaymath}
for every $i\in I$.  Applying Lemma~\ref{lemma:image:of:transverse:plane}
gives \eqref{eq:C:regular:fast:one:step}.

For the slow estimate, assume first that $i+1\in I$.  Since
$\mathcal S_{i+1}\in\mathcal V_{i+1}^-$,
\autoref{prop:C:regular:neighborhoods}{\rm (i)} gives
\begin{displaymath}
\angle(\mathcal S_{i+1},F_i)\geq \frac{1}{2}\,\theta_{i+1}.
\end{displaymath}
Also,
$\mathcal S_i=B_i^{-1}\mathcal S_{i+1}$.
Lemma~\ref{lemma:inverse:estimate:transverse:to:fast:image} 
therefore gives \eqref{eq:C:regular:slow:one:step}.
In the finite case, at the terminal step,
\begin{displaymath}
\mathcal S_b=B_b^{-1}F_b^\perp=K_b.
\end{displaymath}
Hence
$\|B_b|_{\mathcal S_b}\|
=
\|B_b|_{K_b}\|
=
\sigma_{k+1}(B_b)$, which is \eqref{eq:C:regular:slow:terminal:one:step}.
\end{proof}

\begin{lemma}[Product estimates for regular sequences]
\label{lemma:C:regular:product:estimates}
There exist constants
\begin{displaymath}
C_0=C_0(d,k) > 16,
\qquad
C_1=C_1(d,k)<\infty,
\end{displaymath}
with the following property.  Assume the standing setup of this subsection and
assume $C\geq C_0$.  Put
\begin{displaymath}
q(C):=\frac{16}{C}  <  1.  
\end{displaymath}
For $i \leq n$ in $I$, put
\begin{displaymath}
P_{i,n}:= B_{n}\cdots B_i.
\end{displaymath}
Then
\begin{equation}
\label{eq:C:regular:product:fast:lower}
\sigma_k(P_{i,n})
\geq
\mathfrak m(P_{i,n}|_{\mathcal U_i})
\geq
4^{-(n-i+1)}
\prod_{r=i}^{n}\theta_r\sigma_k(B_r),
\end{equation}
and (using the convention $\theta_{n+1}=\pi/2$ when $n=b$ in the finite case),
\begin{equation}
\label{eq:C:regular:product:slow:upper}
\sigma_{k+1}(P_{i,n})
\leq
\|P_{i,n}|_{\mathcal S_i}\|
\leq
4^{n-i+1}
\prod_{r=i}^{n}\theta_{r+1}^{-1}\sigma_{k+1}(B_r).
\end{equation}
Consequently,
\begin{equation}
\label{eq:C:regular:product:gap:growth}
\Delta_k(P_{i,n})
\geq
\frac{
\mathfrak m(P_{i,n}|_{\mathcal U_i})
}{
\|P_{i,n}|_{\mathcal S_i}\|
}
\geq
q(C)^{-(n-i+1)}.
\end{equation}
Moreover,
\begin{equation}
\label{eq:C:regular:product:kernel:near:slow}
\dist_{\Gr}\bigl(
\ker_{d-k}(P_{i,n}),\mathcal S_i
\bigr)
\leq
C_1q(C)^{n-i+1}.
\end{equation}
Finally, if $i+1\in I$, then
\begin{equation}
\label{eq:C:regular:product:kernel:near:kernel}
\dist_{\Gr}\bigl(
\ker_{d-k}(P_{i,n}),K_i
\bigr)
\leq
C_1q(C)^{n-i+1}
+
\frac{16}{\Delta_k(B_i)\theta_{i+1}}.
\end{equation}
\end{lemma}

\begin{proof}
The first inequalities in
\eqref{eq:C:regular:product:fast:lower} and
\eqref{eq:C:regular:product:slow:upper} are the min--max principle.

Conorms are supermultiplicative and norms are submultiplicative.  Applying
\eqref{eq:C:regular:fast:one:step},
\eqref{eq:C:regular:slow:one:step}, and in the finite terminal case
\eqref{eq:C:regular:slow:terminal:one:step}, gives
\eqref{eq:C:regular:product:fast:lower} and
\eqref{eq:C:regular:product:slow:upper}.
Dividing (\ref{eq:C:regular:product:fast:lower}) by
(\ref{eq:C:regular:product:slow:upper}) and using $C$-regularity 
yields (\ref{eq:C:regular:product:gap:growth}). We now
apply
\autoref{lemma:slow:space:from:dominated:splitting}
to \(P_{i,n}:V_i\to V_{n+1}\) and
\(V_i=\mathcal U_i\oplus\mathcal S_i\) to get
\eqref{eq:C:regular:product:kernel:near:slow}.

Finally combine this with
\eqref{eq:C:regular:slow:subspace:near:kernel}
and the triangle inequality to get
\eqref{eq:C:regular:product:kernel:near:kernel}.
\end{proof}

\begin{lemma}[Creating a terminal buffer]
\label{lemma:create:terminal:buffer}
There exist constants $C_{\rm buf}=C_{\rm buf}(d,k)\ge1$ and
$M_{\rm buf}=M_{\rm buf}(d,k)<\infty$ with the following property.

Let $B_1,\dots,B_m$ be a finite $(C,\delta)$-uniformly regular sequence in
dimension $k$, with $0<\delta\le1$. Assume
\begin{equation}
\label{eq:buffer:hyp:C:large}
C\ge C_{\rm buf},
\end{equation}
\begin{equation}
\label{eq:buffer:hyp:C:dominates:delta}
\log C\ge 4M_{\rm buf}|\log\delta|,
\end{equation}
\begin{equation}
\label{eq:buffer:hyp:block:contraction}
-\log\sigma_{k+1}(B_i)\ge L_{\rm blk}
\qquad
\text{for every }1\le i\le m,
\end{equation}
and
\begin{equation}
\label{eq:buffer:hyp:L:dominates:delta}
L_{\rm blk}\ge 2M_{\rm buf}(|\log\delta|+1).
\end{equation}

Let $r$ be an integer with $1\le r<m$, set
\begin{displaymath}
p:=m-r+1,
\end{displaymath}
and define
\begin{displaymath}
\widehat B:=B_mB_{m-1}\cdots B_p.
\end{displaymath}
Then
\begin{equation}
\label{eq:buffer:concl:gap}
\log\Delta_k(\widehat B)\ge \frac12 r\log C,
\end{equation}
\begin{equation}
\label{eq:buffer:concl:angle}
\angle\bigl(\im_k(B_{p-1}),\ker_{d-k}(\widehat B)\bigr)
\ge \frac{\delta}{2}.
\end{equation}
Moreover,
\begin{equation}
\label{eq:buffer:concl:contraction}
-\log\sigma_{k+1}(\widehat B)
\ge
\frac12 rL_{\rm blk}.
\end{equation}
\end{lemma}

\begin{proof}
Since $B_1,\dots,B_m$ is $(C,\delta)$-uniformly regular, it is $C$-regular.
Let 
\begin{displaymath}
q(C)=\frac{16}{C}.
\end{displaymath}
We increase $C_{\rm buf}$ so that \autoref{lemma:C:regular:product:estimates}
applies and
\begin{equation}
\label{eq:buffer:proof:q:gap}
-\log q(C)\geq \frac12\log C
\end{equation}
whenever $C\geq C_{\rm buf}$.

Apply \autoref{lemma:C:regular:product:estimates} to the finite regular
sequence $B_p,\dots,B_m$.  Since
\begin{displaymath}
\widehat B=B_m\cdots B_p
\end{displaymath}
has length $r$, \eqref{eq:C:regular:product:gap:growth} gives
\begin{displaymath}
\Delta_k(\widehat B)\geq q(C)^{-r}.
\end{displaymath}
Together with \eqref{eq:buffer:proof:q:gap}, this gives
\eqref{eq:buffer:concl:gap}.

We next prove the angle estimate.  If $r=1$, then
\begin{displaymath}
\ker_{d-k}(\widehat B)=K_p,
\end{displaymath}
and hence
\begin{displaymath}
\angle\bigl(\im_k(B_{p-1}),\ker_{d-k}(\widehat B)\bigr)
=
\theta_p
\geq \delta.
\end{displaymath}
So assume $r\geq2$.  Applying
\eqref{eq:C:regular:product:kernel:near:kernel} to the finite regular
sequence $B_p,\dots,B_m$ gives
\begin{displaymath}
\dist_{\Gr}\bigl(\ker_{d-k}(\widehat B),K_p\bigr)
\leq
C_1q(C)^r+
\frac{16}{\Delta_k(B_p)\theta_{p+1}}.
\end{displaymath}
By $C$-regularity,
\begin{displaymath}
\Delta_k(B_p)\theta_p\theta_{p+1}\geq C,
\end{displaymath}
and therefore
\begin{displaymath}
\frac{16}{\Delta_k(B_p)\theta_{p+1}}
\leq
\frac{16}{C}\theta_p.
\end{displaymath}
After increasing $C_{\rm buf}$, this is at most
$a_0\theta_p/4$.

For the first term, use $\theta_p\geq\delta$.  Increasing $M_{\rm buf}$ and
then $C_{\rm buf}$ if necessary, the hypothesis
\eqref{eq:buffer:hyp:C:dominates:delta} implies
\begin{displaymath}
C_1q(C)\leq \frac{a_0}{4}\delta.
\end{displaymath}
Hence
\begin{displaymath}
C_1q(C)^r\leq \frac{a_0}{4}\theta_p.
\end{displaymath}
Thus
\begin{displaymath}
\dist_{\Gr}\bigl(\ker_{d-k}(\widehat B),K_p\bigr)
\leq
\frac{a_0}{2}\theta_p.
\end{displaymath}
By the standing smallness choice of $a_0$ in
\autoref{prop:C:regular:neighborhoods}, this implies
\begin{displaymath}
\angle\bigl(\im_k(B_{p-1}),\ker_{d-k}(\widehat B)\bigr)
\geq
\frac12\theta_p
\geq
\frac{\delta}{2}.
\end{displaymath}
This proves \eqref{eq:buffer:concl:angle}.

It remains to prove the contraction estimate.  Apply
\eqref{eq:C:regular:product:slow:upper} to the finite regular sequence
$B_p,\dots,B_m$.  In the terminal factor we use the convention
$\theta_{m+1}=\pi/2$.  Since the sequence is $(C,\delta)$-uniformly regular,
all relevant angles $\theta_{p+1},\dots,\theta_m$ are at least $\delta$.
Thus
\begin{displaymath}
\sigma_{k+1}(\widehat B)
\leq 4^r\delta^{-(r-1)}
\prod_{i=p}^{m}\sigma_{k+1}(B_i).
\end{displaymath}
Using \eqref{eq:buffer:hyp:block:contraction}, we get
\begin{displaymath}
-\log\sigma_{k+1}(\widehat B)
\geq
rL_{\rm blk}
-2 r \log 2
-(r-1)|\log\delta|.
\end{displaymath}
Increase $M_{\rm buf}$ so that
\begin{displaymath}
M_{\rm buf}\geq 2\log 2 +1.
\end{displaymath}
Then \eqref{eq:buffer:hyp:L:dominates:delta} gives
\begin{displaymath}
\frac12L_{\rm blk}
\geq
2 \log 2 +|\log\delta|.
\end{displaymath}
Therefore
\begin{displaymath}
-\log\sigma_{k+1}(\widehat B)
\geq
\frac12rL_{\rm blk},
\end{displaymath}
which proves \eqref{eq:buffer:concl:contraction}.
\end{proof}

\subsection{A parameterized selection lemma}
\label{ssec:selection:lemma}

The purpose of this subsection is to record a one-block selection statement
with an explicit length parameter. The proof has two parts: a short initial
segment produces a definite angle, and the remaining segment increases the
gap while preserving this angle. We also include a blockwise
$\sigma_{k+1}$-contraction estimate.

\begin{lemma}[Kernel stability under a continued gap]
\label{lemma:kernel:stability:continued:gap}
Assume the setup of \autoref{ssec:abstract:good:event}.  Fix
$R<\infty$ and $\Lambda>0$.  There are constants $C<\infty$ and $c_1>0$,
depending only on $R,\Lambda,B_{\rm lin},d,k$ and the height constants, such
that the following holds.

Let $t\geq 1$ and assume $H_t\leq R$.  Let $p_0\geq 1$, $\chi>0$, and
$m\geq 0$.  Let
\begin{displaymath}
T_0:V\to V_t
\end{displaymath}
be an $\mathcal F_t$-measurable linear isomorphism, let $E \subset V$
be a $\mathcal F_t$-measurable $k$-plane 
and define
\begin{displaymath}
T_q:=A_{t+q-1}\cdots A_tT_0:V\to V_{t+q},
\qquad
0\leq q\leq m.
\end{displaymath}
Let $\mathcal G_{\rm gap}$ be the event that
\begin{equation}
\label{eq:continued:gap:hypothesis}
\Delta_k(T_q,E)\geq C e^{\Lambda q+\chi p_0}
\qquad
\text{for every }0\leq q\leq m.
\end{equation}
Let $\mathcal G_{\rm ker}$ be the event that each $T_q$ has a strict
$k$-gap and that
\begin{equation}
\label{eq:kernel:stability:continued:gap}
\angle(\ker_{d-k}T_q,E)\geq\tfrac12\angle(\ker_{d-k}T_0,E)
\end{equation}
Then
\begin{equation}
\label{eq:kernel:stability:probabilistic}
\mathbb P\left(
\mathcal G_{\rm gap}\cap\mathcal G_{\rm ker}^{c}
\,\middle|\,
\mathcal F_t
\right)
\leq
C e^{-c_1\chi p_0}.
\end{equation}
The same estimate holds conditionally after a stopping time $t$, on the event
$H_t\leq R$.
\end{lemma}

\begin{proof}
Put $\Delta_q:=\Delta_k(T_q,E)$ and, on
$\mathcal G_{\rm gap}$, $\theta_q:=\angle(\ker_{d-k}T_q,E)$.
By definition of $\mathcal G_{\rm gap}$,
$\Delta_q\geq C e^{\chi p_0+\Lambda q}>1$.
Set $\kappa_q:=\|A_{t+q}\|\|A_{t+q}^{-1}\|$ and
$r_q:=(\kappa_q+1)/\Delta_{q+1}$.
Lemma~\ref{lemma:GM-step} gives
$\sin\theta_{q+1}\geq(1-r_q)\sin\theta_q$.
By \eqref{eq:height:controlled:operator:norm},
$\kappa_q+1\leq C_0H_{t+q}^{\beta}$ for fixed $C_0,\beta$.
Lemma~\ref{lemma:weighted:height:sum} gives, outside an event of conditional
probability at most $C_2e^{-c_2\chi p_0}$,
\[
\sum_{q=0}^{\infty}H_{t+q}^{\beta}e^{-\Lambda q}\leq C_2e^{\chi p_0}.
\]
On this event and $\mathcal G_{\rm gap}$,
\[
\begin{aligned}
\sum_{i=0}^{m-1}r_i
&\leq \frac{C_0}{C}e^{-\chi p_0}
\sum_{i=0}^{m-1}H_{t+i}^{\beta}e^{-\Lambda(i+1)}\\
&\leq\frac{C_0C_2}{C}\leq\tfrac14,
\end{aligned}
\]
after taking $C\geq4C_0C_2$ in (\ref{eq:continued:gap:hypothesis}).
Therefore, for every $0\leq q\leq m$, 
\[
\sin\theta_q\geq\sin\theta_0\prod_{i=0}^{q-1}(1-r_i) \ge
\sin\theta_0\left(1 - \sum_{i=0}^{q-1}r_i\right)
\geq\tfrac34\sin\theta_0\geq\sin(\theta_0/2).
\]
Thus $\theta_q\geq\theta_0/2$, proving
\eqref{eq:kernel:stability:probabilistic}. The stopping-time version follows from
that of Lemma~\ref{lemma:weighted:height:sum}.  
\end{proof}

\begin{proposition}[Parameterized one-block selection]
\label{prop:parameterized:one:block:selection}
Assume the standing assumptions from \autoref{ssec:abstract:good:event}.
Fix $R<\infty$ sufficiently large.  Then there exist constants
\begin{displaymath}
\gamma>0,
\qquad
\lambda_{\rm ctr}>0,
\qquad
K_{\rm sel}<\infty,
\end{displaymath}
where $\gamma$ and $\lambda_{\rm ctr}$ depend only on $c_0,d,k$, while
$K_{\rm sel}$ may also depend on $R$ and the height-control constants, such
that the following holds.  For every
$\eta>0$, there exist constants
\begin{displaymath}
C_\eta<\infty,\qquad b_\eta>0,\qquad N_0=N_0(\eta,R)\geq 1
\end{displaymath}
such that, for every $t\geq 1$ with $H_t\leq R$, every
$\mathcal F_t$-measurable $k$-plane $E\subset V_t$, and every
$N\geq N_0$, conditionally on $\mathcal F_t$, with probability at least
\begin{displaymath}
1-C_\eta e^{-b_\eta\sqrt N},
\end{displaymath}
there exists an integer
$\frac{N}{2}\leq \ell\leq N$
such that $H_{t+\ell} < R$ and, writing
\begin{displaymath}
B:=A(t,\ell)=A_{t+\ell-1}\cdots A_t:V_t\to V_{t+\ell},
\end{displaymath}
one has
\begin{equation}
\label{eq:parameterized:selection:angle}
\angle(E,\ker_{d-k}B)\geq e^{-\eta\ell},
\end{equation}
\begin{equation}
\label{eq:parameterized:selection:cushion}
\Delta_k(B)\angle(E,\ker_{d-k}B)\geq e^{\gamma\ell},
\end{equation}
and
\begin{equation}
\label{eq:parameterized:selection:block:contraction}
-\log\sigma_{k+1}(B)\geq \lambda_{\rm ctr}\ell.
\end{equation}
Moreover,
\begin{equation}
\label{eq:parameterized:selection:block:tameness}
\max_{0\leq r\leq\ell}
\log\max\left\{
\|A(t,r)\|,\|A(t,r)^{-1}\|
\right\}
\leq K_{\rm sel}\ell.
\end{equation}
The same statement holds conditionally after a stopping time $t$, on the event
$H_t\leq R$.
\end{proposition}

\begin{proof}
For $q\geq 0$, write
\begin{displaymath}
A(t,q):=A_{t+q-1}\cdots A_t:V_t\to V_{t+q},
\qquad
A(t,0)=\Id_{V_t}.
\end{displaymath}
Throughout the proof we condition on $\mathcal F_t$.

Let $a>0$ and $A<\infty$ be the constants from
\autoref{prop:initial:transverse:gap}.  Put
\begin{displaymath}
\lambda:=\frac{c_0}{4}.
\end{displaymath}
Fix $\eta>0$.  Choose $\chi>0$ and $\tau>0$ so small that
\begin{displaymath}
\chi<\frac a3,\qquad
\tau<\frac14,\qquad
4A\tau<\frac{\eta}{2},
\end{displaymath}
and
\begin{displaymath}
4\tau\left(\lambda+A\right)<\frac{\lambda}{4}.
\end{displaymath}
For $N$ large, set
\begin{displaymath}
p:=\lfloor \tau N\rfloor .
\end{displaymath}

Apply \autoref{prop:initial:transverse:gap} with starting time $t$, plane
$E$, and length $p$.  Outside an event of conditional probability at most
$Ce^{-c_1\sqrt p}$, there exists an integer $p_\ast$ with
\begin{displaymath}
p\leq p_\ast\leq 2p
\end{displaymath}
such that $H_{t+p_\ast} < R$ and
\begin{equation}
\label{eq:initial:cushion:lower}
\Delta_k(A(t,p_\ast),E)\,
\geq e^{ap_\ast},
\end{equation}
\begin{equation}
\label{eq:initial:angle:lower}
\angle\bigl(E,\ker_{d-k}A(t,p_\ast)\bigr)
\geq e^{-Ap_\ast}.
\end{equation}
On this event, choose $p_\ast$ to be the first integer in $[p,2p]$ satisfying
these conclusions.  (If no such integer exists, set $p_\ast = 2p$).
Then $p_\ast$ is a bounded stopping time.

Set
\begin{displaymath}
Q:=A(t,p_\ast),
\qquad
K_{p_\ast}:=\ker_{d-k}A(t,p_\ast),
\qquad
\delta_{p_\ast}:=\angle(E,K_{p_\ast}).
\end{displaymath}

We next apply the tail gap estimate after the stopping time $t+p_\ast$.
By \autoref{prop:tail:gap:from:prefix}, with prefix
\begin{displaymath}
Q=A(t,p_\ast):V_t\to V_{t+p_\ast}
\end{displaymath}
and loss $ap_\ast/4$, outside an additional exceptional event of conditional
probability at most $Ce^{-c_1\sqrt p}$, one has
\begin{equation}
\label{eq:selection:continuation:gap}
\log\Delta_k(A(t,p_\ast+q),E)-\log\Delta_k(A(t,p_\ast),E)
\geq
\lambda q-\frac a3p_\ast
\end{equation}
for every $q\geq 0$.  Here we have increased $N_0$ so that the additive
constant $C_{d,k}$ in \autoref{prop:tail:gap:from:prefix} is absorbed by
$ap_\ast/12$.

We now control the motion of the slow kernel during the continuation.  Apply
\autoref{lemma:kernel:stability:continued:gap} after the stopping time
$t+p_\ast$, with $T_q=A(t,p_\ast+q)$, $p_0=p_\ast$, $m=N$ and
$\Lambda=\lambda$. The condition
(\ref{eq:continued:gap:hypothesis})  holds for sufficiently
large $N_0$ since $\chi < a/3$ and by
\eqref{eq:selection:continuation:gap}  and
\eqref{eq:initial:cushion:lower},
\begin{displaymath}
\Delta_k(A(t,p_\ast+q),E)\geq e^{\lambda q - (a/3) p_\ast}
\Delta_k(A(t,p_\ast),E) \ge  e^{\lambda q+(2a/3)p_\ast}
\geq C e^{\lambda q+\chi p_\ast}.
\end{displaymath}
We get, outside an exceptional event
of conditional probability at most
$C e^{-c_1\chi p}$, for every $0\leq q\leq N$,
\begin{equation}
\label{eq:selection:angle:persist}
\angle\bigl(E,\ker_{d-k}A(t,p_\ast+q)\bigr)
\geq
\frac12\delta_{p_\ast}.
\end{equation}

We also need the block contraction estimate.  Apply
\autoref{prop:tail:contracting:from:prefix} with starting time $t$, prefix
$\Id_{V_t}:V_t\to V_t$, and loss $c_0N/16$.  Outside an exceptional event of
conditional probability at most $Ce^{-c_1\sqrt N}$, one has
\begin{equation}
\label{eq:selection:tail:contraction}
-\log\sigma_{k+1}(A(t,q))
\geq
\frac{c_0}{2}q-\frac{c_0}{16}N-C_{d,k}
\end{equation}
for every $q\geq 0$.

Apply \autoref{lemma:abstract:height:occupation}, with exponent $1$, to the
shifted process after time $t$ and the horizon $N$.  There is a constant
$D<\infty$, depending only on $R$ and the height-control constants, such
that, outside an event of conditional probability at most
$Ce^{-c_1\sqrt N}$,
\begin{equation}
\label{eq:selection:height:occupation}
\sum_{i=t}^{t+N-1}(1+\log H_i)\leq DN.
\end{equation}
On this event, \eqref{eq:height:controlled:operator:norm} gives, for every
$0\leq r\leq N$,
\begin{displaymath}
\begin{aligned}
\log\max\left\{
\|A(t,r)\|,\|A(t,r)^{-1}\|
\right\}
&\leq
\sum_{i=t}^{t+r-1}
\log\left(\|A_i\|+\|A_i^{-1}\|\right) \\
&\leq
B_{\rm lin}DN.
\end{aligned}
\end{displaymath}
Set
\begin{displaymath}
K_{\rm sel}:=2B_{\rm lin}D.
\end{displaymath}
Since every selected length satisfies $\ell\geq N/2$, it follows that
\begin{equation}
\label{eq:selection:partial:product:tameness}
\max_{0\leq r\leq\ell}
\log\max\left\{
\|A(t,r)\|,\|A(t,r)^{-1}\|
\right\}
\leq K_{\rm sel}\ell.
\end{equation}

It remains to choose the endpoint.  By \autoref{lemma:abstract:margulis:return:tails},
applied to the shifted process after time $t+\lceil N/2\rceil$, we have
\begin{displaymath}
\mathbb P\left(
H_{t+q} \ge R
\text{ for every }
\left\lceil\frac N2\right\rceil\leq q\leq N
\,\middle|\,
\mathcal F_{t+\lceil N/2\rceil}
\right)
\leq
C H_{t+\lceil N/2\rceil}e^{-c_1N/2}.
\end{displaymath}
Taking conditional expectation with respect to $\mathcal F_t$ and using
\eqref{eq:abstract:margulis:height:drift}, together with $H_t\leq R$, gives
\begin{displaymath}
\mathbb P\left(
H_{t+q} \ge R
\text{ for every }
\left\lceil\frac N2\right\rceil\leq q\leq N
\,\middle|\,
\mathcal F_t
\right)
\leq
Ce^{-c_1N}.
\end{displaymath}
Thus, outside this exceptional event, there exists
\begin{displaymath}
\left\lceil\frac N2\right\rceil\leq \ell\leq N
\end{displaymath}
such that $H_{t+\ell} < R$.  Fix such an $\ell$, and put
\begin{displaymath}
q_\ell:=\ell-p_\ast.
\end{displaymath}
Since $\tau<1/4$, after increasing $N_0$ we have
\begin{displaymath}
0\leq q_\ell\leq N.
\end{displaymath}

We verify the three conclusions at this length $\ell$.  First, by
\eqref{eq:selection:angle:persist} and \eqref{eq:initial:angle:lower},
\begin{displaymath}
\angle\bigl(E,\ker_{d-k}A(t,\ell)\bigr)
\geq
\frac12 e^{-Ap_\ast}.
\end{displaymath}
Since $p_\ast\leq 2p\leq 2\tau N\leq 4\tau\ell$, our choice of $\tau$ and a
further increase of $N_0$ give
\begin{displaymath}
\angle\bigl(E,\ker_{d-k}A(t,\ell)\bigr)
\geq
e^{-\eta\ell}.
\end{displaymath}
This proves \eqref{eq:parameterized:selection:angle}.

Second, using \eqref{eq:selection:angle:persist}
and
\eqref{eq:selection:continuation:gap} with $q=q_\ell$, we obtain
\begin{displaymath}
\begin{aligned}
\Delta_k(A(t,\ell),E)\angle\bigl(E,\ker_{d-k}A(t,\ell)\bigr)
&\geq \tfrac12\Delta_k(A(t,p_\ast),E)\delta_{p_\ast}
\exp\left(\lambda q_\ell-\tfrac a3p_\ast\right) \\
&\geq \tfrac12\exp\left(ap_\ast-Ap_\ast+\lambda q_\ell-\tfrac a3p_\ast\right) \\
&=\tfrac12\exp\left(\lambda\ell+
\left(\tfrac{2a}{3}-\lambda-A\right)p_\ast\right).
\end{aligned}
\end{displaymath}
(In the second inequality we used \eqref{eq:initial:cushion:lower} and
\eqref{eq:initial:angle:lower}; the equality uses $q_\ell=\ell-p_\ast$). 
Since $p_\ast\leq 4\tau\ell$, the choice of $\tau$ implies, after increasing
$N_0$, that (using (\ref{eq:relative:vs:absolute:gap}))
\begin{multline*}
\Delta_k(A(t,\ell))\angle\bigl(E,\ker_{d-k}A(t,\ell)\bigr) \ge
\Delta_k(A(t,\ell),E)\angle\bigl(E,\ker_{d-k}A(t,\ell)\bigr)  
\\  \geq \tfrac12\exp\left(\lambda\ell+
\left(\tfrac{2a}{3}-\lambda-A\right)p_\ast\right).
\end{multline*}
We claim that for $N_0$ sufficiently  large, 
\begin{displaymath}
\tfrac12\exp\left(\lambda\ell+
\left(\tfrac{2a}{3}-\lambda-A\right)p_\ast\right) \ge e^{(\lambda/2)
\ell}.  
\end{displaymath}
Indeed, suppose $N_0$ is large enough so that for  $N  \ge  N_0$,
$\lambda N/8 \ge \log
2$. Then, using $a>0$, $p_\ast \le 4\tau \ell$,
$\ell \ge N/2$ and the chosen condition
$4\tau(\lambda+A)<\lambda/4$:
\begin{displaymath}
\begin{aligned}
\lambda\ell+\left(\tfrac{2a}{3}-\lambda-A\right)p_\ast-\log2
&\geq\lambda\ell-(\lambda+A)p_\ast-\log2 \\
&\geq\bigl(\lambda-4\tau(\lambda+A)\bigr)\ell-\log2 \\
&\geq\tfrac{3\lambda}{4}\ell-\log2 \\
&\geq\tfrac\lambda2\ell.
\end{aligned}
\end{displaymath}
Thus, \eqref{eq:parameterized:selection:cushion} holds with $\gamma =
\frac{\lambda}{2}$. 

Third, applying \eqref{eq:selection:tail:contraction} at $q=\ell$ gives
\begin{displaymath}
-\log\sigma_{k+1}(A(t,\ell))
\geq
\frac{c_0}{2}\ell-\frac{c_0}{16}N-C_{d,k}.
\end{displaymath}
Since $\ell\geq N/2$, after increasing $N_0$ this gives
\begin{displaymath}
-\log\sigma_{k+1}(A(t,\ell))
\geq
\frac{c_0}{4}\ell.
\end{displaymath}
Thus we may take
\begin{displaymath}
\lambda_{\rm ctr}:=\frac{c_0}{4}.
\end{displaymath}
This proves \eqref{eq:parameterized:selection:block:contraction}.
The estimate \eqref{eq:parameterized:selection:block:tameness} is
\eqref{eq:selection:partial:product:tameness}.

The total exceptional probability is bounded by
\begin{displaymath}
Ce^{-c_1\sqrt p}
+
Ce^{-c_1\sqrt N}
+
Ce^{-c_2\chi p}
+
Ce^{-c_1N},
\end{displaymath}
where the second term also absorbs the exceptional event in
\eqref{eq:selection:height:occupation}.
Since $p=\lfloor\tau N\rfloor$, this is at most
\begin{displaymath}
C_\eta e^{-b_\eta\sqrt N}
\end{displaymath}
after changing constants.  The stopping-time version follows from the
stopping-time versions of the tail estimates, the return estimate, and the
kernel-stability estimate.
\end{proof}

\subsubsection{Defining the good events}
\label{sssec:good:event}
\label{sssec:def:event}

Fix $R_{\rm init}<\infty$ sufficiently large, and fix
$\epsilon_{\rm init}>0$.  Let
\begin{displaymath}
\gamma_0>0,
\qquad
\lambda_0>0
\end{displaymath}
be the constants from
\autoref{prop:parameterized:one:block:selection}.  Set
\begin{displaymath}
\gamma:=\frac{\gamma_0}{2},
\qquad
\lambda_{\rm sel}:=\frac{\lambda_0}{2}.
\end{displaymath}

Let
\begin{displaymath}
M_{\rm buf}=M_{\rm buf}(d,k)
\end{displaymath}
be the constant from \autoref{lemma:create:terminal:buffer}, and let
\begin{displaymath}
C_{\rm reg}^\ast=C_{\rm reg}^\ast(d,k)\geq1
\end{displaymath}
be larger than all regularity thresholds occurring in
\autoref{ssec:regular:sequences} and
\autoref{prop:hyperbolicity:admissible}.

Choose $\eta_0>0$ sufficiently small that, for every
$0<\eta<\eta_0$,
\begin{equation}
\label{eq:good:event:eta:gap}
\frac{\gamma}{2}-2\eta
\geq
4M_{\rm buf}\eta
\end{equation}
and
\begin{equation}
\label{eq:good:event:eta:contraction}
\frac{\lambda_{\rm sel}}{2}
\geq
4M_{\rm buf}\eta.
\end{equation}

Fix $0<\eta<\eta_0$.  Apply
\autoref{prop:parameterized:one:block:selection} with parameter $\eta/2$ and
height level $R_{\rm init}$.  Let
\begin{displaymath}
K_\eta<\infty,
\qquad
b_\eta>0,
\qquad
K_{\rm sel}<\infty,
\qquad
N_0=N_0(\eta,R_{\rm init})
\end{displaymath}
be the resulting constants, and set
\begin{equation}
\label{eq:good:event:block:tameness:constant}
K_{\rm blk}:=2K_{\rm sel}.
\end{equation}

We can increase $N_0 = N_0(\eta,R_{\rm init}, \epsilon_{\rm init})$ to be sufficiently large so
that for $N > N_0$, 
\begin{equation}
\label{eq:good:event:N:probability}
K_\eta e^{-b_\eta\sqrt N}\leq\frac12
\end{equation}
and
\begin{equation}
\label{eq:good:event:N:initial:gap}
\frac{2}{\pi}e^{\gamma N}>e^{\epsilon_{\rm init}}.
\end{equation}
Set
\begin{equation}
\label{eq:good:event:delta:C}
\delta_\eta:=e^{-\eta N},
\qquad
C_\eta:=e^{(\gamma/2-2\eta)N},
\qquad
L_\eta:=\frac{\lambda_{\rm sel}}{2}N.
\end{equation}
We further require that $\eta_0$ is small enough and $N_0 = N_0(\eta,R_{\rm
  init}, \epsilon_{\rm init})$ is large enough so that for $0 < \eta < \eta_0$ and $N >
N_0$, 
\begin{equation}
\label{eq:good:event:C:large}
\frac12C_\eta\geq C_{\rm reg}^\ast,
\end{equation}
\begin{equation}
\label{eq:good:event:C:dominates:delta}
\log C_\eta\geq
4M_{\rm buf}|\log\delta_\eta|,
\end{equation}
\begin{equation}
\label{eq:good:event:L:dominates:delta}
L_\eta\geq
2M_{\rm buf}(|\log\delta_\eta|+1),
\end{equation}
\begin{equation}
\label{eq:good:event:deterministic:effective}
\log C_\eta-\log2\geq\frac{\gamma}{128}N,
\end{equation}
\begin{equation}
\label{eq:good:event:deterministic:contraction}
L_\eta+\log\delta_\eta-\log2
\geq
\frac{\lambda_{\rm sel}}{8}N,
\end{equation}
and
\begin{equation}
\label{eq:N:lower:bound:in:terms:of:eta}
N \ge \eta^{-1}.  
\end{equation}
\begin{equation}
\label{eq:pi:2:gamma:N}
\frac{2}{\pi}e^{\gamma N}\ge e^{\gamma N/2}.
\end{equation}
We now fix $N = N(\eta)$ to be an odd integer greater than $N_0(\eta,R_{\rm init}, \epsilon_{\rm init})$. 
Let
\begin{displaymath}
\mathcal O\subset\Gr_k(\bR^d)\times GL(d,\bR)
\end{displaymath}
be the set of pairs $(E,Q)$ satisfying
\begin{equation}
\label{eq:good:event:open:next:gap}
\Delta_k(Q)>e^{\epsilon_{\rm init}},
\end{equation}
\begin{equation}
\label{eq:good:event:open:angle}
\angle\bigl(E,\ker_{d-k}(Q)\bigr)>e^{-\eta N},
\end{equation}
\begin{equation}
\label{eq:good:event:open:cushion}
\Delta_k(Q)
\angle\bigl(E,\ker_{d-k}(Q)\bigr)>e^{\gamma N},
\end{equation}
and
\begin{equation}
\label{eq:good:event:open:contraction}
-\log\sigma_{k+1}(Q)>\lambda_{\rm sel} N.
\end{equation}
The set $\mathcal O$ is open and invariant under orthogonal
changes of coordinates. It depend on the constants of
\S\ref{ssec:abstract:good:event}, $\eta$, $R_{\rm init}$ and $\epsilon_{\rm
  init}$. (The constant $N$ is now fixed in terms of the other
constants).

\begin{lemma}[Positive probability of a good block]
\label{lemma:good:event:positive:probability}
Let $t\geq1$ be an integer-valued stopping time, and let
\begin{displaymath}
E \in  \Gr_k(V_t)
\end{displaymath}
be an $\mathcal F_t$-measurable $k$-plane.  Almost surely on the
event  $H_t \le R_{\rm init}$, 
one has
\begin{multline}
\label{eq:good:event:one:block}
\mathbb P\Biggl(
\exists\,\ell\in\N,\ \frac N2<\ell\leq N
\text{ such that }\;
H_{t+\ell}<R_{\rm init},
(P,A(t,\ell))\in\mathcal O,\\
\max_{0\leq r\leq\ell}
\log\max\left\{
\|A(t,r)\|,\|A(t,r)^{-1}\|
\right\}
<K_{\rm blk}\ell
\ \Bigm|\ \mathcal F_t
\Biggr)
\geq\frac12 .
\end{multline}
\end{lemma}

\begin{proof}
Apply \autoref{prop:parameterized:one:block:selection} at the stopping time
$t$, with height level $R_{\rm init}$, input plane $E$, parameter $\eta/2$,
and target length $N$.  Since $N$ is odd, with conditional probability at
least
\begin{displaymath}
1-K_\eta e^{-b_\eta\sqrt N},
\end{displaymath}
there is an integer $N/2<\ell\leq N$ such that
\begin{displaymath}
H_{t+\ell}<R_{\rm init},
\end{displaymath}
and, writing $Q:=A(t,\ell)$, the estimates
\eqref{eq:parameterized:selection:angle},
\eqref{eq:parameterized:selection:cushion},
\eqref{eq:parameterized:selection:block:contraction}, and
\eqref{eq:parameterized:selection:block:tameness} hold.

Using
\begin{displaymath}
\gamma=\frac{\gamma_0}{2},
\qquad
\lambda_{\rm sel}=\frac{\lambda_0}{2},
\qquad
K_{\rm blk}=2K_{\rm sel},
\end{displaymath}
together with $N/2<\ell\leq N$, these estimates give
\begin{displaymath}
\Delta_k(Q)
\geq
\frac{2}{\pi}e^{\gamma_0\ell}
>
\frac{2}{\pi}e^{\gamma N}
> e^{\epsilon_{\rm init}},
\end{displaymath}
\begin{displaymath}
\angle\bigl(E,\ker_{d-k}(Q)\bigr)>e^{-\eta N},
\end{displaymath}
\begin{displaymath}
\Delta_k(Q)
\angle\bigl(E,\ker_{d-k}(Q)\bigr)>e^{\gamma N},
\end{displaymath}
\begin{displaymath}
-\log\sigma_{k+1}(Q)>\lambda_{\rm sel}N,
\end{displaymath}
and
\begin{displaymath}
\max_{0\leq r\leq\ell}
\log\max\left\{
\|A(t,r)\|,\|A(t,r)^{-1}\|
\right\}
<K_{\rm blk}\ell.
\end{displaymath}
Here the first inequality in the first display
uses that Grassmannian angles are at most
$\pi/2$, and the last inequality in that line follows from
(\ref{eq:good:event:N:initial:gap}).

Thus $(E,Q)\in\mathcal O$, and all the conditions in
\eqref{eq:good:event:one:block} hold.  The conclusion follows from
\eqref{eq:good:event:N:probability}.
\end{proof}

\subsection{Extending uniformly regular prefixes to admissible sequences}
\label{ssec:regular:to:admissible}

\begin{definition}[Sequence with $(C,\epsilon,\lambda,\lambda_{\rm
    ctr},K_{\rm blk},N_{\rm adm})$-admissible  tail]
\label{def:tail:admissible}
Let
\begin{displaymath}
A_i:V_i\to V_{i+1},
\qquad i\geq0,
\end{displaymath}
be a sequence of linear isomorphisms between $d$-dimensional Euclidean
spaces.  We say that the sequence $A_i$ has a
$(C,\epsilon,\lambda,\lambda_{\rm ctr},K_{\rm blk},N_{\rm
  adm})$-admissible tail in dimension
$k$ if there exists an increasing sequence of times
\begin{displaymath}
0=t_0 <t_1 <t_2<t_3<\cdots
\end{displaymath}
such that, writing for $j \ge 0$,
$\ell_j:=t_{j+1}-t_j$,
and
\begin{displaymath}
B_j:=A(t_j,\ell_j) :V_{t_j}\to V_{t_{j+1}},
\end{displaymath}
the following conditions hold.

Assume each $B_j$ has a strict $k$-gap.  For $j\geq1$, set
\begin{displaymath}
\theta_j:=
\angle\bigl(\im_k(B_{j-1}),\ker_{d-k}(B_j)\bigr).
\end{displaymath}

\begin{enumerate}
\item The block lengths satisfy 
\begin{equation}
\label{eq:admissible:length}
\tfrac{N_{\rm adm}}{2}  \le \ell_j\leq\max\{N_{\rm adm},\epsilon t_j\}
\qquad
\text{for every }j\geq1.
\end{equation}
\item Every block is $K_{\rm blk}$-tame:
\begin{equation}
\label{eq:admissible:block:tameness}
\max_{0\leq r\leq\ell_j}
\log\max\left\{
\|A(t_j,r)\|,\|A(t_j,r)^{-1}\|
\right\}
\leq K_{\rm blk}\ell_j
\qquad
\text{for every }j\geq1.
\end{equation}

\item The block sequence is $C$-regular:
\begin{equation}
\label{eq:admissible:regular}
\Delta_k(B_j)\theta_j\theta_{j+1}\geq C
\qquad
\text{for every }j\geq 1.
\end{equation}

\item The regularity is effective at the scale of the block length:
\begin{equation}
\label{eq:admissible:effective}
\log\bigl(\Delta_k(B_j)\theta_j\theta_{j+1}\bigr)
\geq\lambda\ell_j
\qquad
\text{for every }j\geq 1.
\end{equation}

\item The stable contraction is effective after paying the outgoing angle:
\begin{equation}
\label{eq:admissible:contraction}
-\log\sigma_{k+1}(B_j)+\log\theta_{j+1}
\geq\lambda_{\rm ctr}\ell_j
\qquad
\text{for every }j\geq 1.
\end{equation}
\end{enumerate}
\end{definition}

\begin{definition}[$(C,\epsilon,\lambda,\lambda_{\rm
    ctr},K_{\rm blk},N_{\rm adm})$-admissible  sequence]
\label{def:admissible}
We say that the sequence $A_i$ is
$(C,\epsilon,\lambda,\lambda_{\rm ctr},K_{\rm blk},N_{\rm
  adm})$-admissible in dimension $k$ if it has a
$(C,\epsilon,\lambda,\lambda_{\rm ctr},K_{\rm blk},N_{\rm adm})$-admissible tail and we can take $t_1 = 1$.   
\end{definition}

\begin{lemma}[Good prefixes admit admissible extensions]
\label{lemma:good:prefix:admissible:extension}
There exist constants
$\lambda_{\rm adm}>0$,
$\lambda_{\rm ctr}>0$, $\eta_1 > 0$,
depending only on the constants in \S\ref{ssec:abstract:good:event}, 
and for every $0<\eta<\eta_1$ there exist constants
\begin{displaymath}
C_{\rm adm}\geq C_{\rm reg}^\ast,
\qquad
\epsilon_{\rm adm}=\epsilon_{\rm adm}(\eta)>0,
\qquad
\beta>0,
\end{displaymath}
depending on the constants in \S\ref{ssec:abstract:good:event}, $R_{\rm
  init}$, $\epsilon_{\rm init}$ and $\eta$, with 
$\epsilon_{\rm adm}(\eta)\to 0$
as $\eta\to 0$,
such that the following holds. Let $\epsilon_{\rm init}$ and
$N = N(\eta,R_{\rm init},\epsilon_{\rm init})$ be as in \S\ref{sssec:def:event}.
Let $C_\eta$ be as in (\ref{eq:good:event:delta:C}). 
Let
\begin{displaymath}
0=s_0<s_1=1<\cdots<s_m
\end{displaymath}
be integer-valued stopping times.  Suppose that $m$ is sufficiently
large and
\begin{displaymath}
\mathcal E^{(m)}(s_0,\ldots,s_m)
\end{displaymath}
occurs, where $\cE^{(m)}(s_0, \ldots, s_m)$ be as in
Proposition~\ref{prop:good:block:events}(b). 
Then, conditionally on $\mathcal F_{s_m}$, an infinite random
extension of $A_0,A_1,\dots, A_{s_m-1}$
is, with probability at least
$1-e^{-\beta\sqrt m}$,
a $(\tfrac{1}{2}C_\eta,\epsilon_{\rm adm},
\lambda_{\rm adm},\lambda_{\rm ctr},K_{\rm blk},N)$-admissible sequence.
\end{lemma}

\begin{proof}
Put
\begin{displaymath}
n:=m-1,
\qquad
d_j:=s_{j+1}-s_j,
\qquad
B_j:=A(s_j,d_j)
\qquad
(0\leq j\leq n).
\end{displaymath}

For $0<\eta<\eta_1$, put
\begin{equation}
\label{eq:good:prefix:kappa:rho:def}
\kappa:=\eta^2,
\qquad
\rho:=\eta,
\qquad
a_\eta:=\frac{\gamma}{2}-2\eta.
\end{equation}
Choose $\eta_1 < \eta_0$ sufficiently small such that for $\eta <
\eta_1$,  $\rho<1/4$,
$a_\eta\geq\gamma/4$, and
\begin{equation}
\label{eq:good:prefix:eta:random:regular}
\gamma-2(1+2\kappa)\eta\geq\frac{\gamma}{2},
\end{equation}
\begin{equation}
\label{eq:good:prefix:eta:random:contraction}
\lambda_{\rm sel}-2(1+2\kappa)\eta
\geq\frac{\lambda_{\rm sel}}{2},
\end{equation}
\begin{equation}
\label{eq:good:prefix:eta:buffer:regular}
16\eta\kappa\leq\rho a_\eta,
\end{equation}
and
\begin{equation}
\label{eq:good:prefix:eta:buffer:contraction}
32\eta\kappa\leq\rho\lambda_{\rm sel}.
\end{equation}

Let  $C_\eta$ be as in (\ref{eq:good:event:delta:C}), and
set
\begin{equation}
\label{eq:good:prefix:constants}
C_{\rm adm}:=\frac12C_\eta,
\qquad
\lambda_{\rm adm}:=\frac{\gamma}{128},
\qquad
\lambda_{\rm ctr}:=\frac{\lambda_{\rm sel}}{32},
\end{equation}
and
\begin{equation}
\label{eq:good:prefix:epsilon:def}
\epsilon_{\rm adm}:=100(\rho+\kappa).
\end{equation}
By \eqref{eq:good:event:C:large},
\begin{displaymath}
C_{\rm adm}\geq C_{\rm reg}^\ast.
\end{displaymath}
Since $\mathcal E^{(m)}$ occurs, the sequence $s_0,\dots, s_m$ is
$(R_{\rm init}, K_{\rm blk})$-height controlled and is 
adapted to $\cO$. 
Put
\begin{displaymath}
\theta_j:=
\angle\bigl(\im_k(B_{j-1}),\ker_{d-k}(B_j)\bigr).
\end{displaymath}
The height-control condition and the definition of $\cO$ gives
\begin{displaymath}
H_{s_1} \le R_{\rm init}
\end{displaymath}
\begin{equation}
\label{eq:good:prefix:event:return}
H_{s_j}<R_{\rm init}, \qquad\text{$j \ge 2$}
\end{equation}
and for $1 \le  j \le n$,
\begin{equation}
\label{eq:good:prefix:event:length}
\frac N2<d_j\leq N,
\end{equation}
\begin{equation}
\label{eq:good:prefix:event:angle}
\theta_j>e^{-\eta N},
\end{equation}
\begin{equation}
\label{eq:good:prefix:event:cushion}
\Delta_k(B_j)\theta_j> e^{\gamma N} \ge e^{\gamma d_j},
\end{equation}
and
\begin{equation}
\label{eq:good:prefix:event:contraction}
-\log\sigma_{k+1}(B_j)>\lambda_{\rm sel}d_j.
\end{equation}
Moreover,
\begin{equation}
\label{eq:good:prefix:event:tameness}
\max_{0\leq r\leq d_j}
\log\max\left\{
\|A(s_j,r)\|,\|A(s_j,r)^{-1}\|
\right\}
<K_{\rm blk}d_j.
\end{equation}
Since $\theta_j\leq \pi/2$ it follows from
(\ref{eq:good:prefix:event:cushion}) and (\ref{eq:pi:2:gamma:N}) that
\begin{equation}
\label{eq:extension:cushion}
\Delta_k(B_j)>\frac{2}{\pi} e^{\gamma N}\geq e^{\gamma N/2}.
\end{equation}
Let $\delta_\eta$ 
be as in (\ref{eq:good:event:delta:C}). Then,
for $1\leq j\leq n$,
\begin{equation}
\label{eq:extension:angle}
\theta_j> e^{-\eta N}=\delta_\eta,
\end{equation}
and, for $1\leq j\leq n$, in view of (\ref{eq:extension:cushion}), 
\begin{equation}
\label{eq:extension:uniform:regularity}
\Delta_k(B_j)\delta_\eta^2 >
e^{\gamma N/2-2\eta N}
=
C_\eta.
\end{equation}
Thus $B_1,\ldots,B_n$ is
$(C_\eta,\delta_\eta)$-uniformly regular.  Finally, with $L_\eta$ as
in (\ref{eq:good:event:delta:C}), 
\begin{equation}
\label{eq:extension:prefix:contraction}
-\log\sigma_{k+1}(B_j) >
\frac{\lambda_{\rm sel}}2N
=
L_\eta.
\end{equation}

For $n$ sufficiently large, set
\begin{displaymath}
r:=\lfloor\rho n\rfloor,
\qquad
a:=n-r+1,
\end{displaymath}
and define
\begin{displaymath}
\widehat B:=B_nB_{n-1}\cdots B_a.
\end{displaymath}
Its length is
\begin{displaymath}
d_{\rm buf}:=\sum_{j=a}^n d_j.
\end{displaymath}
We have
\begin{equation}
\label{eq:good:prefix:buffer:length}
\frac{rN}{2}<d_{\rm buf}\leq rN.
\end{equation}

Apply \autoref{lemma:create:terminal:buffer} with
\begin{displaymath}
C=C_\eta,
\qquad
\delta=\delta_\eta,
\qquad
L_{\rm blk}=L_\eta.
\end{displaymath}
The hypotheses follow from
\eqref{eq:good:event:C:large},
\eqref{eq:good:event:C:dominates:delta},
\eqref{eq:good:event:L:dominates:delta}, and
\eqref{eq:good:prefix:event:contraction}.  We obtain
\begin{equation}
\label{eq:good:prefix:buffer:gap}
\log\Delta_k(\widehat B)
\geq
\frac12r\log C_\eta,
\end{equation}
\begin{equation}
\label{eq:good:prefix:buffer:angle}
\angle\bigl(
\im_k(B_{a-1}),
\ker_{d-k}\widehat B
\bigr)
\geq
\frac{\delta_\eta}{2},
\end{equation}
\begin{equation}
\label{eq:good:prefix:buffer:cushion}
\log \left( \Delta_k(\widehat B)  \angle\bigl(
\im_k(B_{a-1}),
\ker_{d-k}\widehat B
\bigr)\right) \geq 
\frac12r\log C_\eta+\log\delta_\eta-\log2,
\end{equation}
and
\begin{equation}
\label{eq:good:prefix:buffer:contraction}
-\log\sigma_{k+1}(\widehat B)
\geq
\frac14r\lambda_{\rm sel}N.
\end{equation}

Set
\begin{displaymath}
u_0:=s_{n+1}=s_m.
\end{displaymath}
By \eqref{eq:good:prefix:event:return},
\begin{displaymath}
H_{u_0}<R_{\rm init}.
\end{displaymath}
Suppose recursively that $u_q$ has been defined and
$H_{u_q}\leq R_{\rm init}$.  Put
\begin{displaymath}
m_q:=\lfloor\kappa u_q\rfloor.
\end{displaymath}
For $n$ sufficiently large, $m_q\geq N_0$.  Apply
\autoref{prop:parameterized:one:block:selection} at time $u_q$, with target
length $m_q$, height level $R_{\rm init}$, parameter $\eta/2$, and input
plane
\begin{displaymath}
E_0:=\im_k(\widehat B),
\qquad
E_q:=\im_k(R_{q-1})
\quad
(q\geq1).
\end{displaymath}
Let $\mathcal G_q$ be the success event.  On $\mathcal G_q$, choose the least
good length $d_q'$ and set
\begin{displaymath}
u_{q+1}:=u_q+d_q',
\qquad
R_q:=A(u_q,d_q').
\end{displaymath}
By Proposition~\ref{prop:parameterized:one:block:selection}, one has
\begin{displaymath}
\frac{m_q}{2}\leq d_q'\leq m_q,
\end{displaymath}
\begin{equation}
\label{eq:good:prefix:random:angle}
\angle(E_q,\ker_{d-k}R_q)\geq e^{-\eta d_q'},
\end{equation}
\begin{equation}
\label{eq:good:prefix:random:cushion}
\Delta_k(R_q)\angle(E_q,\ker_{d-k}R_q)
\geq e^{\gamma d_q'},
\end{equation}
\begin{equation}
\label{eq:good:prefix:random:contraction}
-\log\sigma_{k+1}(R_q)
\geq\lambda_{\rm sel}d_q',
\end{equation}
\begin{equation}
\label{eq:good:prefix:random:tameness}
\max_{0\leq r\leq d_q'}
\log\max\left\{
\|A(u_q,r)\|,\|A(u_q,r)^{-1}\|
\right\}
\leq K_{\rm sel}d_q'
<K_{\rm blk}d_q',
\end{equation}
and
\begin{displaymath}
H_{u_{q+1}}<R_{\rm init}.
\end{displaymath}
Moreover,
\begin{equation}
\label{eq:good:prefix:random:failure}
\mathbb P\left(
\mathcal G_q^c
\,\middle|\,
\mathcal F_{u_q}
\right)
\leq
K_\eta e^{-b_\eta\sqrt{m_q}}.
\end{equation}

For all sufficiently large $n$, on the preceding successful blocks,
\begin{displaymath}
m_q\geq\frac{\kappa}{2}u_q,
\qquad
d_q'\geq\frac{\kappa}{4}u_q,
\end{displaymath}
and hence
\begin{displaymath}
u_{q+1}\geq
\left(1+\frac{\kappa}{4}\right)u_q.
\end{displaymath}
It follows from \eqref{eq:good:prefix:random:failure} that
\begin{displaymath}
\mathbb P\left(
\bigcup_{q\geq0}\mathcal G_q^c
\,\middle|\,
\mathcal F_{u_0}
\right)
\leq
C e^{-c_1\sqrt{u_0}}.
\end{displaymath}
Since
\begin{displaymath}
u_0=s_{n+1}>1+\frac{nN}{2},
\end{displaymath}
the last quantity is at most
\begin{displaymath}
e^{-\beta\sqrt n}
\end{displaymath}
for some $\beta>0$, after increasing the lower threshold on $n$.  We work
on the event that every $\mathcal G_q$ occurs.

Set $t_0 = s_0$.  
The admissible blocks are
\begin{displaymath}
B_1,\ldots,B_{a-1},\widehat B,R_0,R_1,\ldots.
\end{displaymath}
Their starting times are
\begin{displaymath}
t_j:=s_j
\qquad
(1\leq j\leq a),
\end{displaymath}
\begin{displaymath}
t_{a+1}:=u_0,
\end{displaymath}
and
\begin{displaymath}
t_{a+q+2}:=u_{q+1}
\qquad
(q\geq0).
\end{displaymath}

The deterministic blocks have length at most $N$.  Since
\begin{displaymath}
t_a=s_a>1+\frac{(n-r)N}{2},
\end{displaymath}
\eqref{eq:good:prefix:buffer:length} gives, for all sufficiently large
$n$, in view of (\ref{eq:good:prefix:epsilon:def}), 
\begin{displaymath}
d_{\rm buf}\leq8\rho t_a
\leq\epsilon_{\rm adm}t_a.
\end{displaymath}
The random blocks satisfy, in view of (\ref{eq:good:prefix:epsilon:def}),
\begin{displaymath}
d_q'\leq m_q\leq\kappa u_q
\leq\epsilon_{\rm adm}u_q.
\end{displaymath}
Thus the upper bound in \eqref{eq:admissible:length} holds with
$N_{\rm adm}=N$. We choose $n$ large enough so that
\begin{displaymath}
m_q \ge  N \qquad \text{for every $q  \ge 0$.} 
\end{displaymath}
Then,
\begin{displaymath}
d_q' \ge \frac{m_q}{2} \ge \frac{N}{2}.
\end{displaymath}
Then the lower bound in \eqref{eq:admissible:length} also holds with
$N_{\rm adm}=N$.

We next verify \eqref{eq:admissible:block:tameness}.  It holds for
$B_1,\ldots,B_{a-1}$ by \eqref{eq:good:prefix:event:tameness}, and for every
random block $R_q$ by \eqref{eq:good:prefix:random:tameness}.  For the buffer
block, put
\begin{displaymath}
\mathfrak q(T):=\max\{\|T\|,\|T^{-1}\|\}.
\end{displaymath}
The inequality
\begin{displaymath}
\mathfrak q(T_2T_1)\leq\mathfrak q(T_2)\mathfrak q(T_1)
\end{displaymath}
and \eqref{eq:good:prefix:event:tameness} imply, by decomposing any partial
product inside $\widehat B$ into complete good blocks and at most one final
partial block, that
\begin{displaymath}
\max_{0\leq r\leq d_{\rm buf}}
\log\mathfrak q\bigl(A(s_a,r)\bigr)
\leq K_{\rm blk}d_{\rm buf}.
\end{displaymath}
Thus \eqref{eq:admissible:block:tameness} holds for every block in the
admissible sequence.

It remains to verify (\ref{eq:admissible:regular}),
(\ref{eq:admissible:effective}) and (\ref{eq:admissible:contraction}).
First, we record two auxiliary estimates. 
Note that by (\ref{eq:good:event:deterministic:effective})
 and
(\ref{eq:good:prefix:constants}),
\begin{equation}
\label{eq:extension:Ceta:lower:bound}
\log C_\eta - \log 2 \ge \max(\tfrac{\gamma}{128}  N, \log C_{\rm adm})
\ge \max(\lambda_{\rm adm}\tfrac{N}{2}, \log C_{\rm adm}).
\end{equation}
Also in view of (\ref{eq:good:event:deterministic:contraction})
and (\ref{eq:good:prefix:constants}), 
\begin{equation}
\label{eq:extension:Leta:lower:bound}
L_\eta + \log \delta_\eta - \log 2 \ge \tfrac{\lambda_{\rm sel}}{8}N \ge 
\lambda_{\rm ctr}\tfrac{N}{2}.  
\end{equation}
We now begin the verification of (\ref{eq:admissible:regular}),
(\ref{eq:admissible:effective}) and (\ref{eq:admissible:contraction}).
For $1\leq j\leq a-2$, using (\ref{eq:extension:angle}) and
(\ref{eq:extension:uniform:regularity}), 
\begin{displaymath}
\log\bigl(\Delta_k(B_j)\theta_j\theta_{j+1}\bigr)
\geq \log C_\eta \ge \max(\lambda_{\rm adm}d_j, \log C_{\rm adm}),
\end{displaymath}
where for the last estimate we used
(\ref{eq:extension:Ceta:lower:bound}) and (\ref{eq:good:prefix:event:length}).
Similarly,  using (\ref{eq:extension:prefix:contraction}), for $1  \le
j \le a-2$, 
\begin{displaymath}
-\log\sigma_{k+1}(B_j)+\log\theta_{j+1} \ge L_\eta + \log \delta_\eta 
\geq \lambda_{\rm ctr}d_j,
\end{displaymath}
where for the last estimate we used
(\ref{eq:extension:Leta:lower:bound}) and (\ref{eq:good:prefix:event:length}).
Thus (\ref{eq:admissible:regular}),
(\ref{eq:admissible:effective}) and (\ref{eq:admissible:contraction})
hold for $B_j$, $1 \le  j  \le  a-2$.  

For $B_{a-1}$, (\ref{eq:extension:uniform:regularity}),
(\ref{eq:extension:angle}) and
\eqref{eq:good:prefix:buffer:angle} give
\begin{displaymath}
\Delta_k(B_{a-1})\theta_{a-1}
\angle\bigl(\im_k(B_{a-1}),\ker_{d-k}\widehat B\bigr)
\geq
\frac12C_\eta
\end{displaymath}
Then, by (\ref{eq:extension:Ceta:lower:bound}) and
(\ref{eq:good:prefix:event:length}), 
\begin{displaymath}
\log\left(\Delta_k(B_{a-1})\theta_{a-1}
\angle\bigl(\im_k(B_{a-1}),\ker_{d-k}\widehat B\bigr)\right) \ge
\max(\lambda_{\rm adm}d_{a-1}, \log C_{\rm adm}). 
\end{displaymath}
By (\ref{eq:extension:prefix:contraction}) and
(\ref{eq:good:prefix:buffer:angle}), 
\begin{displaymath}
-\log\sigma_{k+1}(B_{a-1})
+
\log\angle\bigl(\im_k(B_{a-1}),\ker_{d-k}\widehat B\bigr) \ge L_\eta +
\log \delta_\eta
- \log 2 \geq
\lambda_{\rm ctr}d_{a-1},
\end{displaymath}
where we have used (\ref{eq:extension:Leta:lower:bound}) and
(\ref{eq:good:prefix:event:length}) for the last estimate. 
Thus, (\ref{eq:admissible:regular}),
(\ref{eq:admissible:effective}) and (\ref{eq:admissible:contraction})
hold  for  $B_{a-1}$.  

For the buffer block, \eqref{eq:good:prefix:random:angle}, with $q=0$, gives
\begin{displaymath}
\angle\bigl(\im_k(\widehat B),\ker_{d-k}R_0\bigr)
\geq
e^{-\eta d_0'}.
\end{displaymath}
Combining this with \eqref{eq:good:prefix:buffer:cushion}, we obtain
\begin{equation}
\label{eq:extension:buffer:regularity:initial}
\begin{aligned}
&\log\left(
\Delta_k(\widehat B)
\angle\bigl(\im_k(B_{a-1}),\ker_{d-k}\widehat B\bigr)
\angle\bigl(\im_k(\widehat B),\ker_{d-k}R_0\bigr)
\right) \\
&\qquad\geq
\frac12r\log C_\eta+\log\delta_\eta-\log2-\eta d_0'.
\end{aligned}
\end{equation}
By \eqref{eq:good:prefix:kappa:rho:def} and
\eqref{eq:good:event:delta:C},
\begin{displaymath}
\log C_\eta=a_\eta N,
\qquad
\log\delta_\eta=-\eta N.
\end{displaymath}
Moreover, for all sufficiently large $n$,
\begin{displaymath}
r\geq\frac{\rho n}{2},
\qquad
u_0=s_{n+1}\leq1+nN\leq2nN,
\end{displaymath}
and hence
\begin{displaymath}
d_0'\leq m_0\leq\kappa u_0\leq2\kappa nN.
\end{displaymath}
It follows from \eqref{eq:extension:buffer:regularity:initial} that
\begin{displaymath}
\begin{aligned}
&\log\left(
\Delta_k(\widehat B)
\angle\bigl(\im_k(B_{a-1}),\ker_{d-k}\widehat B\bigr)
\angle\bigl(\im_k(\widehat B),\ker_{d-k}R_0\bigr)
\right) \\
&\qquad\geq
\left(\frac{\rho a_\eta}{4}-2\eta\kappa\right)nN
-\eta N-\log2.
\end{aligned}
\end{displaymath}
By \eqref{eq:good:prefix:eta:buffer:regular},
\begin{displaymath}
2\eta\kappa\leq\frac{\rho a_\eta}{8}.
\end{displaymath}
Therefore, after increasing the lower threshold on $n$,
\begin{displaymath}
\begin{aligned}
&\log\left(
\Delta_k(\widehat B)
\angle\bigl(\im_k(B_{a-1}),\ker_{d-k}\widehat B\bigr)
\angle\bigl(\im_k(\widehat B),\ker_{d-k}R_0\bigr)
\right) \\
&\qquad\geq
\frac{\rho a_\eta}{16}nN
\geq
\frac{\rho\gamma}{64}nN,
\end{aligned}
\end{displaymath}
where the last inequality uses $a_\eta\geq\gamma/4$.

Since
\begin{displaymath}
d_{\rm buf}\leq rN\leq\rho nN
\end{displaymath}
and $\lambda_{\rm adm}=\gamma/128$, the preceding lower bound is at least
\begin{displaymath}
\lambda_{\rm adm}d_{\rm buf}.
\end{displaymath}
After increasing the lower threshold on $n$ once more, it is also at least
$\log C_{\rm adm}$.  Thus
\eqref{eq:admissible:regular} and
\eqref{eq:admissible:effective} hold for the buffer block.

For the contraction estimate, combining
\eqref{eq:good:prefix:buffer:contraction} with the outgoing-angle estimate
gives
\begin{displaymath}
\begin{aligned}
&-\log\sigma_{k+1}(\widehat B)
+
\log\angle\bigl(\im_k(\widehat B),\ker_{d-k}R_0\bigr) \\
&\qquad\geq
\frac14r\lambda_{\rm sel}N-\eta d_0' \\
&\qquad\geq
\left(
\frac{\rho\lambda_{\rm sel}}8-2\eta\kappa
\right)nN.
\end{aligned}
\end{displaymath}
By \eqref{eq:good:prefix:eta:buffer:contraction},
\begin{displaymath}
2\eta\kappa
\leq
\frac{\rho\lambda_{\rm sel}}{16},
\end{displaymath}
and therefore
\begin{displaymath}
-\log\sigma_{k+1}(\widehat B)
+
\log\angle\bigl(\im_k(\widehat B),\ker_{d-k}R_0\bigr)
\geq
\frac{\rho\lambda_{\rm sel}}{16}nN.
\end{displaymath}
Since
\begin{displaymath}
\lambda_{\rm ctr}d_{\rm buf}
\leq
\frac{\lambda_{\rm sel}}{32}\rho nN,
\end{displaymath}
this proves \eqref{eq:admissible:contraction} for the buffer block.

It remains to verify the three admissibility estimates for the random
blocks.  We first compare successive random block lengths.  Since
\begin{displaymath}
u_{q+1}=u_q+d_q',
\qquad
\frac{m_q}{2}\leq d_q'\leq m_q,
\qquad
m_q=\lfloor\kappa u_q\rfloor,
\end{displaymath}
we have
\begin{displaymath}
\begin{aligned}
m_{q+1}
&\leq
\kappa u_{q+1} \\
&=
\kappa u_q+\kappa d_q' \\
&<
m_q+1+\kappa m_q.
\end{aligned}
\end{displaymath}
For all sufficiently large $n$, one has $m_q\geq1/\kappa$ for every
$q\geq0$.  Hence
\begin{displaymath}
m_{q+1}\leq(1+2\kappa)m_q.
\end{displaymath}
Since $d_q'\geq m_q/2$ and $d_{q+1}'\leq m_{q+1}$, it follows that
\begin{equation}
\label{eq:good:prefix:successive:random:lengths}
d_{q+1}'
\leq
2(1+2\kappa)d_q'
\qquad
(q\geq0).
\end{equation}

For $q\geq0$, put
\begin{displaymath}
\vartheta_q'
:=
\angle(E_q,\ker_{d-k}R_q).
\end{displaymath}
Then
\begin{displaymath}
E_{q+1}=\im_k(R_q),
\end{displaymath}
so that $\vartheta_{q+1}'$ is the outgoing angle for the block $R_q$.
By \eqref{eq:good:prefix:random:cushion},
\eqref{eq:good:prefix:random:angle}, and
\eqref{eq:good:prefix:successive:random:lengths},
\begin{displaymath}
\begin{aligned}
\log\left(
\Delta_k(R_q)\vartheta_q'\vartheta_{q+1}'
\right)
&\geq
\gamma d_q'-\eta d_{q+1}' \\
&\geq
\left(
\gamma-2(1+2\kappa)\eta
\right)d_q' \\
&\geq
\frac{\gamma}{2}d_q',
\end{aligned}
\end{displaymath}
where the last inequality is
\eqref{eq:good:prefix:eta:random:regular}.  Since
$\lambda_{\rm adm}=\gamma/128$, this is at least
\begin{displaymath}
\lambda_{\rm adm}d_q'.
\end{displaymath}

Furthermore, for all sufficiently large $n$,
\begin{displaymath}
d_q'
\geq
\frac{\kappa}{4}u_q
\geq
\frac{\kappa}{4}u_0
>
\frac{\kappa}{8}nN.
\end{displaymath}
Consequently,
\begin{displaymath}
\frac{\gamma}{2}d_q'
>
\frac{\gamma\kappa}{16}nN
\geq
\log C_{\rm adm}
\end{displaymath}
after increasing the lower threshold on $n$.  This proves
\eqref{eq:admissible:regular} and
\eqref{eq:admissible:effective} for every random block.

Finally, by \eqref{eq:good:prefix:random:contraction},
\eqref{eq:good:prefix:random:angle}, and
\eqref{eq:good:prefix:successive:random:lengths},
\begin{displaymath}
\begin{aligned}
-\log\sigma_{k+1}(R_q)+\log\vartheta_{q+1}'
&\geq
\lambda_{\rm sel}d_q'-\eta d_{q+1}' \\
&\geq
\left(
\lambda_{\rm sel}-2(1+2\kappa)\eta
\right)d_q' \\
&\geq
\frac{\lambda_{\rm sel}}2d_q' \\
&\geq
\lambda_{\rm ctr}d_q',
\end{aligned}
\end{displaymath}
where the penultimate inequality is
\eqref{eq:good:prefix:eta:random:contraction}, and the last one follows from
$\lambda_{\rm ctr}=\lambda_{\rm sel}/32$.  Thus
\eqref{eq:admissible:contraction} holds for every random block.

We have now verified
\eqref{eq:admissible:length},
\eqref{eq:admissible:block:tameness},
\eqref{eq:admissible:regular},
\eqref{eq:admissible:effective}, and
\eqref{eq:admissible:contraction} for all the blocks in the sequence.
Therefore the infinite extension is  \\
$(C_{\rm adm},\epsilon_{\rm adm},
\lambda_{\rm adm},\lambda_{\rm ctr},K_{\rm blk},N)$-admissible.
\end{proof}

\begin{lemma}[Almost surely admissible tails]
\label{lemma:almost:sure:admissible:tail}
There exist constants
$\lambda_{\rm adm}>0$,
$\lambda_{\rm ctr}>0$
depending only on the constants of
\S\ref{ssec:abstract:good:event}, with the following property.

For every $\epsilon_{\rm adm}>0$ and every
$C_{\rm adm}\geq C_{\rm reg}^\ast$, there exists
\begin{displaymath}
N_\ast=N_\ast(\epsilon_{\rm adm},C_{\rm adm},R_{\rm init})<\infty
\end{displaymath}
such that, for every $N_{\rm adm}\geq N_\ast$, almost every sample
path has a \\
$(C_{\rm adm},\epsilon_{\rm adm},
\lambda_{\rm adm},\lambda_{\rm ctr},
K_{\rm blk},N_{\rm adm})$-admissible tail.
\end{lemma}

\begin{proof}
Set
\begin{displaymath}
\lambda_{\rm adm}:=\frac{\gamma}{128},
\qquad
\lambda_{\rm ctr}:=\frac{\lambda_{\rm sel}}{32},
\end{displaymath}
Fix $\epsilon_{\rm adm}>0$, and choose
\begin{displaymath}
0<\kappa<\frac{\epsilon_{\rm adm}}{2}.
\end{displaymath}
Choose $\eta>0$ sufficiently small that
\begin{displaymath}
\gamma-2(1+2\kappa)\eta\geq\frac{\gamma}{2}
\end{displaymath}
and
\begin{displaymath}
\lambda_{\rm sel}-2(1+2\kappa)\eta
\geq\frac{\lambda_{\rm sel}}{2}.
\end{displaymath}

Let
\begin{displaymath}
K_\eta<\infty,
\qquad
b_\eta>0,
\qquad
N_0=N_0(\eta,R_{\rm init})
\end{displaymath}
be the constants from
\autoref{prop:parameterized:one:block:selection}, applied with parameter
$\eta/2$. We will prove the result for every perscribed 
\begin{displaymath}
N_{\rm adm} \ge N_0.
\end{displaymath}

We first choose a time from which the construction can be restarted.  Let
$\tau$ be a sufficiently late return time to
\begin{displaymath}
\{H < R_{\rm init}\}.
\end{displaymath}
Such a time exists almost surely by
\autoref{lemma:abstract:margulis:return:tails}.  Apply
\autoref{prop:tail:gap:from:prefix} at time $\tau$, with prefix
\begin{displaymath}
Q:=A(0,\tau).
\end{displaymath}
Letting the loss parameter in that proposition tend to infinity shows that,
almost surely,
\begin{displaymath}
\Delta_k(A(0,n))\longrightarrow\infty.
\end{displaymath}
Since the process returns to $\{H < R_{\rm init}\}$ infinitely often,
there are almost surely arbitrarily large times $T$ for which
\begin{displaymath}
H_T < R_{\rm init}
\qquad\text{and}\qquad
\Delta_k(A(0,T))>1.
\end{displaymath}

Choose a deterministic $T_\ast$ sufficiently large that, whenever
$t\geq T_\ast$,
\begin{displaymath}
\lfloor\kappa t\rfloor\geq N_{\rm adm},
\qquad
\lfloor\kappa t\rfloor\geq\frac{\kappa t}{2}, \qquad \frac{\gamma}{4}\lfloor\kappa t\rfloor
\geq
\log C_{\rm adm},
\end{displaymath}
and all the elementary floor estimates used below hold.  

Let $T_0\geq T_\ast$ be the first time such that
\begin{displaymath}
H_{T_0} < R_{\rm init}
\qquad\text{and}\qquad
\Delta_k(A(0,T_0))>1,
\end{displaymath}
and set
\begin{displaymath}
E_0:=\im_k(A(0,T_0)).
\end{displaymath}

Suppose that $T_j$ and an $\mathcal F_{T_j}$-measurable $k$-plane $E_j$
have been defined, with $H_{T_j}\leq R_{\rm init}$.  Put
\begin{displaymath}
m_j:=\lfloor\kappa T_j\rfloor.
\end{displaymath}
Apply \autoref{prop:parameterized:one:block:selection} at time $T_j$, with
target length $m_j$, height level $R_{\rm init}$, parameter $\eta/2$, and
input plane $E_j$.  Let $\mathcal G_j$ be the success event.

On $\mathcal G_j$, choose the least good length $d_j$, and set
\begin{displaymath}
T_{j+1}:=T_j+d_j,
\qquad
E_{j+1}:=\im_k(A(T_j,d_j)).
\end{displaymath}
Then
\begin{displaymath}
\frac{m_j}{2}\leq d_j\leq m_j,
\qquad
H_{T_{j+1}}<R_{\rm init}.
\end{displaymath}

On $\mathcal G_j^c$, let $T_{j+1}$ be the first time
$t\geq T_j+m_j$ such that
\begin{displaymath}
H_t < R_{\rm init}
\qquad\text{and}\qquad
\Delta_k(A(0,t))>1,
\end{displaymath}
and set
\begin{displaymath}
E_{j+1}:=\im_k(A(0,T_{j+1})).
\end{displaymath}
The preceding observations show that $T_{j+1}<\infty$ almost surely.

For every $j$,
\begin{displaymath}
T_{j+1}
\geq
T_j+\frac{m_j}{2}
\geq
\left(1+\frac{\kappa}{4}\right)T_j.
\end{displaymath}
Moreover,
\begin{displaymath}
\mathbb P\left(
\mathcal G_j^c
\,\middle|\,
\mathcal F_{T_j}
\right)
\leq
K_\eta e^{-b_\eta\sqrt{m_j}}.
\end{displaymath}
The conditional probabilities are summable.  By conditional
Borel--Cantelli, only finitely many of the events $\mathcal G_j^c$ occur.

Let $j_0=0$ if no failure occurs, and otherwise let $j_0$ be one greater
than the index of the final failure.  By construction,
\begin{displaymath}
E_{j_0}=\im_k(A(0,T_{j_0})),
\end{displaymath}
and every $\mathcal G_j$ occurs for $j\geq j_0$.

Define
\begin{displaymath}
t_0:=0,
\qquad
t_r:=T_{j_0+r-1}
\quad
(r\geq1).
\end{displaymath}
Then
\begin{displaymath}
B_0:=A(0,t_1)
\end{displaymath}
has a strict $k$-gap.  For $r\geq1$, put
\begin{displaymath}
\ell_r:=t_{r+1}-t_r,
\qquad
B_r:=A(t_r,\ell_r).
\end{displaymath}
Thus $B_r$ is the successful block selected at time
$T_{j_0+r-1}$.

For $r\geq1$, set
\begin{displaymath}
\theta_r:=
\angle\bigl(\im_k(B_{r-1}),\ker_{d-k}(B_r)\bigr).
\end{displaymath}
The choice of the input planes gives, after weakening the conclusions of
\autoref{prop:parameterized:one:block:selection},
\begin{displaymath}
\theta_r\geq e^{-\eta\ell_r},
\end{displaymath}
\begin{displaymath}
\Delta_k(B_r)\theta_r\geq e^{\gamma\ell_r},
\end{displaymath}
\begin{displaymath}
-\log\sigma_{k+1}(B_r)
\geq
\lambda_{\rm sel}\ell_r,
\end{displaymath}
and
\begin{displaymath}
\max_{0\leq q\leq\ell_r}
\log\max\left\{
\|A(t_r,q)\|,\|A(t_r,q)^{-1}\|
\right\}
\leq
K_{\rm blk}\ell_r.
\end{displaymath}

The selected lengths satisfy
\begin{displaymath}
\frac{m_{j_0+r-1}}{2}
\leq
\ell_r
\leq
m_{j_0+r-1}.
\end{displaymath}
Consequently,
\begin{displaymath}
\frac{N_{\rm adm}}{2}\leq\ell_r
\end{displaymath}
and
\begin{displaymath}
\ell_r
\leq
\kappa t_r
\leq
\epsilon_{\rm adm}t_r.
\end{displaymath}
Thus \eqref{eq:admissible:length} and
\eqref{eq:admissible:block:tameness} hold.

After increasing $T_\ast$ if necessary, the floor estimates also give
\begin{displaymath}
\ell_{r+1}
\leq
2(1+2\kappa)\ell_r
\qquad
(r\geq1).
\end{displaymath}
Therefore
\begin{displaymath}
\begin{aligned}
\log\bigl(\Delta_k(B_r)\theta_r\theta_{r+1}\bigr)
&\geq
\gamma\ell_r-\eta\ell_{r+1}\\
&\geq
\left(
\gamma-2(1+2\kappa)\eta
\right)\ell_r\\
&\geq
\frac{\gamma}{2}\ell_r.
\end{aligned}
\end{displaymath}
Our choice of $T_\ast$ implies that the last expression is at least
$\log C_{\rm adm}$, and it is also at least
$\lambda_{\rm adm}\ell_r$.  Hence
\eqref{eq:admissible:regular} and
\eqref{eq:admissible:effective} hold for every $r\geq1$.

Similarly,
\begin{displaymath}
\begin{aligned}
-\log\sigma_{k+1}(B_r)+\log\theta_{r+1}
&\geq
\lambda_{\rm sel}\ell_r-\eta\ell_{r+1}\\
&\geq
\left(
\lambda_{\rm sel}-2(1+2\kappa)\eta
\right)\ell_r\\
&\geq
\frac{\lambda_{\rm sel}}{2}\ell_r\\
&\geq
\lambda_{\rm ctr}\ell_r.
\end{aligned}
\end{displaymath}
Thus \eqref{eq:admissible:contraction} holds for every $r\geq1$.

The times
\begin{displaymath}
0=t_0<t_1<t_2<\cdots
\end{displaymath}
therefore define a
$(C_{\rm adm},\epsilon_{\rm adm},
\lambda_{\rm adm},\lambda_{\rm ctr},
K_{\rm blk},N_{\rm adm})$-admissible tail.
\end{proof}

\subsection{Hyperbolicity properties of admissible sequences}
\label{ssec:admissible:hyperbolicity}

In this subsection we record the hyperbolicity consequences of an
admissible tail. Throughout, let
\begin{displaymath}
A_i:V_i\to V_{i+1}
\end{displaymath}
be a sequence of linear isomorphisms which have a 
$(C_{\rm adm},\epsilon_{\rm adm},\lambda_{\rm adm},
\lambda_{\rm ctr},K_{\rm blk},N_{\rm adm})$-admissible tail
in dimension $k$. Let
\begin{displaymath}
0=t_0 < t_1 <t_2<t_3<\cdots
\end{displaymath}
be admissible times, and write
\begin{displaymath}
\ell_j:=t_{j+1}-t_j \qquad B_j:=A(t_j,\ell_j):V_{t_j}\to V_{t_{j+1}}.
\end{displaymath}
For $j\ge1$, set
\begin{displaymath}
F_{j-1}:=\im_k(B_{j-1})\subset V_{t_j},
\qquad
K_j:=\ker_{d-k}(B_j)\subset V_{t_j},
\end{displaymath}
and
\begin{displaymath}
\theta_j:=\angle(F_{j-1},K_j).
\end{displaymath}

\begin{lemma}[Angle tempering from admissibility]
\label{lemma:angle:tempering:from:admissibility}
There exists a constant
\begin{displaymath}
C_{\rm ang}=C_{\rm ang}(C_{\rm adm},K_{\rm blk},N_{\rm adm})<\infty
\end{displaymath}
such that, for every $j\ge 1$,
\begin{equation}
\label{eq:angle:tempering:from:admissibility}
\theta_j
\ge
C_{\rm ang}^{-1}e^{-2K_{\rm blk}\epsilon_{\rm adm}t_j}.
\end{equation}
\end{lemma}

\begin{proof}
By admissibility,
\begin{displaymath}
\Delta_k(B_j)\theta_j\theta_{j+1}\ge C_{\rm adm}
\qquad
\text{for every }j\ge 1.
\end{displaymath}
Since $\theta_{j+1}\le \pi/2$,
\begin{displaymath}
\theta_j\ge \frac{2}{\pi} C_{\rm adm}\Delta_k(B_j)^{-1}.
\end{displaymath}
Moreover, \eqref{eq:admissible:block:tameness}, applied at
$r=\ell_j$, gives
\begin{displaymath}
\Delta_k(B_j)
\le
\|B_j\|\,\|B_j^{-1}\|
\le
\exp(2K_{\rm blk}\ell_j).
\end{displaymath}
By the admissible length condition,
\begin{displaymath}
\ell_j\le \max\{N_{\rm adm},\epsilon_{\rm adm}t_j\}
\le N_{\rm adm}+\epsilon_{\rm adm}t_j.
\end{displaymath}
Therefore
\begin{displaymath}
\theta_j
\ge
\frac{2}{\pi} C_{\rm adm}e^{-2K_{\rm blk}N_{\rm adm}}
e^{-2K_{\rm blk}\epsilon_{\rm adm}t_j}.
\end{displaymath}
Taking
\begin{displaymath}
C_{\rm ang}:=\frac{\pi}{2} C_{\rm adm}^{-1}e^{2K_{\rm blk}N_{\rm adm}}
\end{displaymath}
proves the lemma.
\end{proof}

\begin{proposition}[Hyperbolicity of admissible tails]
\label{prop:hyperbolicity:admissible}
There exist constants
$C_0=C_0(d,k)\geq1$
and
$N_0=N_0(d,k,\lambda_{\rm adm},\lambda_{\rm ctr})\geq1$
with the following property.  Assume
\begin{displaymath}
C_{\rm adm}\geq C_0,
\qquad
N_{\rm adm}\geq N_0.
\end{displaymath}
Let
\begin{displaymath}
\mathcal U_j,\mathcal S_j\subset V_{t_j},
\qquad j\geq1,
\end{displaymath}
be the fast and slow spaces from
\autoref{def:C:regular:fast:slow:spaces}, applied to the infinite
$C_{\rm adm}$-regular sequence $B_j$, $j\geq1$.  Then
\begin{displaymath}
B_j\mathcal U_j=\mathcal U_{j+1},
\qquad
B_j\mathcal S_j=\mathcal S_{j+1},
\qquad V_{t_j}=\mathcal U_j\oplus\mathcal S_j,
\end{displaymath}
for every $j\geq1$.
There is a constant $C_{\rm split}<\infty$ such that
\begin{equation}
\label{eq:admissible:splitting:angle}
\angle(\mathcal U_j,\mathcal S_j)
\geq
C_{\rm split}^{-1}e^{-2K_{\rm blk}\epsilon_{\rm adm}t_j}
\qquad
(j\geq1).
\end{equation}

Finally, for $j\geq1$ and $r\geq1$, put
\begin{displaymath}
B(j,r):=B_{j+r-1}\cdots B_j
:
V_{t_j}\to V_{t_{j+r}}.
\end{displaymath}
Then
\begin{equation}
\label{eq:admissible:splitting:gap}
\log
\frac{
\mathfrak m\bigl(B(j,r)|_{\mathcal U_j}\bigr)
}{
\bigl\|B(j,r)|_{\mathcal S_j}\bigr\|
}
\geq
\frac{\lambda_{\rm adm}}{2}(t_{j+r}-t_j),
\end{equation}
and
\begin{equation}
\label{eq:admissible:stable:contraction}
\bigl\|B(j,r)|_{\mathcal S_j}\bigr\|
\leq
\exp\left(
-\frac{\lambda_{\rm ctr}}{2}(t_{j+r}-t_j)
\right).
\end{equation}
\end{proposition}

\begin{proof}
Choose $C_0$ large enough so that both
\autoref{lemma:C:regular:fast:slow:basic} and
\autoref{lemma:C:regular:product:estimates} apply whenever
$C_{\rm adm}\geq C_0$.

By admissibility,
\begin{displaymath}
\Delta_k(B_j)\theta_j\theta_{j+1}\geq C_{\rm adm}
\qquad (j\geq1).
\end{displaymath}
For the regular sequence $B_j$, $j\geq1$, the initial angle is by convention
equal to $\pi/2$.  Since $\pi/2\geq\theta_1$, the sequence $B_j$, $j\geq1$, is
$C_{\rm adm}$-regular.

Let $\bar\theta_1:=\pi/2$ and $\bar\theta_j:=\theta_j$ for $j\geq2$.  By
\autoref{lemma:C:regular:fast:slow:basic}, applied to this
$C_{\rm adm}$-regular sequence, we get
\begin{displaymath}
B_j\mathcal U_j=\mathcal U_{j+1},
\qquad
B_j\mathcal S_j=\mathcal S_{j+1},
\qquad
V_{t_j}=\mathcal U_j\oplus\mathcal S_j,
\end{displaymath}
and
\begin{displaymath}
\angle(\mathcal U_j,\mathcal S_j)\geq\frac12\bar\theta_j.
\end{displaymath}
Since $\bar\theta_j\geq\theta_j$ for all $j\geq1$,
\autoref{lemma:angle:tempering:from:admissibility} gives
\eqref{eq:admissible:splitting:angle}, after changing the constant.

It remains to prove the product estimates.  From
\eqref{eq:C:regular:product:fast:lower} and
\eqref{eq:C:regular:product:slow:upper}, we get
\begin{displaymath}
\frac{
\mathfrak m\bigl(B(j,r)|_{\mathcal U_j}\bigr)
}{
\bigl\|B(j,r)|_{\mathcal S_j}\bigr\|
}
\geq
4^{-2r}
\prod_{q=j}^{j+r-1}
\Delta_k(B_q)\theta_q\theta_{q+1}.
\end{displaymath}
Taking logarithms and using admissibility gives
\begin{displaymath}
\log
\frac{
\mathfrak m\bigl(B(j,r)|_{\mathcal U_j}\bigr)
}{
\bigl\|B(j,r)|_{\mathcal S_j}\bigr\|
}
\geq
\sum_{q=j}^{j+r-1}\lambda_{\rm adm}\ell_q
-
4r \log 2.
\end{displaymath}
Choose $N_0$ so large that
\begin{displaymath}
4 \log 2
\leq
\frac{\lambda_{\rm adm}}{2}\ell_q
\end{displaymath}
for every admissible block.  This is possible because
\begin{displaymath}
\ell_q\geq\frac{N_{\rm adm}}{2}\geq\frac{N_0}{2}.
\end{displaymath}
Thus
\begin{displaymath}
\log
\frac{
\mathfrak m\bigl(B(j,r)|_{\mathcal U_j}\bigr)
}{
\bigl\|B(j,r)|_{\mathcal S_j}\bigr\|
}
\geq
\frac{\lambda_{\rm adm}}{2}
\sum_{q=j}^{j+r-1}\ell_q
=
\frac{\lambda_{\rm adm}}{2}(t_{j+r}-t_j),
\end{displaymath}
which proves \eqref{eq:admissible:splitting:gap}.

Similarly, \eqref{eq:C:regular:product:slow:upper} gives
\begin{displaymath}
\bigl\|B(j,r)|_{\mathcal S_j}\bigr\|
\leq
4^{r}
\prod_{q=j}^{j+r-1}
\theta_{q+1}^{-1}\sigma_{k+1}(B_q).
\end{displaymath}
Taking logarithms and using the admissible contraction condition gives
\begin{displaymath}
\log
\bigl\|B(j,r)|_{\mathcal S_j}\bigr\|
\leq
2 r\log 2
-
\sum_{q=j}^{j+r-1}\lambda_{\rm ctr}\ell_q.
\end{displaymath}
Increasing $N_0$ if necessary, we may assume
\begin{displaymath}
2 \log 2
\leq
\frac{\lambda_{\rm ctr}}{2}\ell_q
\end{displaymath}
for every admissible block.  Therefore
\begin{displaymath}
\log
\bigl\|B(j,r)|_{\mathcal S_j}\bigr\|
\leq
-\frac{\lambda_{\rm ctr}}{2}
\sum_{q=j}^{j+r-1}\ell_q
=
-\frac{\lambda_{\rm ctr}}{2}(t_{j+r}-t_j),
\end{displaymath}
which proves \eqref{eq:admissible:stable:contraction}.
\end{proof}

\subsection{From admissibility to uniformly tempered splittings}
\label{ssec:admissible:to:uniformly:tempered}

We now pass from hyperbolicity along the admissible block times to a
one-step uniformly tempered splitting.

\begin{lemma}[Interpolation inside one admissible block]
\label{lemma:interpolation:domination:inside:block}
Let
\begin{displaymath}
T_r:E_r\to E_{r+1},
\qquad
0\le r<L,
\end{displaymath}
be linear isomorphisms between Euclidean spaces. Suppose that we are given
equivariant splittings
\begin{displaymath}
E_r=U_r\oplus S_r,
\qquad
T_rU_r=U_{r+1},
\qquad
T_rS_r=S_{r+1}.
\end{displaymath}
Put
\begin{displaymath}
T(r):=T_{r-1}\cdots T_0:E_0\to E_r,
\qquad
T(0)=\Id,
\end{displaymath}
and
\begin{displaymath}
T(a,b):=T_{b-1}\cdots T_a:E_a\to E_b.
\end{displaymath}
Assume that, for some $\Lambda_{\rm dom},\Lambda_{\rm st}>0$,
\begin{equation}
\label{eq:interpolation:block:domination}
\log
\frac{
\mathfrak m\bigl(T(0,L)|_{U_0}\bigr)
}{
\bigl\|T(0,L)|_{S_0}\bigr\|
}
\ge
\Lambda_{\rm dom}L,
\end{equation}
and
\begin{equation}
\label{eq:interpolation:block:stable}
\bigl\|T(0,L)|_{S_0}\bigr\|
\le
e^{-\Lambda_{\rm st}L}.
\end{equation}
Let
\begin{displaymath}
0<\lambda<\frac12\min\{\Lambda_{\rm dom},\Lambda_{\rm st}\}.
\end{displaymath}
Then there are norms
\begin{displaymath}
\|\cdot\|'_r,
\qquad
0\le r\le L,
\end{displaymath}
on $E_r$  such that: 
\begin{enumerate}
\item $\|  \cdot\|'_0$ agrees with the Euclidean norm on $U_0$ and on $S_0$.
\item $\| \cdot \|'_L$ agrees with the Euclidean norm on $U_L$ and on $S_L$.
\item For every $s\in S_r$,
\begin{equation}
\label{eq:interpolation:stable:one:step}
\|T_rs\|'_{r+1}\le e^{-\lambda}\|s\|'_r.
\end{equation}

\item For every $s\in S_r$ and $u\in U_r$ satisfying
\begin{displaymath}
\|s\|'_r=\|u\|'_r=1,
\end{displaymath}
one has
\begin{equation}
\label{eq:interpolation:domination:one:step}
\|T_rs\|'_{r+1}
\le
e^{-\lambda}\|T_ru\|'_{r+1}.
\end{equation}

\item The norms come from inner products and the splitting is
  orthogonal with respect to these inner products:
\begin{equation}
\label{eq:interpolation:adapted:transverse}
(\|u+s\|'_r)^2 = (\|u\|'_r)^2 + (\|s\|'_r)^2
\end{equation}
for every $u\in U_r$ and $s\in S_r$.

\item Put
\begin{displaymath}
\alpha_0:=\angle(U_0,S_0),
\qquad
\alpha_L:=\angle(U_L,S_L).
\end{displaymath}
If
\begin{equation}
\label{eq:interpolation:block:tameness}
\max_{0\le r\le L}
\log\max\left\{
\|T(0,r)\|,\|T(0,r)^{-1}\|
\right\}
\le K_{\rm blk}L,
\end{equation}
then there is a constant $C_*=C_*(d,K_{\rm blk},\lambda)$ such that, for
every $0\le r\le L$ and every $w\in E_r$,
\begin{equation}
\label{eq:interpolation:norm:comparison}
C_*^{-1}\min\{\alpha_0,\alpha_L\}e^{-C_*L}\|w\|
\le
\|w\|'_r
\le
C_*\max\{\alpha_0^{-1},\alpha_L^{-1}\}e^{C_*L}\|w\|.
\end{equation}
\end{enumerate}
\end{lemma}

\begin{proof}
We construct the norms separately on the two equivariant bundles $S_r$ and
$U_r$, and then combine them by taking the orthrogonal direct sum.

First consider the stable spaces. At the endpoints $r=0$ and $r=L$, let
$\|\cdot\|'_{r,S}$ be the Euclidean norm on $S_r$. For $0<r<L$, if
\begin{displaymath}
s=T(0,r)s_0,
\qquad
s_0\in S_0,
\end{displaymath}
define
\begin{equation}
\label{eq:interpolation:stable:norm:def}
\|s\|'_{r,S}:=e^{-\lambda r}\|s_0\|.
\end{equation}
We claim that
\begin{equation}
\label{eq:interpolation:stable:constructed}
\|T_rs\|'_{r+1,S}\le e^{-\lambda}\|s\|'_{r,S}
\qquad
(s\in S_r,\ 0\le r<L).
\end{equation}
If $0\le r<L-1$ and $s=T(0,r)s_0$, then
\begin{displaymath}
T_rs=T(0,r+1)s_0,
\end{displaymath}
and therefore
\begin{displaymath}
\|T_rs\|'_{r+1,S}
=
e^{-\lambda(r+1)}\|s_0\|
=
e^{-\lambda}\|s\|'_{r,S}.
\end{displaymath}
For the last step, let $r=L-1$ and write $s=T(0,L-1)s_0$. Since the norm at
time $L$ is Euclidean,
\begin{displaymath}
\|T_{L-1}s\|'_{L,S}
=
\|T(0,L)s_0\|.
\end{displaymath}
By \eqref{eq:interpolation:block:stable},
\begin{displaymath}
\|T(0,L)s_0\|
\le
e^{-\Lambda_{\rm st}L}\|s_0\|.
\end{displaymath}
On the other hand,
\begin{displaymath}
\|s\|'_{L-1,S}=e^{-\lambda(L-1)}\|s_0\|.
\end{displaymath}
Thus
\begin{displaymath}
\|T_{L-1}s\|'_{L,S}
\le
e^{-\Lambda_{\rm st}L+\lambda(L-1)}
\|s\|'_{L-1,S}.
\end{displaymath}
Since
\begin{displaymath}
-\Lambda_{\rm st}L+\lambda(L-1)
=
-\lambda-(\Lambda_{\rm st}-\lambda)L
\le -\lambda,
\end{displaymath}
we get \eqref{eq:interpolation:stable:constructed}.

Now define the unstable norms. At the endpoints $r=0$ and $r=L$, let
$\|\cdot\|'_{r,U}$ be the Euclidean norm on $U_r$. For $0<r<L$, if
\begin{displaymath}
u=T(0,r)u_0,
\qquad
u_0\in U_0,
\end{displaymath}
define
\begin{equation}
\label{eq:interpolation:unstable:norm:def}
\|u\|'_{r,U}:=e^{\lambda r}\|u_0\|.
\end{equation}
For $0\le r<L-1$, if $u=T(0,r)u_0$, then
\begin{displaymath}
T_ru=T(0,r+1)u_0,
\end{displaymath}
and hence
\begin{equation}
\label{eq:interpolation:unstable:expands:intermediate}
\|T_ru\|'_{r+1,U}
=
e^\lambda\|u\|'_{r,U}.
\end{equation}

We now verify the one-step domination. Let $0\le r<L$ and suppose
\begin{displaymath}
s\in S_r,\qquad u\in U_r,
\qquad
\|s\|'_{r,S}=\|u\|'_{r,U}=1.
\end{displaymath}
For $0\le r<L-1$, \eqref{eq:interpolation:stable:constructed} gives
\begin{displaymath}
\|T_rs\|'_{r+1,S}\le e^{-\lambda},
\end{displaymath}
while \eqref{eq:interpolation:unstable:expands:intermediate} gives
\begin{displaymath}
\|T_ru\|'_{r+1,U}=e^\lambda.
\end{displaymath}
Therefore
\begin{displaymath}
\|T_rs\|'_{r+1,S}
\le
e^{-\lambda}\|T_ru\|'_{r+1,U}.
\end{displaymath}

It remains to check the last step $r=L-1$. Write
\begin{displaymath}
s=T(0,L-1)s_0,
\qquad
u=T(0,L-1)u_0.
\end{displaymath}
The normalization
\begin{displaymath}
\|s\|'_{L-1,S}=1
\end{displaymath}
means
\begin{displaymath}
\|s_0\|=e^{\lambda(L-1)}.
\end{displaymath}
Similarly,
\begin{displaymath}
\|u\|'_{L-1,U}=1
\end{displaymath}
means
\begin{displaymath}
\|u_0\|=e^{-\lambda(L-1)}.
\end{displaymath}
Since the endpoint norms at time $L$ are Euclidean,
\begin{displaymath}
\|T_{L-1}s\|'_{L,S}
=
\|T(0,L)s_0\|
\le
\|T(0,L)|_{S_0}\|\,e^{\lambda(L-1)},
\end{displaymath}
and
\begin{displaymath}
\|T_{L-1}u\|'_{L,U}
=
\|T(0,L)u_0\|
\ge
\mathfrak m(T(0,L)|_{U_0})\,e^{-\lambda(L-1)}.
\end{displaymath}
Thus the desired inequality
\begin{displaymath}
\|T_{L-1}s\|'_{L,S}
\le
e^{-\lambda}\|T_{L-1}u\|'_{L,U}
\end{displaymath}
follows if
\begin{displaymath}
\|T(0,L)|_{S_0}\|\,e^{\lambda(L-1)}
\le
e^{-\lambda}
\mathfrak m(T(0,L)|_{U_0})\,e^{-\lambda(L-1)}.
\end{displaymath}
Equivalently, it is enough that
\begin{displaymath}
\log
\frac{
\mathfrak m(T(0,L)|_{U_0})
}{
\|T(0,L)|_{S_0}\|
}
\ge
\lambda L+\lambda(L-1).
\end{displaymath}
By the block domination hypothesis \eqref{eq:interpolation:block:domination},
the left-hand side is at least $\Lambda_{\rm dom}L$. Since
$\lambda<\frac12\Lambda_{\rm dom}$, we have
\begin{displaymath}
\lambda L+\lambda(L-1)
\le
2\lambda L
<
\Lambda_{\rm dom}L.
\end{displaymath}
This proves the last-step domination estimate.

Finally define the norm on $E_r=U_r\oplus S_r$ by
\begin{equation}
\label{eq:interpolation:sum:norm:def}
(\|u+s\|'_r)^2
:= (\|u\|'_{r,U})^2+(\|s\|'_{r,S})^2,
\qquad
u\in U_r,\ s\in S_r.
\end{equation}
Then \eqref{eq:interpolation:stable:one:step},
\eqref{eq:interpolation:domination:one:step}, and
\eqref{eq:interpolation:adapted:transverse} follow immediately from the
corresponding estimates on $S_r$ and $U_r$.

It remains to prove the comparison with the Euclidean norms. We first compare
on the stable spaces. If $s=T(0,r)s_0$, then
\begin{displaymath}
\|s\|'_{r,S}=e^{-\lambda r}\|s_0\|.
\end{displaymath}
Since
\begin{displaymath}
\|T(0,r)\|\le e^{K_{\rm blk}L},
\qquad
\|T(0,r)^{-1}\|\le e^{K_{\rm blk}L},
\end{displaymath}
we have
\begin{displaymath}
\|s_0\|\le e^{K_{\rm blk}L}\|s\|,
\qquad
\|s\|\le e^{K_{\rm blk}L}\|s_0\|.
\end{displaymath}
Since $0\le r\le L$, it follows that
\begin{displaymath}
e^{-(K_{\rm blk}+\lambda)L}\|s\|
\le
\|s\|'_{r,S}
\le
e^{K_{\rm blk}L}\|s\|.
\end{displaymath}

The unstable comparison is analogous. If $u=T(0,r)u_0$, then
\begin{displaymath}
\|u\|'_{r,U}=e^{\lambda r}\|u_0\|.
\end{displaymath}
Using the same bounds for $T(0,r)$ and $T(0,r)^{-1}$, we get
\begin{displaymath}
e^{-K_{\rm blk}L}\|u\|
\le
\|u\|'_{r,U}
\le
e^{(K_{\rm blk}+\lambda)L}\|u\|.
\end{displaymath}

The norm \eqref{eq:interpolation:sum:norm:def} is an orthogonal direct
sum norm with respect to
the splitting $E_r=U_r\oplus S_r$. Its comparison with the Euclidean norm
therefore also involves the norms of the projections associated to this
splitting. These projection norms are bounded in terms of
$\angle(U_r,S_r)^{-1}$. Since
\begin{displaymath}
U_r=T(0,r)U_0,
\qquad
S_r=T(0,r)S_0,
\end{displaymath}
and the condition number of $T(0,r)$ is at most $e^{2K_{\rm blk}L}$,
the endpoint angle lower bounds and \eqref{eq:interpolation:block:tameness}
imply
\begin{displaymath}
\angle(U_r,S_r)^{-1}
\le
C_d\,e^{C_dK_{\rm blk}L}
\max\{\alpha_0^{-1},\alpha_L^{-1}\}.
\end{displaymath}
Combining this with the stable and unstable comparisons above gives, after
increasing $C_*=C_*(d,K_{\rm blk},\lambda)$,
\begin{displaymath}
C_*^{-1}\min\{\alpha_0,\alpha_L\}e^{-C_*L}\|w\|
\le
\|w\|'_r
\le
C_*\max\{\alpha_0^{-1},\alpha_L^{-1}\}e^{C_*L}\|w\|.
\end{displaymath}
This is \eqref{eq:interpolation:norm:comparison}.
\end{proof}

\begin{proposition}[Admissible sequences give uniformly tempered splittings]
\label{prop:admissible:to:uniformly:tempered}
There exist constants $C_0=C_0(d,k)\ge1$ and
$L_0=L_0(d,k,\lambda_{\rm adm},\lambda_{\rm ctr})\ge1$ with the following
property. Suppose that $A_i$ has a 
$(C_{\rm adm},\epsilon_{\rm adm},\lambda_{\rm adm},\lambda_{\rm ctr},
K_{\rm blk},N_{\rm adm})$-admissible tail in dimension $k$, with $C_{\rm
  adm}\ge C_0$, and $N_{\rm adm} > 2 L_0$.
Let the admissible times be
\begin{displaymath}
0 = t_0 <  t_1 < t_2 < \dots.
\end{displaymath}
Then the sequence $A_i$ admits a $(C,\lambda,\epsilon)$-uniformly tempered
splitting in dimension $k$, where $\lambda>0$ depends only on
$\lambda_{\rm adm}$ and $\lambda_{\rm ctr}$, and
\begin{displaymath}
\epsilon=O_{d,k,K_{\rm blk},\lambda_{\rm adm}, \lambda_{\rm ctr}}(\epsilon_{\rm adm}). 
\end{displaymath}
The constant
$C<\infty$ may depend on
$C_{\rm adm},N_{\rm adm},K_{\rm blk},d,k$ and on the $A_j$ for $0 \le
j <  t_1$. 
\end{proposition}

\begin{proof}
Choose
\begin{displaymath}
0<\lambda<
\frac14\min\{\lambda_{\rm adm}/2,\lambda_{\rm ctr}/2\}.
\end{displaymath}
We choose $L_0$ large enough so that
$2L_0$ is at least the lower bound for $N_{\rm adm}$ required in
\autoref{prop:hyperbolicity:admissible}.

By Proposition~\ref{prop:hyperbolicity:admissible}, after increasing $C_0$
and $L_0$ if necessary, there are equivariant splittings
$V_{t_j}=\mathcal U_j\oplus\mathcal S_j$ at the admissible times such that
\begin{equation}
\label{eq:admissible:to:ut:block:domination}
\log
\frac{
\mathfrak m\bigl(B_j|_{\mathcal U_j}\bigr)
}{
\|B_j|_{\mathcal S_j}\|
}
\ge
\frac12\lambda_{\rm adm}(t_{j+1}-t_j)
\end{equation}
and
\begin{equation}
\label{eq:admissible:to:ut:block:stable}
\|B_j|_{\mathcal S_j}\|
\le
\exp\left(-\frac12\lambda_{\rm ctr}(t_{j+1}-t_j)\right).
\end{equation}
Moreover,
\begin{equation}
\label{eq:admissible:to:ut:block:angle}
\log\angle(\mathcal U_j,\mathcal S_j)
\ge
-C_{\rm spl}-2K_{\rm blk}\epsilon_{\rm adm}t_j.
\end{equation}

We extend the splitting to all intermediate times by equivariance. If
$t_j\le s \le t_{j+1}$, set
\begin{displaymath}
\mathcal U(s):=A(t_j,s-t_j)\mathcal U_j,
\qquad
\mathcal S(s):=A(t_j,s-t_j)\mathcal S_j.
\end{displaymath}
Thus
\begin{displaymath}
\mathcal U(t_j)=\mathcal U_j,
\qquad
\mathcal S(t_j)=\mathcal S_j.
\end{displaymath}
Moreover,
\begin{displaymath}
A_i\mathcal U(i)=\mathcal U(i+1),
\qquad
A_i\mathcal S(i)=\mathcal S(i+1)
\end{displaymath}
for every $i\geq t_1$.

On each admissible interval $[t_j,t_{j+1}]$ with $j \ge  1$, apply
Lemma~\ref{lemma:interpolation:domination:inside:block} to the maps
\begin{displaymath}
A_{t_j},A_{t_j+1},\dots,A_{t_{j+1}-1}
\end{displaymath}
and to the equivariant splitting just defined. The block hypotheses
\eqref{eq:interpolation:block:domination} and
\eqref{eq:interpolation:block:stable} are precisely
\eqref{eq:admissible:to:ut:block:domination} and
\eqref{eq:admissible:to:ut:block:stable}, while
\eqref{eq:interpolation:block:tameness} follows from
\eqref{eq:admissible:block:tameness}. 
We obtain norms $\|\cdot\|'_i$ on each block interval.  The constructions
agree at common endpoints: at time $t_j$, both adjacent blocks give
\begin{displaymath}
(\|u+s\|'_{t_j})^2 
=(\|u\|)^2+(\|s\|)^2,
\qquad
u\in\mathcal U_j,\quad s\in\mathcal S_j.
\end{displaymath}
Thus the blockwise norms define a single norm $\|\cdot\|'_i$ at every
time $i \ge t_1$.

By construction, for every $i \ge  t_1$ and every $s\in\mathcal S(i)$,
\begin{equation}
\label{eq:interpolation:stable:contraction}
\|A_is\|'_{i+1}\le e^{-\lambda}\|s\|'_i.
\end{equation}
Also, if $s\in\mathcal S(i)$ and $u\in\mathcal U(i)$ satisfy
$\|s\|'_i=\|u\|'_i=1$, then
\begin{equation}
\label{eq:interpolation:gap}
  \|A_is\|'_{i+1} \le e^{-\lambda}\|A_iu\|'_{i+1}.
\end{equation}
Thus the stable contraction and one-step domination requirements in
Definition~\ref{def:uniformly:tempered} hold.

For $0=t_0\leq i<t_1$ we use a simpler pullback construction.  Define
\begin{displaymath}
\mathcal U(i):=A(i,t_1-i)^{-1}\mathcal U(t_1),
\qquad
\mathcal S(i):=A(i,t_1-i)^{-1}\mathcal S(t_1).
\end{displaymath}
For $u\in\mathcal U(i)$ and $s\in\mathcal S(i)$, put
\begin{displaymath}
\|u\|_{i,\mathcal U}'
:= \|A(i,t_1-i) u\|_{t_1}',
\qquad
\|s\|_{i,\mathcal S}'
:=
e^{\lambda(t_1-i)}\|A(i,t_1-i) s\|_{t_1}',
\end{displaymath}
and define
\begin{displaymath}
(\|u+s\|_i')^2
:=(
\|u\|_{i,\mathcal U}')^2 + (\|s\|_{i,\mathcal S}')^2.
\end{displaymath}
Then, (\ref{eq:interpolation:stable:contraction}) and
(\ref{eq:interpolation:gap}) hold for $0 \le i <  t_1$  as well. 
Moreover, for all $i \ge 0$, $u \in \mathcal U(i)$, $s  \in  \mathcal S(i)$,   
\begin{displaymath}
(\|u+s\|_i')^2 =(\|u\|_i')^2+(\|s\|_i')^2
\end{displaymath}
so the orthogonality condition holds.

It remains to check that the adapted norms are subtempered relative to the
background Euclidean norms. Since the  constant $C$ in the definition
of uniformly tempered splitting is allowed to depend
on $A_i$ for $i < t_1$, we only need to consider $i \ge t_1$. 
Let $t_j\le i\le t_{j+1}$, so $j \ge 1$.  The admissible length
condition (\ref{eq:admissible:length}) gives
\begin{displaymath}
t_{j+1}-t_j\le \max\{N_{\rm adm},\epsilon_{\rm adm}t_j\} \le N_{\rm
  adm}  +  \epsilon_{\rm adm}t_j. 
\end{displaymath}
The endpoint angle estimate \eqref{eq:admissible:to:ut:block:angle}
applied at both $t_{j+1}$ and $t_j$, the
block-tameness condition \eqref{eq:admissible:block:tameness}, and the
comparison estimate \eqref{eq:interpolation:norm:comparison} imply that, for
some \\
$C_1=C_1(d,k,K_{\rm blk},C_{\rm adm},N_{\rm adm})<\infty$ and
$C_2=C_2(d,k,K_{\rm blk})<\infty$,
\begin{displaymath}
C_1^{-1}e^{-C_2 N_{\rm adm} - C_2\epsilon_{\rm adm} i}\|v\|
\le
\|v\|'_i
\le
C_1e^{C_2 N_{\rm adm} + C_2\epsilon_{\rm adm} i}\|v\|.
\end{displaymath}
Then, 
\begin{equation}
\label{eq:admissible:to:ut:norm:tempering}
C^{-1}e^{-\epsilon i}\|v\|
\le
\|v\|'_i
\le
Ce^{\epsilon i}\|v\|,
\end{equation}
where $C = C_1  e^{C_2 N_{\rm adm}}$,  $\epsilon:=C_2\epsilon_{\rm adm}$.
By construction, the norms come  from an inner product and the two
summands are orthogonal.  
Therefore the splitting is $(C,\lambda,\epsilon)$-uniformly tempered. This
proves the proposition.
\end{proof}

\begin{proof}[Proof of \autoref{prop:good:block:events}]
Choose $R_{\rm init}$ sufficiently large that
\autoref{lemma:abstract:margulis:return:tails} and
\autoref{prop:parameterized:one:block:selection} apply, and fix, for example,
\begin{displaymath}
\epsilon_{\rm init}:=1.
\end{displaymath}
Let $K_{\rm blk} = 2 K_{\rm sel}$ where $K_{\rm sel}$ is as in
Proposition~\ref{prop:parameterized:one:block:selection} with
$R=R_{\rm init}$.  Then,
$K_{\rm blk}$ depends only on $R_{\rm init}$ and the constants in
\S\ref{ssec:abstract:good:event}. 
Let $\lambda_{\rm adm}$ and $\lambda_{\rm ctr}$ be the constants from
\autoref{lemma:almost:sure:admissible:tail}; these are the same constants
used in \autoref{lemma:good:prefix:admissible:extension}. All of these
constants depend only on $R_{\rm init}$ and the constants in 
\S\ref{ssec:abstract:good:event}.

Let
\begin{displaymath}
\lambda=\lambda(\lambda_{\rm adm}, \lambda_{\rm ctr}) >0, \quad
C_0  = C_0(d,k)>0, \quad L_0 =  L_0(d,k,\lambda_{\rm adm}, \lambda_{\rm ctr}) > 0,
\end{displaymath}
be as in Proposition~\ref{prop:admissible:to:uniformly:tempered}.

We first prove {\rm (a)}.
Fix $\epsilon>0$.  Choose
$\epsilon_{\rm adm}>0$ sufficiently small that the exponent produced by
by \autoref{prop:admissible:to:uniformly:tempered} satisfies
\begin{displaymath}
O_{d,k,K_{\rm blk},\lambda_{\rm adm},  \lambda_{\rm ctr}}(\epsilon_{\rm adm}) \le \epsilon.
\end{displaymath}
Then, choose
\begin{displaymath}
C_{\rm adm} \ge \max\{ C_{\rm reg}^\ast, C_0 \}
\end{displaymath}
and choose
\begin{displaymath}
N_{\rm adm} \ge  \max\left\{N_\ast(\epsilon_{\rm adm}, C_{\rm adm},
  R_{\rm init}), 2L_0 + 1 \right\}, 
\end{displaymath}
where $N_\ast(\cdot,\cdot,\cdot)$ is as in
Lemma~\ref{lemma:almost:sure:admissible:tail}. 

By \autoref{lemma:almost:sure:admissible:tail}, almost every sample path has a\\
$(C_{\rm adm},\epsilon_{\rm adm},
\lambda_{\rm adm},\lambda_{\rm ctr},
K_{\rm blk},N_{\rm adm})$-admissible tail.
Our choices of $C_{\rm adm}$ and $N_{\rm adm}$ ensure that
the hypotheses of \autoref{prop:admissible:to:uniformly:tempered} hold.
Applying that proposition gives a
$(C,\lambda,\epsilon)$-uniformly tempered splitting.
The construction in
\autoref{lemma:almost:sure:admissible:tail} (or
\autoref{lemma:abstract:margulis:return:tails}) also gives
\begin{displaymath}
H_n<R_{\rm init}
\end{displaymath}
for infinitely many $n$.
By Corollary~\ref{cor:subexponential:height} with
$\eta=\delta_\Phi$, almost surely there is a finite $C_H=C_H(\omega)$
such that $H_i\le C_H e^{\delta_\Phi\sqrt i}$ for every $i\ge0$.
Thus \eqref{eq:abstract:systole:height} gives
\[
\Phi_i\le(C_H/\kappa_\Phi)^{1/\delta_\Phi}e^{\sqrt i}
\qquad(i\ge1).
\]
Enlarging the common prefactor $C(\omega,\epsilon)$ to include
$(C_H/\kappa_\Phi)^{1/\delta_\Phi}$ completes the proof of part~{\rm (a)}.

We now prove {\rm (b)}.  Fix $\epsilon>0$.
Let $\eta_0>0$ be as \S\ref{sssec:def:event}, and let $\eta_1>0$  be
as in Lemma~\ref{lemma:good:prefix:admissible:extension}. For $0 <
\eta < \min(\eta_0,\eta_1)$, let $C_\eta$,
$\delta_\eta$ be as in (\ref{eq:good:event:delta:C}). Let
$\epsilon_{\rm adm}(\eta)$ be as in
Lemma~\ref{lemma:good:prefix:admissible:extension}. 
Let $N=N(\eta,R_{\rm init},\epsilon_{\rm init})$ be as in
\S\ref{sssec:def:event}. In view of
(\ref{eq:N:lower:bound:in:terms:of:eta}), we can choose $\eta > 0$
small enough so that
\begin{equation}
\label{eq:tmp:C:eta:C:reg:C0}
\frac{1}{2} C_\eta > \max\left(C_{\rm reg}^\ast, C_0\right),
\end{equation}
\begin{equation}
\label{eq:tmp:N:L0}
N > 2 L_0 + 1
\end{equation}
and so that that the exponent produced by
by \autoref{prop:admissible:to:uniformly:tempered} satisfies
\begin{displaymath}
O_{d,k,K_{\rm blk}, \lambda_{\rm adm}, \lambda_{\rm
    ctr}}(\epsilon_{\rm adm}(\eta)) \le \epsilon.
\end{displaymath}
Let $\cO$ be the open set defined in \S\ref{sssec:def:event} for this
choice of $\eta$, $N$.  
Set $p_0:=\frac12$.
The estimate  (\ref{eq:good:block:events:extension})
is exactly \autoref{lemma:good:event:positive:probability},  applied
at $t = s_m$ with $E=\im_k(A(s_{m-1},s_m-s_{m-1}))$. 

Let $n$ be sufficiently large and suppose that $\mathcal
E^{(n)}(s_0,s_1, \dots, s_n)$ occurs. 
By \autoref{lemma:good:prefix:admissible:extension}, conditionally on
$\mathcal F_{s_n}$, outside an event of probability at most
$e^{-\beta\sqrt n}$, the continuation is
$(\tfrac{1}{2}C_\eta,\epsilon_{\rm adm},
\lambda_{\rm adm},\lambda_{\rm ctr},
K_{\rm blk},N)$-admissible. Let
\begin{displaymath}
0 = t_0 <  t_1 = 1 < t_2 < \dots
\end{displaymath}
be the admissible times. Write for $j \ge 1$,
\begin{displaymath}
\ell_j:=t_{j+1}-t_j \qquad B_j:=A(t_j,\ell_j):V_{t_j}\to V_{t_{j+1}}.
\end{displaymath}
In  view of (\ref{eq:tmp:C:eta:C:reg:C0}) and (\ref{eq:tmp:N:L0}),
we can apply \autoref{prop:admissible:to:uniformly:tempered}. We get a
$(C_{\rm temp},\lambda,\epsilon)$-uniformly tempered splitting.
The constant
$C_{\rm temp}$ may depend on $A_0$, but not on the rest of the
sequence $A_i$. 

Take
$\mathcal T^{(n)}_{\rm old}$ to be the success event constructed in the proof of
Lemma~\ref{lemma:good:prefix:admissible:extension}, and put
$C_\Phi:=(R_{\rm init}/\kappa_\Phi)^{1/\delta_\Phi}$, 
$C_*=C_\Phi A_\Phi^N$.  We define
\[
\mathcal T^{(n)}=\mathcal T_{\rm old}^{(n)}\cap
\{\Phi_i\le C_*e^{\sqrt i}\text{ for every }i\ge s_n\}.
\]
By \eqref{eq:abstract:systole:height},
$\Phi_{s_j}\leq C_\Phi$ for $1 \le j \le n$. Since for $1 \le  j  \le
n-1$, $s_{j+1}-s_j  \le  N$, iterating
\eqref{eq:abstract:systole:growth} gives
\begin{equation}
\label{eq:prefix:bound}
\Phi_i\leq C_\Phi A_\Phi^N \le C_* \qquad\text{for $1\leq i\leq s_n$.}
\end{equation}
On $\mathcal E^{(n)}$ we have $H_{s_n}\le R_{\rm init}$.
Corollary~\ref{cor:subexponential:height}, with $T=s_n$,
$R=R_{\rm init}$, $\eta=\delta_\Phi$ and
$A=\kappa_\Phi C_*^{\delta_\Phi}$, together with
\eqref{eq:abstract:systole:height}, gives
\[
\mathbb P(\exists i\ge s_n:\Phi_i>C_*e^{\sqrt i}
\mid\mathcal F_{s_n})\le K e^{-\delta_\Phi\sqrt{s_n}/2},
\]
with $K$ independent of $n$ and the history.
Since $s_n\ge1+(n-1)N/2$, a union bound with the failure estimate for
$\mathcal T_{\rm old}^{(n)}$ proves
\eqref{eq:good:block:events:tempered:extension} for $\mathcal T^{(n)}$,
after decreasing $\beta$ and increasing the minimum $n$.
Together with (\ref{eq:prefix:bound}), this gives
$\Phi_i\le C_{\rm temp}e^{\sqrt i}$ for all $i\ge1$, after enlarging
$C_{\rm temp}$ to at least $C_*$.
The inclusion $\mathcal T^{(n)}\subseteq\mathcal T_{\rm old}^{(n)}$
also ensures infinitely many returns below $R_{\rm init}$.

It remains to estimate the stable space at time $0$.  Put
\begin{displaymath}
K_0 = \ker_{d-k}(A_0), \qquad
  \theta_1:=
\angle\bigl(\im_k(A_0),\ker_{d-k}(B_1)\bigr).
\end{displaymath}
Since the chosen sequence is $\cO$-adapted,
(\ref{eq:good:event:open:angle}) implies 
\begin{displaymath}
\theta_1\geq \delta_\eta.
\end{displaymath}

Choose $R_\eta\geq1$ sufficiently large that
\begin{equation}
\label{eq:tmp:R:eta:delta:eta}
R_\eta\delta_\eta\geq \frac{1}{2} C_\eta.
\end{equation}
Suppose first that
\begin{equation}
\label{eq:tmp:Delta:k:A0}
\Delta_k(A_0)\geq R_\eta.
\end{equation}
Then, we claim that the sequence of matrices
$A_0, B_1, B_2, \dots$
is $\tfrac{1}{2} C_\eta$-regular. Indeed, the sequence $B_1,B_2,
\dots$ is $\tfrac{1}{2}C_\eta$-regular by the definition of
admissibility, and the extra condition
\begin{displaymath}
\Delta_k(A_0) \theta_1 \ge \frac{1}{2} C_\eta
\end{displaymath}
needed to prepend $A_0$ holds by (\ref{eq:tmp:R:eta:delta:eta}) and
(\ref{eq:tmp:Delta:k:A0}). 

Then, by Lemma~\ref{lemma:C:regular:fast:slow:basic} conclusion
(\ref{eq:C:regular:slow:subspace:near:kernel}), 
\begin{displaymath}
\dist_{\Gr}\bigl(\mathcal S_0,K_0 \bigr)
\leq
\frac{16}{\Delta_k(A_0)\theta_1}
\leq
16\,\delta_\eta^{-1}\Delta_k(A_0)^{-1}.
\end{displaymath}
where $\cS_0$ is the slow subspace defined in
Definition~\ref{def:C:regular:fast:slow:spaces}.

Note that $\cS_1 = \cS(1)$ by construction (since $t_1 = 1$). Since
$\cS(0) = A_0^{-1} \cS(1)$ and $\cS_0 = A_0^{-1}  \cS_1$, it follows
that $\cS_0 = \cS(0)$. Therefore,
\begin{displaymath}
\dist_{\Gr}\bigl(\mathcal S(0),K_0\bigr)
\leq 16\,\delta_\eta^{-1}\Delta_k(A_0)^{-1}.
\end{displaymath}

Suppose instead that
\begin{displaymath}
\Delta_k(A_0)<R_\eta.
\end{displaymath}
Let $D_{d,k}$ be the diameter of the Grassmannian.  Since
$\Delta_k(A_0)>e^{\epsilon_{\rm init}}>1$,
\begin{displaymath}
\min\left\{1,\Delta_k(A_0)^{-1}\right\}
=
\Delta_k(A_0)^{-1}
\geq R_\eta^{-1}.
\end{displaymath}
Therefore
\begin{displaymath}
\dist_{\Gr}\bigl(\mathcal S(0),K_0\bigr)
\leq
D_{d,k}R_\eta
\min\left\{1,\Delta_k(A_0)^{-1}\right\}.
\end{displaymath}

Thus the required estimate holds with
\begin{displaymath}
C_{\rm init}
:=
\max\left\{
16 \delta_\eta^{-1},
D_{d,k}R_\eta
\right\},
\end{displaymath}
which is independent of the fixed prefix $A_0$.
This proves {\rm (b)} and completes the proof.
\end{proof}

\section{Non-concentration of stable manifolds}
\label{sec:non:concentration}

\begin{theorem}
\label{theorem:non:concentration}
Assume the dynamical assumptions of Section~\ref{ssec:dynamical:assumptions}
and the additional dynamical assumptions of
Section~\ref{ssec:extra:dynamical:assumptions}. Let $x\in M$, and let
$H\subset T_xM$ be a hyperplane. Then, with positive probability,
\begin{displaymath}
T_x\mathcal W_{d-k}(x,\omega)\not\subset H.
\end{displaymath}
\end{theorem}

\begin{lemma}[Unstable volume realizes the top singular volume]
\label{lemma:unstable:volume:top:singular}
There exists $C_0=C_0(d,k)<\infty$ with the following property.

Let
\begin{displaymath}
A_i:V_i\to V_{i+1},
\qquad
i\ge 0,
\end{displaymath}
be a sequence of linear isomorphisms admitting a
$(C,\lambda,\epsilon)$-uniformly tempered splitting
\begin{displaymath}
V_i=\mathcal U(i)\oplus\mathcal S(i),
\qquad
\dim\mathcal U(i)=k,
\qquad
\dim\mathcal S(i)=d-k.
\end{displaymath}
Set
\begin{displaymath}
A(0,n):=A_{n-1}\cdots A_0.
\end{displaymath}
Then, for every $n\ge1$,
\begin{equation}
\label{eq:unstable:top:singular:volume}
\log\Jac_k\bigl(A(0,n)|_{\mathcal U(0)}\bigr)
\ge
\sum_{j=1}^k\log\sigma_j(A(0,n))
-
C_0(C+\epsilon n).
\end{equation}
\end{lemma}

\begin{proof}
We use the adapted norms from the definition of uniformly tempered. Since the
adapted norms are $e^{C+\epsilon i}$-comparable with the background Euclidean
norms, the induced norms on exterior powers are
$e^{C_0(C+\epsilon i)}$-comparable, with $C_0=C_0(d,k)$.

Let
\begin{displaymath}
P:=A(0,n).
\end{displaymath}
The domination part of the uniformly tempered splitting gives, for every
non-zero $u\in\mathcal U(0)$ and every non-zero $s\in\mathcal S(0)$,
\begin{displaymath}
\frac{\|Pu\|_n'}{\|u\|_0'}
\ge
e^{\lambda n}
\frac{\|Ps\|_n'}{\|s\|_0'}.
\end{displaymath}
Equivalently,
\begin{displaymath}
\mathfrak m'\bigl(P|_{\mathcal U(0)}\bigr)
\ge
e^{\lambda n}
\bigl\|P|_{\mathcal S(0)}\bigr\|',
\end{displaymath}
where the primes denote operator norm and conorm with respect to the adapted
norms.

Let
\begin{displaymath}
E:=\mathcal U(0).
\end{displaymath}
The splittings are orthogonal in the adapted inner products, and
$P$ is block diagonal. Its singular values are therefore those of
its two restrictions. The domination estimate above shows that every
singular value of $P|_{\mathcal U(0)}$ exceeds every singular value of
$P|_{\mathcal S(0)}$. Consequently,
\[
\Jac_k'(P|_E)=\prod_{j=1}^k\sigma_j'(P).
\]
Passing back from adapted norms to the background Euclidean norms changes both
sides by at most a factor
\begin{displaymath}
\exp\bigl(C_2(C+\epsilon n)\bigr),
\end{displaymath}
where $C_2=C_2(d,k)$. Hence
\begin{displaymath}
\Jac_k\bigl(P|_{\mathcal U(0)}\bigr)
\ge
\exp\bigl(-C_0(C+\epsilon n)\bigr)
\prod_{j=1}^k\sigma_j(P).
\end{displaymath}
Taking logarithms gives \eqref{eq:unstable:top:singular:volume}.
\end{proof}

\begin{lemma}[Stable volume realizes the bottom singular volume]
\label{lemma:stable:volume:bottom:singular}
There exists $C_0=C_0(d,k)<\infty$ with the following property.

Assume that
\begin{displaymath}
A_i:V_i\to V_{i+1},
\qquad
i\ge0,
\end{displaymath}
is a sequence of linear isomorphisms admitting a
$(C,\lambda,\epsilon)$-uniformly tempered splitting
\begin{displaymath}
V_i=\mathcal U(i)\oplus\mathcal S(i),
\qquad
\dim\mathcal U(i)=k,
\qquad
\dim\mathcal S(i)=d-k.
\end{displaymath}
Set
\begin{displaymath}
A(0,n):=A_{n-1}\cdots A_0.
\end{displaymath}
Then, for every $n\ge1$,
\begin{equation}
\label{eq:stable:bottom:singular:volume}
\log\Jac_{d-k}\bigl(A(0,n)|_{\mathcal S(0)}\bigr)
\le
\sum_{j=k+1}^d\log\sigma_j(A(0,n))
+
C_0(C+\epsilon n).
\end{equation}
\end{lemma}

\begin{proof}
Set
\begin{displaymath}
P:=A(0,n).
\end{displaymath}
Since the splitting is equivariant,
\begin{displaymath}
P\mathcal U(0)=\mathcal U(n),
\qquad
P\mathcal S(0)=\mathcal S(n).
\end{displaymath}
Choose bases of $\mathcal U(0)$ and $\mathcal S(0)$. Comparing the
$d$-dimensional volume of their union before and after applying $P$, we get
\begin{displaymath}
\Jac_d(P)
=
\Jac_k\bigl(P|_{\mathcal U(0)}\bigr)
\Jac_{d-k}\bigl(P|_{\mathcal S(0)}\bigr)
\frac{
\operatorname{Ang}(\mathcal U(n),\mathcal S(n))
}{
\operatorname{Ang}(\mathcal U(0),\mathcal S(0))
},
\end{displaymath}
where $\operatorname{Ang}(\mathcal U,\mathcal S)$ denotes the absolute
Jacobian of the map
\begin{displaymath}
\mathcal U\oplus\mathcal S\to V,
\qquad
(u,s)\mapsto u+s,
\end{displaymath}
computed using Euclidean volume forms.
Note that
\begin{displaymath}
\log\Jac_d(P) = \sum_{j=1}^d \log \sigma_j(P). 
\end{displaymath}
The uniform tempering of the splitting gives
\begin{displaymath}
\left|
\log\operatorname{Ang}(\mathcal U(i),\mathcal S(i))
\right|
\le
C_1(C+\epsilon i)
\end{displaymath}
for all $i$, after increasing $C_1=C_1(d,k)$ if necessary. Hence
\begin{displaymath}
\log\Jac_{d-k}\bigl(P|_{\mathcal S(0)}\bigr)
\le
-\log\Jac_k\bigl(P|_{\mathcal U(0)}\bigr) +  \log \Jac_d(P)
+
C_1(C+\epsilon n).
\end{displaymath}
By Lemma~\ref{lemma:unstable:volume:top:singular},
\begin{displaymath}
\log\Jac_k\bigl(P|_{\mathcal U(0)}\bigr)
\ge
\sum_{j=1}^k\log\sigma_j(P)
-
C_2(C+\epsilon n).
\end{displaymath}
Therefore
\begin{align*}
\log\Jac_{d-k}\bigl(P|_{\mathcal S(0)}\bigr)
  & \leq\log\Jac_d(P)-\sum_{j=1}^k\log\sigma_j(P)+C_0(C+\epsilon n) \\
  & = \sum_{j=k+1}^d \log \sigma_j(P) + C_0(C + \epsilon n). 
\end{align*}
\end{proof}

\begin{proof}[Proof of Theorem~\ref{theorem:non:concentration}]
Let $c_{d-1}>0$ be the constant in the additional dynamical assumption giving
uniform expansion in codimension one. Thus, for every $y\in M$ and every
hyperplane $H_y\subset T_yM$,
\begin{displaymath}
\int \alpha_{d-1}(f,y,H_y)\,d\mu(f)\ge c_{d-1}.
\end{displaymath}

Choose $\epsilon>0$ sufficiently small so that the error term in
Lemma~\ref{lemma:stable:volume:bottom:singular} is negligible compared with
$c_{d-1}n$. More precisely, if $C_0$ is the dimensional constant in
Lemma~\ref{lemma:stable:volume:bottom:singular}, choose $\epsilon$ so that
\begin{displaymath}
(C_0+1)\epsilon<\frac12 c_{d-1}.
\end{displaymath}

Assume, toward contradiction, that the conclusion fails. Then there exist
$x\in M$ and a hyperplane $H\subset T_xM$ such that
\begin{equation}
\label{eq:nonconcentration:failure}
\mathbb P\left(T_x\mathcal W_{d-k}(x,\omega)\subset H\right)=1.
\end{equation}
Let
\begin{displaymath}
S:=T_x\mathcal W_{d-k}(x,\omega).
\end{displaymath}
On a full-probability set, the stable tangent space exists, the derivative
cocycle admits a $(C,\lambda,\epsilon)$-uniformly tempered splitting with
stable space $S$ at time $0$, and the conclusions of
Lemma~\ref{lemma:stable:volume:bottom:singular} hold.

For $n\ge1$, write
\begin{displaymath}
A(n):=D_x(f_{\omega_n}\cdots f_{\omega_1})
\end{displaymath}
and
\begin{displaymath}
x_n:=f_{\omega_n}\cdots f_{\omega_1}x.
\end{displaymath}
Also set
\begin{displaymath}
H_n:=A(n)H\subset T_{x_n}M.
\end{displaymath}
The codimension-one Jacobian cocycle is additive along the propagated
hyperplanes:
\begin{displaymath}
\alpha_{d-1}(A(n),H)
=
\sum_{j=0}^{n-1}
\alpha_{d-1}
\left(
f_{\omega_{j+1}},
x_j,
H_j
\right).
\end{displaymath}
Since $H_j$ is measurable with respect to the past up to time $j$, the
additional dynamical assumption gives
\begin{displaymath}
\mathbb E\left[
\alpha_{d-1}
\left(
f_{\omega_{j+1}},
x_j,
H_j
\right)
\middle|
\mathcal F_j
\right]
\ge c_{d-1}.
\end{displaymath}
Taking expectations and summing, we obtain
\begin{equation}
\label{eq:hyperplane:growth:expectation:lower}
\mathbb E\,\alpha_{d-1}(A(n),H)\ge c_{d-1}n.
\end{equation}

We now prove the opposite upper estimate. By
\eqref{eq:nonconcentration:failure}, we have $S\subset H$ almost surely.
Let $E = H  \cap S^\perp$. Then $E$ is 
a $(k-1)$-dimensional subspace, and we have the orthogonal
decomposition 
\begin{displaymath}
H=S\oplus E.
\end{displaymath}
The comparison between the Euclidean volume on $H$ and the product volume on
$S\oplus E$ gives 
\begin{displaymath}
\Jac_{d-1}\bigl(A(n)|_H\bigr)
\le \Jac_{d-k}\bigl(A(n)|_S\bigr)\,
\left\|\wedge^{k-1}A(n)\right\|.
\end{displaymath}
Taking logarithms gives
\begin{displaymath}
\alpha_{d-1}(A(n),H)
\le
\log\Jac_{d-k}\bigl(A(n)|_S\bigr)
+
\sum_{i=1}^{k-1}\log\sigma_i(A(n)).
\end{displaymath}
By Lemma~\ref{lemma:stable:volume:bottom:singular},
\begin{displaymath}
\log\Jac_{d-k}\bigl(A(n)|_S\bigr)
\le
\sum_{i=k+1}^{d}\log\sigma_i(A(n))
+
C_0(C+\epsilon n).
\end{displaymath}
Therefore
\begin{displaymath}
\alpha_{d-1}(A(n),H)
\le
\sum_{i=1}^{k-1}\log\sigma_i(A(n))
+
\sum_{i=k+1}^{d}\log\sigma_i(A(n))
+
C_1+C_0\epsilon n,
\end{displaymath}
where $C_1<\infty$ may depend on the trajectory through the finite uniformly
tempered constant $C$ but is independent of $n$.

We have
\begin{displaymath}
\sum_{i=1}^{d}\log\sigma_i(A(n))=\alpha_d(A(n),T_xM).
\end{displaymath}
Hence
\begin{equation}
\label{eq:hyperplane:growth:upper}
\alpha_{d-1}(A(n),H)
\le
-\log\sigma_k(A(n)) + \alpha_d(A(n),T_xM) 
+
C_1+C_0\epsilon n.
\end{equation}


Let
\begin{displaymath}
D_n:=\alpha_{d-1}(A(n),H)+\log\sigma_k(A(n))
     -\alpha_d(A(n),T_xM).
\end{displaymath}
By \eqref{eq:hyperplane:growth:upper},
\begin{equation}
\label{eq:Dn:growth}  
D_n\leq C_1+C_0\epsilon n.
\end{equation}
By the generator estimates in
\S\ref{sssec:systole_function} together with
\eqref{eq:def:random:margulis:systole-domination}, for any $f$ in the
support of $\mu$ and any $y \in M$,
\begin{displaymath}
\log(\|D_yf\|+\|(D_yf)^{-1}\|)\leq K'(1+\log h(y)),
\end{displaymath}
where $K'$ depends only on the constants in
\S\ref{ssec:assumptions_on_the_total_space} and the constants in the
definition of the Margulis function $h$. Then,
\begin{displaymath}
\frac{|D_n|}{n}\leq K U_n,\qquad
U_n:=\frac1n\sum_{j=0}^{n-1}(1+\log h(x_j)),
\end{displaymath}
where $K=2dK'$. Choose $R\geq h(x)$ sufficiently large
and apply \autoref{lemma:abstract:height:occupation} with $\beta=1$.
For $U_*:=\sup_{n\geq1}U_n$, a union bound gives, for sufficiently large $s$,
\[
\mathbb P(U_*>s)
\leq K_1(x)\sum_{n\geq1}e^{-a_1\sqrt{ns}}
\leq K_2(x)e^{-a_2\sqrt{s}},
\]
Thus $\mathbb E U_*<\infty$. Let $f_n = K U_* - \frac{D_n}{n}$. Then,
$f_n \ge 0$, and Fatou's lemma gives
$\bE (\liminf f_n) \le  \liminf \bE(f_n)$,
hence
\begin{equation}
\label{eq:limisup:Dn}
\limsup \bE\left(\frac{D_n}{n}\right) \le \bE\left(\limsup \frac{D_n}{n}\right) \le C_0 \epsilon,
\end{equation}
where in the last estimate we used (\ref{eq:Dn:growth}).

Choose $N_1$ in \eqref{eq:uniform:gap:positive}
with error $\epsilon$. Then, by the cocycle property, for every $F$, 
\begin{displaymath}
\mathbb E\alpha'_{k-1}(A(mN_1),F)\leq m\epsilon.
\end{displaymath}
Then, by Lemma~\ref{lemma:avg-gamma-sigmak},
\begin{equation}
\label{eq:Elog:sigmak:subseq}
\mathbb E\log\sigma_k(A(mN_1))
\ge -C_{d,k} -\epsilon mN_1
\qquad
\text{for every }m\ge1.
\end{equation}
By (\ref{eq:codim:zero:uniform:expansion}),  
there exists $N_2 \in \bN$
such that
\begin{equation}
\label{eq:E:alpha:d:subseq}
\mathbb E \alpha_d(A(mN_2))
\le \epsilon mN_2
\qquad
\text{for every }m\ge1.
\end{equation}
Let $N$ be the least common multiple of $N_1$ and $N_2$. Then,
\[
\mathbb E\alpha_{d-1}(A(mN),H)
=-\mathbb E\log\sigma_k(A(mN))
 +\mathbb E\alpha_d(A(mN),T_xM)+\mathbb E D_{mN}.
\]
Dividing both sides by $mN$, and using (\ref{eq:limisup:Dn}),
(\ref{eq:Elog:sigmak:subseq}) and (\ref{eq:E:alpha:d:subseq}), we get
\begin{displaymath}
\limsup_{m \to \infty} \frac{\mathbb E\alpha_{d-1}(A(mN),H)}{mN} \le
(C_0 + 2)\epsilon. 
\end{displaymath}
This contradicts \eqref{eq:hyperplane:growth:expectation:lower} when
$(C_0+2)\epsilon<c_{d-1}$.
This contradiction proves that \eqref{eq:nonconcentration:failure} cannot
hold.

Therefore, for every $x\in M$ and every hyperplane $H\subset T_xM$,
\begin{displaymath}
\mathbb P\left(T_x\mathcal W_{d-k}(x,\omega)\not\subset H\right)>0.
\end{displaymath}
\end{proof}

\section{Words in trees}
\label{sec:words:in:trees}

\begin{theorem}
\label{theorem:ac:conclusion}
Let $M$, $\mu$ be as in \S\ref{ssec:assumptions_on_the_total_space}.
Let $\nu$ be an ergodic probability measure on
$M$ which is invariant by any $f$ in the support of $\mu$. We assume
the following:
\begin{itemize}
\item The  conclusions of
  Theorem~\ref{theorem:Good-stables-everywhere} are satisfied.
\item The conclusions of Theorem~\ref{theorem:good:words} are
  satisfied.
\item The conclusions of Theorem~\ref{theorem:non:concentration} are
  satisfied.
\item The conditional measures of $\nu$ along almost any
  $\cW_{d-k}(x,\omega)$ are in the Lebesgue measure class. 
\end{itemize}
Then, $\nu$ is absolutely continuous with respect to volume on $M$,
with smooth density. Futhermore $\operatorname{supp}\nu=M$.
\end{theorem}

\subsubsection{Main goal}
For the absolute  continuity and smooth density result 
it is clearly enough to show the following: given $x_0 \in M$, there
exists a bump function $\rho$ which is identically $1$ in a
neighborhood of $x_0$, and supported in a small neighborhood of $x_0$
such that the measure $\nu_\rho:=\rho \cdot \nu$ has Fourier decay:
\begin{align}
\label{eq:main_estimate}
\left|\int e^{\sqrt{-1}\langle \xi, y\rangle} \, d\nu_\rho(y) \right| \leqapprox_b \frac{1}{1+|\xi|^b}	
\end{align}
for any $b\geq 1$ and any $\xi$, where we work in a fixed coordinate chart near $x$.

The full support result is established separately in
Corollary~\ref{cor:full:support}.

\subsection{Step 1 (choosing the initial prefix $w_0$)}
\label{ssec:Step-1}

Suppose $x_0 \in M$ is arbitrary.  By
Theorem~\ref{theorem:Good-stables-everywhere} (or alternatively by
Lemma~\ref{lemma:abstract:margulis:return:tails}), for any $x_0
\in M$ there exists a word $w$ with $y_0 := f_w^{|w|}(x_0)$ satisfying
$h(y_0) < R_{\rm init}$.  Then a non-zero smooth  density at $y_0$
propogates back to a non-zero smooth  density at $x_0$.  
Thus, it is enough to establish  the main goal and the non-vanishing
of the density  for $x_0 \in M$
satisfying $h(x_0) < R_{\rm init}$.

Fix a point $x_0 \in \supp \nu$ with $h(x_0) < R_{\rm init}$
and let $\xi$ be a unit vector.
Let $\epsilon>0$, $\beta>0$, and $C_{\rm init}$ be the constants from
Theorem~\ref{theorem:good:words}. Let $\xi^\sharp$ denote the set of
$d-k$-planes contained in $\xi^\perp$. 

Applying Theorem~\ref{theorem:non:concentration} with
$H=\xi^\perp$, we find a positive-measure set $E_\xi$ of infinite
words $\omega$ such that
\[
T_{x_0}\mathcal W_{d-k}(x_0,\omega)\not\subset \xi^\perp
\Longrightarrow T_{x_0}\mathcal W_{d-k}(x_0,\omega)\not\in \xi^\sharp
\]
In particular, for every $\omega\in E_\xi$ there exists
$\delta_\omega>0$ such that
\begin{displaymath}
\dist_{\Gr}
\bigl(T_{x_0}\mathcal W_{d-k}(x_0,\omega),\xi^\sharp\bigr)
>
\delta_\omega>0 .
\end{displaymath}
Since Theorem~\ref{theorem:Good-stables-everywhere} holds for every $x_0$
and almost every $\omega$, we may choose $\omega\in E_\xi$ such that
Theorem~\ref{theorem:Good-stables-everywhere} holds for the pair
$(x_0,\omega)$. Thus
\begin{displaymath}
h(f_\omega^m(x_0)) < R_{\rm init} \qquad\text{ for infinitely  many  $m$},
\end{displaymath}
\begin{displaymath}
\Delta_k(Df_\omega^m(x_0))\to\infty
\end{displaymath}
and
\begin{displaymath}
\ker_{d-k}(Df_\omega^m(x_0))
\to
T_{x_0}\mathcal W_{d-k}(x_0,\omega)
\end{displaymath}
in the Grassmannian. We can therefore choose $m\in\mathbb N$ so large that
\begin{displaymath}
h(f_\omega^m(x_0)) < R_{\rm init},
\end{displaymath}
\begin{displaymath}
8 C_{\rm init}
\min\left(1,\Delta_k(Df_\omega^m(x_0))^{-1}\right)
< \delta_\omega/4
\end{displaymath}
and
\begin{displaymath}
\dist_{\Gr}
\bigl(
\ker_{d-k}(Df_\omega^m(x_0)),
T_{x_0}\mathcal W_{d-k}(x_0,\omega)
\bigr)
< \delta_\omega/2.
\end{displaymath}
It follows that
\begin{equation}
\label{eq:step1:condition}
\dist_{\Gr}
\bigl(
\ker_{d-k}(Df_\omega^m(x_0)),\xi^\sharp
\bigr) 
>  8 C_{\rm init}
\min\left(1,\Delta_k(Df_\omega^m(x_0))^{-1}\right).
\end{equation}
The condition \eqref{eq:step1:condition} is open in $\xi$. Therefore, by
compactness of the unit sphere, there exists a finite set
$\mathcal I_0$ of finite words such that for every unit vector $\xi$,
there exists a word $w_0\in\mathcal I_0$ satisfying
\begin{equation}
\label{eq:step1:finite:condition}
\dist_{\Gr}
\bigl(\ker_{d-k}(Df_{w_0}(x_0)),\xi^\sharp\bigr) 
\ge 8 C_{\rm init}
  \min\left(1,\Delta_k(Df_{w_0}(x_0))^{-1}\right).
\end{equation}
Set
\begin{equation}
\label{eq:step1:kappa0}
\kappa_0
:=
\frac12
\min_{w_0\in\mathcal I_0}
C_{\rm init}
\min\left(1,\Delta_k(Df_{w_0}(x_0))^{-1}\right).
\end{equation}
Since $\mathcal I_0$ is finite and each $Df_{w_0}(x_0)$ is invertible, we
have
\begin{displaymath}
\kappa_0>0.
\end{displaymath}
Since $\mathcal I_0$ is finite and the maps
\begin{displaymath}
y\mapsto Df_{w_0}(y),
\qquad
w_0\in\mathcal I_0,
\end{displaymath}
depend continuously on $y$, we may choose $r>0$ so small that, for every
$w_0\in\mathcal I_0$,
\begin{displaymath}
B(x_0,r)\subset BB(w_0,x_0,\epsilon) \cap B(x_0,\tfrac{1}{2}\rho_{\rm init}(|w_0|))
\end{displaymath}
where $\rho_{\rm init}(\cdot)$ is as in Theorem~\ref{theorem:good:words}, 
and such that, for every $w_0\in\mathcal I_0$ and every $y\in B(x_0,r)$,
\begin{equation}
\label{eq:step1:error:scale:lower}
C_{\rm init}
\min\left(1,\Delta_k(Df_{w_0}(y))^{-1}\right)
\ge
\kappa_0.
\end{equation}
After decreasing $r$ further if necessary, we may also ensure that the
estimate \eqref{eq:step1:finite:condition} persists with a factor of $4$
instead of $8$: for every unit vector $\xi$ there exists
$w_0\in\mathcal I_0$ such that for all $y\in B(x_0,r)$, 
\begin{equation}
\label{eq:step1:finite:condition:robust}
\dist_{\Gr}
\bigl(
\ker_{d-k}(Df_{w_0}(y)),\xi^\sharp
\bigr)
\ge
4 C_{\rm init}
\min\left(1,\Delta_k(Df_{w_0}(y))^{-1}\right)
\end{equation}
In the rest of this subsection, all constants (such as $\kappa_0$,
$r_s$, $r$) may depend on the fixed base point, coordinate chart, and
finite family $\mathcal I_0$.

We now make one final choice of the localization radius. The stable manifolds
appearing in Theorem~\ref{theorem:good:words} have uniformly
bounded $C^2$ geometry. Hence, after decreasing a constant depending only on
the constants in Theorem~\ref{theorem:good:words}, we may choose
$r_s>0$ such that, on every stable manifold supplied by
Theorem~\ref{theorem:good:words} for $w_0 \in \cI_0$ and $x \in
B(x_0,r)$, tangent spaces vary by at most
$\kappa_0$ on intrinsic balls of radius $r_s$. In other words, if
$\tilde w$ is a good extension based at $x$ and
\[
y\in \mathcal W_{d-k}(x,\tilde w),
\qquad
d_{\mathcal W}(x,y)<r_s,
\]
then
\begin{equation}
\label{eq:step1:tangent:variation}
\dist_{\Gr}
\bigl(
T_y\mathcal W_{d-k}(x,\tilde w),
T_x\mathcal W_{d-k}(x,\tilde w)
\bigr)
\le \kappa_0.
\end{equation}
We also choose
\begin{equation}
\label{eq:rs:rho:st}
r_s<\min_{w_0\in\mathcal I_0}\rho_{\rm st}(w_0),
\end{equation}
where $\rho_{\rm st}$ is as in Theorem~\ref{theorem:good:words}. 

We shrink $r>0$ further so that it also has the following three properties.

First, the initial word $w_0$ is uniformly controlled on $B(x_0,r)$:
for every $w_0\in\mathcal I_0$ and every pair of points
$x,y\in B(x_0,r)$, the orbit segments defined by all prefixes of $w_0$
remain within the Bowen radius $\epsilon$ from
Theorem~\ref{theorem:good:words}. Equivalently,
\[
y\in BB(w_0,x,\epsilon)
\qquad
\text{whenever } x,y\in B(x_0,r).
\]
Second, $B(x_0,r)$ is small compared with the stable-plaque radius $r_s$.
More precisely, using again the uniform $C^2$ geometry of the stable
manifolds, we choose $r$ so small that, whenever $x\in B(x_0,r)$ and
$\tilde w$ is a good extension based at $x$, one has
\begin{equation}
\label{eq:step1:plaque:localization}
\mathcal W_{d-k}(x,\tilde w)\cap B(x_0,r)
\subset
B_{\mathcal W}(x,r_s).
\end{equation}
Here $B_{\mathcal W}(x,r_s)$ denotes the intrinsic ball of radius $r_s$ in
the stable manifold $\mathcal W_{d-k}(x,\tilde w)$.

Third, for all $w_0 \in \cI_0$ and every $y  \in  B(x_0,r)$
\begin{equation}
\label{eq:hfw0}
h(f_{w_0}^{|w_0|}(y)) < R_{\rm init}
\end{equation}
Fourth, for all $w_0 \in \cI_0$ and  every $y \in B(x_0,r)$,
\begin{equation}
\label{eq:step1:initial:gap}
\Delta_k(D_y f_{w_0}^{|w_0|}(y)) >  e^{\epsilon_{\rm init}}. 
\end{equation}

Finally choose a smooth bump function $\rho$ supported in $B(x_0,r)$ and
identically equal to $1$ on a smaller neighborhood of $x_0$. This will be
the bump function used in \eqref{eq:main_estimate}.

\subsection{Step 2 (building the tree)}
\label{ssec:Step-2}
For the Fourier decay estimate, 
we fix the direction $\xi$ and an initial word, say
$w_0$, such that (\ref{eq:step1:finite:condition}) holds. 
We will establish the bound \eqref{eq:main_estimate} for all multiples of the given $\xi$.

\subsubsection{A tree}
\label{sssec:A-Tree}

Fix the initial word $w_0\in\mathcal I_0$ chosen in Step~1 for the
direction $\xi$. Let $J$ be the finite collection of words supplied by
Theorem~\ref{theorem:good:words}, and let $\epsilon>0$ be the
Bowen radius appearing there. Recall that the bump function $\rho$ is
supported in $B(x_0,r)$.

We choose a smooth partition of unity
\begin{displaymath}
\{\psi_i\}_{i\in I}
\end{displaymath}
with the following smallness property. If two points lie in the support of
the same $\psi_i$, then, after applying any prefix of any word in $J$, their
images remain within distance $\epsilon$. We also assume, as arranged in
Step~1 by shrinking the radius $r$ if necessary, that any two points in
$\supp\rho$ remain $\epsilon$-close under all prefixes of the initial word
$w_0$.

We now define the tree. The tree has a root vertex, denoted by
$\varnothing$. Vertices at level $n\ge1$ are indexed by multi-indices
\begin{displaymath}
\bbi=(i_1,\ldots,i_n)\in I^n.
\end{displaymath}
There is an edge from $\varnothing$ to each vertex $(i_1)$ at level $1$.
More generally, there is an edge from
\begin{displaymath}
\bbi=(i_1,\ldots,i_n)
\end{displaymath}
to each child
\begin{displaymath}
\bbj=(i_1,\ldots,i_n,i_{n+1}).
\end{displaymath}
We write
\begin{displaymath}
\bbi_{\le k}:=(i_1,\ldots,i_k)
\end{displaymath}
for the prefix vertex of length $k$, and we use the convention
$\bbi_{\le0}=\varnothing$.

The words will be assigned to edges of the tree. We will construct an
edge-labeling
\begin{displaymath}
\ell:\mathcal E_{\mathcal T}\to \{w_0\}\cup J
\end{displaymath}
with the convention that all edges leaving a fixed vertex have the same
label. The labels of all edges leaving the root are fixed to be $w_0$:
\begin{displaymath}
\ell(\varnothing,(i_1)):=w_0
\qquad\text{for all } i_1\in I.
\end{displaymath}
All other edge labels will lie in $J$.

Once the edge-labeling has been chosen, we associate a word $w(\bbi)$ to
each vertex $\bbi$ as follows. For a level-one vertex $(i_1)$, set
\begin{displaymath}
w((i_1)):=w_0.
\end{displaymath}
If $\bbj$ is a child of a non-root vertex $\bbi$, then set
\begin{displaymath}
w(\bbj):=\ell(\bbi,\bbj)\,w(\bbi).
\end{displaymath}
Thus $w(\bbi)$ is obtained by concatenating the labels of the edges along
the path from the root to $\bbi$.

We also define smooth functions $\theta_\bbi:M\to[0,1]$ recursively. At the
root, set
\begin{displaymath}
\theta_\varnothing:=\rho.
\end{displaymath}
For a level-one vertex $(i_1)$, set
\begin{displaymath}
\theta_{(i_1)}
:=
\theta_\varnothing\cdot w((i_1))^*\psi_{i_1}
=
\rho\cdot w_0^*\psi_{i_1}.
\end{displaymath}
More generally, if $\bbj=(\bbi,i_{n+1})$ is a child of a non-root vertex
$\bbi$, set
\begin{displaymath}
\theta_\bbj
:=
\theta_\bbi\cdot w(\bbj)^*\psi_{i_{n+1}}.
\end{displaymath}
Thus each $\theta_\bbi$ contains the initial localization factor $\rho$.

The following elementary lemma records the point of this construction.

\begin{lemma}[Tree atoms are Bowen-small]
\label{lemma:theta:bowen:small}
Let $\bbi$ be a non-root vertex of the tree. If
\begin{displaymath}
x,y\in \supp\theta_\bbi,
\end{displaymath}
then
\begin{displaymath}
y\in BB(w(\bbi),x,\epsilon) \cap B(x,\rho_{\rm init}(|w_0|)).
\end{displaymath}
Furthermore, if $1 \le k \le |\bbi|$, the two points
\begin{displaymath}
f_{w(\bbi_{\le k})}(x),
\qquad
f_{w(\bbi_{\le k})}(y)
\end{displaymath}
lie in the support of the same partition element $\psi_{i_k}$.
\end{lemma}

\begin{proof}
Since every $\theta_\bbi$ contains the factor $\rho$, we have
\begin{displaymath}
x,y\in\supp\rho\subset B(x_0,r) \subset B(x_0,\tfrac{1}{2} \rho_{\rm
  init}(|w_0|)). 
\end{displaymath}
Thus, $y \in B(x,\rho_{\rm init}(|w_0|))$. 

By the choice of $r$ in Step~1, the two orbit segments of $x$ and $y$
remain $\epsilon$-close during the initial word $w_0$. 
Now follow the path from the root to $\bbi$. Let $\bbi_{\le k}$ be a
prefix vertex with $1\le k\le |\bbi|$. By the definition of
$\theta_\bbi$, the two points
\begin{displaymath}
f_{w(\bbi_{\le k})}(x),
\qquad
f_{w(\bbi_{\le k})}(y)
\end{displaymath}
lie in the support of the same partition element $\psi_{i_k}$.

If $k<|\bbi|$, then the next edge label along the path is a word in $J$.
By the smallness property of the partition, applying any prefix of this
next edge label keeps the two corresponding images within distance
$\epsilon$. Repeating this argument for each subsequent edge, and combining
it with the initial control during $w_0$, we see that the two orbit
segments remain $\epsilon$-close throughout the whole word $w(\bbi)$.
Therefore
\begin{displaymath}
y\in BB(w(\bbi),x,\epsilon).
\end{displaymath}
\end{proof}

We now choose the remaining edge labels.

\begin{proposition}[Partitions of unity with words]
\label{prop:Partitions-of-Unity-with-Words}
There exists an edge-labeling
\begin{displaymath}
\ell:\mathcal E_{\mathcal T}\to \{w_0\}\cup J
\end{displaymath}
such that all edges leaving a fixed vertex have the same label, all edges
leaving the root have label $w_0$, all other edge labels lie in $J$, and
such that the following holds. For every non-root vertex $\bbi$ and every
$x\in\supp\theta_\bbi$,
\begin{equation}
\label{eq:hfwi}
h(f_{w(\bbi)}^{|w(\bbi)|}(x)) < R_{\rm init}.
\end{equation}
Furthermore,
the word $w(\bbi)$ and the point $x$ satisfy the recursive hypotheses of
Theorem~\ref{theorem:good:words}.
\end{proposition}

\begin{proof}
The labels of all edges leaving the root are fixed to be $w_0$. This
defines $w((i_1))=w_0$ and
\begin{displaymath}
\theta_{(i_1)}=\rho\cdot w_0^*\psi_{i_1}
\end{displaymath}
for every level-one vertex $(i_1)$.

We now choose the labels leaving non-root vertices inductively. Suppose
that $\bbi$ is a non-root vertex for which the labels along the path from
the root to $\bbi$ have already been chosen. Then $w(\bbi)$ and
$\theta_\bbi$ are defined.

If
\begin{displaymath}
\supp\theta_\bbi=\varnothing,
\end{displaymath}
we label all edges leaving $\bbi$ by an arbitrary word in $J$.
Otherwise choose a representative point
\begin{displaymath}
x_\bbi\in\supp\theta_\bbi.
\end{displaymath}
Let
\begin{displaymath}
y_\bbi =  f^{|w(\bbi)|}_{w(\bbi)}(x_\bbi)
\end{displaymath}
By the inductive assumption, $h(y_\bbi) < R_{\rm init}$. 

Let $u(\bbi)$ be the last edge label entering $\bbi$. This, for a
level-one vertex, $u(\bbi) = w_0$, and for deeper vertices, $u(\bbi)
\in J$.  
For every child $\bbj$ of $\bbi$, define the common outgoing label by
\begin{displaymath}
\ell(\bbi,\bbj):=w_{\rm next}(y_\bbi,u(\bbi)).
\end{displaymath}
This is a word in $J$, by Theorem~\ref{theorem:good:words}. Having
chosen this label, define
\begin{displaymath}
w(\bbj)=\ell(\bbi,\bbj)\,w(\bbi)
\end{displaymath}
and
\begin{displaymath}
\theta_\bbj
=
\theta_\bbi\cdot w(\bbj)^*\psi_{i_{n+1}},
\end{displaymath}
where $\bbj=(\bbi,i_{n+1})$.

It remains to verify the recursive hypotheses of
Theorem~\ref{theorem:good:words} and to show that (\ref{eq:hfwi})
holds. To prove (\ref{eq:hfwi}) we use induction on the vertex depth. 
For a level one vertex, the
equation (\ref{eq:hfwi}) follows from (\ref{eq:hfw0}). 

Let $\bbj$ be a non-root vertex,
and let
$
x\in\supp\theta_\bbj.
$
We must check the recursive condition for each non-root proper prefix
vertex $\bbi$ along the path from the root to $\bbj$. Since
$
\supp\theta_\bbj\subset\supp\theta_\bbi,
$
we have
$
x\in\supp\theta_\bbi.
$

Moreover, the representative point $x_\bbi$ chosen at the vertex $\bbi$
also belongs to $\supp\theta_\bbi$. By
Lemma~\ref{lemma:theta:bowen:small},
\begin{displaymath}
x_\bbi\in BB(w(\bbi),x,\epsilon) \cap B(x,\rho_{\rm init}(|w_0|)). 
\end{displaymath}
Since $\bbi$ is a proper prefix
of $\bbj$, the inductive hypothesis applied to
$x_\bbi\in\operatorname{supp}\theta_\bbi$, shows that (\ref{eq:hfwi})
holds. 

By construction, the edge leaving $\bbi$ along the path to $\bbj$ is labeled
by
\begin{displaymath}
w_{\rm next}(f_{w(\bbi)}^{|w(\bbi)|}(x_\bbi),u(\bbi)),
\end{displaymath}
where $u(\bbi)$ is the label of the last edge entering $\bbi$.  
This is precisely the recursive condition required in
Theorem~\ref{theorem:good:words}.

Since this verification applies at every non-root proper prefix vertex
along the path from the root to $\bbj$, the word $w(\bbj)$ and the point
$x$ satisfy the recursive hypotheses of
Theorem~\ref{theorem:good:words}.

It remains to show that (\ref{eq:hfwi}) holds for $\bbj$. Let $\bbi$
be the parent of $\bbj$. By Lemma~\ref{lemma:theta:bowen:small},
$f_{w(\bbi)}^{|w(\bbi)|}(x)$ and $f_{w(\bbi)}^{|w(\bbi)|}(x_\bbi)$ lie in
the support of the same partition element $\psi_k$. Then, by the smallness property of the
partition chosen in \S\ref{sssec:A-Tree}, applied to
$v=\ell(\bbi,\bbj)\in J$,
\begin{displaymath}
d(f_{w(\bbj)}^{|w(\bbj)|}(x), f_{w(\bbj)}^{|w(\bbj)|}(x_\bbi)) < \epsilon.
\end{displaymath}
By induction,
$h(f_{w(\bbi)}^{|w(\bbi)|}(x_\bbi)) < R_{\rm init}$. 
Then, by \eqref{eq:good:words:uniform:return}, 
$h(f_{w(\bbj)}^{|w(\bbj)|}(x)) < R_{\rm init}$.
This  shows  that (\ref{eq:hfwi}) holds  for
$\bbj$, completing the inductive step. 
\end{proof}

\subsubsection{More notation and basic properties}
\label{sssec:More-notation-and-basic-properties}

For every non-root vertex
$
\bbj=(i_1,\ldots,i_n),
$
we introduce the notation
$
\phi_\bbj:=w(\bbj)^*\psi_{i_n}.
$
Thus, if $\bbj=(\bbi,i_{n+1})$ is a child of $\bbi$, then
$
\theta_\bbj=\theta_\bbi\,\phi_\bbj.
$

For each vertex $\bbi$, let
\begin{displaymath}
\Phi_\bbi
:=
\{\phi_\bbj:\ \bbj \text{ is a child of } \bbi\}.
\end{displaymath}
Then $\Phi_\bbi$ is a partition of unity of $M$. Indeed, all edges leaving
$\bbi$ have the same label, so the word $w(\bbj)$ is independent of the
choice of child except for the final index $i_{n+1}$. Hence
\begin{displaymath}
\sum_{\bbj:\,\bbj \text{ child of }\bbi}\phi_\bbj
=
\sum_{i_{n+1}\in I} w(\bbj)^*\psi_{i_{n+1}}
=
w(\bbj)^*\left(\sum_{i_{n+1}\in I}\psi_{i_{n+1}}\right)
=
1.
\end{displaymath}
Let
$
\Theta_n:=\{\theta_\bbi:\ |\bbi|=n\}
$
denote the family of functions attached to vertices at level $n$. With the
present convention, these are not a partition of unity of $M$; rather, they
partition the initial bump function $\rho$:
\begin{equation}
\label{eq:theta:partition:rho}
\sum_{|\bbi|=n}\theta_\bbi=\rho
\qquad
\text{for every } n\ge1.
\end{equation}
Indeed, for $n=1$,
\begin{displaymath}
\sum_{i_1\in I}\theta_{(i_1)}
=
\rho\sum_{i_1\in I} w_0^*\psi_{i_1}
=
\rho.
\end{displaymath}
The induction step follows from
\begin{displaymath}
\sum_{\bbj:\,\bbj \text{ child of }\bbi}\theta_\bbj
=
\theta_\bbi
\sum_{\bbj:\,\bbj \text{ child of }\bbi}\phi_\bbj
=
\theta_\bbi.
\end{displaymath}

For
$
\bbi=(i_1,\ldots,i_n),
$
write
$
\bbi_{\le k}:=(i_1,\ldots,i_k)
$
for the prefix vertex of length $k$. Then
\begin{equation}
\label{eqn:theta_phi_relation}
\theta_\bbi
=
\rho\prod_{k=1}^{n}\phi_{\bbi_{\le k}}.
\end{equation}

Finally, we record the multiplicity consequence that will be used later. If
$m_I$ denotes the maximal multiplicity of the cover
$\{\supp\psi_i\}_{i\in I}$, then for each vertex $\bbi$ and each point
$x\in M$, at most $m_I$ children $\bbj$ of $\bbi$ satisfy
\begin{displaymath}
\phi_\bbj(x)\neq0.
\end{displaymath}
This follows because the children differ only by the final partition element
$\psi_{i_{n+1}}$, and pullback by the fixed map $f_{w(\bbj)}$ preserves the
multiplicity bound.

\subsection{Step 3 (the infinite extension)}

\subsubsection{Good and bad extensions}
\label{sssec:good:and:bad}

Recall  that for every $x_0$ with $h(x_0) < R_{\rm init}$ and any unit
vector $\xi$, we build a tree $\cT = \cT(x_0,\xi)$ with the root label
$w_0 \in \cI_0$ satisfying \eqref{eq:step1:finite:condition:robust}.

For a non-root vertex $\bbi\in\mathcal T$, let
$
n(\bbi):=|\bbi|-1
$
denote the number of edges after the initial edge labeled by $w_0$ along
the path from the root to $\bbi$.

Let $n_\ast(\omega)$ be as in Theorem~\ref{theorem:good:words}, and
let $n_\ast = \max\{ n_\ast(\omega) \st \omega \in \cI_0\}$. 
\begin{proposition}[Good and bad extensions]
\label{prop:Good-and-Bad-extensions}
Let $\bbi\in\mathcal T$ be a non-root vertex with $n(\bbi) \ge n_\ast$ and let
$x\in\supp\theta_\bbi$. Then there exists a decomposition
\begin{displaymath}
\Omega(w(\bbi))=G(w(\bbi),x)\sqcup B(w(\bbi),x)
\end{displaymath}
such that
\begin{displaymath}
\mathbb P_{w(\bbi)}
\bigl(B(w(\bbi),x)\bigr)
\le e^{-\beta \sqrt{n(\bbi)}}.
\end{displaymath}
Moreover, for every $\tilde w\in G(w(\bbi),x)$, the stable manifold
$\mathcal W_{d-k}(x,\tilde w)$
exists and satisfies the estimates in
Theorem~\ref{theorem:good:words}, in particular the plaque size and all $C^r$-geometry bounds, the contraction
\eqref{eq:good:words:stable:contraction}, and the Jacobian bounds
\eqref{eq:good:words:jacobian:bounds}--\eqref{eq:good:words:jacobian:derivatives}. The
constants in these
estimates are uniform over $x \in \supp \theta_\bbi$, and do not
depend on $\bbi$ or on $\tilde{w} \in G(w(\bbi),x)$. 

Futhermore, there exists $c > 0$ such that
the following holds: for every
$x\in\supp\theta_\bbi$, and every $\tilde w\in G(w(\bbi),x)$, 
\begin{equation}
\label{eq:uniform:transversality}
\bigl|\operatorname{proj}_{T_y\cW_{d-k}(x,\tilde{w})}\xi\bigr|\ge c
\qquad\text{for every }y\in B^{\mathcal W}(x,\tilde w,r_s).
\end{equation}
Projection and norm are Euclidean in the fixed coordinates. The constant
$c$ does not depend on $\xi$, $\bbi$, the labels of
$\cT$, $x \in \supp \theta_\bbi$, or $\tilde{w} \in G(w(\bbi),x)$. 

Also, for each fixed vertex $\bbi$, the conditional measure of $\nu$ along
$\mathcal W_{d-k}(x,\tilde w)$ is equivalent to its leaf volume for
$(\nu|_{\operatorname{supp}\theta_\bbi})\otimes\mathbb P_{w(\bbi)}$-almost
every pair $(x,\tilde w)$ with $\tilde w\in G(w(\bbi),x)$.
\end{proposition}

\begin{proof}
The condition $n(\bbi) \ge n_\ast$,
Proposition~\ref{prop:Partitions-of-Unity-with-Words} and equations
\eqref{eq:hfw0} and \eqref{eq:step1:initial:gap}, ensure that the word
$w(\bbi)$ and the point $x$ satisfy the hypotheses of
Theorem~\ref{theorem:good:words}. Applying that theorem gives the
decomposition and the stated estimates.
We take minima of radii and maxima of constants
over the finite set $\mathcal I_0$; the resulting bounds have the
required uniformity properties.

We now begin the proof of (\ref{eq:uniform:transversality}). Since
$x\in\supp\theta_\bbi\subset\supp\rho$,
\eqref{eq:step1:finite:condition:robust} holds. 
Since $\tilde w\in G(w(\bbi),x)$, \eqref{eq:good:words:tangent:control}
gives
\begin{displaymath}
\dist_{\Gr}
\bigl(
T_x\mathcal W_{d-k}(x,\tilde w),
\ker_{d-k}(Df_{w_0}(x))
\bigr)
\le
C_{\rm init}
\min\left(1,\Delta_k(Df_{w_0}(x))^{-1}\right).
\end{displaymath}
Hence, by the triangle inequality, 
\begin{displaymath}
\dist_{\Gr}
\bigl(
T_x\mathcal W_{d-k}(x,\tilde w),\xi^\sharp
\bigr)
\ge
3 C_{\rm init}
\min\left(1,\Delta_k(Df_{w_0}(x))^{-1}\right).
\end{displaymath}
By \eqref{eq:step1:error:scale:lower}, the right-hand side is bounded below
by $3\kappa_0$. Since by \eqref{eq:step1:tangent:variation}
tangent spaces vary by at most $\kappa_0$ on
$B_{\mathcal W}(x,r_s)$, we get
\begin{equation}
\label{eq:leafwise:transversality}
\dist_{\Gr}
\bigl(
T_y\mathcal W_{d-k}(x,\tilde w),\xi^\sharp
\bigr)
\ge
2\kappa_0 \qquad\text{for every
$y\in B^\cW(x,\tilde{w},r_s)$.}
\end{equation}
Equivalently, the restriction of the linear functional
$y\mapsto\langle \xi,y\rangle$ to the stable tangent space has norm bounded
below by a constant multiple of $|\xi|$, uniformly on $B^\cW(x,\tilde{w},
r_s)$. This concludes the proof of
(\ref{eq:uniform:transversality}). 

The last assertion follows by pulling
back the conditional-measure assertion for the random tail under
$f_{w(\bbi)}$: this map preserves $\nu$, and smooth pullback preserves the
Lebesgue class.
\end{proof}

\begin{corollary}
\label{cor:full:support}
Under the hypotheses of Theorem~\ref{theorem:ac:conclusion}, 
$\operatorname{supp}\nu=M$.
\end{corollary}

\begin{remark} In Corollary~\ref{cor:full:support}, instead of
assuming that $\nu$ is a probability measure,  it is enough to 
assume that $\nu$ is locally  finite. 
The proof is quite similar to the one in Zhang \cite[\S5]{Zhang2019_On-stable-transitivity-of-finitely-generated-groups-of-volume-preserving}.
\end{remark}

\begin{proof}[Proof of Corollary~\ref{cor:full:support}]
Let $K = \supp \nu$. If $K \ne M$, there exists a point  $z \in
\partial K$.  By Theorem~\ref{theorem:Good-stables-everywhere} (or
alternatively Lemma~\ref{lemma:abstract:margulis:return:tails}), there
exists a finite word $w$ with $y = f_w^{|w|}(z)$ satisfying $h(y) <
R_{\rm init}$.  Then, by invariance
of the measure, $f_w^{|w|}(K)=K$, so $y \in \partial K$. 
In the rest of the proof,  we work with the Euclidean distance in a
fixed coordinate chart $\cU$ around $y$. 

Let $\ve$ be such that
\begin{displaymath}
\overline{B(y,\ve)} \subset \{ x \st h(x) < R_{\rm init}
\}  \cap \cU.
\end{displaymath}
Since $y \in \partial  K$, there  exists a point $y_0 \in
B(y,\ve/2)$ such that $y_0 \not\in K$. Let $x_0$ be the closest point
in $K \cap \overline{B(y,\ve)}$
to $y_0$, and let $D = B(y_0, d(x_0,y_0))$.  Since $y \in K$,
$d(x_0,y_0) < \ve/2$. 
Then, $x_0 \in K \cap \partial D$, $D \subset K^c$
and $\overline{D} \subset \{x \st h(x) < R_{\rm init}\}$.

Let $\xi$ be the unit vector based at  $x_0$ pointing in the direction
of $y_0$. We choose $w_0 \in \cI_0$ and $r$, $\rho$ as in Step~1, and
so that, in particular, \eqref{eq:step1:finite:condition:robust} holds. We
build the associated tree as in Step~2. Then,
Proposition~\ref{prop:Good-and-Bad-extensions} holds.
Fix $n \in \bN$ with $n \ge n_\ast$.

We claim that there exists a neighborhood $U$ of $x_0$ with the
following  property. Suppose $\bbi$ is any vertex with $n(\bbi) = n$,
$x  \in \supp \theta_\bbi \cap U$, $\tilde{w} \in G(w(\bbi),x)$. Then,
$B^{\cW}(x,\tilde{w},\rho_{\rm st}) \subset \cW_{d-k}(x,\tilde{w})$
intersects $D$.

Suppose no such $U$ exists. Then (since there are only finitely many
$\bbi$ with $n(\bbi)=n$ and $\theta_\bbi$ not identically $0$),
there exists $\bbi$ with $n(\bbi) = n$,
a sequence $x_j \to x_0$ with $x_j \in \supp \theta_\bbi$
and $\tilde{w}_j \in G(w(\bbi), x_j)$ with
\begin{equation}
\label{eq:missing:D}
B^\cW(x_j, \tilde{w}_j, \rho_{\rm st}) \cap D = \emptyset.   
\end{equation}
Let $v_j'$ be the orthogonal projection of $\xi$  to $T_{x_j}
\cW_{d-k}(x_j, \tilde{w}_j)$, and let $v_j = v_j'/\|v_j'\|$. Then
$v_j$ is a unit vector based at $x_j$ tangent to $\cW_{d-k}(x_j,
\tilde{w}_j)$. Write
\begin{displaymath}
v_j = a_j \xi + u_j, \qquad\text{ $u_j  \in \xi^\perp$.} 
\end{displaymath}
In view of the uniform bounds on the $C^2$
geometry of $\cW_{d-k}(x_j, \tilde{w}_j)$ we can pick $\epsilon' > 0$
and  curves
$\gamma_j: [-\epsilon',\epsilon']
\to B^\cW(x_j,\tilde{w}_j, \rho_{\rm st})$ with
uniformly bounded second derivatives, such that $\gamma_j(0) = x_j$
and $\gamma_j'(0) = v_j$. In view of (\ref{eq:missing:D}),
\begin{equation}
\label{eq:gamma:j:image}
\gamma_j([-\epsilon',\epsilon']) \cap D = \emptyset.
\end{equation}
We can pass to a subsequence for which the curves $\gamma_j$ converge
in  the $C^1$ topology to a curve $\gamma$.  Then, $\gamma(0) = x_0$
and in view of (\ref{eq:gamma:j:image}), $\gamma'(0) \in
\xi^\perp$. Then, along the subsequence, $a_j \to 0$. 
However, by
(\ref{eq:uniform:transversality}), $a_j = \langle v_j,\xi\rangle \ge
c$, where $c>0$ does not depend  on $j$. This contradiction
completes the proof of the claim. 

Without loss of generality, we may assume that
\begin{displaymath}
U  \subset B(x_0,r) \cap \{ \rho = 1\}. 
\end{displaymath}
Since $n(\bbi)=n$ on
$\mathcal T_{n+1}$, Proposition~\ref{prop:Good-and-Bad-extensions}
and \eqref{eq:theta:partition:rho} give
\[
\begin{aligned}
&\sum_{\bbi\in\mathcal T_{n+1}}
\int_U\theta_\bbi(x)\,
\mathbb P_{w(\bbi)}\bigl(G(w(\bbi),x)\bigr)\,d\nu(x)\\
&\qquad\geq
\bigl(1-e^{-\beta\sqrt n}\bigr)\int_U\rho\,d\nu>0.
\end{aligned}
\]
The last inequality uses $x_0\in\operatorname{supp}\nu$ and
$\rho=1$ on $U$. Thus, there exists $\bbi \in \cT_{n+1}$ with $(\nu
\cross  \mathbb P_{w(\bbi)})(E_\bbi) > 0$, where
\begin{displaymath}
E_\bbi = \{ (x,\tilde{w}) \st x \in U \cap
\supp \theta_\bbi,
\tilde{w} \in G(w(\bbi),x) \}. 
\end{displaymath}
Let $\nu_{\cW(x,\tilde{w})}$ denote the conditional measure along
$\cW(x,\tilde{w})$. Then, for $\bP_{w(\bbi)}$-almost all
$\tilde{w} \in \Omega(w(\bbi))$, disintegration of $\nu(D) = 0$ with
respect to the conditionals $\nu_{\cW(x,\tilde{w})}$ implies that
for $\nu$-almost all $x$,
\begin{equation}
\label{eq:nu:D:cap:BcW}
\nu_{\cW(x,\tilde{w})}(D \cap
B^\cW(x,\tilde{w}, \rho_{\rm st})) = 0.
\end{equation}
By Fubini, this means  that (\ref{eq:nu:D:cap:BcW}) holds
for $\nu\otimes\mathbb P_{w(\bbi)}$-almost  all $(x,\tilde{w}) \in E_\bbi$.
By the last assertion of Proposition~\ref{prop:Good-and-Bad-extensions},
for $\nu\otimes\mathbb P_{w(\bbi)}$-almost all $(x,\tilde{w})\in
E_\bbi$, $\nu_{\cW(x,\tilde{w})}$ is in the 
Lebesgue measure class on $B^\cW(x,\tilde{w},\rho_{\rm st})$. 
But for $(x,\tilde{w}) \in E_\bbi$, $D \cap  B^\cW(x,\tilde{w},\rho_{\rm
  st})$ is a non-empty relatively open subset of $B^\cW(x,\tilde{w},\rho_{\rm
  st})$. This is a contradiction.
\end{proof}

\begin{proposition}[Pointwise boundedness of $\phi_\bbi$ along stables]
\label{prop:Boundedness-along-stables}
For every $b\ge1$, there exists $C_b<\infty$ with the following
property. Let $\bbj\in\mathcal T$ satisfy $n(\bbj)\ge n_\ast$,
let $x\in\operatorname{supp}\theta_\bbj$, and let
$\tilde w\in G(w(\bbj),x)$.
Then, for every non-root ancestor $\bbi$ of $\bbj$
(including $\bbi=\bbj$) and every
\[
y\in\operatorname{supp}\theta_\bbj
   \cap\mathcal W_{d-k}(x,\tilde w),
\]
one has
\[
\left\|
D_y^b\bigl(
\phi_\bbi|_{\mathcal W_{d-k}(x,\tilde w)}
\bigr)
\right\|\le C_b.
\]
The constant is independent of $\bbi$, $\bbj$, their depths,
$x$, $y$, and the good extension $\tilde w$.
\end{proposition}

\begin{proof}
Write $\bbi=(i_1,\ldots,i_q)$, $m=|w(\bbi)|$, and
$\mathcal W=\mathcal W_{d-k}(x,\tilde w)$.
Since $f_{w(\bbi)}=f_{\tilde w}^m$ and
$y\in B^{\mathcal W}(x,\rho_{\rm st})$ by
\eqref{eq:step1:plaque:localization}, the estimate
\eqref{eq:stable:all:derivatives} in
Theorem~\ref{theorem:stable:manifold:theorem}, with the uniform
constants for good extensions, bounds all derivatives through order
$b$ of $f_{w(\bbi)}|_{\mathcal W}$ at $y$.
By support inclusion and \eqref{eq:hfwi}, $f_{w(\bbi)}(y)$ lies in
the fixed compact set $K=\{h\le R_{\rm init}\}$, on which the
derivatives of all $\psi_i$ through order $b$ are uniformly bounded
by local finiteness. The chain rule applied to
$\phi_\bbi=\psi_{i_q}\circ f_{w(\bbi)}$ now gives the assertion.
\end{proof}

\begin{proposition}[Leafwise Fourier estimate]
\label{prop:leafwise_Fourier_estimate}
For every $b\in\mathbb N$, there exists $C_b<\infty$ such that the following
holds. Suppose $\xi \ne 0$, and let $\cT$ be the tree constructed for the direction $\xi/\|\xi\|$.
Let $\bbi\in\mathcal T$ be a vertex with $n(\bbi) > n_\ast$, let
$x\in\supp\theta_\bbi$, and let
$
\tilde w\in G(w(\bbi),x).
$

Then
\begin{multline*}
\left|
\int_{\mathcal W_{d-k}(x,\tilde w)}
\theta_\bbi(y)J_{\tilde w}(x,y)
e^{\sqrt{-1}\langle\xi,y\rangle}\,
d\Leb_{\mathcal W_{d-k}(x,\tilde w)}(y)
\right|
\le \\
\frac {C_b}{|\xi|^b}
\sum_{j=0}^b
\int_{\mathcal W_{d-k}(x,\tilde w)}
\left\|
J_{\tilde w}(x,y) D_{\mathcal W}^j\theta_\bbi(y)
\right\|\,
d\Leb_{\mathcal W_{d-k}(x,\tilde w)}(y).
\end{multline*}
Here $D_{\mathcal W}^j$ denotes the $j$th derivative along the stable
manifold $\mathcal W_{d-k}(x,\tilde w)$.
\end{proposition}

\begin{proof}
Since $\theta_\bbi$ contains the factor $\rho$, the integrand is supported
on
\begin{displaymath}
\mathcal W_{d-k}(x,\tilde w)\cap\supp\rho .
\end{displaymath}
By \eqref{eq:step1:plaque:localization}, this set is contained
in the intrinsic stable ball
\begin{displaymath}
B_{\mathcal W}(x,r_s)\subset \mathcal W_{d-k}(x,\tilde w).
\end{displaymath}
On this intrinsic ball 
\eqref{eq:uniform:transversality} from
Proposition~\ref{prop:Good-and-Bad-extensions} holds.

By Theorem~\ref{theorem:good:words}, the relevant stable plaque has
uniform $C^{b+1}$ geometry. Since the support of the integrand is contained in a
single intrinsic stable ball, we may parametrize it by one stable chart
\begin{displaymath}
\Psi:U\subset\mathbb R^{d-k}\to \mathcal W_{d-k}(x,\tilde w),
\end{displaymath}
with uniformly bounded $C^{b+1}$ norm and so that uniformly on the chart
$c|v|\le |D\Psi(t)v|\le C|v|$. In this chart, the integral has the form
\begin{displaymath}
\int_U
a(t)e^{\sqrt{-1}\Phi(t)}\,dt,
\end{displaymath}
where
\begin{displaymath}
\Phi(t)=\langle \xi,\Psi(t)\rangle
\text{ and }
a(t)
=
\theta_\bbi(\Psi(t))
J_{\tilde w}(x,\Psi(t))
\Jac_\Psi(t).
\end{displaymath}
Note that the chart is chosen so that
$\operatorname{supp}a\Subset U$. 
The transversality estimate \eqref{eq:leafwise:transversality}
(applied to $\xi/\|\xi\|$ in place of $\xi$) implies that
\begin{displaymath}
|\nabla\Phi(t)|\ge c|\xi|
\end{displaymath}
on the support of $a$, for a uniform constant $c>0$.

Using the integration-by-parts operator
\begin{displaymath}
L
=
\frac{1}{\sqrt{-1}}
\frac{\nabla\Phi}{|\nabla\Phi|^2}\cdot\nabla,
\qquad
L e^{\sqrt{-1}\Phi}=e^{\sqrt{-1}\Phi},
\end{displaymath}
and integrating by parts $b$ times, we obtain
\begin{displaymath}
\left|
\int_U
a(t)e^{\sqrt{-1}\Phi(t)}\,dt
\right|
\le
\frac{C_b}{|\xi|^b}
\sum_{j=0}^b
\int_U |\nabla^j a(t)|\,dt .
\end{displaymath}
The constants are uniform because the stable chart has uniformly bounded
$C^{b+1}$ geometry and because the lower bound for $|\nabla\Phi|/|\xi|$ is
uniform.

For $y  \in B^\cW(x,\tilde{w},\rho_{\rm st})$,
\eqref{eq:good:words:jacobian:bounds} and
\eqref{eq:good:words:jacobian:derivatives} give, for $r \ge 1$, 
\begin{displaymath}
\|D^r_{\mathcal W}J_{\tilde w}(x,y)\|
\le C_{{\rm Jac},r}C_{\rm Jac}J_{\tilde w}(x,y).
\end{displaymath}
The derivatives of
$\Jac_\Psi$ are uniformly bounded by the stable-chart geometry. Hence the
product rule gives
\begin{displaymath}
\sum_{j=0}^b
\int_U |\nabla^j a(t)|\,dt
\le
C_b
\sum_{j=0}^b
\int_{\mathcal W_{d-k}(x,\tilde w)}
\left\| J_{\tilde{w}}(x,y)
D_{\mathcal W}^j\theta_\bbi(y)
\right\|\,
d\Leb_{\mathcal W_{d-k}(x,\tilde w)}(y).
\end{displaymath}
Combining the last two estimates gives
\begin{multline*}
\left|
\int_{\mathcal W_{d-k}(x,\tilde w)}
\theta_\bbi(y)J_{\tilde w}(x,y)
e^{\sqrt{-1}\langle\xi,y\rangle}\,
d\Leb_{\mathcal W_{d-k}(x,\tilde w)}(y)
\right|
\le \\
\frac {C_b}{|\xi|^b}
\sum_{j=0}^b
\int_{\mathcal W_{d-k}(x,\tilde w)}
\left\| J_{\tilde{w}}(x,y)
D_{\mathcal W}^j\theta_\bbi(y)
\right\|\,
d\Leb_{\mathcal W_{d-k}(x,\tilde w)}(y).
\end{multline*}
This proves the proposition.
\end{proof}

\subsubsection{Good and bad extensions and points}
\label{sssec:Good-and-bad-extensions-and-points}

Fix a non-root vertex $\bbi\in\mathcal T$, and write
$
n(\bbi):=|\bbi|-1
$.
For each $x\in\supp\theta_\bbi$,
Proposition~\ref{prop:Good-and-Bad-extensions} gives a decomposition
\begin{displaymath}
\Omega(w(\bbi))
=
G(w(\bbi),x)\sqcup B(w(\bbi),x).
\end{displaymath}

For a fixed extension $\tilde w\in\Omega(w(\bbi))$, we define the good set
$G_\bbi(\tilde w)$ by declaring a point to be good if it lies on a good
local stable plaque.
More preciselty, with $x_0$ and $r$ fixed in Step~1, write
$\cW_{d-k}(x,\tilde w,r)$ for the connected component containing $x$
of $\cW_{d-k}(x,\tilde w)\cap B(x_0,r)$. Here $r$ specifies the
ambient ball centered at $x_0$, not an intrinsic stable ball.
Define
\[
G_\bbi(\tilde w)
=\supp\theta_\bbi\cap
\bigcup_{\substack{x\in\supp\theta_\bbi\\
                    \tilde w\in G(w(\bbi),x)}}
\cW_{d-k}(x,\tilde w,r).
\]
Thus $G_\bbi(\tilde w)$ consists of the points in
$\supp\theta_\bbi$ locally stably related to a good point.

The set of pairs $(y,\tilde w)$ with
$y\in G_\bbi(\tilde w)$ is the projection onto $(y,\tilde w)$ of
the Borel relation
\[
\bigl\{(x,y,\tilde w):x,y\in\supp\theta_\bbi,\quad
\tilde w\in G(w(\bbi),x),\quad
y\in\cW_{d-k}(x,\tilde w,r)\bigr\}.
\]
It is therefore analytic and universally measurable.
The displayed relation is Borel by the measurable dependence
of the good events and stable plaques on $(x,\tilde w)$, with membership
in the connected component containing $x$ expressible by countably many
path conditions in plaque coordinates.

Let
$
B_\bbi(\tilde w)
:=
\supp\theta_\bbi\setminus G_\bbi(\tilde w).
$
In view of the above discussion, the set 
\begin{equation}
\label{eq:set:of:bad:pairs}
\{ (y,\tilde w) \st  y\in B_\bbi(\tilde w) \} 
\end{equation}
is universally  measurable.

\begin{proposition}[Bad points have small average mass]
\label{prop:bad:points:small:average}
For every non-root vertex $\bbi\in\mathcal T$,
\begin{displaymath}
\int_{\Omega(w(\bbi))}
\int_{B_\bbi(\tilde w)}
\theta_\bbi(y)\,d\nu(y)\,
d\mathbb P_{w(\bbi)}(\tilde w)
\le
e^{-\beta \sqrt{n(\bbi)}}
\int \theta_\bbi\,d\nu .
\end{displaymath}
\end{proposition}

\begin{proof}
If $y\in B_\bbi(\tilde w)$, then
\begin{displaymath}
\tilde w\notin G(w(\bbi),y).
\end{displaymath}
Indeed, if $\tilde w\in G(w(\bbi),y)$, then taking $x=y$ in the definition
of $G_\bbi(\tilde w)$ and using 
$y\in\cW_{d-k}(y,\tilde w,r)$ would imply
$y\in G_\bbi(\tilde w)$, a contradiction.
Hence
\begin{displaymath}
B_\bbi(\tilde w)
\subset
\{y\in\supp\theta_\bbi:\tilde w\in B(w(\bbi),y)\}.
\end{displaymath}
(Recall that we use the standing convention that integrals are always taken with
respect to the completion of the measure. The set
(\ref{eq:set:of:bad:pairs}) is universally measurable therefore
measurable with respect to the completion of $\nu\otimes\mathbb
P_{w(\bbi)}$, which allows us to use Fubini in the calculation below).
Therefore, by Fubini and Proposition~\ref{prop:Good-and-Bad-extensions},
\begin{align*}
&\int_{\Omega(w(\bbi))}
\int_{B_\bbi(\tilde w)}
\theta_\bbi(y)\,d\nu(y)\,
d\mathbb P_{w(\bbi)}(\tilde w)
\\
&\qquad\le
\int
\theta_\bbi(y)
\mathbb P_{w(\bbi)}\bigl(B(w(\bbi),y)\bigr)\,
d\nu(y)
\\
&\qquad\le
e^{-\beta \sqrt{n(\bbi)}}
\int\theta_\bbi\,d\nu .
\end{align*}
\end{proof}

\begin{proposition}[Ledrappier-Young conditional measure formula]
\label{prop:Ledrappier:Young:conditional:measure}
Assume the setup of \S\ref{ssec:assumptions_on_the_total_space}. Let
$\nu$ be a measure invariant under every map in the support of $\mu$
such that for almost all $(x,\tilde{w})$, the
conditional measure of $\nu$ along $\cW_{d-k}(x,\tilde{w})$ is in the
Lebesgue measure class. Then, for almost all $(x,\tilde{w})$, 
the conditional measure of $\nu$ on $\cW_{d-k}(x,\tilde{w})$ is, up to
normalization,
\begin{equation}
\label{eq:Ledrappier:Young:density}
J_{\tilde w}(x,y)\,
d\Leb_{\mathcal W_{d-k}(x,\tilde w)}(y).
\end{equation}
\end{proposition}

\begin{proof}

For a single diffeomorphism of a compact manifold, the conditional-density
formula along full unstable manifolds is given in
\cite[Corollaries~6.1.4 and~6.2, pp.~534--535]{LedrappierYoung1985I};
the stable version is obtained by applying it to the inverse map.
The random adaptation is discussed in
\cite[\S6.1, pp.~236--237]{LedrappierYoung1988}, with a detailed
treatment of the random {SRB} characterization in
\cite[Chapter~VI, pp.~128--181]{LiuQian1995} and a treatment of
stationary random systems in \cite{BahnmullerLiu1998}.
For the broader random entropy--dimension formula, see
\cite{QianQianXie2003}.
These references concern compact manifolds; the formulation below in our
possibly non-compact setting, including the selected partial stable
manifolds, requires the corresponding adaptation.
\end{proof}

\begin{proposition}[Good plaques give Fourier decay]
\label{prop:good:plaques:fourier:decay}
For every $b\in\mathbb N$, there exists $C_b<\infty$ such that, for every
non-root vertex $\bbi\in\mathcal T$ and $\bP(w(\bbi))$-almost every
$\tilde w\in\Omega(w(\bbi))$,
\begin{displaymath}
\left|
\int_{G_\bbi(\tilde w)}
\theta_\bbi(y)
e^{\sqrt{-1}\langle\xi,y\rangle}\,d\nu(y)
\right|
\le 
\frac{C_b}{|\xi|^b}
\sum_{j=0}^b
\int_{G_\bbi(\tilde w)}
\left\|
D_{\mathcal W}^j\theta_\bbi(y)
\right\|
\,d\nu(y).
\end{displaymath}
\end{proposition}

\begin{proof}
By definition, $G_\bbi(\tilde w)$ is partitioned by its intersection
with local stable plaques
$\mathcal W_{d-k}(x,\tilde w,r)$ with
\begin{displaymath}
x\in\supp\theta_\bbi,
\qquad
\tilde w\in G(w(\bbi),x).
\end{displaymath}
This is a measurable partition. 
Disintegrate the restriction of $\nu$ to $G_\bbi(\tilde w)$ along these
plaques. By
Proposition~\ref{prop:Ledrappier:Young:conditional:measure}, 
the conditional measure on such a plaque is, up to the normalization of the
conditional measure, given by (\ref{eq:Ledrappier:Young:density}). 
Thus the contribution of each plaque is an integral of the form
\begin{displaymath}
\int_{\mathcal W_{d-k}(x,\tilde w,r)}
\theta_\bbi(y)
J_{\tilde w}(x,y)
e^{\sqrt{-1}\langle\xi,y\rangle}\,
d\Leb_{\mathcal W_{d-k}(x,\tilde w)}(y).
\end{displaymath}
In view of \eqref{eq:step1:plaque:localization}, 
on each $\cW_{d-k}(x,\tilde{w},r)$ the leafwise amplitude is smooth
and compactly supported in its interior, since
$\supp\rho\Subset B(x_0,r)$ and $\theta_\bbi$ contains the factor
$\rho$. Thus integration by parts introduces no boundary term.

Since the plaque is good by construction,
Proposition~\ref{prop:leafwise_Fourier_estimate} applies to this inner
integral. Hence the absolute value of the inner integral is bounded by
\begin{displaymath}
\frac{C_b}{|\xi|^b}
\sum_{j=0}^b
\int_{\mathcal W_{d-k}(x,\tilde w)}
\left\|
D_{\mathcal W}^j\theta_\bbi(y)
\right\|
J_{\tilde w}(x,y)\,
d\Leb_{\mathcal W_{d-k}(x,\tilde w)}(y).
\end{displaymath}
Integrating this estimate over the quotient space of good plaques, and
using again the conditional-measure formula, gives
\begin{displaymath}
\left|
\int_{G_\bbi(\tilde w)}
\theta_\bbi(y)
e^{\sqrt{-1}\langle\xi,y\rangle}\,d\nu(y)
\right|
\le \frac{C_b}{|\xi|^b}
\sum_{j=0}^b
\int_{G_\bbi(\tilde w)}
\left\|
D_{\mathcal W}^j\theta_\bbi(y)
\right\|
\,d\nu(y).
\end{displaymath}
This proves the proposition.
\end{proof}

\begin{corollary}[Good--bad Fourier estimate]
\label{cor:good:bad:fourier:estimate}
For every $b\in\mathbb N$, there exists $C_b<\infty$ such that, for every
non-root vertex $\bbi\in\mathcal T$,
\begin{multline*}
\left|
\int
\theta_\bbi(y)
e^{\sqrt{-1}\langle\xi,y\rangle}\,d\nu(y)
\right|
\le
e^{-\beta \sqrt{n(\bbi)}}
\int\theta_\bbi\,d\nu
\\
+
\frac{C_b}{|\xi|^b}
\sum_{j=0}^b
\int_{\Omega(w(\bbi))}
\int_{G_\bbi(\tilde w)}
\left\|
D_{\mathcal W}^j\theta_\bbi(y)
\right\|
\,d\nu(y)\,
d\mathbb P_{w(\bbi)}(\tilde w).
\end{multline*}
\end{corollary}

\begin{proof}
For each fixed $\tilde w\in\Omega(w(\bbi))$, decompose
\begin{displaymath}
\supp\theta_\bbi
=
G_\bbi(\tilde w)\sqcup B_\bbi(\tilde w).
\end{displaymath}
Since $\nu$ is stationary under the maps in the support of $\mu$, and
$\mathbb P_{w(\bbi)}$ is the Bernoulli measure on extensions of the fixed
prefix $w(\bbi)$, averaging over $\tilde w$ does not change the integral:
\begin{displaymath}
\int
\theta_\bbi(y)
e^{\sqrt{-1}\langle\xi,y\rangle}\,d\nu(y)
=
\int_{\Omega(w(\bbi))}
\int
\theta_\bbi(y)
e^{\sqrt{-1}\langle\xi,y\rangle}\,d\nu(y)\,
d\mathbb P_{w(\bbi)}(\tilde w).
\end{displaymath}
For each $\tilde w$, split the inner integral into the contributions from
$G_\bbi(\tilde w)$ and $B_\bbi(\tilde w)$. The contribution from
$B_\bbi(\tilde w)$ is bounded in absolute value by
$
\int_{B_\bbi(\tilde w)}\theta_\bbi\,d\nu.
$
Averaging over $\tilde w$ and applying
Proposition~\ref{prop:bad:points:small:average} gives the first term on the
right-hand side.

The contribution from $G_\bbi(\tilde w)$ is bounded by
Proposition~\ref{prop:good:plaques:fourier:decay}. Averaging that estimate
over $\Omega(w(\bbi))$ gives the second term. This proves the corollary.
\end{proof}

\subsubsection{Partitions of Unity}
\label{ssec:Partitions-of-Unity}

We next estimate the derivative terms which occur in
Corollary~\ref{cor:good:bad:fourier:estimate} after summing over all
vertices at a fixed level.

\begin{proposition}[Boundedness of the tree partition along good plaques]
\label{prop:Boundedness-of-rho-v-along-stables}
For every $b\in\mathbb Z_{\geq 0}$, there exists $C_b<\infty$ such that,
for every $n\geq n_\ast+1$,
\begin{displaymath}
\sum_{\bbi\in\mathcal T_n}
\int_{\Omega(w(\bbi))}
\int_{G_\bbi(\tilde w)}
\left\|
D_{\mathcal W(y,\tilde{w})}^b\theta_\bbi(y)
\right\|
\,d\nu(y)\,
d\mathbb P_{w(\bbi)}(\tilde w)  \le
C_b(1+n)^b
\int \rho\,d\nu .
\end{displaymath}
Here $D_{\cW(y,\tilde{w})}^b$ denotes the $b$-th derivative
along the stable manifold $\cW(y,\tilde{w})$. 
\end{proposition}

\begin{proof}
Define
\[
E_\bbi(y):=
\bigcup_{\substack{\mathbf{k}\succeq\bbi\\ n(\mathbf{k})\ge n_\ast}}
\{\tilde w\in\Omega(w(\mathbf{k})):y\in G_{\mathbf{k}}(\tilde w)\}.
\]
Here the union is over non-root descendants $\mathbf{k}$ of $\bbi$,
including $\mathbf{k}=\bbi$ when eligible.
Note that for a child $\bbj$ of $\bbi$, $E_\bbj(y)\subseteq E_\bbi(y)$.

For every $n,b\ge0$, set
\[
S_{n,b}:=\sum_{\bbi\in\mathcal T_n}\int_M
\sup_{\tilde w\in E_\bbi(y)}
\left\|D_{\mathcal W(y,\tilde w)}^b\theta_\bbi(y)\right\|\,d\nu(y),
\]
with the convention $\sup\varnothing=0$. 

To justify the $\nu$-measurability of the $\sup$ in the above formula,
note that the supremum above is a universally measurable function, hence
$\nu$-measurable under our completion convention, since its strict
superlevel sets are countable unions of projections of Borel relations
describing good plaques and their leafwise derivatives.
From the definitions,
\begin{displaymath}
\sum_{\bbi\in\mathcal T_n}
\int_{\Omega(w(\bbi))}
\int_{G_\bbi(\tilde w)}
\left\|
D_{\mathcal W(y,\tilde{w})}^b\theta_\bbi(y)
\right\|
\,d\nu(y)\,
d\mathbb P_{w(\bbi)}(\tilde w) \le S_{n,b}. 
\end{displaymath}
It is thus enough to show that for all $n \ge 0$ and all $b
\ge 0$, 
\begin{equation}
\label{eq:Snb:main:estimate}
S_{n,b} \le  C_b(1+n)^b \int \rho\,d\nu.
\end{equation}
We prove the estimate by induction on $b$. The case $b=0$ follows from
$
G_\bbi(\tilde w)\subset\supp\theta_\bbi
$
and from the partition identity
$
\sum_{\bbi\in\mathcal T_n}\theta_\bbi=\rho.
$
Indeed,
\[
S_{n,0}\le\sum_{\bbi\in\mathcal T_n}\int\theta_\bbi\,d\nu
=\int\rho\,d\nu.
\]
(For each vertex, the supremum of the zeroth derivative is
$\theta_\bbi(y)$ when $E_\bbi(y)$ is nonempty, and $0$ otherwise.)

Now fix $b\ge1$ and assume (\ref{eq:Snb:main:estimate}) holds for $S_{n,r}$
for all $n$ and all $r < b$. 
Let $\bbj\in\mathcal T_n$, and write
$
\bbj=(\bbi,i_n),
\bbi\in\mathcal T_{n-1}.
$
Then we have
$
\theta_\bbj=\theta_\bbi\,\phi_\bbj.
$

Fix $y$ and
$\tilde w\in E_\bbj(y)$. Then, there
exists a descendant $\mathbf{k}$ of $\bbj$ ($\mathbf{k}=\bbj$ is allowed) such that $y \in
G_\mathbf{k}(\tilde{w})$. Then, by
Proposition~\ref{prop:Boundedness-along-stables}, for $0 \le r < b$, 
\begin{displaymath}
\|D_{\cW(y,\tilde{w})}^{b-r}\phi_\bbj(y)\| \le C_{b-r} \chi_\bbj(y)
\end{displaymath}
where $\chi_\bbj(y) = 1$ if $y \in \supp \phi_\bbj$ and $\chi_\bbj(y) = 0$ if
$y \not\in \supp \phi_\bbj$. 
On $\mathcal W(y,\tilde w)$, the product
rule gives
\begin{align*}
\Big\|
D_{\cW(y,\tilde{w})}^b & \theta_\bbj(y)
\Big\| \le \\
& \le 
\phi_\bbj(y)\|D_{\cW(y,\tilde{w})}^b\theta_\bbi(y)\| + C_b
\sum_{r=0}^{b-1}
\left\|
D_{\cW(y,\tilde{w})}^r\theta_\bbi(y)
\right\|
\left\|
D_{\cW(y,\tilde{w})}^{b-r}\phi_\bbj(y)
  \right\|  \\
& \le  \phi_\bbj(y)\|D_{\cW(y,\tilde{w})}^b\theta_\bbi(y)\| + C_b'
\sum_{r=0}^{b-1}
\left\|
D_{\cW(y,\tilde{w})}^r\theta_\bbi(y)
\right\|
\chi_\bbj(y).
\end{align*}
Taking the supremum over $\tilde{w} \in E_\bbj(y)$ and noting that $E_\bbj(y)
\subset E_\bbi(y)$, we get
\begin{multline*}
\sup_{\tilde{w} \in E_\bbj(y)} \left\|
D_{\cW(y,\tilde{w})}^b\theta_\bbj(y)
\right\|
\le
 \phi_\bbj(y) \sup_{\tilde{w} \in E_\bbi(y)}
 \|D_{\cW(y,\tilde{w})}^b\theta_\bbi(y)\| \\ + C_b' \chi_\bbj(y)
\sum_{r=0}^{b-1} 
\sup_{\tilde{w} \in E_\bbi(y)}  \left\|
D_{\cW(y,\tilde{w})}^r\theta_\bbi(y)
\right\|.
\end{multline*}
We now integrate both sides with respect to $\nu$ and sum over $\bbj
\in \cT_n$. Note that
\[
\sum_{\bbj:\,\bbj\text{ child of }\bbi}\phi_\bbj=1.
\]
Also, for each point $y$ and each
parent vertex $\bbi$, at most $m_I$ children $\bbj$ have
$
\chi_\bbj(y) \neq0,
$
where $m_I$ is the maximal multiplicity of the cover
$\{\supp\psi_i\}_{i\in I}$.
We get 
\[
S_{n,b}
\le
S_{n-1,b}
+
C_b\sum_{r=0}^{b-1}S_{n-1,r}.
\]
By the induction hypothesis in the derivative order,
\[
S_{n-1,r}
\le
C_r n^r\int\rho\,d\nu
\qquad
\text{for } r<b.
\]
Hence
\[
S_{n,b}
\le
S_{n-1,b}
+
C_b n^{b-1}\int\rho\,d\nu .
\]
Iterating in $n$ gives
\[
S_{n,b}
\le
S_{0,b}
+
C_b\sum_{q=1}^n q^{b-1}\int\rho\,d\nu
\le
C_b(1+n)^b\int\rho\,d\nu .
\]
The initial term $S_{0,b}$ is bounded because $\rho$ is fixed and the
plaque geometry is uniformly bounded: in the fixed ambient
coordinates, for $b\ge1$,
\[
S_{0,b}\le C_b\int_M\sum_{q=1}^b\|D^q\rho(y)\|\,d\nu(y)<C_b' < C_b''
\int  \rho \, d\nu
\]
This proves the proposition.
\end{proof}

\subsubsection{Proof of the main estimate}
\label{ssec:proof:main:estimate}

We now prove \eqref{eq:main_estimate}. Fix $b\in\mathbb N$. Let $n\ge1$ be
chosen later.
Since
$\sum_{\bbi\in\mathcal T_n}\theta_\bbi=\rho$,
we have
\begin{displaymath}
\int
\rho(y)e^{\sqrt{-1}\langle\xi,y\rangle}\,d\nu(y)
=
\sum_{\bbi\in\mathcal T_n}
\int
\theta_\bbi(y)e^{\sqrt{-1}\langle\xi,y\rangle}\,d\nu(y).
\end{displaymath}
Applying Corollary~\ref{cor:good:bad:fourier:estimate} to each vertex
$\bbi\in\mathcal T_n$ and summing gives
\begin{multline*}
\left|
\int
\rho(y)e^{\sqrt{-1}\langle\xi,y\rangle}\,d\nu(y)
\right|
\le
e^{-\beta \sqrt{n-1}}
\sum_{\bbi\in\mathcal T_n}
\int\theta_\bbi\,d\nu
\\
+
\frac{C_b}{|\xi|^b}
\sum_{j=0}^b
\sum_{\bbi\in\mathcal T_n}
\int_{\Omega(w(\bbi))}
\int_{G_\bbi(\tilde w)}
\left\|
D_{\cW(y,\tilde{w})}^j\theta_\bbi(y)
\right\|
\,d\nu(y)\,
d\mathbb P_{w(\bbi)}(\tilde w).
\end{multline*}
Using again
$\sum_{\bbi\in\mathcal T_n}\theta_\bbi=\rho$
for the first term, and applying
Proposition~\ref{prop:Boundedness-of-rho-v-along-stables} to the second
term, we get
\begin{displaymath}
\left|
\int
\rho(y)e^{\sqrt{-1}\langle\xi,y\rangle}\,d\nu(y)
\right|
\le
e^{-\beta \sqrt{n-1}}\int\rho\,d\nu
+
\frac{C_b(1+n)^b}{|\xi|^b}\int\rho\,d\nu .
\end{displaymath}
Absorbing $\int\rho\,d\nu$ into the constant, this becomes
\begin{equation}
\label{eq:main:estimate:before:optimization}
\left|
\int
\rho(y)e^{\sqrt{-1}\langle\xi,y\rangle}\,d\nu(y)
\right|
\le
C_b
\left(
e^{-\beta \sqrt{n-1}}
+
\frac{(1+n)^b}{|\xi|^b}
\right).
\end{equation}

We now choose $n$ as a function of $|\xi|$. If $|\xi|\le2$, the estimate is
immediate after increasing $C_b$. Suppose $|\xi|>2$, and choose
\begin{displaymath}
n=\max\left(n_\ast+1,1+\left\lceil
    \frac{2b}{\beta}\log|\xi|\right\rceil^2\right) . 
\end{displaymath}
Then we have
\begin{displaymath}
e^{-\beta \sqrt{n-1}}\le |\xi|^{-2b}
\text{ while }
(1+n)^b\le C_b(\log|\xi|)^{2b}.
\end{displaymath}
Thus \eqref{eq:main:estimate:before:optimization} gives
\begin{displaymath}
\left|
\int
\rho(y)e^{\sqrt{-1}\langle\xi,y\rangle}\,d\nu(y)
\right|
\le
C_b
\frac{(\log|\xi|)^{2b}}{|\xi|^b}.
\end{displaymath}
Since $b$ was arbitrary, this implies the stated rapid polynomial decay.
Indeed, applying the preceding estimate with $2b$ in place of $b$ gives
\begin{displaymath}
\left|
\int
\rho(y)e^{\sqrt{-1}\langle\xi,y\rangle}\,d\nu(y)
\right|
\le
C_b
\frac{1}{|\xi|^b}
\end{displaymath}
for $|\xi|>2$, after increasing $C_b$. Together with the trivial bound for
bounded $|\xi|$, this proves
\begin{displaymath}
\left|
\int
e^{\sqrt{-1}\langle\xi,y\rangle}\,d\nu_\rho(y)
\right|
\le
\frac{C_b}{1+|\xi|^b}.
\end{displaymath}
This is \eqref{eq:main_estimate}.



\appendix
\section{Some lemmas from linear algebra}
	\label{appendix:some_lemmas_linear_algebra}

\subsubsection{Angles and projections}
	\label{sssec:Angles-and-projections}

Suppose that $U,S\subset V$ are subspaces of an Euclidean vector space $V$, of complementary dimensions and giving a decomposition $V=U\oplus S$.
We define:
\[
	\angle(U,S):=\arccos\left(\frac{\max_{u\in U, s\in S} |\langle u,s\rangle|}{\|u\|\|s\|}\right) \in [0,\pi/2].
\]
Then, the norm of the projection from $V$ onto $U$ along $S$ is bounded by
\begin{align}
	\label{eqn:projection-norm-bound}
	\norm{\pi_U^S} \le \frac{1}{\sin(\angle(U,S))}\leq \frac{2}{\angle(U,S)}
\end{align}
Indeed, for arbitrary $v = u+s$ with $u\in U,s\in S$, we will show that $\sin(\angle(U,S))\|u\| \leq \|u+s\|$, and use that $\sin(\theta)\in[\tfrac 2\pi \theta, \theta]$ for $\theta \in [0,\pi/2]$.
But we have (with $\theta:=\angle(U,S)$):
\begin{align*}
	\norm{u+s}^2
	&= \norm{u}^2+\norm{s}^2+2\langle u,s\rangle \\
	&\geq \norm{u}^2+\norm{s}^2-2\norm{u}\norm{s}\cos\theta \\
	&= \norm{u}^2\sin^2\theta +(\norm{s}-\norm{u}\cos\theta)^2 \\
	&\geq \norm{u}^2\sin^2\theta.
\end{align*}
To a linear isomorphism $A\colon V\to W$ between Euclidean spaces of the same dimension, we associate the $k$-gaps:
\begin{align}
	\label{eqn:Delta_k_subspaces}
	\Delta_k^{U,S}(A) := \frac{\frakm(A\vert_U)}{\norm{A\vert_S}} \leq \Delta_k(A) := \frac{\sigma_k(A)}{\sigma_{k+1}(A)}
\end{align}
where $\dim U=k,\dim S = d-k$.

For two subspaces $S,K\in \Gr_{d-k}(V)$ we use the angular gap distance
\begin{align}
	\label{eqn:grassmannian-distance-def}
	\dist_{\Gr}(S,K):=
	\max\left\{
	\max_{\substack{s\in S}}
	\min_{\substack{k\in K}}\angle(s,k),
	\max_{\substack{k\in K}}
	\min_{\substack{s\in S}}\angle(k,s)
	\right\}.
\end{align}
Since $S$ and $K$ have the same dimension, the two one-sided quantities are
equal; this is the largest principal angle between $S$ and $K$. Equivalently,
if $P_S$ and $P_K$ are the orthogonal projections onto $S$ and $K$, then
\[
\sin\bigl(\dist_{\Gr}(S,K)\bigr)=\|P_S-P_K\|.
\]
In particular, for $v\in K$ we have:
\begin{align}
	\label{eqn:bound-norm-projection}
	\norm{v - P_S v} \le \sin\bigl(\dist_{\Gr}(K,S)\bigr) \norm{v}.
\end{align}
and there exists a unit vector $v\in K$ such that equality holds.

\begin{lemma}[Slow singular space from a dominated splitting]
	\label{lemma:slow:space:from:dominated:splitting}
Let $A\colon V\to W$ be a linear isomorphism between $d$-dimensional Euclidean
spaces, and suppose we have a decomposition
$ V=U\oplus S$ with 
$\dim U=k, \dim S=d-k$.

Then we have:
\begin{equation}
\label{eq:slow:space:from:dominated:splitting}
	\dist_{\Gr}\bigl(\ker_{d-k}(A),S\bigr)
	\leq
	\frac{8}{\Delta_k^{U,S}(A)}
\end{equation}
More precisely, if $\frakm(A|_U)>\|A|_S\|$, then
\begin{equation}
\label{eq:slow:space:projection:gap}
	\left\|
	P_{\ker_{d-k}(A)}-P_S
	\right\|
	\leq
	2 \min\left\{1,\frac{1}{\Delta_k^{U,S}(A)-1}\right\}.
\end{equation}
Moreover,
\begin{displaymath}
	\|A|_S\|\le 2  \frac{\sigma_{k+1}(A)}{\angle(AS,\im_k(A))}.
\end{displaymath}
\end{lemma}

\begin{proof}

$K:=\ker_{d-k}(A)$
and note that since $K$ is a bottom $(d-k)$-dimensional singular subspace, and $\dim S=d-k$, we have
\begin{displaymath}
\|A|_K\|=\sigma_{k+1}(A) \leq \norm{A\vert_S}.
\end{displaymath}
Let $v\in K$ be any unit vector, decompose it as
$
v=u+s,
u\in U, s\in S$.
Then on the one hand, we have
$\norm{Au}\geq \frakm(A\vert_U)\norm{u}$
and on the other we have:
\begin{align*}
	\norm{Au} & \leq \norm{Av} + \norm{As} \leq \norm{A\vert_K} + \norm{A\vert_S}\norm{s} \\
	& \leq \norm{A\vert_S}(1+\norm{s}) \leq \norm{A\vert_S}(2 + \norm{u}).
\end{align*}
where we used in the last line $\norm{u}\leq \norm{v}+\norm{s} = 1+\norm{s}$.
Combining the upper and lower bounds on $\norm{Au}$ gives
\[
	\frakm(A\vert_U)\norm{u} \leq \norm{A\vert_S}(2 + \norm{u})
\]
Rearranging, and assuming $\frakm(A\vert_U)>\norm{A\vert_S}$ we find
\[
	\norm{u}\leq \frac{2\norm{A\vert_S}}{\frakm(A\vert_U) - \norm{A\vert_S}} = \frac{2}{\frac{\frakm(A\vert_U)}{\norm{A\vert_S}}-1}.
\]
We thus obtain the bound, for any unit vector $v\in K$,
\[
	\norm{(P_K-P_S)(v)} \leq 2 \min\left\{1,\frac{1}{\frac{\frakm(A\vert_U)}{\norm{A\vert_S}}-1}\right\}
\]
and thus \autoref{eq:slow:space:projection:gap} follows, and hence \autoref{eq:slow:space:from:dominated:splitting} by \autoref{eqn:bound-norm-projection}.

For the ``moreover'' part, we will show that for any $v\in S$ we have
\[
	\norm{Av} \leq \frac{\norm{A{\vert}_K}}{\sin\angle\left(AS,\im_k(A)\right)}\norm{v}
\]
which implies the desired bound.
For this, we will use the orthogonal splittings $V=K\oplus A^{-1}F$ and $W=AK\oplus F$ where $F:=\im_k(A)$ and the splittings are matched by $A$ since they are defined by the singular value decomposition.

Write an arbitrary $v\in S$ as $v = k + A^{-1}f$ where $k\in K$ and $f\in F$.
Rewrite $Av = Ak + f$ as $Ak = Av - f\in AS\oplus F$ and observe that $Av = \pi^{F}_{AS}(Ak)$.
Therefore we have:
\[
	\norm{Av} \leq \norm{\pi^{F}_{AS}}\norm{Ak}
	\leq \frac{\norm{A{\vert}_K}}{\sin\angle\left(AS,\im_k(A)\right)}\norm{k}
\]
and we are done because $\norm{k}\leq \norm{v}$ since the decomposition $V=K\oplus A^{-1}F$ is orthogonal.
\end{proof}

\begin{lemma}[Inverse image of a plane transverse to the fast image]
\label{lemma:inverse:estimate:transverse:to:fast:image}
Suppose $A\colon V\to W$ is a linear isomorphism with a strict $k$-gap,
so that $
\im_k(A)\subset W,
\ker_{d-k}(A)\subset V$
are well-defined.

Then for any $S'\in\Gr_{d-k}(W)$ with $S:=A^{-1}S'$ we have
\begin{displaymath}
\dist_{\Gr}(S,\ker_{d-k}(A))
\le
\frac{8}{\Delta_k(A)\cdot\angle(S',\im_k(A))}
\end{displaymath}
and
\begin{displaymath}
\|A|_S\|\le \frac{2\sigma_{k+1}(A)}{\angle(S',\im_k(A))}.
\end{displaymath}
\end{lemma}

\begin{proof}
We apply \autoref{lemma:slow:space:from:dominated:splitting} with $U
= A^{-1} \im_k(A) = \ker_{d-k}(A)^\perp$,
using the estimate:
\[
	\Delta_k^{U,S}(A) 
	=
	\frac{\sigma_k(A)}{\|A|_S\|} \geq 
	\frac{\sigma_k(A)}{2\sigma_{k+1}(A)/\angle(S',\im_k(A))} 
	=
	\tfrac 12\Delta_k(A)\cdot\angle(S',\im_k(A))
\]
where in the inequality we used the ``moreover'' part of \autoref{lemma:slow:space:from:dominated:splitting}.
\end{proof}

\begin{lemma}[Image of a transverse plane]
\label{lemma:image:of:transverse:plane}
Suppose $B\colon V\to W$ is a linear isomorphism with a strict $k$-gap,
so that $
\im_k(B)\subset W,
\ker_{d-k}(B)\subset V$
are well-defined.

The for any $E\in\Gr_k(V)$ we have
\begin{displaymath}
\dist_{\Gr}(BE,\im_k(B))
\le
\frac{8}{\Delta_k(B)\cdot \angle(E,\ker_{d-k}(B))}
\end{displaymath}
and
\begin{displaymath}
	\mathfrak m(B|_E)\ge \tfrac 12 \angle(E,\ker_{d-k}(B))\,\sigma_k(B).
\end{displaymath}
\end{lemma}

\begin{proof}
This is \autoref{lemma:inverse:estimate:transverse:to:fast:image}
applied to $A^{-1}:=B$, $k$ in place of $d-k$ and
$E=S'$, noting that $\ker_{d-k}(B)=  \im_{d-k}(A)$, and $\sigma_k(B)^{-1}=\sigma_{d-k+1}(A)$.
\end{proof}

\begin{lemma}[Local Lipschitz estimate for inverse images]
	\label{lemma:inverse:local:lipschitz}
	Let $B\colon V\to W$ be a linear isomorphism with a strict $k$-gap, and set
	$F:=\im_k(B)\subset W$.
	Assume $H\in\Gr_{d-k}(W)$ satisfies
	\begin{displaymath}
		\angle(H,F)=\theta>0.
	\end{displaymath}
	For any	$S_1,S_2\in\Gr_{d-k}(W)$ with $\dist_{\Gr}(S_i,H)\le \tfrac 12\theta, i=1,2$
	we have:
	\begin{equation}
		\label{eq:inverse:local:lipschitz}
		\dist_{\Gr}\left(B^{-1}S_1,B^{-1}S_2\right)
		\le
		\frac{16}{\Delta_k(B)\theta^2}
		\dist_{\Gr}(S_1,S_2).
	\end{equation}
\end{lemma}

\begin{proof}
	For every $S\in \Gr_{d-k}(W)$ with $\angle(S,F)>0$ there exists a linear map $u_S\colon H\to F$ such that $S$ is its graph: $S=\{h + u_S(h)\colon h\in H \}$.
	For any two such $S_1,S_2$ we first establish the bound
	\begin{equation}
		\label{eq:local:graph:difference}
		\norm{u_{S_1} - u_{S_2}} \leq 8 \frac{\dist_{\Gr}(S_1,S_2)}{\theta}.
	\end{equation}
			Set $v_i=v_i(h):=h + u_{S_i}(h)$ and our goal is to estimate from above $\norm{(u_{S_1} - u_{S_2})(h)}=\norm{v_1-v_2}$, set therefore $f:=v_1-v_2$.
	For a vector $v$ and plane $E$, we denote by $\dist(v,E)$ the distance from $v$ to $E$, i.e. the norm of the orthogonal projection of $v$ to $E^\perp$, and by $\angle(v,E)$ the angle between $v$ and $E$, i.e. the angle between $v$ and its orthogonal projection to $E$ (valued in $[0,\pi/2]$).
	We then have:
	\[
		\norm{f}\leq \frac{\dist(f,S_2)}{\sin\angle(f,S_2)}
		\leq
		\frac{\dist(f,S_2)}{\sin(\theta/2)}
	\]
	Now $\dist(f,S_2)=\dist(v_1,S_2)$ since $v_1 - f = v_2 \in S_2$.
	Next, we note that
	\[
		\dist(v_1,S_2) \leq \sin\dist_{\Gr}(S_1,S_2)\norm{v_1}\leq \dist_{\Gr}(S_1,S_2)\norm{v_1}
	\]
	So we arrive at
	\[
		\norm{u_{S_1}(h)-u_{S_2}(h)} \leq \frac{\dist_{\Gr}(S_1,S_2)}{\sin(\theta/2)}\norm{v_1}
	\]
	and proceed to bound $\norm{v_1}$ in terms of $\norm{h}$.
	For this, since $v_1 = h + u_{S_1}(h)$, we have that $\dist(v_1,H)=\dist(u_{S_1}(h),H)$, and we have
	\[
		\sin\theta\norm{u_{S_1}(h)}\leq \dist(u_{S_1}(h),H)=\dist(v_1,H)\leq \sin(\theta/2)\norm{v_1}
	\]
		In particular $\norm{u_{S_1}(h)}\leq \tfrac{\sin(\theta/2)}{\sin (\theta)} \norm{v_1}$.
		Note that $\sin(\theta/2)/\sin(\theta) = (2\cos(\theta/2))^{-1}\leq 1/\sqrt{2}$ since $\theta\in(0,\pi/2]$.
		Thus we have
		\[
			\norm{v_1} \leq \norm{h} + \tfrac{1}{\sqrt2} \norm{v_1}
			\text{and so} \norm{v_1}\leq (2+\sqrt2)\norm{h}.
		\]
		Together with the bound $\sin(\theta/2)\geq \frac{\sqrt2}{\pi}\theta$ this completes the proof of \autoref{eq:local:graph:difference}.
	
	Set now $A:=B^{-1}$.
	We will estimate from above the distance between $AS_1$ and $AS_2$ in the Grassmannian using the definition in \autoref{eqn:grassmannian-distance-def}, and specifically for every nonzero $s_1\in S_1$ we will use a corresponding $s_2\in S_2$ obtained via the isomorphism of each $S_i$ to $H$.
	Concretely, using $v_i(h):=h + u_{S_i}(h)\in S_i$, we have that
	\[
		\dist_{\Gr}(AS_1,AS_2) 
		\leq
		\sup_{h\in H\setminus 0} \angle( A v_1(h), A v_2(h) )
	\]
	Fix one such $h\in H\setminus 0$ and write $v_i:=h + u_{S_i}(h)\in S_i$.
	Quite generally, we have the elementary bound 
	\[
		\angle(x,y) \leq \frac{\pi}{2 \sqrt{2}} \frac{\|x-y\|}{\min\{\|x\|,\|y\|\}}.
	\]
	We will apply this to the vectors $Av_i$.
		First, we note that $v_1-v_2=u_{S_1}(h)-u_{S_2}(h)\in F$ and therefore
	\begin{align*}
		\norm{A(v_1-v_2)} 
		& \leq \sigma_k(B)^{-1}\norm{u_{S_1}(h)-u_{S_2}(h)}\\
			& \leq 8\,\sigma_k(B)^{-1}\frac{\dist_{\Gr}(S_1,S_2)}{\theta}\norm{h}
	\end{align*}
	where the first line used that $F=\im_k(B)$ so $\norm{B^{-1}\vert_{F}}\leq \sigma_k^{-1}(B)$, and the second line used the bound \eqref{eq:local:graph:difference}.

	Next, denote by $P_{F^\perp}$ the orthogonal projection to $F^\perp$.
	Note that $A$ sends $F$ and $F^\perp$ to orthogonal singular subspaces in $V$, and since by assumption $H$ is $\theta$-transverse to $F$, we have that $\norm{P_{F^\perp}h}\geq \sin{\theta}\norm{h}$ for any $h\in H$.
	Using that $u_S(h)\in F$ we estimate:
	\begin{align*}
		\norm{A v_i}
			& \geq \norm{A(P_{F^\perp} v_i)}
			= \norm{A(P_{F^\perp} h)} 
		\geq \sigma_{k+1}(B)^{-1}\norm{P_{F^\perp}h}\\
		& \geq \sigma_{k+1}(B)^{-1}\sin{\theta}\norm{h}
	\end{align*}
	Combining the above estimates, we obtain
		\[
			\dist_{\Gr}(AS_1,AS_2)
				\leq \frac{\pi}{2 \sqrt{2}}\frac{8}{\Delta_k(B)\sin{\theta}}\frac{1}{\theta}\dist_{\Gr}(S_1,S_2)
			\le
			\frac{\sqrt{2}\pi^2}{\Delta_k(B)\theta^2}\dist_{\Gr}(S_1,S_2)
		\]
	which finishes the proof.

\end{proof}

\subsubsection{The quotient map  $\bar{B}_E$.}
Suppose $B: V \to W$ is a linear  isomorphism between Euclidean
spaces, and suppose $E \subset
V$ is a linear subspace with  $\dim E = k$. Let $\bar{B}_E$ denote
the induced  map $V/E \to W/BE$. We equip both $V/E$ and $W/BE$ with
the quotient norm. 

\begin{lemma}
\label{lemma:bar:B:bound}
We have
\begin{displaymath}
\|\bar{B}_E\|  \ge \sigma_{k+1}(B). 
\end{displaymath}
\end{lemma}

\begin{proof}
Let $d:=\dim V=\dim W$.  Suppose $k <d-1$.
Without loss of generality, we may assume
$\Delta_{k+1}(B) > 1$. (If not  we may approximate $B$ by linear maps
for which $\Delta_{k+1} > 1$). Then, let $F=\im_{k+1}(B) \subset W$.
If $k=d-1$, we let $F = W$.  In both cases, we proceed as follows:

We have $\dim F = k+1$, $\dim (BE)^\perp = d-k$. Thus $F  \cap (BE)^\perp \ne
\{0\}$. Choose $w \in F \cap (BE)^\perp$ with $\|w\|=1$.  Let $v =
B^{-1}  w$. Then $v \in \ker_{d-k-1}(B)^\perp$. Hence,
\begin{displaymath}
1 = \|w \|  = \|B v \| \ge \sigma_{k+1}(B) \|v\|.  
\end{displaymath}
Let $v+E$ denote the coset of $v$ in $V/E$, and let $w+BE$ denote the
coset of $w$ in $W/BE$. Then, 
\begin{displaymath}
\|v+E\| \le \|v\| \le \sigma_{k+1}(B)^{-1}, \text{ and } \|w+BE\| = \|w\|=1.
\end{displaymath}
Since $\bar{B}_E(v+E) = w+BE$, it follows that $\|\bar{B}_E\| \ge
\sigma_{k+1}(B)$. 
\end{proof}

\subsubsection{The relative gap $\Delta_k(B,E)$.}
\label{sssec:relative:gap}
Suppose $B: V \to W$ is a linear  isomorphism, and suppose $E \subset
V$ is a linear subspace with  $\dim E = k$. Let $\bar{B}_E$ denote
the induced  map $V/E \to W/BE$.  Let
\begin{displaymath}
\Delta_k(B,E):=\frac{\mathfrak m(B|_E)}{\|\bar B_E\|}, 
\end{displaymath}
with Euclidean exterior and quotient norms.

By the min-max principle $\mathfrak m(B|_E) \le \sigma_k(B)$. Then, in
view of Lemma~\ref{lemma:bar:B:bound},
\begin{equation}
\label{eq:relative:vs:absolute:gap}
\Delta_k(B,E) \le \frac{\sigma_k(B)}{\sigma_{k+1}(B)} = \Delta_k(B).   
\end{equation}

A basic observation is the following:
\begin{lemma}
\label{lemma:basic:relative:gap}
Suppose $\Delta_k(B) > 1$ and $E$ is not transverse to $\ker_{d-k}(B)$.  Then,
$\Delta_k(B,E) \le 1$.   
\end{lemma}

\begin{proof}
Let $K = \ker_{d-k}(B)$.  
Choose $v \in E \cap K$ with $\|v\|=1$. Then, since  $v  \in K$, $\|B
v  \|  \le  \sigma_{k+1}$. Hence, $\mathfrak{m}(B|_E)  \le
\sigma_{k+1}$. Now the estimate follows from
Lemma~\ref{lemma:bar:B:bound}. 
\end{proof}

The following is a quantitative version of
Lemma~\ref{lemma:basic:relative:gap}. 
\begin{lemma}
\label{lemma:relative:gap:to:angle}
Suppose $\Delta_k(B,E) > 1$. Then
\begin{displaymath}
\angle(E,  \ker_{d-k}(B)) \ge (\Delta_k(B,E)-1) \frac{\sigma_{k+1}(B)}{\|B\|}. 
\end{displaymath}
\end{lemma}

\begin{proof}
By \eqref{eq:relative:vs:absolute:gap}, $\Delta_k(B) \ge \Delta_k(B,E)
>  1$. Then, $\ker_{d-k}(B)$ is well-defined. Let $K  =
\ker_{d-k}(B)$, and note that $BK$ and $B(K^\perp)$ are orthogonal. 

Let $\theta = \angle(E,  K)$. Then, there
exists $v \in E\setminus \{0\}$
with $v = u + w$, $u \in K^\perp$, $w \in K$  and
$\|u\| =  \sin \theta \|v\|$.
Then,
\begin{multline*}
\|B v\|^2 = \|Bu\|^2 + \|B  w \|^2 \le \|B\|^2 \|u\|^2  +
\sigma_{k+1}(B)^2 \|w\|^2 \\
\le (\|B\|^2 \sin^2 \theta + \sigma_{k+1}(B)^2) \|v\|^2
\end{multline*}
Thus,
\begin{displaymath}
\mathfrak{m}(B|_E) \le \frac{\|Bv\|}{\|v\|} \le \|B\|\sin \theta + \sigma_{k+1}(B),    
\end{displaymath}
so
\begin{displaymath}
\sin \theta \ge \frac{\mathfrak{m}(B|_E) - \sigma_{k+1}(B)}{\|B\|}.
\end{displaymath}
Also, by Lemma~\ref{lemma:bar:B:bound}, 
$\|\bar{B}_E\| \ge \sigma_{k+1}(B)$, hence
\begin{displaymath}
\mathfrak{m}(B|_E) = \Delta_k(B,E) \|\bar{B}_E\| \ge
\Delta_k(B,E)\sigma_{k+1}(B).
\end{displaymath}
Thus,
\begin{displaymath}
\theta \ge \sin \theta \ge \frac{\Delta_k(B,E)\sigma_{k+1}(B) - \sigma_{k+1}(B)}{\|B\|}. 
\end{displaymath}
\end{proof}

\begin{lemma}[Relative Goldsheid--Margulis type one-step estimate]
\label{lemma:GM-step}
Consider
\[
	V \xrightarrow{P} W \xrightarrow{A} X
\]
linear isomorphisms between $d$-dimensional Euclidean spaces, and set
$Q:=AP$. Let $E \subset V$ be a $k$-plane.  Assume $\Delta_k(P,E) > 1$
and $\Delta_k(Q,E) >  1$.  
Then
\begin{displaymath}
\sin \angle(E,\ker_{d-k}(Q)) \ge \left(1 -
  \frac{\|A\|\|A^{-1}\|+1}{\Delta_k(Q,E)} \right) \sin \angle(E, \ker_{d-k}(P)).
\end{displaymath}
\end{lemma}

\begin{proof}
It follows from the hypotheses and
(\ref{eq:relative:vs:absolute:gap}) that
$P$ and $Q$ have strict $k$-gaps. Let
\begin{displaymath}
S_P:=\ker_{d-k}P,\qquad S_Q:=\ker_{d-k}Q.
\end{displaymath}
It follows from Lemma~\ref{lemma:basic:relative:gap} that $V=S_P\oplus E$. 
Take a non-zero vector
$v\in S_Q$ and decompose it as
$v=s+u$ with  $s\in S_P,u\in E$.
Since $v\in S_Q$, we have
\begin{displaymath}
\|Qv\|\le \sigma_{k+1}(Q) \|v\|
\end{displaymath}
Also,
\begin{displaymath}
\|Qs\|\leq\|A\|\|Ps\|\le \|A\|\sigma_{k+1}(P)  \|s\|. 
\end{displaymath}
Hence
\begin{displaymath}
\|Qu\|
=
\|Q(v-s)\|
\le
\|Qv\|+\|Qs\|
\le
\sigma_{k+1}(Q)\|v\|+\|A\|\sigma_{k+1}(P)\|s\|
\end{displaymath}
On the other hand, since $u \in E$, 
\begin{displaymath}
\|Qu\|\ge \mathfrak{m}(Q|_E) \|u\|.
\end{displaymath}
Combining the last two estimates gives
\begin{displaymath}
\|u\| \le \frac{\sigma_{k+1}(Q)\|v\| + \|A\|
  \sigma_{k+1}(P)\|s\|}{\mathfrak{m}(Q|_E)}
\end{displaymath}
By Lemma~\ref{lemma:bar:B:bound}, $\mathfrak{m}(Q|_E) \ge
\Delta_k(Q,E) \sigma_{k+1}(Q)$. Also $\sigma_{k+1}(P) \le \|A^{-1}\|
\sigma_{k+1}(Q)$. Therefore,
\begin{displaymath}
\|u\| \le \frac{\|v\|}{\Delta_k(Q,E)} + \frac{\|A\|
  \|A^{-1}\|\|s\|}{\Delta_k(Q,E)}.
\end{displaymath}
The triangle inequality gives
\begin{displaymath}
\|v\|\leq\|s\|+\|u\|
\leq\Delta_k(Q,E)^{-1}\|v\|+(1+\|A\|\|A^{-1}\|\Delta_k(Q,E)^{-1})\|s\|,
\end{displaymath}
so
\begin{displaymath}
\frac{\|s\|}{\|v\|}  \ge
\frac{\Delta_k(Q,E)-1}{\Delta_k(Q,E)+\|A\|\|A^{-1}\|} \ge \left( 1 -
  \frac{\|A\| \|A^{-1}\| + 1}{\Delta_k(Q,E)} \right). 
\end{displaymath}
Since $v-s\in E$,
\[
\frac{\dist(v,E)}{\|v\|}
=\frac{\dist(s,E)}{\|v\|}
\geq\sin \angle(E, S_P) \frac{\|s\|}{\|v\|}
\geq\left( 1 -
  \frac{\|A\| \|A^{-1}\| + 1}{\Delta_k(Q,E)} \right)\sin \angle(E, S_P).
\]
Taking the infimum over nonzero $v\in S_Q$ yields
\[
\sin\angle(E,S_Q)\geq
\left( 1 -
  \frac{\|A\| \|A^{-1}\| + 1}{\Delta_k(Q,E)} \right)\sin\angle(E,S_P).
\]
\end{proof}


\section{Auxiliary probability estimates}
\label{app:auxiliary:probability:estimates}

\subsection{An abstract Azuma--Margulis estimate}
\label{app:azuma:margulis}

Throughout this subsection we use the setup of
\autoref{ssec:abstract:good:event}.  In particular, $H_i$ is an adapted height
process satisfying \eqref{eq:abstract:margulis:height:drift} for all
$i \ge 0$. 

\begin{lemma}[Return tails]
\label{lemma:abstract:margulis:return:tails}
Fix $R<\infty$ sufficiently large, and put
\begin{displaymath}
  \tau_R:=\inf\{n\geq 0\st H_n < R\}.
\end{displaymath}
There is a constant $c_1>0$ such that, for every initial height $H_0$
and every $n\geq 0$,
\begin{equation}
\label{eq:abstract:margulis:return:tails}
\mathbb P(\tau_R>n)\leq H_0 e^{-c_1 n}.
\end{equation}
The same estimate holds conditionally for the shifted process after
any stopping time.
\end{lemma}

\begin{proof}
Choose $R$ sufficiently large so that $\theta:=c+\frac{b}{R}<1$.
On the event $H_i\ge R$, the Margulis drift gives
\begin{displaymath}
  \mathbb E(H_{i+1}\mid\mathcal F_i)
  \leq cH_i+b
  \leq \theta H_i.
\end{displaymath}
	
Set
$Y_n:=H_n\mathbf 1_{\{\tau_R>n\}}$.
Since $\{\tau_R> n\}$ is $\mathcal F_n$-measurable, the preceding
estimate implies
\begin{displaymath}
  \mathbb E(Y_{n+1})\leq \theta \mathbb E(Y_n)
\text{ and hence }
  \mathbb E(Y_n)\leq \theta^n H_0.
\end{displaymath}
On the event $\{\tau_R>n\}$ we have $H_n \ge R$, and therefore
\begin{displaymath}
  \mathbb P(\tau_R>n)
  \leq R^{-1}\mathbb E(Y_n)
  \leq R^{-1}H_0\theta^n.
\end{displaymath}
This proves \eqref{eq:abstract:margulis:return:tails}, with
$c_1=-\log\theta$.  The conditional statement
follows by applying the same argument to the shifted process.
\end{proof}

\begin{lemma}[Stretch exponential tails]
\label{lemma:subweibull:summation}
Let $0<\alpha\leq1$. Let $(\mathcal G_j)_{j\geq0}$ be a filtration, and
let $X_j$, $0\leq j<N$, be non-negative random variables
such that $X_j$ is $\cG_{j+1}$-measurable.  Suppose that there are constants
$C_0,c_0>0$ such that, for every $j$ and every $u\geq 0$,
\begin{displaymath}
\mathbb P(X_j>u\mid\mathcal G_j)\leq C_0 e^{-c_0u^\alpha}.
\end{displaymath}
Then there are constants $C,c,D_0>0$, depending only on
$C_0,c_0,\alpha$, such that, for every $x\geq D_0N$,
\begin{equation}
\label{eq:subweibull:summation}
\mathbb P\left(\sum_{j=0}^{N-1}X_j>x\right)
\leq
C e^{-cx^\alpha}.
\end{equation}
\end{lemma}

\begin{proof}
After changing constants, it is enough to consider $x$ large.  The
probability that $X_j$ is larger than $x$ for some $0\leq j<N$ is at
most
$NC_0e^{-c_0x^\alpha}$,
which is bounded by $Ce^{-cx^\alpha}$, since $x\geq D_0N$.
	
On the complementary event $X_j\leq x$ for all $j$, set
$\lambda:=\epsilon x^{\alpha-1}$,
where $0<\epsilon<c_0$.
Using the layer cake representation theorem, the
tail assumption gives
\begin{displaymath}
  \mathbb E\left(e^{\lambda X_j}\mathbf 1_{\{X_j\leq x\}}
    \,\middle|\,\mathcal G_j\right)
  \leq 1+C_0\lambda\int_0^{x}e^{\lambda u-c_0u^{\alpha}}du\leq 1+C\lambda
\end{displaymath}
with $C$ depending only on $C_0,c_0,\alpha$.  Iterating conditional
expectations gives
	\begin{displaymath}
		\mathbb E\left(
		\exp\left(\lambda\sum_{j=0}^{N-1}X_j\right)
		\mathbf 1_{\{X_j\leq x\text{ for all }j\}}
		\right)
		\leq
		e^{CN\lambda}.
	\end{displaymath}
	Since $x\geq D_0N$, choosing $D_0$ sufficiently large gives
		$CN\lambda\leq \frac{\lambda x}{2}$.
	Thus Markov's inequality gives
	\begin{displaymath}
		\mathbb P\left(
		\sum_{j=0}^{N-1}X_j>x,\ X_j\leq x\text{ for all }j
		\right)
		\leq
		e^{-\lambda x/2}
		=
		e^{-\epsilon x^\alpha/2}.
	\end{displaymath}
	Together with the large-jump estimate, this proves the lemma.
\end{proof}

\begin{lemma}[Height occupation estimate]
\label{lemma:abstract:height:occupation}
Fix $R<\infty$ sufficiently large and let $\beta>0$.  There are constants
$C<\infty$, $c_1>0$, and $D<\infty$, depending on $R$, $\beta$, and the
height constants, such that, if $H_0\leq R$, then, for every horizon
$N\geq 1$ and every $v\geq DN$,
\begin{equation}
\label{eq:abstract:height:occupation}
\mathbb P\left(
  \sum_{i=0}^{N-1}(1+\log H_i)^\beta>v
\right)
\leq
Ce^{-c_1 v^{1/(\beta+1)}}.
\end{equation}
The same estimate holds conditionally for the shifted process after
any stopping time.
\end{lemma}

\begin{proof}
Set $\tau_0 = 0$, and define for  $j \ge 0$,
$\tau_{j+1}:=\inf\{n>\tau_j:H_n<R\}$. Then, for $j \ge 0$,
\begin{equation}
\label{eq:Hj:conditional:estimate}
\mathbb E(H_{\tau_j+1}\mid\mathcal F_{\tau_j})\leq cR+b.  
\end{equation}
For an excursion
beginning at $\tau_j$, set
\begin{displaymath}
  X_j:=
  \sum_{i=\tau_j}^{\tau_{j+1}-1}(1+\log H_i)^\beta .
\end{displaymath}
We claim that
\begin{equation}
\label{eq:height:occupation:excursion:tail}
\mathbb P(X_j>u\mid\mathcal F_{\tau_j})
\leq
Ce^{-c_1u^{1/(\beta+1)}}.
\end{equation}
Put $r:=u^{1/(\beta+1)}$. If $X_j>u$, then either
the excursion has length larger than $r$, or
\[
1+\log H_{\tau_j+k}>r
\]
for some integer $0\leq k<r$. Conditionally on $\mathcal F_{\tau_j}$,
the first event has probability at most $Ce^{-c_1r}$ by
\autoref{lemma:abstract:margulis:return:tails}, applied at $\tau_j+1$,
followed by the tower property and \eqref{eq:Hj:conditional:estimate}.
For the second event, iterating the Margulis drift gives
\[
\mathbb E(H_{\tau_j+k}\mid\mathcal F_{\tau_j})\leq C
\]
for every integer $k\geq0$, and hence Markov's inequality gives
\[
\mathbb P(1+\log H_{\tau_j+k}>r\mid\mathcal F_{\tau_j})\leq Ce^{-r}.
\]
Taking a union bound over at most $r+1$ times proves
\eqref{eq:height:occupation:excursion:tail}, after adjusting the constants.

Up to time $N$, there are at most $N+1$ excursions.  Applying
\autoref{lemma:subweibull:summation} with
\begin{displaymath}
  \alpha=\frac{1}{\beta+1}
\end{displaymath}
gives \eqref{eq:abstract:height:occupation}, after increasing $D$ and
changing constants.  The conditional statement follows by applying the
same argument to the shifted process.
\end{proof}

\begin{lemma}[Azuma with random variance]
\label{lemma:azuma:random:variance}
Let $\xi_i$, $0\leq i<N$, be martingale differences with respect to
$\{\mathcal F_i\}_{i\geq 0}$; thus $\xi_i$ is $\mathcal F_{i+1}$-measurable and
\begin{displaymath}
  \mathbb E(\xi_i\mid\mathcal F_i)=0.
\end{displaymath}
Assume that $M_i$ is $\mathcal F_i$-measurable and that
\begin{displaymath}
  |\xi_i|\leq M_i.
\end{displaymath}
Then, for every $L>0$ and $v>0$,
\begin{equation}
\label{eq:azuma:random:variance}
\mathbb P\left(
  \exists\,0\leq n\leq N:
  \sum_{i=0}^{n-1}\xi_i\leq -L
  \text{ and }
  \sum_{i=0}^{n-1}M_i^2\leq v
\right)
\leq
\exp\left(-\frac{L^2}{2v}\right).
\end{equation}
\end{lemma}

\begin{proof}
For any $\lambda\in\mathbb{R}$ Hoeffding's lemma gives
\begin{displaymath}
  \mathbb E\left(
    \exp\left(
      -\lambda\xi_i-\frac{\lambda^2M_i^2}{2}
    \right)
    \,\middle|\,\mathcal F_i
  \right)
  \leq 1.
\end{displaymath}
Thus
\begin{displaymath}
  Z_n:=
  \exp\left(
    -\lambda\sum_{i=0}^{n-1}\xi_i
    -\frac{\lambda^2}{2}\sum_{i=0}^{n-1}M_i^2
  \right)
\end{displaymath}
is a non-negative supermartingale. Take $\lambda > 0$.  

Let $\tau$ be the first $0\leq n\leq N$ such that
$\sum_{i=0}^{n-1}\xi_i\leq -L$ and $\sum_{i=0}^{n-1}M_i^2\leq
v$.  Let $E$ be the event that such a time exists. (Define
$\tau = N$ if $E$ does not hold). 
Since $Z_n$ is a non-negative supermartingale and $\tau$ is bounded on $E$, Doob's Optional Stopping Theorem gives
\begin{displaymath}
  \mathbb E\left(Z_\tau\mathbf 1_E\right)\leq Z_0=1.
\end{displaymath}
On $E$, one has
\begin{displaymath}
  Z_\tau
  \geq
  \exp\left(\lambda L-\frac{\lambda^2v}{2}\right).
\end{displaymath}
Therefore
\begin{displaymath}
  \mathbb P(E)
  \leq
  \exp\left(-\lambda L+\frac{\lambda^2v}{2}\right).
\end{displaymath}
Taking $\lambda=L/v$ proves \eqref{eq:azuma:random:variance}.
\end{proof}

\begin{proposition}[Azuma--Margulis estimate]
\label{prop:azuma:margulis}
Fix $R$ sufficiently large.  Assume $H_0\leq R$.  Let $\psi_i$,
$i\geq 0$, be real-valued random variables such that $\psi_i$ is
$\mathcal F_{i+1}$-measurable.  Assume that there are constants $a>0$ and
$B<\infty$ such that, for every $i\geq 0$,
\begin{equation}
  \label{eq:azuma:margulis:drift}
  \mathbb E(\psi_i\mid\mathcal F_i)\geq a
\end{equation}
and
\begin{equation}
  \label{eq:azuma:margulis:size}
  |\psi_i|\leq B(1+\log H_i).
\end{equation}
Then there are constants $C<\infty$ and $c_1>0$, depending only on
$R,a,B$ and the constants in \eqref{eq:abstract:margulis:height:drift}, such
that, for every $L>0$,
\begin{equation}
\label{eq:azuma:margulis:general}
  \mathbb P\left(
    \exists\,n\geq 0:
    \sum_{i=0}^{n-1}\psi_i<\frac{a}{2}n-L
  \right)
  \leq
  C e^{-c_1\sqrt L}.
\end{equation}
In particular, for every $\rho>0$, after changing the constants,
\begin{equation}
\label{eq:azuma:margulis:linear:loss}
\mathbb P\left(
  \exists\,0\leq n\leq N:
  \sum_{i=0}^{n-1}\psi_i<\frac{a}{2}n-\rho N
\right)
\leq
C e^{-c_1\sqrt N}.
\end{equation}
The same estimates hold conditionally after any stopping time $T$, on the
event $H_T\leq R$, for the shifted future process.
\end{proposition}

\begin{proof}
Set
\begin{displaymath}
  m_i:=\mathbb E(\psi_i\mid\mathcal F_i),
  \qquad
  \xi_i:=\psi_i-m_i.
\end{displaymath}
Then $\xi_i$ is a martingale difference.  By
\eqref{eq:azuma:margulis:size},
\begin{displaymath}
  |\xi_i|
  \leq
  |\psi_i|+\mathbb E(|\psi_i|\mid\mathcal F_i)
  \leq
  2B(1+\log H_i).
\end{displaymath}
Put
\begin{displaymath}
  M_i:=2B(1+\log H_i).
\end{displaymath}
	
We first record a finite-horizon estimate.  Let $N\geq 1$ and $L>0$.  Choose
\begin{displaymath}
  v:=D(N+L)^{3/2},
\end{displaymath}
with $D$ sufficiently large.  By \autoref{lemma:abstract:height:occupation} with $\beta=2$, after changing constants,
\begin{equation}
\label{eq:azuma:margulis:variance:event}
\mathbb P\left(
  \sum_{i=0}^{N-1}M_i^2>v
\right)
\leq
Ce^{-c_1\sqrt{N+L}}.
\end{equation}
If, for some $0\leq n\leq N$,
\begin{displaymath}
  \sum_{i=0}^{n-1}\psi_i<\frac{a}{2}n-L,
\end{displaymath}
then, by \eqref{eq:azuma:margulis:drift},
\begin{displaymath}
  \sum_{i=0}^{n-1}\xi_i
  =
  \sum_{i=0}^{n-1}\psi_i-\sum_{i=0}^{n-1}m_i
  <
  -\frac{a}{2}n-L
  \leq
  -L.
\end{displaymath}
For $n\leq N$,  $\sum_{i=0}^{n-1}M_i^2\leq \sum_{i=0}^{N-1}M_i^2$, thus \autoref{lemma:azuma:random:variance}, together with
\eqref{eq:azuma:margulis:variance:event}, gives
\begin{equation}
\label{eq:azuma:margulis:finite:horizon}
\begin{aligned}
  &\mathbb P\left(
    \exists\,0\leq n\leq N:
    \sum_{i=0}^{n-1}\psi_i<\frac{a}{2}n-L
    \right) \\
  &\hspace{2cm}\leq
    Ce^{-c_1\sqrt{N+L}}
    +
    \exp\left(-\frac{L^2}{C(N+L)^{3/2}}\right).
\end{aligned}
\end{equation}
	
To extend this to the infinite-horizon estimate (\autoref{eq:azuma:margulis:general}), we partition the time scale into dyadic intervals \([2^j, 2^{j+1})\). Let $\mathcal E_j$ be the event that there exists
$n\in [2^j, 2^{j+1})$ such that
\begin{displaymath}
  \sum_{i=0}^{n-1}\psi_i<\frac{a}{2}n-L.
\end{displaymath}
On $\mathcal E_j$ one has
\begin{displaymath}
  \sum_{i=0}^{n-1}\xi_i
  <
  -L-\frac{a}{2}n
  \leq
  -L-\frac{a}{2}2^j
\end{displaymath}
for some $n<2^{j+1}$.  Applying \autoref{lemma:azuma:random:variance} on the
horizon $N=2^{j+1}$, and using \autoref{lemma:abstract:height:occupation} with
variance cutoff comparable to $(L+2^j)^{3/2}$, gives
\begin{displaymath}
  \mathbb P(\mathcal E_j)
  \leq
  C e^{-c_1\sqrt{L+2^j}}+
  \exp\left(-\frac{(L+a2^{j-1})^2}{C(2^{j}+L)^{3/2}}\right).
\end{displaymath}
Since
\begin{displaymath}
\frac{(L+a2^{j-1})^2}{C(L+2^j)^{3/2}}
\geq \frac{\min\{1,a/2\}^2}{C}\sqrt{L+2^j},
\end{displaymath}
the preceding estimate gives
\begin{displaymath}
\mathbb P(\mathcal E_j)\leq C e^{-c_1\sqrt{L+2^j}},
\end{displaymath}
after increasing $C$ and decreasing $c_1$.
The event at $n=0$ is empty since $L>0$.
Therefore
\begin{displaymath}
  \mathbb P\left(
    \exists\,n\geq 0:
    \sum_{i=0}^{n-1}\psi_i<\frac{a}{2}n-L
  \right)
  \leq
  \sum_{j=0}^{\infty}Ce^{-c_1\sqrt{L+2^j}}
  \leq
  C e^{-c_1\sqrt L},
\end{displaymath}
after changing $C$ and $c_1$. Taking $L=\rho N$ gives
\eqref{eq:azuma:margulis:linear:loss}, after changing the constants.  The
conditional statement follows from the conditional versions of the preceding
lemmas applied to the shifted process.
\end{proof}

\begin{lemma}[Weighted height-sum estimate]
\label{lemma:weighted:height:sum}
Assume the setup of \autoref{ssec:abstract:good:event}.  Fix
$R<\infty$, $\Lambda>0$, and $\beta>0$.  Then there exist constants
$C=C(R,\Lambda,\beta)<\infty$ and $c_1=c_1(R,\Lambda,\beta)>0$ such that the
following holds.
	
Let $m\geq 0$ and assume $H_m\leq R$, then for every $\chi>0$ and every
$p\geq 1$,
\begin{equation}
\label{eq:weighted:height:sum}
\mathbb P\left(
  \sum_{q=0}^{\infty} H_{m+q}^{\beta} e^{-\Lambda q}
  >
  C e^{\chi p}
  \,\middle|\,
  \mathcal F_m
\right)
\leq
C e^{-c_1\chi p}.
\end{equation}
The same estimate holds conditionally after any stopping time $m$, on
the event $H_m\leq R$.
\end{lemma}

\begin{proof}
By iterating \eqref{eq:abstract:margulis:height:drift}, on the event
$H_m\leq R$ we have
\begin{displaymath}
  \mathbb E(H_{m+q}\mid\mathcal F_m)
  \leq c^qH_m+b\frac{1-c^q}{1-c}
  \leq R+\frac{b}{1-c}
\end{displaymath}
for every $q\geq 0$.
	
For $q\geq 0$, Markov's inequality gives
\begin{displaymath}
  \mathbb P\left(
    H_{m+q}^{\beta}>e^{\chi p}e^{\Lambda q/2}
    \,\middle|\,
    \mathcal F_m
  \right)
  \leq
  C e^{-\chi p/\beta}e^{-\Lambda q/(2\beta)}.
\end{displaymath}
Hence, by summing over $q$,
\begin{displaymath}
  \mathbb P\left(
    \exists q\geq 0 \st
    H_{m+q}^{\beta}>e^{\chi p}e^{\Lambda q/2}
    \,\middle|\,
    \mathcal F_m
  \right)
  \leq
  C e^{-c_1\chi p}.
\end{displaymath}
	
On the complementary event, for every $q\geq 0$ one has
\begin{displaymath}
  H_{m+q}^{\beta}e^{-\Lambda q}
  \leq
  e^{\chi p}e^{-\Lambda q/2}.
\end{displaymath}
Therefore
\begin{displaymath}
  \sum_{q=0}^{\infty}H_{m+q}^{\beta}e^{-\Lambda q}
  \leq
  e^{\chi p}\sum_{q=0}^{\infty}e^{-\Lambda q/2}
  \leq
  C e^{\chi p}.
\end{displaymath}
This proves \eqref{eq:weighted:height:sum}.  The stopping-time version
follows by applying the same argument to the shifted process.
\end{proof}

\begin{corollary}[Subexponential height bound]
\label{cor:subexponential:height}
Fix $R\ge1$ and $\eta>0$. There is $K<\infty$, depending only on
$R,\eta$ and the height drift constants, such that, for every finite
integer-valued stopping time $T$ and every $A>0$, on $H_T\le R$,
\[
\mathbb P(\exists i\ge T:H_i>Ae^{\eta\sqrt i}\mid\mathcal F_T)
\le \frac{K}{A}e^{-\eta\sqrt T/2}.
\]
Also, for each $\eta>0$,
$\sup_{i\ge0}H_i e^{-\eta\sqrt i}<\infty$ almost surely.
\end{corollary}

\begin{proof}
Iterating \eqref{eq:abstract:margulis:height:drift} by the tower property,
starting at $T$, gives
$\mathbb E(H_{T+q}\mid\mathcal F_T)\le R+b/(1-c)$ on $H_T\le R$.
Markov's inequality and a union bound bound the probability above by
\[
\frac{R+b/(1-c)}{A}\sum_{q\ge0}e^{-\eta\sqrt{T+q}}
\le \frac{K}{A}e^{-\eta\sqrt T/2}.
\]
Taking $T=0$, letting $A\to\infty$, and then taking the union over
integer $R\ge1$ proves the final assertion.
\end{proof}


\section{Euclidean stable manifolds}
\label{app:euclidean:stable:manifold}

\begin{theorem}[Euclidean stable manifolds for a sequence of maps]
	\label{theorem:euclidean:stable:manifold:theorem}
	Fix a splitting $\mathbb R^n=\mathbb R^p\oplus\mathbb R^{n-p}$,
	write $B_j(r)$ for the Euclidean radius-$r$ ball in $\mathbb R^j$,
	and put $Q=B_p(1)\times B_{n-p}(1)$. Fix constants
	\[
	0<\alpha<\theta<1,\qquad 0<\kappa<1,\qquad 0<\ell<1.
	\]
	There is $\delta>0$, depending only on these constants and the
	dimensions, with the following property. Let $F_i$, $i\ge0$, be
	smooth maps from neighborhoods of $\overline Q$ to $\mathbb R^n$,
	with $F_i(0)=0$ and invertible block-diagonal derivatives
	\[
	L_i=D_0F_i=\operatorname{diag}(B_i,D_i),\qquad
	\|B_i\|\le\alpha,\qquad \|D_i^{-1}\|\|B_i\|\le\kappa.
	\]
	Assume that the normalized nonlinear parts
	\[
	R_i:=L_i^{-1}F_i-\Id
	\]
	satisfy
	\[
	\sup_i\|R_i\|_{C^2(Q)}\le\delta,\qquad
	\sup_i\|R_i\|_{C^r(Q)}\le M_r<\infty\quad(r\ge2).
	\]
	In particular $R_i(0)=D_0R_i=0$. No expansion of $D_i$ or uniform
	bound on $\|D_i\|$ or $\|D_i^{-1}\|$ separately is required.
	Then there is a unique family of graphs
	\[
	\Psi_i(s)=(s,w_i(s)),\qquad s\in B_p(1),
	\]
	with $w_i(0)=0$, $\operatorname{Lip}(w_i)\le\ell$, and
	$F_i\Psi_i(B_p(1))\subset\Psi_{i+1}(B_p(1))$.
	The maps $w_i$ are smooth, $D_0w_i=0$, and
	$\sup_i\|w_i\|_{C^r(B_p(1/2))}\le K_r$ for every $r$.
	More precisely, setting
	\[
	k_i(s)=s+R_i^s(s,w_i(s)),\qquad \beta_i=B_i k_i,
	\]
	one has $F_i\Psi_i=\Psi_{i+1}\beta_i$, $Dk_i$ uniformly close to
	the identity, and $\|D\beta_i\|\le\theta$. If
	$S_0=\Id$ and $S_m=\beta_{m-1}\cdots\beta_0$, then on $B_p(1/2)$
	\[
	|S_m(s)|\le\theta^m|s|,\qquad
	\|DS_m\|\le\theta^m,\qquad
	\|D^rS_m\|\le K_r\theta^m\quad(r\ge2).
	\]
	Consequently the iterates along the graphs are well-defined and
	contract exponentially, with
	\[
	\bigl\|D^r(F_{m-1}\cdots F_0\circ\Psi_0)\bigr\|
	\le K_r\theta^m\qquad(r\ge1).
	\]
	These constants depend only on the stated parameters and finitely
	many of the bounds $M_j$ for each fixed $r$.
	The theorem also controls Jacobians for changing volume densities.
	Let $v_i(s)\,ds$ be smooth positive densities in the graph
	parameters, defined on $B_p(1)$, and suppose that
	\[
	\|D\log v_i\|_{C^0(B_p(1/2))}\le H e^{\tau i},
	\qquad \tau\ge0,\qquad e^\tau\theta<1.
	\]
	Assume also that the $C^r$ norms of $\log(v_0/v_0(0))$ on
	$B_p(1/2)$ are bounded by $V_r<\infty$. Define the Jacobian from
	the density $v_0(s)\,ds$ to $v_m(s)\,ds$ by
	\[
	\mathcal J_m(s)=
	\frac{v_m(S_m(s))}{v_0(s)}\,|\det DS_m(s)|.
	\]
	Then $\mathcal J_m(s)/\mathcal J_m(0)$ converges uniformly on
	$B_p(1/2)$ to a positive smooth function $J_v(s)$, with uniform
	upper and lower bounds and bounds for all its derivatives.
	The bounds depend additionally on $H,\tau$ and the appropriate
	$V_r$. For Euclidean leaf volumes, take
	\[
	v_i(s)=\sqrt{\det\bigl(\Id_p+(Dw_i(s))^{\mathsf T}Dw_i(s)\bigr)}.
	\]
	The graph estimates give the density hypotheses with $\tau=0$,
	so this includes the usual Euclidean stable-Jacobian ratio.
\end{theorem}
\begin{proof}
	We follow Pugh--Shub--Wilkinson \cite{PughShubWilkinson1997}.
	The construction of invariant sections is in \S3 (p.~526), and
	the weighted norm and graph transform are in \S4 (pp.~528--529).
	We use the stable version of this transform.
	The higher smoothness of individual stable leaves is stated in
	\S6 (p.~545).
	Here we give the uniform estimates needed to apply this method to
	a sequence of maps with bounded normalized remainders: the linear
	blocks themselves need not be bounded. Thus the argument below
	adapts their construction; it does not invoke their transverse
	H\"older-regularity theorem. The Jacobian assertion is proved
	separately at the end.
	For a graph $w_{i+1}$ through zero with Lipschitz constant at most
	$\ell$, the backward graph transform is the solution $u=w_i(s)$ of
	\[
	u=D_i^{-1}w_{i+1}\bigl(B_i(s+R_i^s(s,u))\bigr)-R_i^u(s,u).
	\]
	For sufficiently small $\delta$, the right side maps the ball
	$|u|\le\ell|s|$ to itself and is a contraction in $u$.
	The same estimates preserve the Lipschitz bound $\ell$.
	Here the factor involving the next graph is controlled by
	$\|D_i^{-1}\|\|B_i\|\le\kappa$, rather than by
	$\|D_i^{-1}\|$ alone. The transform on graph families contracts
	the complete metric
	\[
	d(w,\widetilde w)=
	\sup_{i\ge0,\ 0<|s|<1}
	\frac{|w_i(s)-\widetilde w_i(s)|}{|s|}
	\]
	by a factor at most $\kappa(1+O(\delta))<1$.
	This is the stable counterpart of the fiber-contraction estimate
	\cite[\S4, equation~(8), p.~529]{PughShubWilkinson1997}:
	for the linear maps the contraction factor is
	$\|D_i^{-1}\|\|B_i\|$, and the normalized remainders change
	it by a uniform $O(\delta)$ term. Taking the supremum over the
	time index gives the invariant-section contraction from their
	\S3 directly on the space of graph families.
	Its unique fixed point is therefore the asserted invariant family.
	For its base maps,
	$\|D\beta_i\|\le\alpha(1+O(\delta))\le\theta$.
	To prove smoothness, iterate the transform starting with zero
	graphs. All these iterates are smooth and tangent to the stable
	coordinate plane at zero. Differentiate the implicit graph
	equation $r$ times and move the terms containing the highest
	derivative of the new graph to the left. The resulting linear
	operator is uniformly close to the identity. The coefficient
	of the highest derivative of the old graph is bounded by
	\[
	q_r\le\kappa_*\theta^{r-1}<1,\qquad
	\kappa_*:=\frac{1+\kappa}{2}.
	\]
	Indeed, that coefficient is bounded by
	$(1+O(\delta))\|D_i^{-1}\|\|B_i\|^r(1+O(\delta))^r$.
	One factor $\|B_i\|$ combines with $\|D_i^{-1}\|$ to give
	$\kappa_*<1$, while each of the remaining $r-1$ factors is at
	most $\theta$. A single choice of $\delta$ works for every $r$.
	The relevant higher-order inequality is therefore
	$\|D_i^{-1}\|\|B_i\|^r\le\kappa\alpha^{r-1}<1$.
	It concerns differentiation within an individual plaque; no
	higher-order bunching for transverse regularity is needed.
	All other terms contain only lower-order derivatives of the
	graphs and derivatives of $R_i$ of order at most $r$.
	Thus induction gives
	\[
	\|D^r w_i^{\mathrm{new}}\|
	\le q_r\|D^r w_{i+1}^{\mathrm{old}}\|+K_r.
	\]
	The iterated graphs have bounds of every order, uniformly in time
	and iteration number. The contraction and the bounds one order
	higher imply $C^r$ convergence on compact subballs for every $r$.
	This proves smoothness of the same fixed-point graphs, with
	uniform bounds on $B_p(1/2)$.
	The maps $\beta_i$ consequently have uniform bounds of every
	order. The chain rule gives the bounds on $S_m$: at order $r\ge2$,
	the induction hypothesis bounds every inhomogeneous term by
	$K_r\theta^{2m}$, because each contains at least two derivatives
	of $S_m$. Hence the highest derivative $X_m:=\|D^rS_m\|_{C^0(B_p(1/2))}$ satisfies $X_{m+1}\le\theta X_m+K_r\theta^{2m}$, giving
	$X_m\le K_r\theta^m$. Applying the chain rule once more to
	$F_{m-1}\cdots F_0\circ\Psi_0=\Psi_m\circ S_m$ proves the
	claimed derivative bounds along the plaques.
	For the Jacobian assertion, put
	\[
	h_i(s)=\log|\det Dk_i(s)|-\log|\det Dk_i(0)|.
	\]
	Since $Dk_i$ is uniformly close to the identity, the functions
	$h_i$ have uniform bounds of every order and vanish at zero.
	The constant factors $|\det B_i|$ cancel, and the chain rule gives
	\[
	\begin{aligned}
		\log\frac{\mathcal J_m(s)}{\mathcal J_m(0)}
		={}&\sum_{i=0}^{m-1}h_i(S_i(s))\\
		&+\log\frac{v_m(S_m(s))}{v_m(0)}
		-\log\frac{v_0(s)}{v_0(0)}.
	\end{aligned}
	\]
	The middle term is bounded in absolute value by
	$H(e^\tau\theta)^m|s|$ and tends uniformly to zero.
	For every $r\ge0$,
	\[
	\|h_i\circ S_i\|_{C^r(B_p(1/2))}\le K_r\theta^i.
	\]
	At order zero use $h_i(0)=0$; at positive orders every chain-rule
	term contains at least one derivative of $S_i$.
	It follows that
	$$\log J_v(s)=\sum_{i=0}^{\infty}h_i(S_i(s))-\log\frac{v_0(s)}{v_0(0)}$$
	is smooth, with bounds in every $C^r$ norm. Only uniform
	convergence of the terminal volume term was needed to identify
	the limit; no bounds on its higher derivatives are required.
	Exponentiating proves all the assertions about $J_v$.
\end{proof}


\bibliographystyle{sfilip_bibstyle}
\bibliography{ac_from_uniform_gap}

\end{document}